\documentclass[11pt,reqno]{amsart}
\usepackage{amsmath,amssymb,amsthm,mathtools,mathrsfs,esint}
\usepackage[margin=1.10in]{geometry}
\usepackage{hyperref}
\usepackage{xcolor}
\newif\ifreviewchanges
\reviewchangesfalse

\newif\ifauditchanges
\auditchangestrue

\newif\ifsectionchanges
\sectionchangestrue

\hypersetup{
	bookmarksnumbered=true,
	bookmarksopen=true,
	pdfauthor={Chao Zhang},
	pdftitle={Gradient Hölder regularity for singular fractional p-Laplace equations},
	pdfsubject={Singular fractional p-Laplacian and gradient regularity}
}

\numberwithin{equation}{section}

\newcommand{\R}{\mathbb R}
\newcommand{\Jp}{\mathcal J_p}
\newcommand{\dd}{\,d}
\newcommand{\Lip}{\operatorname{Lip}}
\newcommand{\osc}{\operatorname*{osc}}
\newcommand{\esssup}{\operatorname*{ess\,sup}}

\newtheorem{theorem}{Theorem}[section]
\newtheorem{proposition}[theorem]{Proposition}
\newtheorem{lemma}[theorem]{Lemma}
\newtheorem{corollary}[theorem]{Corollary}
\theoremstyle{remark}
\newtheorem{remark}[theorem]{Remark}

\makeatletter
\@namedef{subjclassname@2020}{%
	\textup{2020} Mathematics Subject Classification}
\makeatother

\title[Gradient regularity for singular fractional equations]
{Gradient H\"older regularity for singular fractional
	\(p\)-Laplace equations}

\author[C. Zhang]{Chao Zhang}
\address{Chao Zhang  \hfill\break School of Mathematics and Institute for Advanced Study in
	Mathematics, Harbin Institute of Technology, Harbin 150001, China}
\email{czhangmath@hit.edu.cn}


\begin{document}
	
	\begin{abstract}
		We prove interior $C^{1,\alpha}$ regularity for globally bounded weak
		solutions to the fractional $p$-Laplace equation $(-\Delta_p)^s u=0$ under the assumptions $n\ge2$,
		$1<p<2$, $0<s<1$, and $sp>p-1$, where $\alpha=\alpha(n,p,s)>0$. This resolves the open problem of gradient H\"{o}lder continuity for globally bounded solutions in the singular range $sp>p-1$.

		The proof is based on affine approximation and estimates independent
		of the normalized affine slope. When these slopes tend to infinity,
		we prove convergence to a solution of a linear nonlocal equation of
		order $sp-p+2$, together with convergence of the normalized energies
		to the corresponding quadratic energy. The regularity of the limiting equation yields an
		improved affine approximation. For bounded slopes, we prove that
		every globally Lipschitz weak solution of $(-\Delta_p)^s u=0$ in
		$\R^n$ is affine and use this result in a compactness argument.
		The two estimates are combined in an iteration that preserves the
		required energy and tail bounds and gives H\"older continuity of
		the gradient.
	\end{abstract}
	
	\maketitle
	
	\tableofcontents

	\section{Introduction and main result}
	\label{sec:introduction}
	
	Let $n\ge2$, $1<p<2$, and $0<s<1$.
	We consider weak solutions of the fractional
	$p$-Laplace equation
	\begin{equation}
		(-\Delta_p)^s u=0
		\qquad\text{in }\Omega,
		\label{eq:introduction-equation}
	\end{equation}
	where, up to a positive normalization constant,
	\[
	(-\Delta_p)^s u(x)
	=
	\operatorname{P.V.}\int_{\R^n}
	\frac{\Jp(u(x)-u(y))}
	{|x-y|^{n+sp}}\dd y,
	\qquad
	\Jp(t)=|t|^{p-2}t.
	\]
	We use the tail space
	\[
	L^{p-1}_{sp}(\R^n)
	:=
	\left\{
	v\in L^{p-1}_{\mathrm{loc}}(\R^n):
	\int_{\R^n}
	\frac{|v(y)|^{p-1}}
	{(1+|y|)^{n+sp}}\dd y<\infty
	\right\}.
	\]
	A function
	\[
	u\in W^{s,p}_{\mathrm{loc}}(\Omega)
	\cap L^{p-1}_{sp}(\R^n)
	\]
	is a weak solution of
	\eqref{eq:introduction-equation} if
	\begin{equation*}
		\iint_{\R^n\times\R^n}
		\frac{
			\Jp(u(x)-u(y))
			(\varphi(x)-\varphi(y))
		}{
			|x-y|^{n+sp}
		}\dd x\dd y=0
	\end{equation*}
	for every $\varphi\in C_c^\infty(\Omega)$.
	
	The regularity theory for fractional
	$p$-Laplace equations has developed considerably
	in recent years.
	Local boundedness, Harnack inequalities,
	and H\"older continuity have been established
	in \cite{DiCastroKuusiPalatucci2014,
		DiCastroKuusiPalatucci2016,
		IannizzottoMosconiSquassina2016,
		Lindgren2016}.
	Boundary regularity in the singular range was
	further studied in \cite{IannizzottoMosconi2024}.
	Higher H\"older regularity and differentiability
	properties have been obtained in
	\cite{BrascoLindgren2017,
		BrascoLindgrenSchikorra2018,
		GarainLindgren2024,
		DieningKimLeeNowak2025a,
		BogeleinDuzaarLiaoMolicaBisciServadei2025a,
		BogeleinDuzaarLiaoMolicaBisciServadei2025b}.
	Nonlinear potential estimates and related
	regularity results can be found in
	\cite{KuusiMingioneSire2015a,
		KuusiMingioneSire2015b,
		KuusiMingioneSire2018}.
	
	An important question is whether solutions
	have H\"older continuous gradients.
	Several recent results have brought the
	regularity theory closer to this question.
	Calder\'on--Zygmund estimates and higher
	integrability results have been developed
	in \cite{DieningNowak2025,
		BogeleinDuzaarLiaoMoring2025,
		BogeleinDuzaarLiaoMoring2026}.
	In particular, local Lipschitz continuity
	has been established for the fractional
	$p$-Laplace equation in
	\cite{BiswasTopp2025,BiswasSen2025}.
	Related gradient estimates for linear nonlocal
	equations, nonlinear equations with different
	structural conditions, and mixed local--nonlocal
	problems were obtained in
	\cite{KuusiNowakSire2022,
		DieningKimLeeNowak2025b,
		DeFilippisMingione2024}.
	These results provide important information
	about the gradient, but they do not establish
	interior H\"older continuity of the gradient
	for the singular fractional $p$-Laplacian.
	
	The question of interior $C^{1,\alpha}$
	regularity for fractional $p$-harmonic functions
	has remained open for a long time
	\cite{BogeleinDuzaarLiaoMolicaBisciServadei2024,
		KuusiMingioneSire2018}.
	It is also discussed in
	\cite[Section~1.5]{DieningNowak2025}.
	A recent result
	\cite{GiovagnoliJesusSilvestre2025}
	establishes interior $C^{1,\alpha}$ regularity
	when
	\[
	2\le p<\frac{2}{1-s}.
	\]
	The singular range $1<p<2$ is not covered
	by that result and is explicitly left open
	in \cite[Section~1.2]{GiovagnoliJesusSilvestre2025}.
	Although Lipschitz continuity and several
	differentiability estimates are available,
	the continuity of the gradient in this range
	requires further arguments.
	
	The singular case presents a particular
	difficulty.
	Indeed,
	\[
	\Jp'(t)=(p-1)|t|^{p-2}\qquad(t\ne0),
	\]
	and hence
	\[
	\Jp'(t)\longrightarrow\infty
	\qquad\text{as }t\to0.
	\]
	Thus the coefficient obtained by linearizing
	the equation can become unbounded.
	Moreover, the equation is nonlinear, so
	subtracting an affine function does not
	preserve the original equation for the
	remaining error.
	Consequently, estimates for the difference
	between a solution and an affine function
	must take account of the affine slope.
	Obtaining such estimates with constants
	independent of that slope is a main
	difficulty in the singular range.
	
	In this paper, we establish interior
	H\"older continuity of the gradient for
	globally bounded weak solutions of
	\eqref{eq:introduction-equation}
	under the condition $sp>p-1$.
	For a bounded measurable function $v$
	and a measurable set $E$ of positive measure,
	we write
	\[
	\osc_E v
	:=
	\operatorname*{ess\,sup}_E v
	-
	\operatorname*{ess\,inf}_E v.
	\]
	
	Our main result is the following.
	
	\begin{theorem}
		\label{thm:main}
		Let $n\ge2$, $1<p<2$, and $0<s<1$ satisfy
		\begin{equation}
			sp>p-1.
			\label{eq:main-structural-range}
		\end{equation}
		There exist constants $\alpha_0\in(0,1)$
		and $C\ge1$, depending only on $n,p,s$,
		such that, for every $x_0\in\R^n$
		and $R>0$, any weak solution
		$u\in L^\infty(\R^n)$ of
		\[
		(-\Delta_p)^s u=0
		\qquad\text{in }B_{4R}(x_0)
		\]
		belongs to $C^{1,\alpha_0}(B_R(x_0))$
		and satisfies
		\begin{equation}
			\|\nabla u\|_{L^\infty(B_R(x_0))}
			+
			R^{\alpha_0}
			[\nabla u]_{C^{\alpha_0}(B_R(x_0))}
			\le
			\frac{C}{R}\osc_{\R^n}u.
			\label{eq:fixed-order-gradient-holder-estimate}
		\end{equation}
	\end{theorem}
	
	The condition $sp>p-1$ plays an important
	role in our argument.
	It guarantees that
	\[
	\int_1^\infty r^{p-2-sp}\dd r<\infty,
	\]
	which allows us to control the nonlocal flux
	associated with affine functions at infinity.
	It also gives
	\[
	\tau_s:=sp-p+2\in(1,2).
	\]
	The number $\tau_s$ is the order of the
	linear operator obtained in the large-slope
	limit.
	For this limiting operator, the regularity
	estimates allow every exponent
	\[
	0<\alpha<\tau_s-1=sp-p+1.
	\]
	The exponent $\alpha_0$ in
	Theorem~\ref{thm:main} may be smaller,
	because the argument for bounded slopes
	uses a qualitative compactness result.

	\begin{remark}
		\label{rem:bounded-exterior-data}
		The global boundedness assumption provides
		the initial estimates for the exterior
		affine excess used in our iteration.
		Although the condition
		$u\in L^{p-1}_{sp}(\R^n)$
		is sufficient for the local Lipschitz
		estimate and the initial relative flux,
		it does not directly control the unweighted
		$L^{q_{\rm ext}}$ averages used to define the exterior
		affine excess
		in Section~\ref{sec:common-recursion}.
		The argument also applies locally when
		the initial quantities in
		\eqref{eq:common-tail-state} are finite.
		Uniform estimates on compact subsets hold
		when these quantities are locally uniformly
		bounded.
		An extension to general data in
		$L^{p-1}_{sp}(\R^n)$ would require
		a different estimate for the exterior
		affine excess.
	\end{remark}
	
	We next describe the main steps of the proof.
	Our argument is based on approximation by
	affine functions.
	For
	\[
	U(X)=Q\cdot X+v(X),
	\]
	the error $v$ satisfies an equation involving
	the shifted flux (where $\delta v(X,Y):=v(X)-v(Y)$)
	\[
	\begin{aligned}
		\mathcal D_Qv(X,Y)
		&=
		\Jp\bigl(Q\cdot(X-Y)+\delta v(X,Y)\bigr)-
		\Jp\bigl(Q\cdot(X-Y)\bigr).
	\end{aligned}
	\]
	We measure the corresponding energy using
	the Bregman remainder
	\[
	D_\Phi(a;t)
	=
	\Phi(a+t)-\Phi(a)-\Jp(a)t,
	\qquad
	\Phi(t)=\frac{|t|^p}{p}.
	\]
	A key step is to establish a local
	boundedness estimate for the error $v$
	with a constant independent of $Q$
	(Lemma~\ref{lem:shift-uniform-local-boundedness}).
	This estimate allows us to study the equation
	when the normalized affine slopes become
	arbitrarily large.
	
	We distinguish two cases according to the
	size of the normalized affine slopes.
	
	When the slopes tend to infinity,
	the normalized nonlinear energies converge
	to a quadratic energy associated with
	a linear stable operator of order
	\[
	\tau_s=sp-p+2>1.
	\]
	Its kernel has the form
	\[
	\frac{|e\cdot\omega|^{p-2}}
	{|X-Y|^{n+\tau_s}},
	\qquad
	\omega=\frac{X-Y}{|X-Y|},
	\quad e\in\mathbb S^{n-1}.
	\]
	Although the angular coefficient
	$|e\cdot\omega|^{p-2}$ is unbounded,
	\[
	|e\cdot\omega|^{p-2}\ge1\quad\text{a.e.},\qquad
	\int_{\mathbb S^{n-1}}|e\cdot\omega|^{p-2}\dd\omega<\infty.
	\]
	The limiting operator therefore satisfies
	the spectral ellipticity conditions needed
	for the available linear regularity estimates.
	
	The main compactness result is established in
	Proposition~\ref{prop:fixed-order-infinite-slope-compactness}.
	In particular, we prove convergence of both
	the solutions and the normalized energies.
	The energy convergence requires a comparison
	with suitable approximating functions,
	using the gluing estimate in
	Lemma~\ref{lem:uniform-shifted-bregman-gluing}.
	The regularity of the limiting equation
	then yields an improved affine approximation,
	which is transferred to the nonlinear equation
	in Proposition~\ref{prop:large-slope-affine-improvement}.
	This improvement can be iterated while
	controlling the nonlocal tails
	(Lemma~\ref{lem:iterable-large-slope-state}).
	
	The second case concerns bounded normalized
	slopes.
	Here the limiting equation remains nonlinear,
	and the argument based on the linear stable
	operator cannot be applied directly.
	We prove that every globally Lipschitz weak
	solution of $(-\Delta_p)^s u=0$ in $\R^n$
	is affine
	(Theorem~\ref{thm:fixed-order-centered-closure}).
	This result is used to establish a Liouville
	theorem under a Campanato growth bound
	(Proposition~\ref{prop:affine-campanato-liouville-main}),
	which gives the compactness improvement
	of affine approximation in
	Proposition~\ref{prop:bounded-slope-affine-invariant-block-main}.
	
	The estimates for bounded and unbounded
	slopes are combined with suitable bounds
	for the exterior terms.
	At each scale, we subtract an affine function
	and normalize the remaining error.
	The resulting estimates allow us to improve
	the affine approximation while preserving
	the required energy and tail bounds.
	This leads to a common iteration
	(Proposition~\ref{prop:unified-variable-block-recursion}),
	which gives
	\[
	\|u-\ell_{x_0}\|_{L^\infty(B_r(x_0))}
	\le Cr^{1+\beta}
	\]
	for a suitable affine function $\ell_{x_0}$
	and some $\beta>0$.
	The corresponding estimates for the affine
	slopes yield the gradient H\"older estimate
	in Theorem~\ref{thm:main}.
	
	The paper is organized as follows.
	Section~\ref{sec:setting-contact} introduces
	the basic notation and proves preliminary results
	for globally Lipschitz solutions.
	Section~\ref{sec:large-slope-and-blowdown}
	develops estimates independent of the affine
	slope, proves compactness for large slopes,
	and studies affine limits at infinity.
	Section~\ref{sec:dilation-liouville}
	proves the Liouville theorem for globally
	Lipschitz solutions.
	Finally, Section~\ref{sec:common-recursion}
	combines the estimates for bounded and large
	slopes and proves Theorem~\ref{thm:main}.
	
	\section{Preliminaries and affine solutions}
	\label{sec:setting-contact}
	
	In this section, we introduce the notation and recall some basic
	facts. We then prove that a globally Lipschitz weak solution
	of \eqref{eq:equation} is affine if its Lipschitz constant is
	attained at two distinct points.
	
	\subsection{Notation and basic facts}
	\label{subsec:notation-basic-facts}
	Throughout the paper, $C>0$ denotes a constant that may change from line
	to line.  We write \(A\lesssim B\) if \(A\le CB\), and
	\(A\asymp B\) if \(cB\le A\le CB\). For a measurable set \(E\) of positive measure, we write
	\(\fint_E g:=|E|^{-1}\int_E g\). We also set
	\(\delta w(x,y):=w(x)-w(y)\).

	We consider
	\begin{equation}\label{eq:equation}
		\operatorname{P.V.}\int_{\mathbb R^n}
		\frac{\Jp(U(x)-U(y))}
		{|x-y|^{n+sp}}\,dy=0
		\qquad\text{in }\mathbb R^n.
	\end{equation}
	A weak solution of \eqref{eq:equation} is a function
	\[
	U\in W^{s,p}_{\mathrm{loc}}(\R^n)
	\cap L^{p-1}_{sp}(\R^n)
	\]
	satisfying the weak formulation in
	Section~\ref{sec:introduction} with $\Omega=\R^n$.
	For globally Lipschitz $U$, the tail is finite since
	\[
	\int_{|y|>1}\frac{|U(y)-U(0)|^{p-1}}{|y|^{n+sp}}\dd y
	\le C[U]_{C^{0,1}(\R^n)}^{p-1}
	\int_1^\infty r^{p-2-sp}\dd r<\infty.
	\]
	
	\subsection{A criterion for affine solutions}
	\label{subsec:pointwise-contact}
	
	We first justify the pointwise equation at points where the
	solution has an affine expansion with a quadratic remainder
	and nonzero slope. 
	\begin{lemma}
		\label{lem:pointwise-contact}
		Under assumption \eqref{eq:main-structural-range}, let
		$U$ be a globally Lipschitz  weak solution of
		\eqref{eq:equation}.  If, at $x_0$,
		\[
		U(x_0+z)=U(x_0)+q\cdot z+O(|z|^2),
		\qquad q\ne0,
		\]
		then the principal value in \eqref{eq:equation} exists at $x_0$ and is
		zero.
	\end{lemma}
	
	\begin{proof} 
		Take $
		R(z):=U(x_0+z)-U(x_0)-q\cdot z$. Choose $K>0$ and $0<r_0\le1$ such that
		\[
		|R(z)|\le K|z|^2
		\qquad (|z|<r_0).
		\]
		For $1<p<2$, we use the inequality
		\[
		|\Jp(a+b)-\Jp(a)|
		\le C|a|^{p-2}|b|,
		\qquad a\ne0.
		\]
		Indeed, if $|b|\le |a|/2$, this follows
		from the mean value theorem. If
		$|b|>|a|/2$, the $(p-1)$-H\"older
		continuity of $\Jp$ gives
		\[
		|\Jp(a+b)-\Jp(a)|
		\le C|b|^{p-1}
		\le C|a|^{p-2}|b|.
		\]
		
		Since $q\ne0$ and $p>1$, the function $|q\cdot\omega|^{p-2}$ is integrable over $\mathbb{S}^{n-1}$.		Thus, for any $\eta$ satisfying
		$|\eta(z)|\le K|z|^2$ in $B_{r_0}$ and $0<\rho\le r_0$, we have 
		\begin{equation*}
			\int_{B_\rho}
			\frac{
				|\Jp(-q\cdot z+\eta(z))
				-\Jp(-q\cdot z)|
			}{|z|^{n+sp}}\dd z\le
			CK\int_0^\rho r^{p-sp-1}\dd r
			\int_{\mathbb S^{n-1}}
			|q\cdot\omega|^{p-2}\dd\omega\le C\rho^{p-sp}.
		\end{equation*}
		Taking $\eta=-R$, we obtain
		\begin{equation}\label{eq-lemma22}
			\int_{B_{r_0}}
			\frac{
				|\Jp(-q\cdot z-R(z))
				-\Jp(-q\cdot z)|
			}{|z|^{n+sp}}\dd z<\infty.
		\end{equation}
		Moreover, since $\Jp$ is odd,
		\[
		\int_{\varepsilon<|z|<r_0}
		\frac{\Jp(-q\cdot z)}
		{|z|^{n+sp}}\dd z=0
		\qquad (0<\varepsilon<r_0).
		\]
		Let $L$ be the global Lipschitz
		constant of $U$. Then
		\[
		\begin{aligned}
			\int_{\R^n\setminus B_{r_0}}
			\frac{
				|\Jp(U(x_0)-U(x_0+z))|
			}{|z|^{n+sp}}\dd z\le
			CL^{p-1}\int_{r_0}^\infty
			r^{p-sp-2}\dd r<\infty,
		\end{aligned}
		\]
		where we used $sp>p-1$.
		Consequently, the principal value
		\[
		I:=
		\operatorname{P.V.}
		\int_{\R^n}
		\frac{
			\Jp(U(x_0)-U(y))
		}{|x_0-y|^{n+sp}}\dd y
		\]
		exists.
		
		Since $U$ is globally Lipschitz
		and $sp>p-1$, both $U$ and $-U$
		are continuous, locally bounded,
		and belong to
		$W_{\mathrm{loc}}^{s,p}(\R^n)
		\cap L_{sp}^{p-1}(\R^n)$.
		As both are weak supersolutions,
		\cite[Theorem~1.2]{KorvenpaaKuusiLindgren2019}
		shows that they are viscosity
		supersolutions on every bounded ball.
		Hence $U$ is a viscosity solution.
		
		Define
		\[
		P_\pm(x_0+z)
		:=U(x_0)+q\cdot z\pm K|z|^2.
		\]
		Then
		\[
		P_-\le U\le P_+
		\quad\text{in }B_{r_0}(x_0),
		\qquad
		P_\pm(x_0)=U(x_0).
		\]
		Moreover, $\nabla P_\pm(x_0)=q\ne0$,
		so these test functions satisfy
		\cite[Definition~3(iii)(a)]{KorvenpaaKuusiLindgren2019}.
		
		For $0<\rho<r_0$, let
		\[
		U_\rho^\pm(x):=
		\begin{cases}
			P_\pm(x),&x\in B_\rho(x_0),\\
			U(x),&x\notin B_\rho(x_0),
		\end{cases}
		\]
		and take
		\[
		I_\rho^\pm:=
		\operatorname{P.V.}
		\int_{\R^n}
		\frac{
			\Jp(U_\rho^\pm(x_0)-U_\rho^\pm(y))
		}{|x_0-y|^{n+sp}}\dd y.
		\]
		The viscosity inequalities yield
		\[
		I_\rho^+\le0\le I_\rho^-.
		\]
		
		Since $U_\rho^\pm=U$ outside
		$B_\rho(x_0)$ and at $x_0$,
		we have
		\[
		\begin{aligned}
			|I_\rho^\pm-I|
			&\le
			\int_{B_\rho}
			\frac{
				|\Jp(-q\cdot z\mp K|z|^2)
				-\Jp(-q\cdot z)|
			}{|z|^{n+sp}}\dd z+
			\int_{B_\rho}
			\frac{
				|\Jp(-q\cdot z-R(z))
				-\Jp(-q\cdot z)|
			}{|z|^{n+sp}}\dd z\\
			&\le C\rho^{p-sp}
			\xrightarrow[\rho\to0]{}0.
		\end{aligned}
		\]
		Therefore,
		\[
		I=\lim_{\rho\to0}I_\rho^+
		\le0\le
		\lim_{\rho\to0}I_\rho^-=I.
		\]
		Hence $I=0$.
	\end{proof}

	We next use Lemma~\ref{lem:pointwise-contact} to compare $U$
	with a translated solution. Under the expansion assumptions
	below, we show that $U(x+h)-U(x)$ is constant if it attains
	its supremum or infimum.

	\begin{lemma}
		\label{lem:attained-translation}
		Under assumption \eqref{eq:main-structural-range},
		let $U$ be a globally Lipschitz weak solution of
		\eqref{eq:equation}, and let
		$h\in\mathbb R^n\setminus\{0\}$.
		Suppose that
		\[
		M_h:=\sup_{x\in\mathbb R^n}
		\bigl(U(x+h)-U(x)\bigr)
		\]
		is attained at $x_0$.
		Assume that there exists $q\in\mathbb R^n\setminus\{0\}$
		such that
		\[
		U(x_i+z)=U(x_i)+q\cdot z+O(|z|^2)
		\qquad (z\to0),
		\]
		for both $x_i=x_0$ and $x_i=x_0+h$.
		Then
		\[
		U(x+h)-U(x)=M_h
		\qquad\text{for every }x\in\mathbb R^n.
		\]
		The analogous conclusion holds for an attained
		translation minimum.
	\end{lemma}
	
	\begin{proof}
		The function $V(x):=U(x+h)-M_h$ is a globally Lipschitz weak solution
		of \eqref{eq:equation}, with
		\[
		V\le U\quad\text{in }\R^n,\qquad V(x_0)=U(x_0).
		\]
		The assumed expansions at $x_0$ and $x_0+h$
		give
		\[
		\begin{aligned}
			U(x_0+z)
			=U(x_0)+q\cdot z+O(|z|^2)\quad \text{and}\quad 
			V(x_0+z)
			=V(x_0)+q\cdot z+O(|z|^2).
		\end{aligned}
		\]
		Since $q\ne0$, Lemma~\ref{lem:pointwise-contact}
		applies to both $U$ and $V$ at $x_0$.
		
		For $W\in\{U,V\}$, we take
		\[
		R_W(z):=W(x_0+z)-W(x_0)-q\cdot z.
		\]
		There exist $C>0$ and $r_0>0$ such that
		\[
		|R_U(z)|+|R_V(z)|\le C|z|^2
		\qquad (|z|<r_0).
		\]
		Since both expansions have slope $q$,  it follows from \eqref{eq-lemma22} that 
		\[
		\begin{aligned}
			&\int_{B_{r_0}}
			\frac{
				|\Jp(V(x_0)-V(x_0+z))
				-\Jp(U(x_0)-U(x_0+z))|
			}{
				|z|^{n+sp}
			}\dd z\\
			&\quad\le
			\sum_{W\in\{U,V\}}
			\int_{B_{r_0}}
			\frac{
				|\Jp(-q\cdot z-R_W(z))
				-\Jp(-q\cdot z)|
			}{
				|z|^{n+sp}
			}\dd z
			<\infty.
		\end{aligned}
		\]
		The exterior integral is bounded by
		$C\int_{r_0}^\infty r^{p-2-sp}\dd r<\infty$.
		By Lemma~\ref{lem:pointwise-contact}, both
		principal values at $x_0$ are zero.
		Since their difference is absolutely integrable,
		we obtain
		\[
		0=
		\int_{\R^n}
		\frac{
			\Jp(V(x_0)-V(y))
			-\Jp(U(x_0)-U(y))
		}{
			|x_0-y|^{n+sp}
		}\dd y.
		\]
		Since $V\le U$ and $V(x_0)=U(x_0)$,
		the integrand is nonnegative.
		It must therefore vanish almost everywhere.
		By the strict monotonicity of $\Jp$,
		we obtain $V=U$ almost everywhere in $\R^n$.
		Continuity then gives $V=U$ everywhere.
		Consequently,
		\[
		U(x+h)-U(x)=M_h
		\qquad (x\in\R^n).
		\]
		Apply the same argument to $-U$ for the minimum.
	\end{proof}

	We now apply Lemma~\ref{lem:attained-translation} when the
	Lipschitz bound is attained at two distinct points.
	For a nonconstant solution, this equality gives affine
	expansions with quadratic remainders and a common nonzero
	slope at interior points of the segment joining them.

	\begin{lemma}
		\label{thm:extremal-secant-rigidity} 
		Under assumption \eqref{eq:main-structural-range}, let $U$ be a globally Lipschitz weak solution of
		\eqref{eq:equation}, and let $L:=[U]_{C^{0,1}(\R^n)}$.
		If there exist distinct points $a,b\in\R^n$ such that
		\[
		|U(b)-U(a)|=L|b-a|,
		\]
		then $U$ is affine in $\R^n$.
	\end{lemma}
	\begin{proof}
		If $L=0$, then $U$ is constant.
		Replacing $U$ by $-U$ if necessary,
		we may assume that $U(b)-U(a)=L|b-a|$.
		Take $
		d:=|b-a|$ and $e:=\frac{b-a}{d}$.
		For every $t\in[0,d]$, the Lipschitz bound gives
		\[
		\begin{aligned}
			Ld
			&=U(b)-U(a)=U(a+te)-U(a)+U(b)-U(a+te)\le Lt+L(d-t)=Ld.
		\end{aligned}
		\]
		Hence, we obtain
		\begin{equation}
			\label{eq:linearity-on-extremal-segment}
			U(a+te)=U(a)+Lt
			\qquad (0\le t\le d).
		\end{equation}
		Fix $t\in(0,d)$. By the Lipschitz bound,
		\[
		\begin{aligned}
			U(b)-L|(d-t)e-z|
			\le U(a+te+z)\le U(a)+L|te+z|.
		\end{aligned}
		\]
		Since
		\[	|te+z|=t+e\cdot z+O(|z|^2)\quad \text{and}\quad 
		|(d-t)e-z|=(d-t)-e\cdot z+O(|z|^2).
		\]
		Taylor's formula and
		\eqref{eq:linearity-on-extremal-segment} give
		\begin{equation}
			\label{eq:automatic-quadratic-expansion}
			U(a+te+z)
			=
			U(a+te)+Le\cdot z+O(|z|^2).
		\end{equation}
		
		Now fix $\lambda\in(0,d)$ and choose
		$t\in(0,d-\lambda)$. Set $
		x_0:=a+te$ and $ h:=\lambda e$.
		By the Lipschitz bound and
		\eqref{eq:linearity-on-extremal-segment},
		\[
		M_h:=\sup_{x\in\R^n}
		\bigl(U(x+h)-U(x)\bigr)
		=L\lambda
		=U(x_0+h)-U(x_0).
		\]
		Since $0<t<t+\lambda<d$,
		\eqref{eq:automatic-quadratic-expansion}
		shows that $U$ and
		$V(x):=U(x+h)-M_h$ have the same
		nonzero slope $Le$ at $x_0$.
		Lemma~\ref{lem:attained-translation} therefore gives
		\[
		U(x+\lambda e)-U(x)=L\lambda
		\qquad (x\in\R^n).
		\]
		We next extend this identity to every
		$t\in\R$.
		For $t>0$, choose $N\in\mathbb{N}$ such that $t/N<d$.
		Then
		\[
		\begin{aligned}
			U(x+te)-U(x)
			&=\sum_{j=0}^{N-1}
			\left[
			U\left(x+\frac{(j+1)t}{N}e\right)
			-
			U\left(x+\frac{jt}{N}e\right)
			\right]=N L\frac{t}{N}=Lt.
		\end{aligned}
		\]
		For $t<0$, apply the same identity with
		$-t>0$ at the point $x+te$.
		Thus
		\begin{equation}
			\label{eq:global-extremal-direction}
			U(x+te)=U(x)+Lt
			\qquad (x\in\R^n,\ t\in\R).
		\end{equation}
		
		Now let $z\in\R^n$ satisfy $z\cdot e=0$
		and take $c_z:=U(z)-U(0)$. By \eqref{eq:global-extremal-direction}
		and the Lipschitz bound,
		\[
		\begin{aligned}
			|c_z+Lt|=|U(z+te)-U(0)|\le L|z+te|
			=L\sqrt{|z|^2+t^2}
		\end{aligned}
		\]
		for every $t\in\R$.  Thus, we obtain
		\[
		c_z^2+2Ltc_z\le L^2|z|^2
		\qquad(t\in\R).
		\]
		Since $L>0$ and this inequality holds
		for every $t\in\R$, we must have $c_z=0$.
		Hence  $
		U(z)=U(0)$ whenever $z\cdot e=0$.

		Finally, for any $x\in\R^n$,  we take $
		x=z+te$, $
		t=e\cdot x$ and $
		z=x-(e\cdot x)e$. Then $z\cdot e=0$. By
		\eqref{eq:global-extremal-direction},  we have
		\[
		\begin{aligned}
			U(x)
			=U(z+te)=U(z)+Lt=U(0)+Le\cdot x.
		\end{aligned}
		\]
		Thus $U$ is affine in $\R^n$.
	\end{proof}
	
	\section{Estimates for large slopes and affine limits}
	\label{sec:large-slope-and-blowdown}
	
	In this section, we establish local boundedness and energy estimates
	for the shifted equation, with constants independent of $Q$.
	We then prove compactness as $|Q|\to\infty$ and use the limiting
	linear equation to improve the affine approximation for large slopes.
	By iteration, we obtain pointwise expansions for globally Lipschitz
	solutions at differentiability points with nonzero gradient.
	We also show that a globally Lipschitz solution with an affine
	limit at infinity is affine.
	
	\subsection{The shifted equation and energy estimates}
	\label{subsec:shifted-bregman-orlicz}
	
	Let
	\[
	\Phi(t):=\frac{|t|^p}{p},\qquad
	D_\Phi(a;b):=\Phi(a+b)-\Phi(a)-\Jp(a)b,
	\qquad \lambda_Q:=1+|Q|.
	\]
	For $Q\in\R^n$ and a function $v$, define
	\[
	\mathcal D_Qv(X,Y)
	:=
	\Jp\bigl(Q\cdot(X-Y)+\delta v(X,Y)\bigr)
	-\Jp\bigl(Q\cdot(X-Y)\bigr).
	\]
	Let $
	v\in W^{s,p}_{\mathrm{loc}}(D)
	\cap L^{p-1}_{sp}(\R^n)$.
	We say that $v$ solves the equation shifted
	by $Q$ in $D$ if
	\[
	\iint_{\R^n\times\R^n}
	\frac{
		\mathcal D_Qv(X,Y)
		(\varphi(X)-\varphi(Y))
	}{
		|X-Y|^{n+sp}
	}\dd X\!\dd Y=0
	\]
	for every $\varphi\in C_c^\infty(D)$. Equivalently, $U(X)=Q\cdot X+v(X)$ is a weak solution of
	$(-\Delta_p)^s U=0$ in $D$, since
	\[
	\iint_{\R^n\times\R^n}
	\frac{
		\Jp(Q\cdot(X-Y))
		(\varphi(X)-\varphi(Y))
	}{
		|X-Y|^{n+sp}
	}\dd X\!\dd Y=0.
	\] We define the normalized energy
	and the exterior quantity by
	\begin{equation*}
		\mathcal E_{s,Q}(v;B)
		:=(1-s)\lambda_Q^{2-p}
		\iint_{B\times B}
		\frac{D_\Phi(Q\cdot(X-Y);\delta v(X,Y))}
		{|X-Y|^{n+sp}}\dd X\!\dd Y,
	\end{equation*}
	and 
	\begin{align} \label{eq:large-slope-pointwise-tail}
		\mathcal T^{\rm pt}_{s,Q}(v)
		:=\max\left\{
		{\operatorname*{ess\,sup}}_{X\in B_{5/4}}
		\int_{\R^n\setminus B_{4/3}}
		\frac{\lambda_Q^{2-p}|\mathcal D_Qv(X,Y)|}
		{|X-Y|^{n+sp}}\dd Y,
		\right.\notag\\[-2mm]
		\hspace{31mm}\left.
		{\operatorname*{ess\,sup}}_{X\in B_{3/2}}
		\int_{\R^n\setminus B_2}
		\frac{\lambda_Q^{2-p}|\mathcal D_Qv(X,Y)|}
		{|X-Y|^{n+sp}}\dd Y\right\}.
	\end{align}
	The two terms are used in
	Proposition~\ref{prop:fixed-order-infinite-slope-compactness}
	and Lemma~\ref{lem:shift-uniform-local-boundedness},
	respectively.
	If \(m(X)=c+b\cdot X\) is affine and \(0<\rho\le1\), define
	\begin{align} \label{eq:block-zero-modular}
		\mathfrak Z(v,m;\rho)
		:=\fint_{B_{2\rho}}
		\left|\frac{v-m}{\rho}\right|^p\dd X,
	\end{align}
	and
	\begin{equation*}
		\mathfrak B_{s,Q}(v,m;\rho)
		:=(1-s)\rho^{sp-p-n}
		\iint_{B_{2\rho}\times B_{2\rho}}
		\frac{D_\Phi\bigl((Q+b)\cdot(X-Y);
			\delta(v-m)(X,Y)\bigr)}
		{|X-Y|^{n+sp}}\dd X\!\dd Y.\notag
	\end{equation*}
	We let \(\mathfrak a_{s,Q}(v,m;\rho)\) be the infimum of all \(a>0\)
	such that
	\begin{equation}
		\mathfrak Z(v,m;\rho)\le a^p,\qquad
		\mathfrak B_{s,Q}(v,m;\rho)
		\le (|Q+b|+a)^{p-2}a^2.
		\label{eq:block-intrinsic-amplitude}
	\end{equation}
	The energy bound has the two regimes
	\[
	(|Q+b|+a)^{p-2}a^2\asymp
	\begin{cases}
		|Q+b|^{p-2}a^2,&0<a\le |Q+b|,\\
		a^p,&a\ge |Q+b|.
	\end{cases}
	\]
	We also define $\Psi_p(t):=\frac{t^2}{(1+t)^{2-p}}$ for $t\ge0$.
	Direct differentiation gives
	\begin{equation}
		\Psi_p(t)\asymp\min\{t^2,t^p\},\qquad
		p\le\frac{t\Psi_p'(t)}{\Psi_p(t)}\le2\quad(t>0).
		\label{eq:orlicz-two-indices}
	\end{equation}
	For the estimates below, define
	\[
	\mathcal H_{Q,Z}(t):=\lambda_Q^{2-p}
	\bigl[D_\Phi(Q\cdot Z;t)+D_\Phi(Q\cdot Z;-t)\bigr].
	\]
	
	\begin{lemma}
		\label{lem:angularly-averaged-shifted-modular}
		There exist constants $c,C>0$, depending only
		on $n,p$, such that, for every $Q\in\R^n$,
		$|Z|\le4$, and $t,t_1,t_2\in\R$,
		\begin{align}
			c\Psi_p(|t|)
			&\le \mathcal H_{Q,Z}(t),
			\label{eq:shifted-modular-lower}\\
			\mathcal H_{Q,Z}(t_1+t_2)
			&\le C\bigl[
			\mathcal H_{Q,Z}(t_1)
			+\mathcal H_{Q,Z}(t_2)
			\bigr].
			\label{eq:shifted-modular-delta-two}
		\end{align}
		If $Q=qe$, where $q\ge0$ and
		$e\in\mathbb S^{n-1}$, then, for $0<r\le4$,
		\begin{equation}
			r^{-p}\int_{\mathbb S^{n-1}}
			\mathcal H_{Q,r\omega}(rt)\dd\omega
			\le C\bigl(|t|^p+|t|^2\bigr).
			\label{eq:shifted-modular-angular-upper}
		\end{equation}
		The same estimates hold if each
		$D_\Phi(a;\pm t)$ in the definition of
		$\mathcal H_{Q,Z}(t)$ is replaced by
		$[\Jp(a\pm t)-\Jp(a)](\pm t)$,
		where $a=Q\cdot Z$.
	\end{lemma}
	
	\begin{proof}
		Since $p-2>-1$, the function $\Jp$ is
		locally absolutely continuous.
		The fundamental theorem of calculus
		and Fubini's theorem give
		\[
		D_\Phi(a;t)
		=
		(p-1)t^2
		\int_0^1(1-\theta)
		|a+\theta t|^{p-2}\dd\theta.
		\]
		Applying this formula to $t$ and $-t$,
		we obtain
		\[
		\begin{aligned}
			D_\Phi(a;t)+D_\Phi(a;-t)
			=(p-1)t^2\int_0^1(1-\theta)
			\bigl[
			|a+\theta t|^{p-2}
			+|a-\theta t|^{p-2}
			\bigr]\dd\theta.
		\end{aligned}
		\]
		
		We claim that
		\begin{equation}
			D_\Phi(a;t)+D_\Phi(a;-t)
			\asymp
			(|a|+|t|)^{p-2}|t|^2.
			\label{eq:symmetric-shifted-density}
		\end{equation}
		The right-hand side is understood to be
		zero when $a=t=0$. For $t\ne0$, the lower bound
		follows from $p-2<0$ and $|a\pm\theta t|\le |a|+|t|$. For the upper bound, if $|a|\ge2|t|$, then
		$|a\pm\theta t|\ge |a|/2$, so
		\[
		\int_0^1 |a\pm\theta t|^{p-2}\dd\theta
		\le C|a|^{p-2}.
		\]
		If $|a|<2|t|$, a change of variables gives
		\[
		\int_0^1 |a\pm\theta t|^{p-2}\dd\theta
		\le |t|^{p-2}\int_{-3}^3 |z|^{p-2}\dd z
		\le C|t|^{p-2}.
		\]
		In both cases, the integral is bounded by
		$C(|a|+|t|)^{p-2}$.
		This proves \eqref{eq:symmetric-shifted-density}.
		
		Let $a=Q\cdot Z$. By
		\eqref{eq:symmetric-shifted-density},
		\[
		\mathcal H_{Q,Z}(t)
		\ge
		c\lambda_Q^{2-p}
		(|a|+|t|)^{p-2}|t|^2.
		\]
		Since $|Z|\le4$,
		\[
		|a|+|t|
		\le4\lambda_Q(1+|t|).
		\]
		Using $p-2<0$, we obtain
		\[
		\begin{aligned}
			\mathcal H_{Q,Z}(t)
			\ge
			c\lambda_Q^{2-p}
			\lambda_Q^{p-2}
			(1+|t|)^{p-2}|t|^2=c\Psi_p(|t|).
		\end{aligned}
		\]
		This proves \eqref{eq:shifted-modular-lower}.
		
		Next, \eqref{eq:symmetric-shifted-density}
		gives $
		\mathcal H_{Q,Z}(2t)
		\le C\mathcal H_{Q,Z}(t)$.	Moreover,
		\[
		\mathcal H_{Q,Z}(t)
		=
		\lambda_Q^{2-p}
		\bigl[
		\Phi(a+t)+\Phi(a-t)-2\Phi(a)
		\bigr],
		\]
		so $\mathcal H_{Q,Z}$ is convex.
		Hence
		\[
		\begin{aligned}
			\mathcal H_{Q,Z}(t_1+t_2)
			\le
			\frac12\mathcal H_{Q,Z}(2t_1)
			+\frac12\mathcal H_{Q,Z}(2t_2)\le
			C\bigl[
			\mathcal H_{Q,Z}(t_1)
			+\mathcal H_{Q,Z}(t_2)
			\bigr].
		\end{aligned}
		\]
		This proves \eqref{eq:shifted-modular-delta-two}.
		
		We now prove
		\eqref{eq:shifted-modular-angular-upper}.
		Suppose first that $0\le q\le1$.
		The $(p-1)$-H\"older continuity
		of $\Jp$ gives $ D_\Phi(a;t)\le C|t|^p$.
		Since $\lambda_Q=1+q\le2$,
		\[
		\mathcal H_{Q,r\omega}(rt)
		\le Cr^p|t|^p.
		\]
		Consequently,
		\[
		r^{-p}\int_{\mathbb S^{n-1}}
		\mathcal H_{Q,r\omega}(rt)\dd\omega
		\le C|t|^p.
		\]
		
		Suppose now that $q>1$.
		By \eqref{eq:symmetric-shifted-density}
		and $p-2<0$, for almost every $\omega$,
		\[
		\begin{aligned}
			\mathcal H_{Q,r\omega}(rt)
			&\le
			C\lambda_Q^{2-p}
			\bigl(qr|e\cdot\omega|+r|t|\bigr)^{p-2}
			r^2|t|^2\le
			Cr^p\lambda_Q^{2-p}q^{p-2}
			|e\cdot\omega|^{p-2}|t|^2\le
			Cr^p|e\cdot\omega|^{p-2}|t|^2,
		\end{aligned}
		\]
		where we used
		\[
		\lambda_Q^{2-p}q^{p-2}
		=
		\left(\frac{1+q}{q}\right)^{2-p}
		\le2^{2-p}.
		\]
		Therefore,
		\[
		r^{-p}\int_{\mathbb S^{n-1}}
		\mathcal H_{Q,r\omega}(rt)\dd\omega
		\le C|t|^2.
		\]
		Combining the two cases proves
		\eqref{eq:shifted-modular-angular-upper}.
		
		Finally, for $t\ne0$, the fundamental
		theorem of calculus gives
		\[
		\begin{aligned}
			[\Jp(a+t)-\Jp(a)]t
			+[\Jp(a-t)-\Jp(a)](-t)&=(p-1)t^2
			\int_{-1}^1|a+\theta t|^{p-2}\dd\theta\\
			&\asymp
			(|a|+|t|)^{p-2}|t|^2.
		\end{aligned}
		\]
		The same identity is trivial when $t=0$.
		By \eqref{eq:symmetric-shifted-density},
		this expression is comparable to
		$D_\Phi(a;t)+D_\Phi(a;-t)$.
		The remaining estimates follow from
		this comparison.
	\end{proof}
	
	\begin{lemma}
		\label{lem:large-slope-scalar-angular}
		Assume $1<p<2$ and $n\ge2$.
		The following estimates hold.
		
		\begin{enumerate}
			\item
			For all $a,b\in\R$,
			\begin{align}
				D_\Phi(a;b)
				&\asymp (|a|+|b|)^{p-2}|b|^2,
				\label{eq:large-slope-bregman-equivalence}\\
				|\Jp(a+b)-\Jp(a)|
				&\le C(|a|+|b|)^{p-2}|b|.
				\label{eq:large-slope-flux-bound}
			\end{align}
			The expressions on the right are defined
			to be zero when $a=b=0$.
			
			\item
			Let $Q=qe$, where $q\ge1$ and
			$e\in\mathbb S^{n-1}$, and let
			$Z=r\omega$, where $r>0$ and
			$\omega\in\mathbb S^{n-1}$.
			Then, for almost every $\omega$,
			\begin{equation}
				\lambda_Q^{2-p}
				|\Jp(Q\cdot Z+t)-\Jp(Q\cdot Z)|
				\le
				Cr^{p-2}|e\cdot\omega|^{p-2}|t|.
				\label{eq:large-slope-angular-flux}
			\end{equation}
			Moreover, for $m>1$,
			\begin{equation}
				\sup_{e\in\mathbb S^{n-1}}
				\int_{\mathbb S^{n-1}}
				|e\cdot\omega|^{(p-2)m}\dd\omega
				<\infty
				\quad\Longleftrightarrow\quad
				m<\frac1{2-p}.
				\label{eq:large-slope-angular-integrability}
			\end{equation}
			When finite, the integral depends only
			on $n,p,m$.
		\end{enumerate}
	\end{lemma}
	\begin{proof}
		The estimates are trivial when $b=0$.
		For $b\ne0$, set $S=|a|+|b|$.
		The fundamental theorem of calculus gives
		\[
		D_\Phi(a;b)
		=(p-1)b^2\int_0^1(1-\theta)
		|a+\theta b|^{p-2}\dd\theta
		\]
		and
		\[
		|\Jp(a+b)-\Jp(a)|
		=(p-1)|b|\int_0^1
		|a+\theta b|^{p-2}\dd\theta.
		\]
		Since $p-2<0$ and $|a+\theta b|\le S$,
		\[
		\int_0^1(1-\theta)
		|a+\theta b|^{p-2}\dd\theta
		\ge \frac12 S^{p-2}.
		\]
		On the other hand,
		\[
		\int_0^1|a+\theta b|^{p-2}\dd\theta
		\le CS^{p-2}.
		\]
		Indeed, if $|a|\ge2|b|$, then
		$|a+\theta b|\ge S/3$.
		If $|a|<2|b|$, the estimate follows
		by the change of variables
		$z=a/b+\theta$, since $p-2>-1$.
		These bounds prove
		\eqref{eq:large-slope-bregman-equivalence}
		and \eqref{eq:large-slope-flux-bound}.
		
		We next consider
		\eqref{eq:large-slope-angular-flux}.
		For $t\ne0$ and $e\cdot\omega\ne0$,
		\eqref{eq:large-slope-flux-bound} gives
		\[
		\begin{aligned}
			\lambda_Q^{2-p}
			|\Jp(Q\cdot Z+t)-\Jp(Q\cdot Z)|
			&\le
			C\lambda_Q^{2-p}
			\bigl(qr|e\cdot\omega|+|t|\bigr)^{p-2}|t|\\
			&\le
			C\lambda_Q^{2-p}q^{p-2}
			r^{p-2}|e\cdot\omega|^{p-2}|t|\le
			Cr^{p-2}|e\cdot\omega|^{p-2}|t|,
		\end{aligned}
		\]
		where the last inequality follows from
		$q\ge1$ and
		\[
		\lambda_Q^{2-p}q^{p-2}
		=
		\left(\frac{1+q}{q}\right)^{2-p}
		\le2^{2-p}.
		\]
		The estimate is also immediate when $t=0$.

		Finally, rotational invariance allows us
		to take $e=e_1$. By spherical coordinates,
		\[
		\int_{\mathbb S^{n-1}}
		|e\cdot\omega|^{(p-2)m}\dd\omega=
		|\mathbb S^{n-2}|
		\int_{-1}^1
		|z|^{(p-2)m}
		(1-z^2)^{\frac{n-3}{2}}\dd z.
		\]
		Since $n\ge2$, the factor
		$(1-z^2)^{(n-3)/2}$ is integrable
		near $z=\pm1$ and is bounded above
		and below by positive constants
		near $z=0$.
		Consequently, the integral is finite if and only if $(p-2)m>-1$.
		For $m>1$, this is equivalent to $1<m<\frac1{2-p}$, which proves
		\eqref{eq:large-slope-angular-integrability}.
	\end{proof}

	\begin{lemma}
		\label{lem:shift-uniform-local-boundedness}
		Let $n\ge2$, $1<p<2$, $0<s<1$, and $sp>p-1$.
		Suppose that $U(x)=Q\cdot x+v(x)$ is a weak solution of
		$(-\Delta_p)^sU=0$ in $B_2$, and that
		\[
		\fint_{B_2}|v|^p\dd x\le1,
		\qquad
		\mathcal E_{s,Q}(v;B_2)\le1,
		\qquad
		\mathcal T^{\rm pt}_{s,Q}(v)\le L
		\]
		for some $L\ge0$. Then
		\begin{equation}
			\|v\|_{L^\infty(B_{4/3})}
			\le C(n,p,s)\bigl(1+L^{1/(p-1)}\bigr).
			\label{eq:shift-uniform-local-boundedness}
		\end{equation}
		The constant is independent of $Q$.
		Only the second term in the definition of
		$\mathcal T^{\rm pt}_{s,Q}(v)$ is needed.
	\end{lemma}
	
	\begin{proof}
		Let $\kappa:=2-p$ and $\sigma:=sp+\kappa$, so that
		$0<\kappa<1$ and $1<\sigma<2$. Let $A:=1+L^{1/(p-1)}$, $g:=v/A$, and $P:=Q/A$.
		Then $P\cdot x+g=U/A$ solves the same equation.
		Since $\lambda_P\le\lambda_Q$, homogeneity gives $	\fint_{B_2}|g|^p\dd x\le1$,
		\[
		\mathcal E_{s,P}(g;B_2)
		=A^{-p}\left(\frac{\lambda_P}{\lambda_Q}\right)^\kappa
		\mathcal E_{s,Q}(v;B_2)\le1\quad \text{and}\quad 
		\mathcal T^{\rm pt}_{s,P}(g)
		=A^{1-p}\left(\frac{\lambda_P}{\lambda_Q}\right)^\kappa
		\mathcal T^{\rm pt}_{s,Q}(v)\le A^{1-p}L\le1.
		\]
		It therefore suffices to prove the estimate under the additional
		assumption $\mathcal T^{\rm pt}_{s,Q}(v)\le1$.
		
		Write $q:=|Q|$ and $\lambda:=1+q$.
		The $L^p$ bound and the second term in
		\eqref{eq:large-slope-pointwise-tail} allow us to choose a Lebesgue
		point $x_*\in B_1$ such that $|v(x_*)|\le C$ and
		\[
		\int_{\R^n\setminus B_2}
		\frac{|\mathcal D_Qv(x_*,y)|}{|x_*-y|^{n+sp}}\dd y
		\le\lambda^{-\kappa}.
		\]
		Since
		\[
		|U(x_*)-U(y)|^{p-1}
		\le |\mathcal D_Qv(x_*,y)|
		+|Q\cdot(x_*-y)|^{p-1},
		\]
		and $sp>p-1$, we obtain
		\[
		\int_{\R^n\setminus B_2}
		\frac{|U(y)-U(x_*)|^{p-1}}{|y|^{n+sp}}\dd y
		\le C\bigl(\lambda^{-\kappa}+q^{p-1}\bigr)
		\le C\lambda^{p-1},\]
		and\[
		\|U-U(x_*)\|_{L^p(B_2)}\le C\lambda.
		\]
		Here $|x_*-y|\asymp|y|$ for $y\notin B_2$.

		By the local boundedness estimate
		\cite[Theorem~1.1]{DiCastroKuusiPalatucci2016},
		in the local setting of
		\cite[Theorem~4]{KorvenpaaKuusiPalatucci2016},
		the local $L^p$ and tail bounds above give
		\[
		\|U-U(x_*)\|_{L^\infty(B_{15/8})}
		\le C\lambda.
		\]
		The local tails are controlled by splitting
		the integrals over $B_2$ and $\R^n\setminus B_2$.
		
		By the H\"older continuity result
		\cite[Theorem~1.2]{DiCastroKuusiPalatucci2016}
		(see also \cite[Theorem~5]{KorvenpaaKuusiPalatucci2016}),
		$U$ has a continuous representative.
		Thus \cite[Proposition~1.5]{BiswasTopp2025}
		implies that $U$ is also a viscosity solution.
		Since $sp>p-1$,
		\cite[Theorem~2.1]{BiswasTopp2025}
		and a finite covering give
		\[
		\Lip(U;B_{7/4})\le C\lambda.
		\]
		Since $U=Q\cdot x+v$ and $|v(x_*)|\le C$,
		we also have
		\[
		\|v\|_{L^\infty(B_{7/4})}\le C\lambda.
		\]
		Consequently,
		\begin{equation}
			\Lip(U;B_{7/4})
			+\|v\|_{L^\infty(B_{7/4})}
			\le C\lambda.
			\label{eq:h1-coarse-lipschitz}
		\end{equation}
		If $q\le1$, the proof is complete.
		Henceforth assume $q>1$ and set $e:=Q/q$.

		For $t\ge0$, define
		\[
		\varphi_q(t):=\lambda^\kappa\frac{t^2}{(q+t)^\kappa},
		\qquad V_q(t):=\sqrt{\varphi_q(t)}.
		\]
		Direct differentiation shows that $\varphi_q$ is increasing
		and convex. Moreover,
		\[
		\varphi_q(2t)\le4\varphi_q(t),
		\qquad \varphi_q(1)=1.
		\]
		By \eqref{eq:large-slope-bregman-equivalence} and $|x-y|\le4$,
		\begin{equation}
			\lambda^\kappa
			D_\Phi(Q\cdot(x-y);\delta v(x,y))
			\ge c\varphi_q(|\delta v(x,y)|)
			\qquad(x,y\in B_2).
			\label{eq:h1-full-modular-coercivity}
		\end{equation}
		Let $\bar v:=\fint_{B_2}v\dd x$, so that $|\bar v|\le1$.
		Jensen's inequality, the doubling estimate, and the energy bound
		yield
		\begin{equation}
			\begin{aligned}
				\int_{B_2}\varphi_q(|v|)\dd x
				&\le C+C\int_{B_2}\varphi_q(|v-\bar v|)\dd x\\
				&\le C+\frac{C}{|B_2|}
				\iint_{B_2\times B_2}\varphi_q(|\delta v(x,y)|)\dd x\dd y
				\le C.
			\end{aligned}
			\label{eq:h1-full-modular-bound}
		\end{equation}
		By \eqref{eq:h1-coarse-lipschitz}, $q+|v|\le C\lambda$
		in $B_{7/4}$, and hence $|v|^2\le C\varphi_q(|v|)$ there.
		Consequently,
		\begin{equation}
			\|v\|_{L^2(B_{7/4})}\le C.
			\label{eq:h1-local-quadratic-state}
		\end{equation}
		
		We define $F_Q(x,y):=\lambda^\kappa\mathcal D_Qv(x,y)$.
		For distinct $x,y\in B_{7/4}$, define the symmetric kernel
		\[
		K_Q(x,y):=
		\begin{cases}
			\displaystyle
			(1-s)\frac{F_Q(x,y)}
			{\delta v(x,y)|x-y|^{n+sp}},
			&\delta v(x,y)\ne0,\\[2mm]
			\displaystyle
			(1-s)|x-y|^{-n-\sigma},
			&\delta v(x,y)=0.
		\end{cases}
		\]
		Then
		\[
		K_Q(x,y)\delta v(x,y)
		=(1-s)\frac{F_Q(x,y)}{|x-y|^{n+sp}}.
		\]
		By \eqref{eq:h1-coarse-lipschitz},
		\[
		|Q\cdot(x-y)|+|\delta v(x,y)|
		\le C\lambda|x-y|.
		\]
		When $\delta v(x,y)\ne0$, the fundamental theorem of calculus gives
		\[
		\lambda^\kappa
		\frac{\mathcal D_Qv(x,y)}{\delta v(x,y)}
		=(p-1)\lambda^\kappa\int_0^1
		|Q\cdot(x-y)+\theta\delta v(x,y)|^{-\kappa}\dd\theta\ge c|x-y|^{-\kappa}.
		\]
		Together with \eqref{eq:large-slope-angular-flux} and the definition
		on $\{\delta v=0\}$, this proves
		\begin{equation}
			\frac{c}{|x-y|^{n+\sigma}}
			\le K_Q(x,y)
			\le\frac{C|e\cdot\omega|^{-\kappa}}{|x-y|^{n+\sigma}},
			\qquad \omega:=\frac{x-y}{|x-y|}.
			\label{eq:h1-secant-kernel}
		\end{equation}
		The upper bound is understood as $+\infty$ when $e\cdot\omega=0$.

		We first estimate $F_Q(x,y)$ for $x,y\in B_2$.
		Take $
		t:=\delta v(x,y)$ and $ d:=|e\cdot(x-y)|$. Since $d\le|x-y|\le4$, we have
		\[
		qd+|t|\ge\frac d4(q+|t|).
		\]
		As $\lambda=1+q\le2q$,
		\eqref{eq:large-slope-flux-bound} gives,
		for almost every $(x,y)\in B_2\times B_2$,
		\begin{equation}	\label{eq:h1-flux-square-root}
			\begin{aligned}
				|F_Q(x,y)|
				\le C\lambda^\kappa
				\frac{|t|}{(qd+|t|)^\kappa}\le Cd^{-\kappa}
				\left(\frac{\lambda}{q+|t|}\right)^{\kappa/2}
				V_q(|t|)\le Cd^{-\kappa}V_q(|t|).
			\end{aligned}
		\end{equation}
		
		Define
		\[
		G(y):=V_q(v_+(y))\mathbf1_{B_2}(y).
		\]
		By \eqref{eq:h1-full-modular-bound},
		$\|G\|_{L^2(\R^n)}\le C$.
		For $\tau\in\R$, set
		\[
		b(\tau):=
		\left(
		\int_{e^\perp}|G(\tau e+z)|^2\dd z
		\right)^{1/2},
		\]
		and define
		\[
		k_\kappa(\tau):=
		|\tau|^{-\kappa}\mathbf1_{\{|\tau|\le4\}},
		\qquad h:=k_\kappa*b.
		\]
		Since $\kappa<1$, we have
		$k_\kappa\in L^1(\R)$.
		Fubini's theorem and Young's inequality give
		\begin{equation}
			\begin{aligned}
				\|h\|_{L^2(\R)}
				\le\|k_\kappa\|_{L^1(\R)}
				\|b\|_{L^2(\R)}
				=\|k_\kappa\|_{L^1(\R)}
				\|G\|_{L^2(\R^n)}
				\le C.
			\end{aligned}
			\label{eq:h1-mixed-source-norm}
		\end{equation}

		Fix $4/3<S<R\le3/2$ and $k\ge0$.
		For $x\in B_S\cap\{v>k\}$, we estimate
		the integral over $B_2\setminus B_R$.
		Since $\Jp$ is increasing,
		\[
		F_Q(x,y)<0
		\quad\Longrightarrow\quad
		v(y)>v(x)>k.
		\]
		Hence
		$|\delta v(x,y)|\le v_+(y)$ whenever
		$F_Q(x,y)<0$.
		Using \eqref{eq:h1-flux-square-root},
		the monotonicity of $V_q$, and
		$|x-y|\ge R-S$ for $y\in B_2\setminus B_R$,
		we obtain
		\[
		\begin{aligned}
			\int_{B_2\setminus B_R}
			\frac{(-F_Q(x,y))_+}
			{|x-y|^{n+sp}}\dd y\le C(R-S)^{-n-sp}
			\int_{B_2}|e\cdot(x-y)|^{-\kappa}
			G(y)\dd y.
		\end{aligned}
		\]
		Writing $y=\tau e+z$, where
		$\tau\in\R$ and $z\in e^\perp$,
		and applying Cauchy--Schwarz on
		the sections perpendicular to $e$,
		we obtain
		\begin{equation}
			\int_{B_2\setminus B_R}
			\frac{(-F_Q(x,y))_+}
			{|x-y|^{n+sp}}\dd y
			\le C(R-S)^{-n-sp}h(e\cdot x).
			\label{eq:h1-annulus-negative-flux}
		\end{equation}
		
		Since $\sigma>1$, Sobolev embedding
		in the direction $e$ gives
		\begin{equation}
			\sup_{\tau\in\R}
			\|\phi(\tau e+\cdot)\|_{L^2(e^\perp)}
			\le C\|\phi\|_{H^{\sigma/2}(\R^n)}
			\qquad
			\bigl(\phi\in H^{\sigma/2}(\R^n)\bigr).
			\label{eq:h1-directional-trace}
		\end{equation}
		Indeed, this follows from the Fourier transform
		in the direction $e$ and Cauchy--Schwarz, since
		\[
		\int_\R
		(1+|\xi|^2)^{-\sigma/2}\dd\xi<\infty.
		\]
		
		If $\phi=0$ almost everywhere outside
		a measurable set $E$ of finite measure,
		Fubini's theorem and Cauchy--Schwarz give
		\[
		\begin{aligned}
			\int_{\R^n}h(e\cdot x)|\phi(x)|\dd x\le
			\|h\|_{L^2(\R)}|E|^{1/2}
			\sup_{\tau\in\R}
			\|\phi(\tau e+\cdot)\|_{L^2(e^\perp)}.
		\end{aligned}
		\]
		Thus, by \eqref{eq:h1-mixed-source-norm}
		and \eqref{eq:h1-directional-trace},
		\begin{equation}
			\int_{\R^n}
			h(e\cdot x)|\phi(x)|\dd x
			\le C|E|^{1/2}
			\|\phi\|_{H^{\sigma/2}(\R^n)}.
			\label{eq:h1-mixed-trace-pairing}
		\end{equation}
		The constant is independent of $e$.
		
		We now derive an energy estimate.
		Fix $S,R,k$ as above and choose
		$\eta\in C_c^\infty(B_S)$
		with $0\le\eta\le1$.
		Set
		\[
		w:=(v-k)_+,
		\qquad \phi:=\eta w,
		\qquad \psi:=\eta^2w,
		\]
		and
		\[
		E:=\{v>k\}\cap B_R,
		\qquad
		Y:=\int_{B_R}w^2\dd x.
		\]
		
		By \eqref{eq:h1-coarse-lipschitz},
		$\psi$ is Lipschitz and compactly supported
		in $B_{7/4}$. Smooth approximation
		and dominated convergence justify using
		$\psi$ as a test function.
		Indeed, the near-diagonal integral is finite
		since $s<1$, and the exterior terms
		are controlled by the tail bound.
		The shifted weak formulation gives
		\[
		\iint_{B_R\times B_R}
		K_Q(x,y)\delta v(x,y)\delta\psi(x,y)
		\dd x\dd y+2(1-s)\int_{B_R}\psi(x)
		\int_{\R^n\setminus B_R}
		\frac{F_Q(x,y)}{|x-y|^{n+sp}}
		\dd y\dd x=0.
		\]
		
		The truncation inequality gives
		\[
		\begin{aligned}
			\delta v(x,y)\delta(\eta^2w)(x,y)
			&\ge
			\delta w(x,y)\delta(\eta^2w)(x,y)\\
			&=
			|\delta\phi(x,y)|^2
			-w(x)w(y)|\delta\eta(x,y)|^2.
		\end{aligned}
		\]
		By \eqref{eq:h1-secant-kernel}
		and \eqref{eq:large-slope-angular-integrability},
		the last term contributes at most
		$C\|\nabla\eta\|_\infty^2Y$,
		since $\kappa<1$ and $\sigma<2$.
		
		For $y\in\R^n\setminus B_2$,
		the second term in
		\eqref{eq:large-slope-pointwise-tail}
		bounds the contribution by $	C\int_E\eta^2w\dd x$.
		For $y\in B_2\setminus B_R$,
		we discard the nonnegative contribution
		and use \eqref{eq:h1-annulus-negative-flux}.
		Consequently,
		\[
		\begin{aligned}
			[\phi]_{H^{\sigma/2}(B_R)}^2
			\le C\|\nabla\eta\|_\infty^2Y
			+C\int_E\eta^2w\dd x+C(R-S)^{-n-sp}
			\int_{\R^n}h(e\cdot x)\eta^2w(x)\dd x.
		\end{aligned}
		\]
		
		Since $0\le\eta^2w\le\phi$,
		\eqref{eq:h1-mixed-trace-pairing}
		and Young's inequality give
		\[
		\begin{aligned}
			C(R-S)^{-n-sp}
			\int_{\R^n}h(e\cdot x)\eta^2w(x)\dd x\le
			\varepsilon\|\phi\|_{H^{\sigma/2}(\R^n)}^2
			+C_\varepsilon(R-S)^{-2(n+sp)}|E|.
		\end{aligned}
		\]
		
		Moreover, since
		$\operatorname{supp}\phi\subset B_S$,
		\[
		\begin{aligned}
			\|\phi\|_{H^{\sigma/2}(\R^n)}^2
			\le
			[\phi]_{H^{\sigma/2}(B_R)}^2+
			C\bigl(1+(R-S)^{-\sigma}\bigr)Y.
		\end{aligned}
		\]
		Also,
		\[
		\int_E\eta^2w\dd x
		\le |E|^{1/2}Y^{1/2}
		\le C(Y+|E|).
		\]
		Choosing $\varepsilon>0$ sufficiently small
		and absorbing the corresponding term,
		we obtain
		\begin{equation}
			\begin{aligned}
				\|\eta w\|_{H^{\sigma/2}(\R^n)}^2
				&\le C\bigl[
				1+\|\nabla\eta\|_\infty^2\\
				&\qquad +(R-S)^{-2(n+sp)-\sigma}
				\bigr](Y+|E|).
			\end{aligned}
			\label{eq:h1-localized-energy}
		\end{equation}
		
		We now iterate \eqref{eq:h1-localized-energy}.
		Let $K\ge1$ and, for $j\ge0$, define
		\[
		\begin{aligned}
			R_j&:=\frac43+\frac16\,2^{-j},
			&S_j&:=\frac{R_j+R_{j+1}}2,\\
			k_j&:=K(1-2^{-j-1}),
			&w_j&:=(v-k_j)_+.
		\end{aligned}
		\]
		Choose $\eta_j\in C_c^\infty(B_{S_j})$
		such that
		\[
		0\le\eta_j\le1,\qquad
		\eta_j=1\quad\text{on }B_{R_{j+1}},
		\qquad
		|\nabla\eta_j|\le C2^j.
		\]
		Write
		\[
		E_j:=\{v>k_j\}\cap B_{R_j},
		\qquad
		Y_j:=\int_{B_{R_j}}w_j^2\dd x.
		\]
		Since
		\[
		d_j:=k_{j+1}-k_j=K2^{-j-2},
		\]
		we have
		\[
		|E_{j+1}|\le d_j^{-2}Y_j,
		\qquad
		Y_{j+1}\le Y_j.
		\]
		
		By \eqref{eq:h1-localized-energy},
		there exists $B>1$, depending only
		on $n,p,s$, such that
		\[
		\|\eta_jw_j\|_{H^{\sigma/2}(\R^n)}^2
		\le CB^j(Y_j+|E_j|).
		\]
		
		Since $1<\sigma<2\le n$,
		the fractional Sobolev exponent
		\[
		\chi:=\frac{n}{n-\sigma}>1
		\]
		is finite.
		The fractional Sobolev inequality
		and H\"older's inequality give
		\[
		\begin{aligned}
			Y_{j+1}
			&\le |E_{j+1}|^{\sigma/n}
			\|\eta_jw_j\|_{L^{2\chi}(\R^n)}^2\\
			&\le CB^j(Y_j+|E_j|)
			|E_{j+1}|^{\sigma/n}.
		\end{aligned}
		\]
		
		For $j\ge1$, we also have
		\[
		|E_j|\le d_{j-1}^{-2}Y_{j-1},
		\qquad
		Y_j\le Y_{j-1}.
		\]
		Combining these inequalities
		and increasing $B$ if necessary,
		we obtain
		\begin{equation}
			Y_{j+1}
			\le CB^jK^{-2\sigma/n}
			Y_{j-1}^{1+\sigma/n},
			\qquad j\ge1.
			\label{eq:h1-two-step-iteration}
		\end{equation}
		
		Set $Z_j:=Y_j/K^2$.
		By \eqref{eq:h1-local-quadratic-state},
		\[
		Z_{j+1}\le CB^jZ_{j-1}^{1+\sigma/n},
		\qquad
		Z_0+Z_1\le CK^{-2}.
		\]
		For $K=K(n,p,s)$ sufficiently large,
		the nonlinear iteration applied separately
		to the even and odd indices gives
		\[
		Y_j\longrightarrow0.
		\]
		Since $R_j\to4/3$ and $k_j\to K$,
		we conclude that
		\[
		v\le K
		\qquad\text{a.e. in }B_{4/3}.
		\]
		Applying the same argument to $(-Q,-v)$
		and undoing the normalization proves
		\[
		\|v\|_{L^\infty(B_{4/3})}
		\le C(n,p,s)\bigl(1+L^{1/(p-1)}\bigr).
		\]
	\end{proof}

	We shall also use the following consequence of
	\eqref{eq:large-slope-bregman-equivalence}: for \(|X-Y|\le4\) and \(t\in\R\),
	\begin{equation}
		\lambda_Q^{2-p}D_\Phi(Q\cdot(X-Y);t)
		\ge c\Psi_p(|t|).
		\label{eq:shift-uniform-orlicz-coercivity}
	\end{equation}
	Indeed, \(|Q\cdot(X-Y)|+|t|\le4|Q|+|t|\). Apply
	\eqref{eq:large-slope-bregman-equivalence} separately when
	\(|t|\le\lambda_Q\) and when \(|t|>\lambda_Q\).

	The following gluing estimate is stated on a fixed scale,
	with constants independent of $Q$.
	
	\begin{lemma}
		\label{lem:uniform-shifted-bregman-gluing}
		Let $\Omega\subset\R^n$ be a bounded open set, with
		$\operatorname{diam}\Omega\le4$, and let
		$\eta\in C^{0,1}(\R^n)$ satisfy
		$0\le\eta\le1$.
		Define
		\[
		I_\eta:=\{\eta=1\}\cap\Omega,
		\qquad
		O_\eta:=\{\eta=0\}\cap\Omega,
		\qquad
		C_\eta:=\{0<\eta<1\}\cap\Omega,
		\]
		and
		\[
		\begin{aligned}
			\Gamma_\eta
			:=(C_\eta\times\Omega)
			\cup(\Omega\times C_\eta),\quad
			\Xi_\eta
			:=(I_\eta\times O_\eta)
			\cup(O_\eta\times I_\eta).
		\end{aligned}
		\]
		Assume that $C_\eta\Subset\Omega$ and $
		d_\eta:=\operatorname{dist}(I_\eta,O_\eta)>0$.
		For $E\subset\Omega^2$, define
		\[
		\mathscr G_{s,Q}(u;E)
		:=
		(1-s)\lambda_Q^{2-p}
		\iint_E
		\frac{
			D_\Phi(Q\cdot(x-y);\delta u(x,y))
		}{
			|x-y|^{n+sp}
		}\dd x\!\dd y.
		\]
		Suppose that
		\[
		\mathscr G_{s,Q}(v;\Omega^2),
		\quad
		\mathscr G_{s,Q}(z;I_\eta^2),
		\quad
		\mathscr G_{s,Q}(z;\Gamma_\eta)
		\]
		are finite, and that
		$z-v\in L^p(\Omega)\cap L^m(\Omega)$,
		where
		\[
		a>1,\qquad
		(p-2)a>-1,\qquad
		m:=\frac{2a}{a-1}>2.
		\]
		Then $w:=\eta z+(1-\eta)v$ satisfies, for every $\varepsilon\in(0,1)$,
		\begin{equation}	\label{eq:uniform-shifted-bregman-gluing}
			\begin{split}
				\mathscr G_{s,Q}(w;\Omega^2)
				-\mathscr G_{s,Q}(v;\Omega^2)
				&\le
				\mathscr G_{s,Q}(z;I_\eta^2)
				-\mathscr G_{s,Q}(v;I_\eta^2)
				+\varepsilon\mathscr G_{s,Q}(v;\Xi_\eta)
				\\
				&
				+C\bigl[
				\mathscr G_{s,Q}(z;\Gamma_\eta)
				+\mathscr G_{s,Q}(v;\Gamma_\eta)
				\bigr]\\
				&\qquad
				+C_{\varepsilon,\eta,\Omega}
				\bigl(
				\|z-v\|_{L^p(\Omega)}^p
				+\|z-v\|_{L^m(\Omega)}^2
				\bigr).
			\end{split}
		\end{equation}
		Here $C=C(p)$. The constants are
		independent of $Q$ and are locally
		uniform for $s$ in compact subsets
		of $(0,1)$.
		The last constant may depend on
		$d_\eta^{-1}$ and
		$\operatorname{Lip}(\eta;\R^n)$.
	\end{lemma}
	
	\begin{proof}
		For $Q,Z\in\R^n$, define
		\[
		\mathcal B_{Q,Z}(t)
		:=
		\lambda_Q^{2-p}D_\Phi(Q\cdot Z;t).
		\]
		By \eqref{eq:large-slope-bregman-equivalence},
		\[
		\mathcal B_{Q,Z}(t)
		\asymp
		\lambda_Q^{2-p}
		(|Q\cdot Z|+|t|)^{p-2}|t|^2.
		\]
		The constants depend only on $p$.
		
		The convexity of $\Phi$ implies that $\mathcal B_{Q,Z}$ is convex,
		and $\mathcal B_{Q,Z}(0)=0$.
		Moreover,
		\[
		\partial_t\mathcal B_{Q,Z}(t)
		=
		\lambda_Q^{2-p}
		\bigl[
		\Jp(Q\cdot Z+t)-\Jp(Q\cdot Z)
		\bigr].
		\]
		The scalar estimates in
		Lemma~\ref{lem:large-slope-scalar-angular}
		therefore give
		\[
		0\le
		t\partial_t\mathcal B_{Q,Z}(t)
		\le
		\kappa_p\mathcal B_{Q,Z}(t)
		\qquad(t\ne0),
		\]
		where $\kappa_p\ge1$ depends only on $p$.
		
		Integrating this inequality gives
		\[
		\mathcal B_{Q,Z}(\lambda t)
		\le
		\lambda^{\kappa_p}\mathcal B_{Q,Z}(t)
		\qquad(\lambda\ge1).
		\]
		For $0<\delta<1$, convexity gives
		\[
		\begin{aligned}
			\mathcal B_{Q,Z}(t+r)
			&\le
			(1-\delta)
			\mathcal B_{Q,Z}
			\left(\frac{t}{1-\delta}\right)+
			\delta\mathcal B_{Q,Z}
			\left(\frac r\delta\right)\\
			&\le
			(1-\delta)^{1-\kappa_p}
			\mathcal B_{Q,Z}(t)
			+\delta^{1-\kappa_p}
			\mathcal B_{Q,Z}(r).
		\end{aligned}
		\]
		Choose $\delta>0$ sufficiently small
		that
		\[
		(1-\delta)^{1-\kappa_p}
		\le1+\varepsilon.
		\]
		We obtain
		\begin{equation}
			\mathcal B_{Q,Z}(t+r)
			\le
			(1+\varepsilon)\mathcal B_{Q,Z}(t)
			+C_\varepsilon\mathcal B_{Q,Z}(r).
			\label{eq:relative-gluing-epsilon-delta-two}
		\end{equation}
		Taking a fixed $\varepsilon$ also gives
		\[
		\mathcal B_{Q,Z}(t+r)
		\le
		C\bigl[
		\mathcal B_{Q,Z}(t)
		+\mathcal B_{Q,Z}(r)
		\bigr].
		\]
		Furthermore, convexity yields
		\begin{equation}
			\begin{aligned}
				\mathcal B_{Q,Z}(\theta t+(1-\theta)r)\le
				\theta\mathcal B_{Q,Z}(t)
				+(1-\theta)\mathcal B_{Q,Z}(r)
				\qquad(0\le\theta\le1).
			\end{aligned}
			\label{eq:relative-gluing-convex-combination}
		\end{equation}
		
		Let $d:=z-v$.
		Since $w=v+\eta d$, we have
		\begin{equation}
			\begin{aligned}
				\delta w(x,y)
				=\eta(x)\delta z(x,y)
				+(1-\eta(x))\delta v(x,y)+d(y)(\eta(x)-\eta(y)).
			\end{aligned}
			\label{eq:gluing-increment-decomposition}
		\end{equation}
		
		We use the disjoint decomposition
		\[
		\Omega^2=I_\eta^2\mathbin{\dot\cup}O_\eta^2
		\mathbin{\dot\cup}\Gamma_\eta\mathbin{\dot\cup}\Xi_\eta.
		\]
		On $I_\eta^2$, $\delta w=\delta z$; on $O_\eta^2$, $\delta w=\delta v$.
		
		For $(x,y)\in\Gamma_\eta$,
		\eqref{eq:gluing-increment-decomposition},
		the bound for $\mathcal B_{Q,x-y}$ above, and
		\eqref{eq:relative-gluing-convex-combination}
		give
		\[
		\begin{aligned}
			\mathcal B_{Q,x-y}(\delta w(x,y))
			\le C\bigl[
			\mathcal B_{Q,x-y}(\delta z(x,y))
			+\mathcal B_{Q,x-y}(\delta v(x,y))
			\bigr]+
			C\mathcal B_{Q,x-y}
			\bigl(d(y)(\eta(x)-\eta(y))\bigr).
		\end{aligned}
		\]
		If $(x,y)\in I_\eta\times O_\eta$, then $
		\delta w(x,y)=\delta v(x,y)+d(x)$.
		Hence \eqref{eq:relative-gluing-epsilon-delta-two}
		gives
		\[
		\begin{aligned}
			&\mathcal B_{Q,x-y}(\delta w(x,y))
			-\mathcal B_{Q,x-y}(\delta v(x,y))\le
			\varepsilon\mathcal B_{Q,x-y}(\delta v(x,y))
			+C_\varepsilon\mathcal B_{Q,x-y}(d(x)).
		\end{aligned}
		\]
		The same estimate holds on
		$O_\eta\times I_\eta$ after interchanging
		$x$ and $y$, since
		$D_\Phi(-a;-t)=D_\Phi(a;t)$.
		
		It remains to estimate the terms containing $d$.
		Set
		\[
		L_\eta:=\operatorname{Lip}(\eta;\R^n),
		\qquad
		D_\Omega:=\operatorname{diam}\Omega.
		\]
		For $x\ne y$, write $
		r:=|x-y|$ and $
		\omega:=\frac{x-y}{|x-y|}$.
		If $Q\ne0$, let $e=Q/|Q|$;
		otherwise choose any $e\in\mathbb S^{n-1}$. 	When $|Q|\le1$, the bound
		$D_\Phi(a;t)\le C|t|^p$ gives $
		\mathcal B_{Q,r\omega}(t)\le C|t|^p$.
		When $|Q|>1$,
		\eqref{eq:large-slope-bregman-equivalence}
		gives, for almost every $\omega$,
		$
		\mathcal B_{Q,r\omega}(t)
		\le Cr^{p-2}|e\cdot\omega|^{p-2}|t|^2$.
		Thus, in both cases,
		\[
		\mathcal B_{Q,r\omega}(t)
		\le C\left(
		|t|^p+
		r^{p-2}|e\cdot\omega|^{p-2}|t|^2
		\right).
		\]
		
		We first consider $\Gamma_\eta$.
		Since
		\[
		|\eta(x)-\eta(y)|\le L_\eta r,
		\]
		we obtain
		\[
		\begin{aligned}
			\mathcal B_{Q,x-y}
			\bigl(d(y)(\eta(x)-\eta(y))\bigr)\le C\bigl[
			L_\eta^p r^p|d(y)|^p
			+L_\eta^2r^p
			|e\cdot\omega|^{p-2}|d(y)|^2
			\bigr].
		\end{aligned}
		\]
		Using polar coordinates in $x$
		and \eqref{eq:large-slope-angular-integrability},
		we obtain
		\[
		\begin{aligned}
			&(1-s)\iint_{\Gamma_\eta}
			\frac{
				\mathcal B_{Q,x-y}
				\bigl(d(y)(\eta(x)-\eta(y))\bigr)
			}{|x-y|^{n+sp}}\dd x\dd y\\
			&\quad\le
			C(1-s)
			\int_0^{D_\Omega}r^{p-sp-1}\dd r
			\bigl[
			L_\eta^p\|d\|_{L^p(\Omega)}^p
			+L_\eta^2\|d\|_{L^2(\Omega)}^2
			\bigr]\le
			C_{\eta,\Omega}\bigl[
			\|d\|_{L^p(\Omega)}^p
			+\|d\|_{L^m(\Omega)}^2
			\bigr],
		\end{aligned}
		\]
		where we used $m>2$ and the boundedness of $\Omega$.
		
		Next, consider $\Xi_\eta$.
		Since
		\[
		|x-y|\ge d_\eta>0
		\qquad ((x,y)\in\Xi_\eta),
		\]
		the same bound and polar coordinates give
		\[
		\begin{aligned}
			&(1-s)\iint_{\Xi_\eta}
			\frac{\mathcal B_{Q,x-y}(d(x))}
			{|x-y|^{n+sp}}\dd x\dd y\\
			&\quad\le C(1-s)
			\left[
			\left(\int_{d_\eta}^{D_\Omega}
			r^{-sp-1}\dd r\right)
			\|d\|_{L^p(\Omega)}^p+
			\left(\int_{d_\eta}^{D_\Omega}
			r^{p-sp-3}\dd r\right)
			\|d\|_{L^2(\Omega)}^2
			\right]\\
			&\quad\le C_{\eta,\Omega}
			\bigl[
			\|d\|_{L^p(\Omega)}^p
			+\|d\|_{L^m(\Omega)}^2
			\bigr].
		\end{aligned}
		\]
		The same estimate holds for $d(y)$
		after interchanging $x$ and $y$.
		
		Finally,
		\[
		(1-s)\int_0^{D_\Omega}
		r^{p-sp-1}\dd r
		=\frac{D_\Omega^{p-sp}}p.
		\]
		The integrals over
		$[d_\eta,D_\Omega]$ are also bounded
		locally uniformly for $s\in(0,1)$,
		since $d_\eta>0$.
		
		Combining the estimates on
		$I_\eta^2$, $O_\eta^2$,
		$\Gamma_\eta$, and $\Xi_\eta$
		proves \eqref{eq:uniform-shifted-bregman-gluing}.
	\end{proof}
	
	\begin{lemma}
		\label{lem:shift-uniform-fixed-level-clearing}
		Let $n\ge2$, $1<p<2$, and
		$I\Subset(0,1)$.
		Let $D'\Subset D_0\Subset D\subset\R^n$, where $D$ is bounded,
		and let $M\ge1$.
		Suppose that $s_j\in I$, $Q_j\in\R^n$,
		and
		\[
		\mathscr G_{s_j,Q_j}(v_j;D^2)<\infty.
		\]
		Assume that $f_j\in L^\infty(D_0)$ and
		\begin{equation}
			\begin{aligned}
				(1-s_j)\lambda_{Q_j}^{2-p}
				\iint_{D\times D}
				\frac{
					\mathcal D_{Q_j}v_j(x,y)
					\delta\varphi(x,y)
				}{
					|x-y|^{n+s_jp}
				}\dd x\!\dd y=
				\int_{D_0}f_j\varphi\dd x
			\end{aligned}
			\label{eq:shifted-regional-source}
		\end{equation}
		for every $\varphi\in C_c^\infty(D_0)$,
		and
		\begin{equation}
			\sup_j
			\left(
			\|v_j\|_{L^\infty(D)}
			+\|f_j\|_{L^\infty(D_0)}
			\right)
			\le M.
			\label{eq:shifted-regional-source-bounds}
		\end{equation}
		If
		\[
		v_j\to v
		\qquad\text{strongly in }L^p_{\rm loc}(D_0)
		\]
		and $v$ is continuous in $D_0$, then
		\begin{equation}
			\|v_j-v\|_{L^\infty(D')}
			\longrightarrow0.
			\label{eq:shift-uniform-local-uniform-upgrade}
		\end{equation}
		
	\end{lemma}
	
	\begin{proof}
		For $Z=X-Y$, $a=P\cdot Z$, and
		$\lambda=1+|P|$, define
		\[F_a(t)
		:=\lambda^{2-p}
		\bigl[\Jp(a+t)-\Jp(a)\bigr]\quad \text{and}\quad 
		G_a(t)
		:=\mathcal H_{P,Z}(t).
		\]
		By the scalar estimates in
		Lemma~\ref{lem:large-slope-scalar-angular},
		\[
		F_a(t)t
		\asymp G_a(t)
		\asymp
		\lambda^{2-p}
		(|a|+|t|)^{p-2}|t|^2.
		\]
		In particular, $F_a$ is increasing,
		$F_a(0)=0$, and $G_a$ is even,
		nonnegative, and increasing in $|t|$.
		
		We first prove that, for every $\varepsilon>0$ and $0<\theta\le1$,
		\begin{equation}
			|F_a(t)b|
			\le
			\varepsilon G_a(t)
			+C_{p,\varepsilon}G_a(b)\quad \text{and}\quad 
			\theta^2G_a(b/\theta)
			\le C_pG_a(b).
			\label{eq:uncapped-shifted-young}
		\end{equation}
		The constants are independent of
		$P$ and $Z$.
		
		To prove the first inequality,
		choose $\delta>0$.
		If $|b|\le\delta|t|$, then $
		|F_a(t)b|
		\le C\delta G_a(t)$.
		Choosing $\delta$ sufficiently small
		gives the desired bound. If $|b|>\delta|t|$, the scalar
		estimates give
		\[
		\begin{aligned}
			\frac{|F_a(t)b|}{G_a(b)}
			\le
			C\frac{|t|}{|b|}
			\left(
			\frac{|a|+|b|}
			{|a|+|t|}
			\right)^{2-p}\le C_\delta.
		\end{aligned}
		\]
		For the second inequality in \eqref{eq:uncapped-shifted-young},
		\[
		\theta^2G_a(b/\theta)
		\le C\lambda^{2-p}(|a|+|b|/\theta)^{p-2}|b|^2
		\le CG_a(b).
		\]

		We next prove an estimate for a bounded function
		$g$ on $B_2$.
		Suppose that
		\[
		\|g\|_{L^\infty(B_2)}\le N,
		\qquad |k|\le N+1.
		\]
		Let $w:=(g-k)_+$ and $0\le\eta\le1$.
		For $X,Y\in B_2$, we have
		\[
		|\delta g(X,Y)|+w(X)+w(Y)
		\le C(1+N).
		\]

		We claim that, for $X,Y\in B_2$,
		\begin{equation}
			\begin{aligned}
				&F_a(\delta g(X,Y))
				\delta(\eta^2w)(X,Y)\\
				&\quad\ge
				c_N|\delta(\eta w)(X,Y)|^2-
				C_N
				\mathbf1_{\{w(X)>0,\ w(Y)>0\}}
				G_a(C_N\delta\eta(X,Y)),
			\end{aligned}
			\label{eq:bounded-uncapped-pair}
		\end{equation}
		where $a=P\cdot(X-Y)$.
		The constants depend on $p,N$,
		but not on $P$.
		
		Suppose first that $w(X)>0$ and $w(Y)>0$.
		Then $\delta g=\delta w=:t$.
		Write
		\[
		m:=\frac{\eta(X)^2+\eta(Y)^2}{2},
		\qquad
		\theta:=\frac{\eta(X)+\eta(Y)}2.
		\]
		A direct calculation gives
		\[
		\delta(\eta^2w)
		=
		mt+\theta\delta\eta\,
		(w(X)+w(Y)),
		\]
		where $
		\theta^2\le m\le2\theta^2$. 	If $\theta=0$, both sides vanish.
		Otherwise, apply
		\eqref{eq:uncapped-shifted-young}
		with
		\[
		b=\frac{
			(w(X)+w(Y))\delta\eta
		}{\theta}.
		\]
		Choosing $\varepsilon$ sufficiently
		small, we obtain
		\[
		\begin{aligned}
			F_a(t)\delta(\eta^2w)&\ge
			cmG_a(t)
			-C\theta^2G_a(b)\ge
			cmG_a(t)
			-CG_a((w(X)+w(Y))\delta\eta).
		\end{aligned}
		\]
		Since $|X-Y|\le4$ and the increments are bounded in terms of $N$,
		the scalar estimates give
		\[
		G_a(t)\ge c_N|t|^2.
		\]
		The same lower bound applies to
		bounded cutoff increments.
		Using
		\[
		|\delta(\eta w)|^2
		\le C\bigl[m|t|^2+(w(X)+w(Y))^2|\delta\eta|^2\bigr],
		\]
		we obtain
		\eqref{eq:bounded-uncapped-pair}
		for these pairs.
		
		Suppose next that $w(X)>0$ and $w(Y)=0$.
		Then $\delta g(X,Y)\ge w(X)$.
		By monotonicity of $F_a$,
		\[
		\begin{aligned}
			F_a(\delta g)
			\delta(\eta^2w)=
			F_a(\delta g)\eta(X)^2w(X)\ge
			F_a(w(X))\eta(X)^2w(X)\ge
			c_N|\delta(\eta w)|^2.
		\end{aligned}
		\]
		The case $w(Y)>0=w(X)$ follows
		by interchanging $X$ and $Y$.
		If both values vanish, the
		inequality is immediate.
		This proves
		\eqref{eq:bounded-uncapped-pair}.
		
		Fix $
		B_{2r}(x_0)\Subset D_0$ and $0<r\le1$.
		We use the change of variables
		\[
		g_j(X):=v_j(x_0+rX),
		\qquad
		P_j:=rQ_j,
		\qquad
		\widehat D:=\frac{D-x_0}{r}.
		\]
		Then, for every $\varphi\in C_c^\infty(B_2)$,
		the rescaled equation is
		\[
		\begin{aligned}
			(1-s_j)\lambda_{P_j}^{2-p}
			\iint_{\widehat D\times\widehat D}
			\frac{
				\mathcal D_{P_j}g_j(X,Y)\delta\varphi(X,Y)
			}{
				|X-Y|^{n+s_jp}
			}\dd X\dd Y=
			\int_{B_2}\widehat f_j(X)\varphi(X)\dd X,
		\end{aligned}
		\]
		where
		\[
		\widehat f_j(X)
		=
		r^{s_jp}
		\left(
		\frac{\lambda_{P_j}}{\lambda_{Q_j}}
		\right)^{2-p}
		f_j(x_0+rX).
		\]
		Since $r\le1$, $
		\|\widehat f_j\|_{L^\infty(B_2)}
		\le CM$.
		We also have $|g_j|\le M$ almost everywhere in $\widehat D$.
		
		We next control the integral over
		$\widehat D\setminus B_2$.
		If $X\in B_{3/2}$ and $Y\in\widehat D\setminus B_2$,
		then $|X-Y|\ge\frac12$. Since $|\delta g_j(X,Y)|\le2M$,
		we consider two cases.
		
		If $|P_j|\le1$, the $(p-1)$-H\"older
		continuity of $\Jp$ gives
		\[
		\lambda_{P_j}^{2-p}
		|\mathcal D_{P_j}g_j(X,Y)|
		\le CM^{p-1}.
		\]
		
		If $|P_j|>1$, let
		\[
		e_j:=\frac{P_j}{|P_j|},
		\qquad
		\rho:=|X-Y|,
		\qquad
		\omega:=\frac{X-Y}{|X-Y|}.
		\]
		By \eqref{eq:large-slope-angular-flux},
		\[
		\lambda_{P_j}^{2-p}
		|\mathcal D_{P_j}g_j(X,Y)|
		\le CM\rho^{p-2}
		|e_j\cdot\omega|^{p-2}
		\]
		for almost every $\omega\in\mathbb S^{n-1}$.
		
		Using polar coordinates centered
		at $X$, we obtain
		\begin{equation}
			\begin{aligned}
				\esssup_{X\in B_{3/2}}
				\int_{\widehat D\setminus B_2}
				\frac{
					\lambda_{P_j}^{2-p}
					|\mathcal D_{P_j}g_j(X,Y)|
				}{
					|X-Y|^{n+s_jp}
				}\dd Y\le C_{M,I,r,D}.
			\end{aligned}
			\label{eq:uncapped-separated-flux}
		\end{equation}

		We now prove an estimate for
		the positive part of $g_j$.
		Fix $A>0$ and $|k|\le M+1$.
		For $i\ge0$, define
		\[
		\begin{aligned}
			r_i:=\frac43+\frac16\,2^{-i},\quad 	k_i:=k+A(1-2^{-i}),\\
			w_i:=(g_j-k_i)_+, \quad 
			E_i:=\{g_j>k_i\}\cap B_{r_i}.
		\end{aligned}
		\]
		Choose
		$\eta_i\in C_c^\infty(B_{r_i})$
		such that
		\[
		0\le\eta_i\le1,
		\qquad
		\eta_i=1\quad\text{on }B_{r_{i+1}},
		\qquad
		|\nabla\eta_i|\le C2^i.
		\]

		Since $|g_j|\le M$ and
		$|k_i|\le M+A+1$, we may apply
		\eqref{eq:bounded-uncapped-pair}
		with $N=M+A$.
		The constants depend only on $p,M,A$
		and are independent of $i,j,P_j$.

		We first justify using $\psi_i:=\eta_i^2w_i$
		as a test function.
		Since $|\delta w_i|\le|\delta g_j|$ and
		$0\le\eta_i\le1$, we have
		\[
		|\delta\psi_i(X,Y)|
		\le|\delta g_j(X,Y)|
		+C_M|\delta\eta_i(X,Y)|.
		\]
		The growth estimates of $G_a$ give
		\[
		G_a(\delta\psi_i)
		\le C\bigl[
		G_a(\delta g_j)
		+G_a(C_M\delta\eta_i)
		\bigr],
		\qquad a=P_j\cdot(X-Y).
		\]
		The first term is integrable by the energy
		assumption. The second is integrable by
		\eqref{eq:shifted-modular-angular-upper},
		since
		\[
		|\delta\eta_i(X,Y)|
		\le C2^i|X-Y|.
		\]
		As $\operatorname{supp}\psi_i\Subset B_{3/2}$,
		convolution and Jensen's inequality,
		together with \eqref{eq:uncapped-shifted-young},
		justify approximation by smooth functions
		supported in $B_2$.
		The integral over $\widehat D\setminus B_2$
		is controlled by \eqref{eq:uncapped-separated-flux},
		and the source term is controlled by
		$\|\widehat f_j\|_\infty$.
		Thus $\psi_i$ is an admissible test function.
		
		We next derive the energy estimate.
		By \eqref{eq:bounded-uncapped-pair},
		the local integral gives a positive term
		containing $|\delta(\eta_iw_i)|^2$.
		The error term is nonzero only if
		\[
		w_i(X)>0,\qquad
		w_i(Y)>0,\qquad
		\delta\eta_i(X,Y)\ne0.
		\]
		For such pairs, at least one of $X,Y$
		belongs to $E_i$. By symmetry, it suffices
		to integrate over $X\in E_i$ and $Y\in B_2$.
		
		Since $G_a$ is increasing in $|t|$ and
		$|\delta\eta_i|\le C2^i|X-Y|$,
		\eqref{eq:shifted-modular-angular-upper}
		and polar coordinates give
		\[
		\begin{aligned}
			(1-s_j)\int_{E_i}\int_{B_2}
			\frac{
				\mathcal H_{P_j,X-Y}
				(C_{M,A}\delta\eta_i(X,Y))
			}{
				|X-Y|^{n+s_jp}
			}\dd Y\dd X&\le
			C_{M,A}4^i|E_i|(1-s_j)
			\int_0^4\rho^{p-s_jp-1}\dd\rho\\
			&\le C_{M,A}4^i|E_i|,
		\end{aligned}
		\]
		where we used
		\[
		(1-s_j)\int_0^4\rho^{p-s_jp-1}\dd\rho
		=\frac{4^{p(1-s_j)}}p\le C.
		\]
		
		The source term is bounded by
		$C_M|E_i|$, since $\widehat f_j$ and
		$w_i$ are bounded. By
		\eqref{eq:uncapped-separated-flux},
		the integral over $\widehat D\setminus B_2$
		has the same bound.
		Combining these estimates, we obtain
		\begin{equation}
			\begin{aligned}
				(1-s_j)
				\iint_{B_2\times B_2}
				\frac{
					|\delta(\eta_iw_i)(X,Y)|^2
				}{
					|X-Y|^{n+s_jp}
				}\dd X\dd Y\le C_{M,A,I,r,D}4^i|E_i|.
			\end{aligned}
			\label{eq:uncapped-fixed-level-energy}
		\end{equation}
		
		Let
		\[
		s_-:=\inf I,
		\qquad
		\chi:=\frac{n}{n-s_-p}>1.
		\]
		Since $s_-p<2\le n$, the fractional
		Sobolev inequality applies.
		Moreover, $\eta_iw_i$ vanishes outside
		$B_{3/2}$, and
		\[
		\|\eta_iw_i\|_{L^2(B_2)}^2
		\le C_M|E_i|.
		\]
		Comparing the kernels for $s_j\ge s_-$
		and using \eqref{eq:uncapped-fixed-level-energy},
		we obtain
		\[
		\|\eta_iw_i\|_{L^{2\chi}(B_2)}^2
		\le C_{M,A,I,r,D}4^i|E_i|.
		\]
		
		On $E_{i+1}$, we have $\eta_i=1$ and
		\[
		w_i\ge k_{i+1}-k_i=A2^{-i-1}.
		\]
		Consequently,
		\[
		\begin{aligned}
			A^2 2^{-2i-2}|E_{i+1}|^{1/\chi}
			\le
			\|\eta_iw_i\|_{L^{2\chi}(B_2)}^2\le C_{M,A,I,r,D}4^i|E_i|.
		\end{aligned}
		\]
		It follows that
		\[
		|E_{i+1}|\le CB^i|E_i|^\chi,
		\qquad \chi>1,
		\]
		where $C>0$ and $B>1$ are independent
		of $i,j,P_j$.
		
		Applying this estimate with initial level
		$k+A/2$ and remaining height $A/2$,
		the nonlinear sequence lemma gives
		$\delta_A>0$, independent of $j,P_j$,
		such that
		\[
		\begin{aligned}
			|\{g_j>k+A/2\}\cap B_2|<\delta_A
			\quad\Longrightarrow\quad
			g_j\le k+A
			\quad\text{a.e. in }B_{4/3}.
		\end{aligned}
		\]
		The same argument applied to $-g_j$
		gives the corresponding lower estimate.
		
		We now prove \eqref{eq:shift-uniform-local-uniform-upgrade}.
		Fix $x_0\in D'$ and $0<\varepsilon<1$.
		By the continuity of $v$, choose
		$0<r\le1$ such that
		\[
		B_{2r}(x_0)\Subset D_0,
		\qquad
		\operatorname{osc}_{B_{2r}(x_0)}v
		\le\frac{\varepsilon}{4}.
		\]
		Take
		\[
		k:=v(x_0)+\frac{\varepsilon}{4},
		\qquad
		A:=\frac{\varepsilon}{2}.
		\]
		Since $v_j\to v$ in $L^p_{\rm loc}(D_0)$
		and $|v_j|\le M$, continuity gives
		$|v(x_0)|\le M$.
		Thus $|k|\le M+1$.
		
		By the oscillation bound and
		Chebyshev's inequality,
		\[
		\begin{aligned}
			|\{g_j>k+A/2\}\cap B_2|&=
			r^{-n}
			|\{v_j>v(x_0)+\varepsilon/2\}
			\cap B_{2r}(x_0)|\\
			&\le
			r^{-n}\left(\frac4\varepsilon\right)^p
			\|v_j-v\|_{L^p(B_{2r}(x_0))}^p
			\longrightarrow0.
		\end{aligned}
		\]
		The level-set estimate therefore gives,
		for all sufficiently large $j$,
		\[
		v_j\le v(x_0)+\frac{3\varepsilon}{4}
		\qquad\text{a.e. in }B_{4r/3}(x_0).
		\]
		Since
		\[
		v\ge v(x_0)-\frac{\varepsilon}{4}
		\qquad\text{in }B_{4r/3}(x_0),
		\]
		we obtain
		\[
		v_j-v\le\varepsilon
		\qquad\text{a.e. in }B_{4r/3}(x_0).
		\]
		Applying the lower estimate gives
		\[
		v-v_j\le\varepsilon
		\qquad\text{a.e. in }B_{4r/3}(x_0).
		\]
		Hence,	for all sufficiently large $j$, we have
		\[
		\|v_j-v\|_{L^\infty(B_{4r/3}(x_0))}
		\le\varepsilon.
		\]
		Finally, a finite covering of $D'$
		by such balls gives
		\[
		\|v_j-v\|_{L^\infty(D')}
		\longrightarrow0,
		\]
		which proves
		\eqref{eq:shift-uniform-local-uniform-upgrade}.
	\end{proof}

	\subsection{Compactness as the slope tends to infinity}
	\label{subsec:infinite-slope-compactness}
	
	We next study the shifted equation when the affine
	slope tends to infinity. The following proposition
	shows that the solutions converge locally to a
	solution of a linear nonlocal equation and that
	the normalized energies also converge. These
	results will be used to prove the affine
	improvement estimate.

	\begin{proposition}
		\label{prop:fixed-order-infinite-slope-compactness}
		Let \(s_j\to s\in(0,1)\) with \(sp>p-1\), and let
		\(Q_j=q_je_j\), \(q_j\to\infty\), and
		\(e_j\to e\in\mathbb S^{n-1}\).
		Suppose that \(v_j\) solves the shifted equation in \(B_2\) and, for
		some \(L<\infty\),
		\[
		\|v_j\|_{L^\infty(B_{4/3})}
		+\mathcal E_{s_j,Q_j}(v_j;B_2)
		+\mathcal T^{\rm pt}_{s_j,Q_j}(v_j)\le L.
		\]
		Then, after a subsequence, \(v_j\to v\) locally uniformly in
		\(B_{5/4}\) and strongly in \(L^p_{\rm loc}(B_{4/3})\).  Set
		\(D=B_{4/3}\).  In \(B_{5/4}\),
		the limit solves the regional equation
		\begin{equation}
			L^D_{s,e}v=f,\qquad f\in L^\infty(B_{5/4}),
			\label{eq:large-slope-limit-equation}
		\end{equation}
		whose weak formulation is
		\begin{equation}
			(p-1)(1-s)
			\iint_{D\times D}
			\frac{|e\cdot(X-Y)|^{p-2}\delta v(X,Y)\delta\varphi(X,Y)}
			{|X-Y|^{n+sp}}\dd X\!\dd Y
			=\int_Df\varphi\dd X,
			\label{eq:large-slope-limit-form}
		\end{equation}
		for every \(\varphi\in C_c^\infty(B_{5/4})\).
		The operator has order \(\tau_s\). The normalized energies converge on every ball $B\Subset B_{5/4}$. For each fixed affine map
		\(m=c+b\cdot X\),
		\begin{equation}
			\begin{split}
				&\mathcal E_{s_j,Q_j+b}(v_j-m;B)\\
				&\quad\longrightarrow
				\frac{(p-1)(1-s)}2
				\iint_{B\times B}
				\frac{|e\cdot(X-Y)|^{p-2}
					|\delta(v-m)(X,Y)|^2}
				{|X-Y|^{n+sp}}\dd X\!\dd Y.
			\end{split}
			\label{eq:infinite-slope-recentered-energy-convergence}
		\end{equation}
	\end{proposition}
	
	\begin{proof}
		Since $s_j\to s\in(0,1)$, $sp>p-1$,
		and $q_j\to\infty$, we may assume that  $0<	\inf_j s_j :=s_- \le s_+:=\sup_j s_j<1$,
		and $\inf_j\{s_jp-(p-1)\}>0$ with $q_j\ge1$.
		
		Write $D=B_{4/3}$ and, for $E\subset D^2$, we define
		\[	\mathscr G_j(u;E):=\mathscr G_{s_j,Q_j}(u;E), \quad  \mathscr G_j(u):=\mathscr G_j(u;D^2)\]
		and \[
		\mathscr G_\infty(u;E):=
		\frac{(p-1)(1-s)}2
		\iint_E\frac{|e\cdot(X-Y)|^{p-2}|\delta u|^2}
		{|X-Y|^{n+sp}}\dd X\!\dd Y.
		\]
		We proceed in several steps.

		\emph{Step 1: compactness in $L^p$.}
		By \eqref{eq:shift-uniform-orlicz-coercivity} and
		\eqref{eq:orlicz-two-indices},
		\[
		J_j:=\iint_{B_2\times B_2}
		\frac{\Psi_p(|\delta v_j|)}{|X-Y|^{n+s_jp}}
		\dd X\!\dd Y\le\frac{CL}{1-s_j}\le C.
		\]
		Fix a ball $B\Subset D$ and $0<t<s_-p/2$. Set
		\[
		A_j^-:=B^2\cap\{|\delta v_j|\le1\},\qquad
		A_j^+:=B^2\cap\{|\delta v_j|>1\}.
		\]
		H\"older's inequality with exponents $2/p$ and $2/(2-p)$ gives
		\[
		\iint_{A_j^-}\frac{|\delta v_j|^p}{|X-Y|^{n+tp}}
		\dd X\!\dd Y
		\le CJ_j^{p/2}
		\left(\int_0^4
		r^{-1+\frac{p(s_jp-2t)}{2-p}}\dd r\right)^{(2-p)/2}
		\le C,\]
		and \[
		\iint_{A_j^+}\frac{|\delta v_j|^p}{|X-Y|^{n+tp}}
		\dd X\!\dd Y
		\le C\iint_{A_j^+}
		\frac{\Psi_p(|\delta v_j|)}{|X-Y|^{n+s_jp}}
		\dd X\!\dd Y\le CJ_j.
		\]
		Thus $\sup_j\|v_j\|_{W^{t,p}(B)}<\infty$. Rellich's theorem and
		a diagonal subsequence yield
		\begin{equation}
			v_j\to v\quad\text{in }L^p_{\rm loc}(D)
			\quad\text{and a.e. in }D.
			\label{eq:infinite-slope-strong-compactness}
		\end{equation}
		For $X\in B_{5/4}$, define
		\begin{equation}
			f_j(X):=-2(1-s_j)\int_{\R^n\setminus B_{4/3}}
			\frac{\lambda_{Q_j}^{2-p}\mathcal D_{Q_j}v_j(X,Y)}
			{|X-Y|^{n+s_jp}}\dd Y.
			\label{eq:large-slope-crossed-source}
		\end{equation}
		The shifted equation gives, for $\varphi\in C_c^\infty(B_{5/4})$,
		\[
		(1-s_j)\lambda_{Q_j}^{2-p}
		\iint_{D^2}
		\frac{\mathcal D_{Q_j}v_j\,\delta\varphi}
		{|X-Y|^{n+s_jp}}\dd X\!\dd Y
		=\int_{B_{5/4}}f_j\varphi\dd X.
		\]
		The first component of \eqref{eq:large-slope-pointwise-tail} implies
		\begin{equation}
			\|f_j\|_{L^\infty(B_{5/4})}
			\le2(1-s_j)\mathcal T^{\rm pt}_{s_j,Q_j}(v_j)\le CL.
			\label{eq:large-slope-crossed-source-bound}
		\end{equation}

		\emph{Step 2: Lower bound and approximation of the energies.}
		
		Fix a ball $B\Subset B_{5/4}$.
		For $t\in\R$ and $h\in\R^n$ with $e\cdot h\ne0$,
		Taylor's formula gives
		\begin{equation}
			\begin{aligned}
				\lambda_{Q_j}^{2-p}D_\Phi(Q_j\cdot h;t)
				&=(p-1)
				\left(\frac{\lambda_{Q_j}}{q_j}\right)^{2-p}
				t^2\int_0^1(1-\theta)
				\left|e_j\cdot h+\frac{\theta t}{q_j}\right|^{p-2}
				\dd\theta\\
				&\longrightarrow
				\frac{p-1}{2}|e\cdot h|^{p-2}t^2.
			\end{aligned}
			\label{eq:infinite-slope-pointwise-density}
		\end{equation}
		Indeed, $q_j\to\infty$, $e_j\to e$, and
		$\lambda_{Q_j}/q_j\to1$.
		Since $e\cdot h\ne0$, the integrands
		are uniformly bounded for all large $j$,
		so the convergence follows by dominated
		convergence. The same argument applies
		when $t_j\to t$.
		
		Choose $
		1<a<\frac1{2-p}$, and $
		m:=\frac{2a}{a-1}>2$.
		By \eqref{eq:large-slope-angular-integrability}, we have 
		\[
		\sup_j\int_{\mathbb S^{n-1}}
		|e_j\cdot\omega|^{(p-2)a}\dd\omega
		\le C.
		\]
		Moreover, \eqref{eq:infinite-slope-strong-compactness}
		and the uniform $L^\infty$ bound imply $
		\|v\|_{L^\infty(D)}\le L$
		and
		\[
		\|v_j-v\|_{L^q(B)}^q
		\le(2L)^{q-p}
		\|v_j-v\|_{L^p(B)}^p
		\longrightarrow0,
		\qquad q=2,m.
		\]
		
		Since $\delta v_j(X,Y)\to\delta v(X,Y)$
		for almost every $(X,Y)\in B\times B$,
		\eqref{eq:infinite-slope-pointwise-density}
		gives almost everywhere convergence
		of the corresponding energy densities.
		As $D_\Phi$ is nonnegative,
		Fatou's lemma yields
		\begin{equation}
			\mathscr G_\infty(v;B^2)
			\le\liminf_{j\to\infty}
			\mathscr G_j(v_j;B^2).
			\label{eq:infinite-slope-form-liminf}
		\end{equation}
		
		We next construct smooth approximations
		of $v$ whose energies converge. Recall that $
		\tau_s:=sp-p+2\in(1,2)$. For $h=r\omega\ne0$, where
		$r=|h|$ and $\omega=h/|h|$, we have
		\[
		\frac{|e\cdot h|^{p-2}}{|h|^{n+sp}}
		=
		\frac{|e\cdot\omega|^{p-2}}
		{|h|^{n+\tau_s}}.
		\]
		Since $p-2<0$ and
		$|e\cdot\omega|\le1$, we have $
		|e\cdot\omega|^{p-2}\ge1$
		for almost every $\omega$.
		Together with
		\eqref{eq:large-slope-angular-integrability},
		this gives
		\[
		\begin{aligned}
			c|\xi|^{\tau_s}
			\le
			\int_{\mathbb S^{n-1}}
			|\xi\cdot\omega|^{\tau_s}
			|e\cdot\omega|^{p-2}\dd\omega\le C|\xi|^{\tau_s},
			\qquad \xi\in\R^n.
		\end{aligned}
		\]
		By Plancherel's theorem, the corresponding
		quadratic energy on $\R^n$ is comparable
		to the usual fractional quadratic energy
		of order $\tau_s$.

		We next construct a sequence $z_j$ converging
		to $v$ such that
		\[
		\mathscr G_j(z_j;B^2)
		\longrightarrow
		\mathscr G_\infty(v;B^2).
		\]
		This sequence will be used in Step~4
		to prove the energy convergence of $v_j$.
		
		First, let $\zeta$ be a fixed smooth function.
		For $X,Y\in B$, write $
		\rho:=|X-Y|$ and $
		\omega:=\frac{X-Y}{|X-Y|}$.
		Since
		\[
		|\delta\zeta(X,Y)|
		\le\rho\|\nabla\zeta\|_\infty,
		\]
		\eqref{eq:large-slope-bregman-equivalence}
		and $p-2<0$ give, for almost every $\omega$,
		\[
		\begin{aligned}
			\lambda_{Q_j}^{2-p}
			D_\Phi(Q_j\cdot\rho\omega;\delta\zeta)\le
			C\lambda_{Q_j}^{2-p}q_j^{p-2}
			\rho^{p-2}|e_j\cdot\omega|^{p-2}
			|\delta\zeta|^2\le
			C\|\nabla\zeta\|_\infty^2
			\rho^p|e_j\cdot\omega|^{p-2},
		\end{aligned}
		\]
		where we used $
		\lambda_{Q_j}^{2-p}q_j^{p-2}
		\le2^{2-p}$.

		For $0<\rho\le4$, we have
		\[
		(1-s_j)\rho^{p-s_jp-1}
		\le C(1+\rho^{\beta-1}),
		\qquad
		\beta:=p(1-s_+)>0.
		\]
		Moreover, by H\"older's inequality and
		\eqref{eq:large-slope-angular-integrability},
		\[
		\int_A|e_j\cdot\omega|^{p-2}\dd\omega
		\le C|A|^{1-1/a}
		\]
		for every measurable $A\subset\mathbb S^{n-1}$.
		
		By \eqref{eq:infinite-slope-pointwise-density},
		Vitali's theorem allows us to pass to the
		limit in the integrals over $\mathbb S^{n-1}$.
		Since $1+\rho^{\beta-1}$ is integrable
		on $(0,4)$, dominated convergence
		in $\rho$ and $X$ gives
		\[
		\mathscr G_j(\zeta;B^2)
		\longrightarrow
		\mathscr G_\infty(\zeta;B^2).
		\]
		
		We now approximate the limit function $v$.
		By \eqref{eq:infinite-slope-form-liminf}, $
		\mathscr G_\infty(v;B^2)<\infty$.
		Since $|e\cdot\omega|^{p-2}\ge1$
		for almost every $\omega$, the limiting
		energy controls the usual fractional
		quadratic energy of order $\tau_s$ on $B$.
		
		Extend $v$ from $B$ to $\R^n$ using
		the standard extension theorem.
		After truncating the extension to $[-L,L]$,
		the quadratic energy comparison above
		shows that the extension has finite
		limiting energy on $\R^n$.
		Mollification then gives smooth functions
		$\zeta^{(k)}$, uniformly bounded by $L$,
		such that
		\[
		\|\zeta^{(k)}-v\|_{L^2(B)}
		\longrightarrow0
		\]
		and
		\[
		\begin{aligned}
			\iint_{B\times B}
			\frac{
				|e\cdot(X-Y)|^{p-2}
				|\delta(\zeta^{(k)}-v)(X,Y)|^2
			}{
				|X-Y|^{n+sp}
			}\dd X\dd Y\longrightarrow0.
		\end{aligned}
		\]
		Consequently,
		\[
		\mathscr G_\infty(\zeta^{(k)};B^2)
		\longrightarrow
		\mathscr G_\infty(v;B^2).
		\]
		
		For each fixed $k$, the convergence
		proved above for smooth functions gives
		\[
		\mathscr G_j(\zeta^{(k)};B^2)
		\longrightarrow
		\mathscr G_\infty(\zeta^{(k)};B^2)
		\qquad(j\to\infty).
		\]
		Choosing $k=k(j)\to\infty$ sufficiently
		slowly and setting
		\[
		z_j:=\zeta^{(k(j))},
		\]
		we obtain
		\[
		z_j\to v\quad\text{in }L^2(B),
		\qquad
		\mathscr G_j(z_j;B^2)
		\longrightarrow
		\mathscr G_\infty(v;B^2).
		\]
		Since $z_j$ and $v$ are uniformly bounded,
		we also have
		\[
		z_j\to v
		\qquad\text{in }L^p(B)\cap L^m(B).
		\]
		
		\emph{Step 3: convergence of the sources.}
		By \eqref{eq:large-slope-crossed-source}--\eqref{eq:large-slope-crossed-source-bound},
		a subsequence satisfies
		\[
		f_j\stackrel{*}{\rightharpoonup}f
		\text{ in }L^\infty(B_{5/4}),\qquad
		\|f\|_{L^\infty(B_{5/4})}\le CL.
		\]
		Extend $f_j,f$ by zero to $D$. Then
		\[
		\int_Df_j\varphi\dd X\to\int_Df\varphi\dd X,
		\qquad \varphi\in C_c^\infty(B_{5/4}).
		\]

		\emph{Step 4: The limiting equation and convergence of the energies.}
		
		For $\varphi\in C_c^\infty(B_{5/4})$,
		the first variation of $\mathscr G_j$
		at $v_j$ in the direction $\varphi$ is defined by
		\[
		\delta\mathscr G_j(v_j)[\varphi]
		:=
		\left.
		\frac{\dd}{\dd t}
		\mathscr G_j(v_j+t\varphi)
		\right|_{t=0}.
		\]
		Since
		\[
		\frac{\dd}{\dd t}D_\Phi(a;b+tc)
		=
		[\Jp(a+b+tc)-\Jp(a)]c,
		\]
		we obtain
		\[
		\begin{aligned}
			\delta\mathscr G_j(v_j)[\varphi]
			&=(1-s_j)\lambda_{Q_j}^{2-p}
			\iint_{D\times D}
			\frac{
				\mathcal D_{Q_j}v_j(X,Y)
				\delta\varphi(X,Y)
			}{
				|X-Y|^{n+s_jp}
			}\dd X\dd Y.
		\end{aligned}
		\]
		By the regional equation established in Step~1,
		\begin{equation}
			\delta\mathscr G_j(v_j)[\varphi]
			=\int_D f_j\varphi\dd X,
			\qquad
			\varphi\in C_c^\infty(B_{5/4}).
			\label{eq:regional-euler-large-slope}
		\end{equation}
		
		For each fixed $j$,
		\cite[Theorem~1.1]{BiswasTopp2025}
		gives $v_j\in C^{0,1}_{\rm loc}(B_{5/4})$.
		Since
		\[
		\int_0^1 r^{p-s_jp-1}\dd r<\infty,
		\]
		smooth approximation extends
		\eqref{eq:regional-euler-large-slope}
		to every compactly supported Lipschitz
		function in $B_{5/4}$.
		
		By the convexity of $\Phi$, we have
		\[
		D_\Phi(a;b+c)-D_\Phi(a;b)
		\ge[\Jp(a+b)-\Jp(a)]c.
		\]
		Taking
		\[
		a=Q_j\cdot(X-Y),\qquad
		b=\delta v_j(X,Y),\qquad
		c=\delta\psi(X,Y),
		\]
		and integrating over $D\times D$, we obtain
		\[
		\begin{aligned}
			\mathscr G_j(v_j+\psi)-\mathscr G_j(v_j)
			\ge\delta\mathscr G_j(v_j)[\psi]=\int_D f_j\psi\dd X
		\end{aligned}
		\]
		for every compactly supported Lipschitz
		function $\psi$ in $B_{5/4}$.

		Choose
		\[
		0<r<r_0<r_1<r'<r''<5/4.
		\]
		Let $\chi\in C_c^\infty(B_{r''})$ satisfy
		$0\le\chi\le1$ and
		$\chi=1$ in a neighborhood of $\overline{B_{r'}}$.
		Define
		\[
		\widetilde v(X):=
		\begin{cases}
			\chi(X)v(X),&X\in D,\\
			0,&X\in\R^n\setminus D,
		\end{cases}
		\]
		and set
		\[
		v^{(k)}:=\widetilde v*\rho_{\varepsilon_k},
		\qquad \varepsilon_k\to0,
		\]
		where $\rho_{\varepsilon_k}$ is a standard mollifier. 
		Since $\chi=1$ in a neighborhood of
		$\overline{B_{r'}}$ and $\widetilde v$
		has finite limiting energy,
		standard mollification gives
		\[
		\|v^{(k)}\|_{L^\infty(B_{r'})}
		\le \|v\|_{L^\infty(B_{r''})},
		\qquad
		v^{(k)}\to v
		\quad\text{in }L^p(B_{r'})\cap L^m(B_{r'}).
		\]
		Moreover, the convergence also holds
		in the limiting energy:
		\[
		\mathscr G_\infty(v^{(k)}-v;B_{r'}^2)
		\longrightarrow0.
		\]
		Consequently,
		\[
		\mathscr G_\infty(v^{(k)};B_\rho^2)
		\longrightarrow
		\mathscr G_\infty(v;B_\rho^2),
		\qquad \rho=r_1,r'.
		\]
		
		For each fixed $k$, the function $v^{(k)}$
		is independent of $j$, while
		$\mathscr G_j$ depends on $s_j$ and $Q_j$.
		Thus the convergence established in Step~2
		for smooth functions gives
		\[
		\mathscr G_j(v^{(k)};B_\rho^2)
		\longrightarrow
		\mathscr G_\infty(v^{(k)};B_\rho^2)
		\qquad(j\to\infty),
		\]
		for $\rho=r_1,r'$. By the preceding convergence results, we may
		choose integers $k_j\to\infty$ sufficiently
		slowly and define
		\[
		z_j:=v^{(k_j)}
		\qquad\text{in }B_{r'}.
		\]
		Then
		\[
		z_j\to v
		\qquad\text{in }L^p(B_{r'})\cap L^m(B_{r'}),
		\]
		and
		\[
		\mathscr G_j(z_j;B_\rho^2)
		\longrightarrow
		\mathscr G_\infty(v;B_\rho^2),
		\qquad \rho=r_1,r'.
		\]
		Since $v_j\to v$ in $L^p(B_{r'})$
		and both $z_j$ and $v_j$ are uniformly
		bounded in $B_{r'}$, we also have
		\[
		z_j-v_j\to0
		\qquad\text{in }L^p(B_{r'})\cap L^m(B_{r'}).
		\]
		In particular,
		\[
		\sup_j\mathscr G_j(z_j;B_{r'}^2)<\infty.
		\]
		We extend $z_j$ to $D\setminus B_{r'}$
		by setting $z_j=v_j$ there.

		For an integer $N\ge1$, define
		\[
		\rho_i:=r_0+\frac{i}{N}(r_1-r_0),
		\qquad i=0,\ldots,N,
		\]
		and
		\[
		C_i:=B_{\rho_i}\setminus\overline{B_{\rho_{i-1}}},
		\qquad
		\Gamma_i:=(C_i\times D)\cup(D\times C_i),
		\qquad 1\le i\le N.
		\]
		Since the sets $C_i$ are pairwise disjoint,
		each point $X\in D$ belongs to at most one
		$C_i$. Therefore, for every $(X,Y)\in D^2$,
		\[
		\sum_{i=1}^N\mathbf1_{\Gamma_i}(X,Y)\le2.
		\]
		Indeed, $(X,Y)\in\Gamma_i$ only if
		$X\in C_i$ or $Y\in C_i$.
		
		Since $D_\Phi(a;b)\ge0$ by the convexity
		of $\Phi$, we obtain
		\[
		\begin{aligned}
			\sum_{i=1}^N\mathscr G_j(v_j;\Gamma_i)&=(1-s_j)\lambda_{Q_j}^{2-p}
			\iint_{D^2}
			\frac{
				D_\Phi(Q_j\cdot(X-Y);\delta v_j(X,Y))
			}{
				|X-Y|^{n+s_jp}
			}
			\sum_{i=1}^N\mathbf1_{\Gamma_i}(X,Y)
			\,\dd X\dd Y\\
			&\le2\mathscr G_j(v_j;D^2)
			\le C.
		\end{aligned}
		\]
		Here $C$ is independent of $j$ and $N$.
		
		For $x\in B_{r_1}$ and $y\in D\setminus B_{r'}$,
		\[
		\delta z_j(x,y)=\delta v_j(x,y)+(z_j-v_j)(x),\qquad
		|x-y|\ge r'-r_1>0.
		\]
		The separated-pair estimate in Lemma~\ref{lem:uniform-shifted-bregman-gluing} give
		\[
		\begin{aligned}
			\sum_{i=1}^N\mathscr G_j(z_j;\Gamma_i)
			&\le2\mathscr G_j(z_j;B_{r'}^2)+C\mathscr G_j(v_j;D^2)+C_{r'-r_1}\bigl(
			\|z_j-v_j\|_{L^p(B_{r'})}^p+
			\|z_j-v_j\|_{L^m(B_{r'})}^2\bigr)\le C.
		\end{aligned}
		\]
		Along a subsequence realizing $\limsup_j\mathscr G_j(v_j;B_r^2)$,
		choose a minimizing index and pass to a further subsequence:
		\[
		i_j\in\operatorname*{arg\,min}_{1\le i\le N}
		\bigl[\mathscr G_j(z_j;\Gamma_i)+\mathscr G_j(v_j;\Gamma_i)\bigr],
		\qquad i_j=i(N).
		\]
		For this fixed $i=i(N)$,
		\begin{equation}
			\mathscr G_j(z_j;\Gamma_i)+\mathscr G_j(v_j;\Gamma_i)
			\le\frac CN.
			\label{eq:small-energy-subcollar}
		\end{equation}
		
		Take
		\[
		\eta(X):=
		\begin{cases}
			1,&|X|\le\rho_{i-1},\\
			\dfrac{\rho_i-|X|}{\rho_i-\rho_{i-1}},
			&\rho_{i-1}<|X|<\rho_i,\\
			0,&|X|\ge\rho_i,
		\end{cases}
		\qquad w_j:=\eta z_j+(1-\eta)v_j.
		\]
		Then
		\[
		\operatorname{supp}(w_j-v_j)\subset\overline B_{\rho_i}
		\Subset B_{5/4},\qquad
		w_j-v_j\in C_c^{0,1}(B_{5/4}).
		\]
		For fixed $N,\varepsilon$, define
		\[
		\begin{aligned}
			R_{j,N,\varepsilon}:={}&C_{N,\varepsilon}
			\bigl(\|z_j-v_j\|_{L^p(B_{r'})}^p
			+\|z_j-v_j\|_{L^m(B_{r'})}^2\bigr)+\|f_j\|_{L^\infty(B_{5/4})}
			\|z_j-v_j\|_{L^1(B_{r'})}\longrightarrow0.
		\end{aligned}
		\]
		Local minimality, Lemma~\ref{lem:uniform-shifted-bregman-gluing}, and
		\eqref{eq:small-energy-subcollar} yield
		\[
		\begin{aligned}
			\mathscr G_j(v_j;B_r^2)
			&\le\mathscr G_j(v_j;I_\eta^2)\\
			&\le\mathscr G_j(z_j;I_\eta^2)+\frac CN
			+\varepsilon\mathscr G_j(v_j;\Xi_\eta)+R_{j,N,\varepsilon}\\
			&\le\mathscr G_j(z_j;B_{r_1}^2)
			+\frac CN+C\varepsilon+R_{j,N,\varepsilon}.
		\end{aligned}
		\]
		Taking limits in the order
		\[
		j\to\infty,\qquad N\to\infty,\qquad
		\varepsilon\to 0,\qquad r_1\to r,
		\]
		through the chosen recovery radii gives
		\begin{equation}
			\limsup_{j\to\infty}\mathscr G_j(v_j;B_r^2)
			\le\mathscr G_\infty(v;B_r^2).
			\label{eq:localized-relative-gluing-limsup}
		\end{equation}
		Together with \eqref{eq:infinite-slope-form-liminf}, the same
		construction in any ball $B\Subset B_{5/4}$ gives
		\[
		\mathscr G_j(v_j;B^2)\to\mathscr G_\infty(v;B^2).
		\]

		Fix a Lipschitz subdomain
		$V\Subset B_R\Subset B_{5/4}$ and $\delta>0$.
		By Lemma~\ref{lem:large-slope-scalar-angular},
		for some $a>1$,
		\[
		\sup_j\int_{\mathbb S^{n-1}}
		|e_j\cdot\omega|^{(p-2)a}
		\,\dd\omega<\infty.
		\]
		Since $v_j$ are uniformly bounded,
		this estimate implies that the integrands
		defining $\mathscr G_j(v_j;\cdot)$ are
		uniformly integrable on
		$B_R^2\cap\{|X-Y|\ge\delta\}$.
		Their almost everywhere convergence
		and Vitali's theorem give
		\[
		\mathscr G_j(v_j;E)
		\longrightarrow
		\mathscr G_\infty(v;E)
		\]
		for every fixed measurable set
		$E\subset B_R^2\cap\{|X-Y|\ge\delta\}$. 
		Subtracting the convergent energies on
		$B_R^2\cap\{|X-Y|\ge\delta\}$
		from those on $B_R^2$, we obtain
		\[
		\begin{aligned}
			\lim_{j\to\infty}
			\mathscr G_j(v_j;B_R^2\cap\{|X-Y|<\delta\})=
			\mathscr G_\infty(v;B_R^2\cap\{|X-Y|<\delta\}).
		\end{aligned}
		\]
		Since $\mathscr G_\infty(v;B_R^2)<\infty$,
		the right-hand side tends to zero
		as $\delta\to0$.
		
		Since $V\subset B_R$ and the energy
		densities are nonnegative, we obtain
		\[
		\begin{aligned}
			&\limsup_{j\to\infty}
			\bigl|
			\mathscr G_j(v_j;V^2)
			-\mathscr G_\infty(v;V^2)
			\bigr|\le
			2\mathscr G_\infty
			(v;B_R^2\cap\{|X-Y|<\delta\}).
		\end{aligned}
		\]
		Letting $\delta\to0$ gives
		\[
		\mathscr G_j(v_j;V^2)
		\longrightarrow
		\mathscr G_\infty(v;V^2).
		\]
		
		To identify the equation, fix $\varphi\in C_c^\infty(B_{5/4})$ and set
		\[
		F_{j,Z}(t):=\lambda_{Q_j}^{2-p}
		[\Jp(Q_j\cdot Z+t)-\Jp(Q_j\cdot Z)].
		\]
		By Taylor's formula and \eqref{eq:infinite-slope-strong-compactness},
		\[
		F_{j,X-Y}(\delta v_j)
		\to(p-1)|e\cdot(X-Y)|^{p-2}\delta v
		\quad\text{a.e. in }D^2.
		\]
		For every $\delta>0$, the angular flux estimate gives
		\[
		\sup_j\left\|
		\frac{(1-s_j)F_{j,X-Y}(\delta v_j)\delta\varphi}
		{|X-Y|^{n+s_jp}}
		\right\|_{L^a(D^2\cap\{|X-Y|\ge\delta\})}
		\le C_{\delta,\varphi,L}.
		\]
		Vitali's theorem therefore implies
		\[
		\begin{aligned}
			&(1-s_j)\iint_{D^2\cap\{|X-Y|\ge\delta\}}
			\frac{F_{j,X-Y}(\delta v_j)\delta\varphi}
			{|X-Y|^{n+s_jp}}\dd X\!\dd Y\\
			&\quad\longrightarrow(p-1)(1-s)
			\iint_{D^2\cap\{|X-Y|\ge\delta\}}
			\frac{|e\cdot(X-Y)|^{p-2}\delta v\,\delta\varphi}
			{|X-Y|^{n+sp}}\dd X\!\dd Y.
		\end{aligned}
		\]
		By \eqref{eq:uncapped-shifted-young} and the scalar Bregman equivalence,
		\[
		|F_{j,Z}(t)b|
		\le C\varepsilon\lambda_{Q_j}^{2-p}D_\Phi(Q_j\cdot Z;t)
		+C_\varepsilon\mathcal H_{Q_j,Z}(b).
		\]
		Taking $t=\delta v_j$, $b=\delta\varphi$, and using
		\eqref{eq:shifted-modular-angular-upper}, we obtain, for $0<\delta<1$, $\beta=p(1-s_+)>0$,
		\[
		\begin{aligned}
			(1-s_j)\iint_{D^2\cap\{|X-Y|<\delta\}}
			\frac{|F_{j,X-Y}(\delta v_j)\delta\varphi|}
			{|X-Y|^{n+s_jp}}\dd X\!\dd Y&\le C\varepsilon L+C_{\varepsilon,\varphi}
			(1-s_j)\int_0^\delta r^{p-s_jp-1}\dd r\\
			&\le C\varepsilon L+C_{\varepsilon,\varphi}\delta^\beta.
		\end{aligned}
		\]
		Hence
		\[
		\lim_{\delta\to 0}\sup_j
		(1-s_j)\iint_{D^2\cap\{|X-Y|<\delta\}}
		\frac{|F_{j,X-Y}(\delta v_j)\delta\varphi|}
		{|X-Y|^{n+s_jp}}\dd X\!\dd Y=0.
		\]
		Step 2 and exhaustion of $D$ give $\mathscr G_\infty(v;D^2)\le L$.
		Cauchy--Schwarz yields
		\[
		\begin{aligned}
			&\iint_{D^2\cap\{|X-Y|<\delta\}}
			\frac{|e\cdot(X-Y)|^{p-2}|\delta v\,\delta\varphi|}
			{|X-Y|^{n+sp}}\dd X\!\dd Y\\
			&\quad\le C\mathscr G_\infty(v;D^2\cap\{|X-Y|<\delta\})^{1/2}
			\mathscr G_\infty(\varphi;D^2\cap\{|X-Y|<\delta\})^{1/2}
			\longrightarrow0.
		\end{aligned}
		\]
		Letting $j\to\infty$, then $\delta\to0$, in the weak equation
		proves \eqref{eq:large-slope-limit-form} and
		\eqref{eq:large-slope-limit-equation}.

		Set $\sigma=\tau_s/2\in(1/2,1)$. Since
		$\tau_s=sp-p+2$, the limiting kernel can be written as
		\[
		\frac{|e\cdot(X-Y)|^{p-2}}{|X-Y|^{n+sp}}
		=\frac{|e\cdot\omega|^{p-2}}{|X-Y|^{n+2\sigma}},
		\qquad
		\omega=\frac{X-Y}{|X-Y|}.
		\]
		As $1<p<2$, we have $|e\cdot\omega|^{p-2}\ge1$
		for almost every $\omega\in\mathbb S^{n-1}$. Moreover,
		Lemma~\ref{lem:large-slope-scalar-angular} gives
		\[
		\int_{\mathbb S^{n-1}}|e\cdot\omega|^{p-2}\dd\omega\le C,
		\qquad
		\inf_{\nu\in\mathbb S^{n-1}}
		\int_{\mathbb S^{n-1}}
		|\nu\cdot\omega|^{2\sigma}
		|e\cdot\omega|^{p-2}\dd\omega\ge c>0.
		\]
		Thus the kernel satisfies the ellipticity conditions
		in \cite{RosOtonSerra2016}.
		
		Let $v_D$ be the extension of $v$ by zero outside $D$.
		For $E\Subset B_{5/4}$, set
		$d_E=\operatorname{dist}(E,D^c)>0$.
		Since $|X-Y|\ge d_E$ for $X\in E$ and $Y\in D^c$,
		polar coordinates give
		\[
		\begin{aligned}
			&\operatorname*{ess\,sup}_{X\in E}
			\int_{D^c}
			\frac{|e\cdot(X-Y)|^{p-2}|v(X)|}
			{|X-Y|^{n+sp}}\dd Y\le
			C\|v\|_{L^\infty(D)}
			\int_{d_E}^{\infty}r^{-1-\tau_s}\dd r
			\le C d_E^{-\tau_s}\|v\|_{L^\infty(D)}.
		\end{aligned}
		\]
		By \eqref{eq:large-slope-limit-form}, the function $v_D$
		satisfies the corresponding full-space equation in
		$B_{5/4}$ with right-hand side
		\[
		f_D(X)=f(X)+2(p-1)(1-s)v(X)
		\int_{D^c}
		\frac{|e\cdot(X-Y)|^{p-2}}
		{|X-Y|^{n+sp}}\dd Y.
		\]
		The preceding estimate shows that
		$f_D\in L^\infty_{\rm loc}(B_{5/4})$.
		Since $v_D$ is bounded on $\R^n$,
		\cite[Theorem~1.1(a)]{RosOtonSerra2016}
		implies that $v$ has a continuous representative
		in $B_{5/4}$.
		
		We now apply
		Lemma~\ref{lem:shift-uniform-fixed-level-clearing}
		with $D'=B\Subset B_{5/4}$, $D_0=B_{5/4}$,
		and $D=B_{4/3}$. The uniform bounds
		\eqref{eq:large-slope-crossed-source-bound},
		the strong convergence
		\eqref{eq:infinite-slope-strong-compactness},
		and the continuity of $v$ yield
		\begin{equation}
			v_j\to v
			\qquad\text{locally uniformly in }B_{5/4}.
			\label{eq:infinite-slope-local-uniform-convergence}
		\end{equation}
		
		\emph{Step 5: subtraction of an affine map.} Fix $m(X)=c+b\cdot X$. Since
		\[
		(Q_j+b)\cdot X+(v_j-m)(X)
		=Q_j\cdot X+v_j(X)-c,
		\]
		the function $v_j-m$ solves the equation shifted by
		$Q_j+b$. Moreover, as $q_j=|Q_j|\to\infty$
		and $Q_j/q_j\to e$, we have
		\[
		\frac{|Q_j+b|}{q_j}\to1,\qquad
		\frac{Q_j+b}{|Q_j+b|}\to e,\qquad
		\frac{\lambda_{Q_j+b}}{\lambda_{Q_j}}\to1.
		\]
		In particular, after discarding finitely many indices,
		we may assume that
		$\lambda_{Q_j+b}\le2\lambda_{Q_j}$.
		
		We first use the following inequality:
		\begin{equation}
			D_\Phi(a+c;t-c)
			\le C_p\bigl[D_\Phi(a;t)+D_\Phi(a;c)\bigr].
			\label{eq:bregman-change-affine-base}
		\end{equation}
		Indeed, for $V_p(t)=|t|^{(p-2)/2}t$ with
		$V_p(0)=0$, the scalar Bregman estimate gives
		$D_\Phi(a;t)\asymp|V_p(a+t)-V_p(a)|^2$.
		Hence the triangle inequality yields
		\[
		\begin{aligned}
			D_\Phi(a+c;t-c)
			&\le C_p|V_p(a+t)-V_p(a+c)|^2\\
			&\le C_p\bigl(
			|V_p(a+t)-V_p(a)|^2
			+|V_p(a+c)-V_p(a)|^2
			\bigr)\\
			&\le C_p\bigl[D_\Phi(a;t)+D_\Phi(a;c)\bigr].
		\end{aligned}
		\]
		
		For $h=X-Y=r\omega$, Lemma~\ref{lem:large-slope-scalar-angular}
		and $q_j\ge1$ give
		\[
		\lambda_{Q_j}^{2-p}
		D_\Phi(Q_j\cdot h;b\cdot h)
		\le C|b|^2r^p|e_j\cdot\omega|^{p-2}
		\]
		for almost every $\omega\in\mathbb S^{n-1}$.
		Since
		\[
		\sup_j\int_{\mathbb S^{n-1}}
		|e_j\cdot\omega|^{p-2}\dd\omega\le C,
		\qquad
		(1-s_j)\int_0^4r^{p-s_jp-1}\dd r
		=\frac{4^{p-s_jp}}p\le C,
		\]
		applying \eqref{eq:bregman-change-affine-base}
		with $a=Q_j\cdot h$, $c=b\cdot h$,
		and $t=\delta v_j$ gives
		\[
		\begin{aligned}
			\mathcal E_{s_j,Q_j+b}(v_j-m;B_2)
			&\le C\mathcal E_{s_j,Q_j}(v_j;B_2)
			+C|b|^2\le C(L+|b|^2).
		\end{aligned}
		\]
		Also,
		$\|v_j-m\|_{L^\infty(B_{4/3})}
		\le L+|c|+\frac43|b|$.
		
		It remains to estimate the exterior terms.
		By the definition of $\mathcal D_Q$, we have
		\[
		\begin{aligned}
			\mathcal D_{Q_j+b}(v_j-m)(X,Y)
			=\mathcal D_{Q_j}v_j(X,Y)-\bigl[
			\Jp((Q_j+b)\cdot h)-\Jp(Q_j\cdot h)
			\bigr].
		\end{aligned}
		\]
		The scalar estimate in
		Lemma~\ref{lem:large-slope-scalar-angular} gives
		\[
		\lambda_{Q_j}^{2-p}
		|\Jp((Q_j+b)\cdot h)-\Jp(Q_j\cdot h)|
		\le C|b|r^{p-1}|e_j\cdot\omega|^{p-2}.
		\]
		In both exterior integrals defining
		$\mathcal T^{\rm pt}_{s_j,Q_j+b}$,
		we have $r=|X-Y|\ge1/12$.
		Furthermore, the assumption $sp>p-1$
		and the convergence $s_j\to s$ imply
		\[
		\sup_j\int_{1/12}^{\infty}
		r^{p-2-s_jp}\dd r
		=
		\sup_j
		\frac{(1/12)^{p-1-s_jp}}{s_jp-p+1}
		<\infty.
		\]
		Using again the integrability of
		$|e_j\cdot\omega|^{p-2}$ and
		$\lambda_{Q_j+b}\le2\lambda_{Q_j}$,
		we obtain
		\[
		\mathcal T^{\rm pt}_{s_j,Q_j+b}(v_j-m)
		\le C\mathcal T^{\rm pt}_{s_j,Q_j}(v_j)
		+C|b|
		\le C_m(L+1).
		\]
		
		Thus the pair $(Q_j+b,v_j-m)$ satisfies
		the uniform bounds required in
		Steps~1--4. Since $m$ is fixed,
		$v_j-m\to v-m$ locally uniformly
		in $B_{5/4}$, and
		$(Q_j+b)/|Q_j+b|\to e$.
		Applying the energy convergence proved
		above, we conclude that, for every
		ball $B\Subset B_{5/4}$,
		\[
		\begin{aligned}
			&\mathcal E_{s_j,Q_j+b}(v_j-m;B)\\
			&\quad\longrightarrow
			\frac{(p-1)(1-s)}2
			\iint_{B\times B}
			\frac{|e\cdot(X-Y)|^{p-2}
				|\delta(v-m)(X,Y)|^2}
			{|X-Y|^{n+sp}}\dd X\dd Y.
		\end{aligned}
		\]
		This proves
		\eqref{eq:infinite-slope-recentered-energy-convergence}.
		
	\end{proof}

	\subsection{Improvement of flatness for large slopes}
	\label{subsec:large-slope-improvement}

	We now establish an affine approximation estimate
	for solutions with large slopes.
	By Proposition~\ref{prop:fixed-order-infinite-slope-compactness},
	the limiting function solves a linear nonlocal equation
	and satisfies a $C^{1,\bar\alpha}$ estimate.
	Using the convergence of the solutions and their energies,
	we obtain an affine approximation with constants
	independent of the slope.

	\begin{proposition}
		\label{prop:large-slope-affine-improvement}
		Let $n\ge2$, $1<p<2$, $s\in(0,1)$ with
		$sp>p-1$, and fix $0<\alpha<sp-p+1$ and $L\ge1$.
		Choose $\bar\alpha$ such that
		\begin{equation}
			\alpha<\bar\alpha<sp-p+1.
			\label{eq:large-slope-intermediate-exponent}
		\end{equation}
		There exist constants $C_{\rm L},C\ge1$,
		depending only on $n,p,s,\alpha,\bar\alpha,L$,
		such that, for every $0<\rho<1/8$ satisfying
		\begin{equation}
			C_{\rm L}\rho^{\bar\alpha-\alpha}\le1,
			\label{eq:prescribed-large-slope-radius-main}
		\end{equation}
		there is $Q_{\rm th}=Q_{\rm th}(\rho)<\infty$
		with the following property.
		
		Suppose that $v$ solves the equation shifted
		by $Q$ in $B_2$, $|Q|\ge Q_{\rm th}$, and
		\[
		\fint_{B_2}|v|^p\dd X\le1,\qquad
		\mathcal E_{s,Q}(v;B_2)\le1,\qquad
		\mathcal T^{\rm pt}_{s,Q}(v)\le L.
		\]
		Then there exists an affine function
		$m(X)=c+b\cdot X$, with $|c|+|b|\le C$,
		such that
		\begin{equation}
			\mathfrak a_{s,Q}(v,m;r)
			+r^{-1}\|v-m\|_{L^\infty(B_{2r})}
			\le C_{\rm L}r^{\bar\alpha},
			\qquad \rho\le r\le1/2.
			\label{eq:large-slope-intermediate-block}
		\end{equation}
		In particular,
		\begin{equation}
			\mathfrak a_{s,Q}(v,m;\rho)\le\rho^\alpha.
			\label{eq:large-slope-affine-improvement}
		\end{equation}
		The constants $C_{\rm L}$ and $C$ are
		independent of $\rho$ and $Q$.
	\end{proposition}
	
	\begin{proof}
		We first obtain estimates for the limiting
		equation. Let $Q_j=q_je_j$, where $q_j\to\infty$,
		and let $v_j$ satisfy the assumptions.
		By Lemma~\ref{lem:shift-uniform-local-boundedness}
		and Proposition~\ref{prop:fixed-order-infinite-slope-compactness},
		after passing to a subsequence,
		\[
		e_j\to e\in\mathbb S^{n-1},\qquad
		v_j\to v\quad\text{locally uniformly in }B_{5/4}.
		\]
		Moreover, $v$ solves $L^D_{s,e}v=f$ in $B_{5/4}$,
		where $D=B_{4/3}$ and
		\[
		\|v\|_{L^\infty(D)}+\|f\|_{L^\infty(B_{5/4})}\le C.
		\]
		
		Set $\tau_s=sp-p+2\in(1,2)$.
		As shown in the proof of
		Proposition~\ref{prop:fixed-order-infinite-slope-compactness},
		the zero extension of $v$ satisfies
		a full-space linear equation with bounded
		right-hand side in $B_{19/16}$.
		Since $1+\bar\alpha<\tau_s$,
		\cite[Theorem~1.1(a)]{RosOtonSerra2016} gives
		\[
		\|v\|_{C^{1,\bar\alpha}(B_{9/8})}\le C.
		\]
		Let $m(X)=v(0)+\nabla v(0)\cdot X$.
		For $0<r\le1/2$, we have
		\[
		|m(0)|+|\nabla m|\le C,\qquad
		r^{-1}\|v-m\|_{L^\infty(B_{2r})}
		+\mathfrak Z(v,m;r)^{1/p}\le Cr^{\bar\alpha}.
		\]
		Moreover, for $X,Y\in B_{2r}$,
		\[
		|\delta(v-m)(X,Y)|\le Cr^{\bar\alpha}|X-Y|.
		\]
		Using Lemma~\ref{lem:large-slope-scalar-angular}
		and polar coordinates, we obtain
		\[
		\begin{aligned}
			&(1-s)r^{sp-p-n}
			\iint_{B_{2r}^2}
			\frac{|e\cdot(X-Y)|^{p-2}|\delta(v-m)|^2}
			{|X-Y|^{n+sp}}\dd X\dd Y\\
			&\qquad\le Cr^{sp-p+2\bar\alpha}(1-s)
			\int_0^{4r}t^{p-sp-1}\dd t
			\le Cr^{2\bar\alpha}.
		\end{aligned}
		\]
		Thus we may fix $C_*\ge1$, depending only
		on the stated parameters, such that
		\[
		|m(0)|+|\nabla m|\le C_*,
		\qquad
		r^{-1}\|v-m\|_{L^\infty(B_{2r})}
		+\mathfrak Z(v,m;r)^{1/p}\le C_*r^{\bar\alpha},
		\]
		and
		\begin{equation}
			\begin{aligned}
				&(1-s)r^{sp-p-n}
				\iint_{B_{2r}^2}
				\frac{|e\cdot(X-Y)|^{p-2}|\delta(v-m)|^2}
				{|X-Y|^{n+sp}}\dd X\dd Y
				\le C_*^2r^{2\bar\alpha}.
			\end{aligned}
			\label{eq:large-slope-limit-full-block}
		\end{equation}
		These bounds hold for the limit of every
		sequence satisfying the assumptions,
		with the same constant $C_*$.
		
		Set $C_{\rm L}=3C_*$ and fix
		$0<\rho<1/8$. We claim that there is a
		threshold $Q_{\rm th}(\rho)$ such that
		the required estimate holds for some
		affine function $m$ with
		$|m(0)|+|\nabla m|\le C_*$.
		
		Suppose that the claim is false.
		Then, for every $j$, there are $Q_j$ and
		$v_j$ satisfying the assumptions with
		$|Q_j|\ge j$, such that no affine function
		$m$ with $|m(0)|+|\nabla m|\le C_*$
		satisfies the claimed estimate for all
		$r\in[\rho,1/2]$.
		
		By Proposition~\ref{prop:fixed-order-infinite-slope-compactness},
		we may assume that $v_j\to v$ locally
		uniformly in $B_{5/4}$.
		Choose the affine function
		\[
		m(X)=v(0)+\nabla v(0)\cdot X,\qquad b=\nabla m.
		\]
		The preceding estimates give
		$|m(0)|+|\nabla m|\le C_*$.
		Hence, by the choice of $v_j$, there is
		$r_j\in[\rho,1/2]$ such that
		\[
		\mathfrak a_{s,Q_j}(v_j,m;r_j)
		+r_j^{-1}\|v_j-m\|_{L^\infty(B_{2r_j})}
		>C_{\rm L}r_j^{\bar\alpha}.
		\]
		After passing to a further subsequence,
		we may assume that $r_j\to r\in[\rho,1/2]$.
		
		The local uniform convergence gives
		\[
		\begin{aligned}
			\mathfrak Z(v_j,m;r_j)&\to\mathfrak Z(v,m;r),\\
			r_j^{-1}\|v_j-m\|_{L^\infty(B_{2r_j})}
			&\to r^{-1}\|v-m\|_{L^\infty(B_{2r})}.
		\end{aligned}
		\]
		We next consider the energies. For any fixed
		$0<\varepsilon<\min\{r/2,1/16\}$
		and all sufficiently large $j$,
		\[
		\begin{aligned}
			&\mathcal E_{s,Q_j+b}(v_j-m;B_{2(r-\varepsilon)})
			\le\mathcal E_{s,Q_j+b}(v_j-m;B_{2r_j})\le\mathcal E_{s,Q_j+b}
			(v_j-m;B_{2(r+\varepsilon)}).
		\end{aligned}
		\]
		Applying
		\eqref{eq:infinite-slope-recentered-energy-convergence}
		at the two fixed radii and then letting
		$\varepsilon\to0$, we obtain
		\[
		\mathcal E_{s,Q_j+b}(v_j-m;B_{2r_j})
		\longrightarrow\mathscr G_\infty(v-m;B_{2r}^2).
		\]
		
		Fix $A>C_*r^{\bar\alpha}$.
		Since $|Q_j+b|\to\infty$,
		\[
		\frac{\lambda_{Q_j+b}}{|Q_j+b|+A}\longrightarrow1.
		\]
		By the definitions of $\mathfrak B_{s,Q_j}$
		and $\mathcal E_{s,Q_j+b}$,
		\[
		\begin{aligned}
			&\frac{\mathfrak B_{s,Q_j}(v_j,m;r_j)}
			{(|Q_j+b|+A)^{p-2}}=r_j^{sp-p-n}
			\left(\frac{\lambda_{Q_j+b}}{|Q_j+b|+A}\right)^{p-2}
			\mathcal E_{s,Q_j+b}(v_j-m;B_{2r_j}).
		\end{aligned}
		\]
		The energy convergence and
		\eqref{eq:large-slope-limit-full-block} yield
		\[
		\begin{aligned}
			&\lim_{j\to\infty}
			\frac{\mathfrak B_{s,Q_j}(v_j,m;r_j)}
			{(|Q_j+b|+A)^{p-2}}=r^{sp-p-n}\mathscr G_\infty(v-m;B_{2r}^2)
			\le\frac{p-1}{2}C_*^2r^{2\bar\alpha}<A^2.
		\end{aligned}
		\]
		Similarly,
		\[
		\mathfrak Z(v_j,m;r_j)\to\mathfrak Z(v,m;r)
		\le C_*^pr^{p\bar\alpha}<A^p.
		\]
		Thus, for all sufficiently large $j$,
		both conditions in
		\eqref{eq:block-intrinsic-amplitude}
		hold with amplitude $A$.
		Consequently,
		$\mathfrak a_{s,Q_j}(v_j,m;r_j)\le A$.
		As $A>C_*r^{\bar\alpha}$ is arbitrary,
		\[
		\limsup_{j\to\infty}
		\mathfrak a_{s,Q_j}(v_j,m;r_j)
		\le C_*r^{\bar\alpha}.
		\]
		
		Combining this with the convergence of
		the $L^\infty$ term, we obtain
		\[
		\begin{aligned}
			C_{\rm L}r^{\bar\alpha}
			&\le\limsup_{j\to\infty}
			\left[\mathfrak a_{s,Q_j}(v_j,m;r_j)
			+r_j^{-1}\|v_j-m\|_{L^\infty(B_{2r_j})}\right]\\
			&\le C_*r^{\bar\alpha}
			+r^{-1}\|v-m\|_{L^\infty(B_{2r})}\\
			&\le2C_*r^{\bar\alpha}
			<C_{\rm L}r^{\bar\alpha},
		\end{aligned}
		\]
		a contradiction.
		
		Therefore, for every fixed $\rho\in(0,1/8)$,
		there exists $Q_{\rm th}(\rho)$ such that
		\eqref{eq:large-slope-intermediate-block}
		holds with $C_{\rm L}=3C_*$ and $C=C_*$.
		These constants are independent of $\rho$.
		
		Finally, for any $\rho$ satisfying
		\eqref{eq:prescribed-large-slope-radius-main},
		we have
		\[
		\mathfrak a_{s,Q}(v,m;\rho)
		\le C_{\rm L}\rho^{\bar\alpha}
		\le\rho^\alpha,
		\]
		which proves
		\eqref{eq:large-slope-affine-improvement}.
	\end{proof}

	\subsection{Iteration for large slopes and rigidity of affine limits at infinity}
	\label{subsec:noncritical-blowdown}
	
	Proposition~\ref{prop:large-slope-affine-improvement}
	gives an affine approximation at a fixed scale
	when the slope is large. We next apply this
	estimate repeatedly at smaller scales. To do so,
	we need to keep the normalized $L^p$ norm,
	energy, and exterior terms uniformly bounded
	after each rescaling. The following lemma
	establishes these bounds and gives an affine
	approximation with an error of order $r^{1+\alpha}$.
	
	\begin{lemma}
		\label{lem:iterable-large-slope-state}
		Let \(0<\alpha<sp-p+1\) and \(L\ge1\).  There are
		\(\rho_\infty\in(0,1/16)\), \(Q_{\rm it}>0\), and \(C<\infty\),
		depending only on \(n,p,s,\alpha,L\), with the following property.
		Suppose that \(v\) solves the equation shifted by \(Q\) in \(B_2\),
		\(|Q|\ge Q_{\rm it}\), and
		\[
		\fint_{B_2}|v|^p\dd X\le1,\qquad
		\mathcal E_{s,Q}(v;B_2)\le1,\qquad
		\mathcal T^{\rm pt}_{s,Q}(v)\le L.
		\]
		Set \(r_k=\rho_\infty^k\) and \(H_k=r_k^\alpha\).
		There are affine maps \(m_k=A_k+B_k\cdot X\), with \(m_0=0\), such that
		\[
		w_k(X)=\frac{v(r_kX)-m_k(r_kX)}{r_kH_k},
		\qquad P_k=\frac{Q+B_k}{H_k}
		\]
		satisfy the same unit bounds for the averaged \(L^p\) norm and energy,
		and \(\mathcal T^{\rm pt}_{s,P_k}(w_k)\le2L\) for every \(k\).
		Moreover, \(|A_k|+|B_k|\le C\), and the maps \(m_k\) converge to an
		affine map \(m_\infty\) satisfying
		\begin{equation}
			\|v-m_\infty\|_{L^\infty(B_t)}
			\le Ct^{1+\alpha}\qquad(0<t\le1).
			\label{eq:large-slope-iteration-pointwise-expansion}
		\end{equation}
	\end{lemma}
	
	\begin{proof}
		Set $\gamma=sp-p+1>0$ and choose
		$\bar\alpha$ such that
		$\alpha<\bar\alpha<\gamma$.
		Let $L_{\rm bd}=2L$.
		
		We first choose the constants used in the iteration.
		Apply Proposition~\ref{prop:large-slope-affine-improvement}
		with $L=L_{\rm bd}$, and fix $C_{\rm L}$
		large enough to control both the affine approximation
		and the coefficients of the affine function.
		The constant $C_{\rm L}$ is independent of the radius.
		We also fix $C_{\rm T}\ge1$ sufficiently large
		for the estimates of the exterior integrals below.
		These constants depend only on the stated parameters.
		
		Since $\bar\alpha-\alpha>0$ and
		$\gamma-\alpha>0$, we can choose a dyadic
		number $\rho_\infty\in(0,1/16)$ sufficiently
		small so that
		\begin{equation}
			C_{\rm L}\rho_\infty^{\bar\alpha-\alpha}
			\le\frac12
			\label{eq:large-slope-iteration-endpoint-slack}
		\end{equation}
		and
		\begin{equation}
			\begin{aligned}
				&C_{\rm T}\bigl(
				\rho_\infty^{\bar\alpha-\alpha}
				+\rho_\infty^{\gamma-\alpha}\bigr)
				+C_{\rm T}L_{\rm bd}
				\rho_\infty^{\gamma-\alpha}
				\le L_{\rm bd}.
			\end{aligned}
			\label{eq:large-slope-iteration-tail-budget}
		\end{equation}
		Set $\eta_\infty=\rho_\infty^\alpha$.
		Proposition~\ref{prop:large-slope-affine-improvement}
		then gives a threshold $Q_{\rm bd}$
		corresponding to this fixed radius.
		
		For $k\ge0$, define
		\[
		r_k=\rho_\infty^k,\qquad
		H_k=\eta_\infty^k=r_k^\alpha,\qquad
		C_{\rm aff}=\frac{C_{\rm L}}{1-\eta_\infty}.
		\]
		Choose $Q_{\rm it}$ sufficiently large that
		$|Q|\ge Q_{\rm it}$ implies
		\begin{equation}
			|Q|\ge Q_{\rm bd}+C_{\rm aff}+2C_{\rm L}+2.
			\label{eq:large-slope-iteration-initial-slope-cutoff}
		\end{equation}
		
		We construct affine functions
		$m_k(X)=A_k+B_k\cdot X$, starting with $m_0=0$.
		At each step, define
		\begin{equation}
			w_k(X)=\frac{v(r_kX)-m_k(r_kX)}{r_kH_k},
			\qquad
			P_k=\frac{Q+B_k}{H_k}.
			\label{eq:large-slope-iteration-normalized-state}
		\end{equation}
		The function $w_k$ solves the equation shifted
		by $P_k$ in $B_2$.
		
		We prove by induction that, for every $k\ge0$,
		\begin{equation}
			\fint_{B_2}|w_k|^p\dd X\le1,\qquad
			\mathcal E_{s,P_k}(w_k;B_2)\le1,\qquad
			\mathcal T^{\rm pt}_{s,P_k}(w_k)\le L_{\rm bd},
			\label{eq:large-slope-iteration-reproduced-state}
		\end{equation}
		and
		\begin{equation}
			|B_k|\le C_{\rm L}\sum_{i=0}^{k-1}H_i
			\le C_{\rm aff}.
			\label{eq:large-slope-iteration-coefficient-bound}
		\end{equation}
		For $k=0$, we have $w_0=v$, $P_0=Q$,
		and $B_0=0$. Hence both assertions follow
		from the assumptions of the lemma.
		
		Suppose that they hold at level $k$.
		Since $H_k\le1$, the bound for $B_k$ and
		\eqref{eq:large-slope-iteration-initial-slope-cutoff}
		give
		\[
		|P_k|=\frac{|Q+B_k|}{H_k}
		\ge\frac{|Q|-C_{\rm aff}}{H_k}
		\ge Q_{\rm bd}+2C_{\rm L}+2.
		\]
		Thus Proposition~\ref{prop:large-slope-affine-improvement}
		applies to $w_k$ with slope $P_k$ and
		tail bound $L_{\rm bd}$.
		It gives an affine function
		$\mu_k(X)=c_k+b_k\cdot X$ satisfying
		$|c_k|+|b_k|\le C_{\rm L}$ and
		\begin{equation}
			\mathfrak a_{s,P_k}(w_k,\mu_k;t)
			+t^{-1}\|w_k-\mu_k\|_{L^\infty(B_{2t})}
			\le C_{\rm L}t^{\bar\alpha},
			\qquad \rho_\infty\le t\le1/2.
			\label{eq:large-slope-iteration-intermediate-block}
		\end{equation}
		In particular, by
		\eqref{eq:large-slope-iteration-endpoint-slack},
		\[
		\mathfrak a_{s,P_k}(w_k,\mu_k;\rho_\infty)
		\le C_{\rm L}\rho_\infty^{\bar\alpha}
		\le\frac{\eta_\infty}{2}.
		\]
		Therefore, by the definition of
		$\mathfrak a_{s,P_k}$, we may choose
		an admissible number
		$0<\widehat a_k<\eta_\infty$ such that
		\[
		\mathfrak Z(w_k,\mu_k;\rho_\infty)
		\le\widehat a_k^p
		\]
		and
		\[
		\mathfrak B_{s,P_k}(w_k,\mu_k;\rho_\infty)
		\le(|P_k+b_k|+\widehat a_k)^{p-2}
		\widehat a_k^2.
		\]
		
		We now define the next affine function by
		\[
		m_{k+1}(X)=m_k(X)
		+r_kH_k\mu_k\left(\frac{X}{r_k}\right).
		\]
		Since $\mu_k(X)=c_k+b_k\cdot X$, this gives
		\[
		A_{k+1}=A_k+r_kH_kc_k,\qquad
		B_{k+1}=B_k+H_kb_k.
		\]
		Consequently,
		\[
		|B_{k+1}|
		\le C_{\rm L}\sum_{i=0}^kH_i
		\le C_{\rm aff}.
		\]
		The corresponding function and slope are
		\[
		w_{k+1}(X)=
		\frac{w_k(\rho_\infty X)-\mu_k(\rho_\infty X)}
		{\rho_\infty\eta_\infty},
		\qquad
		P_{k+1}=\frac{P_k+b_k}{\eta_\infty}.
		\]
		These expressions agree with
		\eqref{eq:large-slope-iteration-normalized-state}
		at level $k+1$.
		
		We next verify the two unit bounds.
		By a change of variables and the definition
		of $\mathfrak Z$,
		\[
		\begin{aligned}
			\fint_{B_2}|w_{k+1}|^p\dd X
			=\frac{1}{\rho_\infty^p\eta_\infty^p}
			\fint_{B_{2\rho_\infty}}|w_k-\mu_k|^p\dd X\le\left(\frac{\widehat a_k}{\eta_\infty}\right)^p
			\le1.
		\end{aligned}
		\]
		Set $\mathfrak q_k=|P_k+b_k|$.
		The homogeneity of $D_\Phi$ and a change
		of variables give
		\[
		\mathcal E_{s,P_{k+1}}(w_{k+1};B_2)
		=\eta_\infty^{-p}\lambda_{P_{k+1}}^{2-p}
		\mathfrak B_{s,P_k}(w_k,\mu_k;\rho_\infty).
		\]
		Since
		$\lambda_{P_{k+1}}
		=(\mathfrak q_k+\eta_\infty)/\eta_\infty$,
		the admissibility of $\widehat a_k$ yields
		\begin{equation}
			\begin{aligned}
				\mathcal E_{s,P_{k+1}}(w_{k+1};B_2)
				\le\left(
				\frac{\mathfrak q_k+\eta_\infty}
				{\mathfrak q_k+\widehat a_k}
				\right)^{2-p}
				\left(\frac{\widehat a_k}{\eta_\infty}\right)^2\le\left(\frac{\widehat a_k}{\eta_\infty}\right)^p
				\le1.
			\end{aligned}
			\label{eq:large-slope-iteration-lp-energy-reproduction}
		\end{equation}
		Here we used
		\[
		\frac{\mathfrak q_k+\eta_\infty}
		{\mathfrak q_k+\widehat a_k}
		\le\frac{\eta_\infty}{\widehat a_k},
		\]
		which follows from
		$0<\widehat a_k<\eta_\infty$.
		Thus the $L^p$ and energy bounds are preserved.
		
		It remains to prove the bound for
		$\mathcal T^{\rm pt}_{s,P_{k+1}}(w_{k+1})$.
		Write $\rho=\rho_\infty$, $\eta=\eta_\infty$,
		and $h_k=w_k-\mu_k$.
		For $Z=\rho X$ and $W=\rho Y$,
		the definition of the shifted difference
		gives
		\[
		\mathcal D_{P_{k+1}}w_{k+1}(X,Y)
		=(\rho\eta)^{1-p}
		\mathcal D_{P_k+b_k}h_k(Z,W).
		\]
		Since $\gamma=sp-p+1$, a change of variables
		in the exterior integral yields
		\begin{equation}
			\begin{aligned}
				\lambda_{P_{k+1}}^{2-p}
				\frac{|\mathcal D_{P_{k+1}}w_{k+1}(X,Y)|}
				{|X-Y|^{n+sp}}\dd Y=
				\rho^\gamma\eta^{1-p}\lambda_{P_{k+1}}^{2-p}
				\frac{|\mathcal D_{P_k+b_k}h_k(Z,W)|}
				{|Z-W|^{n+sp}}\dd W.
			\end{aligned}
			\label{eq:large-slope-iteration-exact-flux-scaling}
		\end{equation}
		The factor on the right satisfies
		\[
		\begin{aligned}
			\rho^\gamma\eta^{1-p}
			\lambda_{P_{k+1}}^{2-p}
			=\rho^{\gamma-\alpha}
			(\mathfrak q_k+\eta)^{2-p}\le\rho^{\gamma-\alpha}
			\lambda_{P_k+b_k}^{2-p}.
		\end{aligned}
		\]
		
		For the first exterior integral, we have
		$X\in B_{5/4}$ and $Y\notin B_{4/3}$,
		so $Z\in B_{5\rho/4}$ and
		$W\notin B_{4\rho/3}$.
		For the second integral,
		$X\in B_{3/2}$ and $Y\notin B_2$,
		so $Z\in B_{3\rho/2}$ and
		$W\notin B_{2\rho}$.
		In both cases,
		\[
		|Z-W|\ge\frac1{16}|W|.
		\]
		We estimate the integrals separately
		over $\rho\le|W|\le1/2$,
		$1/2<|W|<4/3$, and $|W|\ge4/3$.
		
		Consider first $\rho\le|W|\le1/2$.
		We divide this region into annuli
		$t\le|W|\le2t$, with
		$\rho\le t\le1/2$.
		Both $Z$ and $W$ belong to $B_{2t}$.
		By \eqref{eq:large-slope-iteration-intermediate-block},
		\[
		|h_k(Z)|+|h_k(W)|
		\le C_{\rm L}t^{1+\bar\alpha}.
		\]
		Moreover, $|P_k+b_k|\ge|P_k|-C_{\rm L}
		\ge|P_k|/2\ge1$.
		Thus Lemma~\ref{lem:large-slope-scalar-angular}
		and polar coordinates give
		\[
		\begin{aligned}
			\lambda_{P_k+b_k}^{2-p}
			\int_{\{t\le|W|\le2t\}}
			\frac{|\mathcal D_{P_k+b_k}h_k(Z,W)|}
			{|Z-W|^{n+sp}}\dd W\le
			Ct^{1+\bar\alpha}
			\int_{ct}^{Ct}r^{p-sp-3}\dd r
			\le Ct^{\bar\alpha-\gamma}.
		\end{aligned}
		\]
		Here we used the bound
		\[
		\int_{\mathbb S^{n-1}}
		|\nu\cdot\omega|^{p-2}\dd\omega\le C,
		\qquad \nu\in\mathbb S^{n-1}.
		\]
		Since $\bar\alpha<\gamma$, summing over
		the annuli and using
		\eqref{eq:large-slope-iteration-exact-flux-scaling}
		give
		\begin{equation}
			C\rho^{\gamma-\alpha}
			\sum_{\rho\le t\le1/2}
			t^{\bar\alpha-\gamma}
			\le C_{\rm T}\rho^{\bar\alpha-\alpha},
			\label{eq:large-slope-iteration-intervening-annuli}
		\end{equation}
		where the sum is over the chosen dyadic radii.
		
		Next, consider $1/2<|W|<4/3$.
		By Lemma~\ref{lem:shift-uniform-local-boundedness}
		and the induction hypotheses,
		\[
		\|w_k\|_{L^\infty(B_{4/3})}\le C.
		\]
		Since $|c_k|+|b_k|\le C_{\rm L}$,
		we also have $|h_k|\le C$ in $B_{4/3}$.
		Furthermore, $|Z-W|\ge c>0$
		on this region.
		The estimate in
		Lemma~\ref{lem:large-slope-scalar-angular}
		therefore bounds the corresponding
		integral by $C_{\rm T}\rho^{\gamma-\alpha}$.
		
		Finally, consider $|W|\ge4/3$.
		Since $Z\in B_{3\rho/2}\subset B_{5/4}$,
		the first exterior integral in
		$\mathcal T^{\rm pt}_{s,P_k}(w_k)$
		controls the original shifted difference.
		Indeed,
		\begin{equation}
			\begin{aligned}
				\mathcal D_{P_k+b_k}h_k(Z,W)
				&=\mathcal D_{P_k}w_k(Z,W)\\
				&\quad+\Jp(P_k\cdot(Z-W))
				-\Jp((P_k+b_k)\cdot(Z-W)).
			\end{aligned}
			\label{eq:large-slope-iteration-flux-decomposition}
		\end{equation}
		Since $|P_k+b_k|\asymp|P_k|$,
		the normalization factors
		$\lambda_{P_k+b_k}^{2-p}$
		and $\lambda_{P_k}^{2-p}$ are comparable.
		Thus the first term on the right
		is bounded by
		$C\rho^{\gamma-\alpha}
		\mathcal T^{\rm pt}_{s,P_k}(w_k)$.
		
		For the second term, the estimate in
		Lemma~\ref{lem:large-slope-scalar-angular}
		gives
		\[
		\begin{aligned}
			\int_{\mathbb S^{n-1}}
			\left|\Jp(P_k\cdot r\omega)
			-\Jp((P_k+b_k)\cdot r\omega)\right|
			\dd\omega\le C|b_k||P_k|^{p-2}r^{p-1}.
		\end{aligned}
		\]
		Since $|Z-W|\ge1$ in the present region,
		the corresponding normalized integral
		is bounded by
		\[
		C|b_k|\rho^{\gamma-\alpha}
		\int_1^\infty r^{p-sp-2}\dd r
		\le C_{\rm T}\rho^{\gamma-\alpha}.
		\]
		Here the integral is finite because
		$p-sp-2=-\gamma-1<-1$.
		
		Combining the three estimates, we obtain
		\begin{equation}
			\begin{aligned}
				\mathcal T^{\rm pt}_{s,P_{k+1}}(w_{k+1})
				&\le C_{\rm T}
				\bigl(\rho_\infty^{\bar\alpha-\alpha}
				+\rho_\infty^{\gamma-\alpha}\bigr)+C_{\rm T}\rho_\infty^{\gamma-\alpha}
				\mathcal T^{\rm pt}_{s,P_k}(w_k).
			\end{aligned}
			\label{eq:large-slope-iteration-tail-recurrence}
		\end{equation}
		By the induction hypothesis and
		\eqref{eq:large-slope-iteration-tail-budget},
		the right-hand side is at most $L_{\rm bd}$.
		Thus the tail bound is also preserved.
		This completes the induction.
		
		We now prove convergence of the affine functions.
		From the construction,
		\[
		|A_{k+1}-A_k|\le C_{\rm L}r_kH_k,\qquad
		|B_{k+1}-B_k|\le C_{\rm L}H_k.
		\]
		Since $H_k=\eta_\infty^k$ and
		$r_kH_k=(\rho_\infty\eta_\infty)^k$,
		both series are convergent.
		Hence $A_k\to A_\infty$ and $B_k\to B_\infty$.
		Define
		$m_\infty(X)=A_\infty+B_\infty\cdot X$.
		The coefficient bounds above imply
		$|A_\infty|+|B_\infty|\le C$.
		Moreover,
		\[
		\begin{aligned}
			\|m_\infty-m_k\|_{L^\infty(B_{r_k})}
			&\le C\sum_{i=k}^\infty
			(r_iH_i+r_kH_i)
			\le Cr_kH_k.
		\end{aligned}
		\]
		
		By the first inequality in
		\eqref{eq:large-slope-iteration-reproduced-state},
		a change of variables gives
		\begin{equation}
			\left(\fint_{B_{2r_k}}|v-m_k|^p\dd X\right)^{1/p}
			\le r_kH_k=r_k^{1+\alpha}.
			\label{eq:large-slope-iteration-iterated-error}
		\end{equation}
		Furthermore, the local boundedness estimate
		and the uniform bounds for $w_k$ imply
		\[
		\|v-m_k\|_{L^\infty(B_{r_k})}
		=r_kH_k\|w_k\|_{L^\infty(B_1)}
		\le Cr_kH_k.
		\]
		Consequently,
		\[
		\|v-m_\infty\|_{L^\infty(B_{r_k})}
		\le Cr_kH_k=Cr_k^{1+\alpha}.
		\]
		
		For any $0<t\le1$, choose $k\ge0$
		such that $r_{k+1}<t\le r_k$.
		Since $r_{k+1}=\rho_\infty r_k$, we have
		$r_k<t/\rho_\infty$. Therefore,
		\[
		\begin{aligned}
			\|v-m_\infty\|_{L^\infty(B_t)}
			\le\|v-m_\infty\|_{L^\infty(B_{r_k})}\le Cr_k^{1+\alpha}
			\le C\rho_\infty^{-1-\alpha}t^{1+\alpha}.
		\end{aligned}
		\]
		Since $\rho_\infty$ is fixed,
		its contribution can be included in $C$.
		This proves
		\eqref{eq:large-slope-iteration-pointwise-expansion}.
	\end{proof}

	We now apply Lemma~\ref{lem:iterable-large-slope-state}
	at differentiability points where the gradient
	is nonzero. After a suitable normalization,
	the lemma gives the following pointwise
	affine approximation.
	
	\begin{proposition}
		\label{prop:noncritical-pointwise-expansion}
		Let $U$ be a globally Lipschitz weak solution of \eqref{eq:equation}
		in $\R^n$, and let $x_0$ be a
		differentiability point with \(q=\nabla U(x_0)\ne0\).  For every
		\(0<\alpha<sp-p+1\), there are \(r_0>0\) and \(C_{x_0}<\infty\) such that
		\begin{equation}
			\|U-U(x_0)-q\cdot(\,\cdot-x_0)\|_{L^\infty(B_r(x_0))}
			\le C_{x_0}r^{1+\alpha}
			\qquad(0<r<r_0).
			\label{eq:large-slope-noncritical-expansion}
		\end{equation}
	\end{proposition}

	\begin{proof}
		For $r>0$, define
		\[
		V_r(X):=\frac{U(x_0+rX)-U(x_0)}r-q\cdot X,
		\qquad d_r:=\|V_r\|_{L^\infty(B_4)}.
		\]
		Since $U$ is differentiable at $x_0$, for every
		fixed $R>0$,
		\[
		\|V_r\|_{L^\infty(B_R)}
		\le R\sup_{0<|h|\le Rr}
		\frac{|U(x_0+h)-U(x_0)-q\cdot h|}{|h|}
		\longrightarrow0
		\]
		as $r\to0$. Moreover, the global Lipschitz
		continuity of $U$ and $|q|\le\operatorname{Lip}(U;\R^n)$
		give
		\[
		|V_r(X)-V_r(Y)|
		\le2\operatorname{Lip}(U;\R^n)|X-Y|.
		\]
		In particular, $d_r\to0$, and the Lipschitz
		constants of $V_r$ are bounded independently
		of $r$.
		
		We next estimate the energy and the exterior
		terms. Define
		\[
		E_r:=\lambda_q^{p-2}\mathcal E_{s,q}(V_r;B_2),
		\qquad
		T_r:=\lambda_q^{p-2}\mathcal T^{\rm pt}_{s,q}(V_r).
		\]
		By the definition of the normalized energy,
		\[
		E_r=(1-s)\iint_{B_2^2}
		\frac{D_\Phi(q\cdot(X-Y);\delta V_r(X,Y))}
		{|X-Y|^{n+sp}}\dd X\dd Y.
		\]
		For $X,Y\in B_2$, the uniform Lipschitz bound
		and the definition of $d_r$ imply
		\[
		|\delta V_r(X,Y)|\le C\min\{d_r,|X-Y|\}.
		\]
		Using $D_\Phi(a;t)\le C|t|^p$ and polar
		coordinates, we obtain
		\[
		\begin{aligned}
			E_r
			&\le C\int_0^4\min\{d_r,t\}^p t^{-sp-1}\dd t\\
			&=C\int_0^{d_r}t^{p-sp-1}\dd t
			+Cd_r^p\int_{d_r}^4t^{-sp-1}\dd t\\
			&\le Cd_r^{p-sp}\longrightarrow0,
		\end{aligned}
		\]
		where we used $p-sp=p(1-s)>0$.
		
		We now prove that $T_r\to0$.
		Fix $R>4$ and set
		$\varepsilon_R(r):=\|V_r\|_{L^\infty(B_R)}$.
		The differentiability of $U$ at $x_0$
		gives $\varepsilon_R(r)\to0$ as $r\to0$.
		
		Consider either exterior integral in the
		definition of $T_r$. The distance between
		its interior and exterior points is at
		least $1/12$. We first restrict the exterior
		integration to $B_R$. For these pairs,
		$|\delta V_r(X,Y)|\le2\varepsilon_R(r)$.
		Since $J_p$ is $(p-1)$-H\"older continuous,
		\[
		|\mathcal D_qV_r(X,Y)|
		\le C|\delta V_r(X,Y)|^{p-1}
		\le C\varepsilon_R(r)^{p-1}.
		\]
		The positive distance between the two
		integration regions therefore bounds this
		part of each exterior integral by
		$C\varepsilon_R(r)^{p-1}$.
		
		For $|Y|\ge R$, the interior points satisfy
		$|X|\le3/2$, and hence
		$|X-Y|\ge|Y|/2$. Since $V_r(0)=0$
		and the global Lipschitz constants of $V_r$
		are uniformly bounded,
		\[
		|\delta V_r(X,Y)|\le C|Y|.
		\]
		Applying again the H\"older continuity of
		$J_p$, we obtain
		\[
		\begin{aligned}
			\int_{\R^n\setminus B_R}
			\frac{|\mathcal D_qV_r(X,Y)|}
			{|X-Y|^{n+sp}}\dd Y
			\le C\int_R^\infty t^{p-2-sp}\dd t\le CR^{p-1-sp}.
		\end{aligned}
		\]
		Consequently,
		\[
		T_r\le C\varepsilon_R(r)^{p-1}
		+CR^{p-1-sp}.
		\]
		For each fixed $R$, the first term tends
		to zero as $r\to0$. Since $sp>p-1$,
		the second term tends to zero as $R\to\infty$.
		Letting first $r\to0$ and then $R\to\infty$,
		we conclude that
		\[
		d_r+E_r+T_r\longrightarrow0.
		\]
		
		Choose a fixed constant $C_q\ge1$,
		to be specified below, and define
		\[
		a_r:=C_q(d_r+E_r^{1/2}+T_r+r),
		\qquad v_r:=\frac{V_r}{a_r},
		\qquad Q_r:=\frac q{a_r}.
		\]
		Since $a_r>0$ and $a_r\to0$, we have
		$a_r\le1$ for all sufficiently small $r$.
		The function
		\[
		Q_r\cdot X+v_r(X)
		=\frac{U(x_0+rX)-U(x_0)}{ra_r}
		\]
		solves the same homogeneous equation.
		Thus $v_r$ solves the equation shifted
		by $Q_r$.
		
		We verify the assumptions of
		Lemma~\ref{lem:iterable-large-slope-state}.
		First, since $|V_r|\le d_r$ in $B_2$,
		\[
		\fint_{B_2}|v_r|^p\dd X
		\le\left(\frac{d_r}{a_r}\right)^p
		\le C_q^{-p}.
		\]
		Next, the $p$-homogeneity of $D_\Phi$
		and $J_p$ gives
		\[
		\mathcal E_{s,Q_r}(v_r;B_2)
		=\frac{(a_r+|q|)^{2-p}}{a_r^2}E_r
		\]
		and
		\[
		\mathcal T^{\rm pt}_{s,Q_r}(v_r)
		=\frac{(a_r+|q|)^{2-p}}{a_r}T_r.
		\]
		By the definition of $a_r$,
		$E_r^{1/2}\le a_r/C_q$ and
		$T_r\le a_r/C_q$. Since $a_r\le1$,
		\[
		\mathcal E_{s,Q_r}(v_r;B_2)
		\le\frac{(1+|q|)^{2-p}}{C_q^2},
		\qquad
		\mathcal T^{\rm pt}_{s,Q_r}(v_r)
		\le\frac{(1+|q|)^{2-p}}{C_q}.
		\]
		We now fix $C_q$ sufficiently large,
		depending on $q$, so that all three
		bounds are at most one. Hence
		\[
		\fint_{B_2}|v_r|^p\dd X\le1,\qquad
		\mathcal E_{s,Q_r}(v_r;B_2)\le1,\qquad
		\mathcal T^{\rm pt}_{s,Q_r}(v_r)\le1.
		\]
		
		Since $q\ne0$ and $a_r\to0$,
		we have $|Q_r|=|q|/a_r\to\infty$.
		Choose $r_0>0$ sufficiently small that
		$|Q_{r_0}|\ge Q_{\rm it}$ and the three
		bounds above hold.
		Lemma~\ref{lem:iterable-large-slope-state},
		applied to $v_{r_0}$ with $L=1$,
		gives an affine function
		$m_\infty(X)=A_\infty+B_\infty\cdot X$
		such that
		\[
		\|v_{r_0}-m_\infty\|_{L^\infty(B_t)}
		\le Ct^{1+\alpha}\qquad(0<t\le1).
		\]
		Multiplying by $a_{r_0}$, we obtain
		\[
		\|V_{r_0}-a_{r_0}m_\infty\|_{L^\infty(B_t)}
		\le Ca_{r_0}t^{1+\alpha}.
		\]
		
		We claim that $m_\infty=0$.
		Since $V_{r_0}(0)=0$, evaluating the last
		estimate at the origin gives
		$|A_\infty|\le Ct^{1+\alpha}$.
		Letting $t\to0$, we obtain $A_\infty=0$.
		Thus $m_\infty(X)=B_\infty\cdot X$.
		
		Moreover, $V_{r_0}$ is differentiable at
		the origin and
		\[
		\nabla V_{r_0}(0)=\nabla U(x_0)-q=0.
		\]
		Hence
		\[
		\frac{\|V_{r_0}\|_{L^\infty(B_t)}}{t}
		\longrightarrow0\qquad(t\to0).
		\]
		Using the estimate above, we have
		\[
		a_{r_0}t|B_\infty|
		\le\|V_{r_0}\|_{L^\infty(B_t)}
		+Ca_{r_0}t^{1+\alpha}.
		\]
		Dividing by $t$ and letting $t\to0$
		gives $a_{r_0}|B_\infty|=0$.
		Since $a_{r_0}>0$, we obtain $B_\infty=0$.
		Therefore $m_\infty=0$, and
		\[
		\|V_{r_0}\|_{L^\infty(B_t)}
		\le Ca_{r_0}t^{1+\alpha}
		\qquad(0<t\le1).
		\]
		
		Finally, by the definition of $V_{r_0}$,
		\[
		\begin{aligned}
			\|U-U(x_0)-q\cdot(\,\cdot-x_0)\|
			_{L^\infty(B_{r_0t}(x_0))}=r_0\|V_{r_0}\|_{L^\infty(B_t)}
			\le Cr_0a_{r_0}t^{1+\alpha}.
		\end{aligned}
		\]
		Setting $\rho=r_0t$, we obtain
		\[
		\|U-U(x_0)-q\cdot(\,\cdot-x_0)\|
		_{L^\infty(B_\rho(x_0))}
		\le C_{x_0}\rho^{1+\alpha},
		\]
		where $C_{x_0}=Ca_{r_0}r_0^{-\alpha}$.
		This proves
		\eqref{eq:large-slope-noncritical-expansion}.
	\end{proof}

	We next consider the behavior of globally
	Lipschitz solutions under rescaling at
	large scales. The following proposition
	shows that if a sequence of rescaled
	solutions converges to an affine function,
	then the original solution is also affine.
	
	\begin{proposition}
		\label{prop:affine-blowdown-rigidity}
		Let $U$ be a globally Lipschitz weak solution
		of \eqref{eq:equation} in $\R^n$, with $U(0)=0$.
		Set $U_R(X)=U(RX)/R$.
		Suppose that, for some $R_j\to\infty$ and
		$b\in\R^n$,
		\begin{equation}
			U_{R_j}\longrightarrow b\cdot X
			\qquad\text{locally uniformly in }\R^n.
			\label{eq:affine-blowdown-convergence}
		\end{equation}
		Then $U(X)=b\cdot X$ in $\R^n$.
	\end{proposition}
	
	\begin{proof}
		Set $\Lambda=[U]_{C^{0,1}(\R^n)}$
		and $\gamma=sp-p+1>0$.
		Since $U(0)=0$, the rescaled functions
		$U_{R_j}$ satisfy
		\[
		U_{R_j}(0)=0,\qquad
		[U_{R_j}]_{C^{0,1}(\R^n)}=\Lambda.
		\]
		Moreover, each $U_{R_j}$ is a weak solution
		of the same homogeneous equation.
		
		Suppose first that $b=0$.
		Then $U_{R_j}\to0$ locally uniformly.
		For every fixed $M>4$, the global Lipschitz
		bound gives $|U_{R_j}(Y)|\le\Lambda|Y|$.
		Consequently,
		\[
		\begin{aligned}
			\int_{\R^n\setminus B_2}
			\frac{|U_{R_j}(Y)|^{p-1}}{|Y|^{n+sp}}\dd Y\le
			C_M\|U_{R_j}\|_{L^\infty(B_M)}^{p-1}
			+C\Lambda^{p-1}\int_M^\infty t^{p-sp-2}\dd t\le o_j(1)+C\Lambda^{p-1}M^{-\gamma}.
		\end{aligned}
		\]
		For fixed $M$, the first term tends to zero
		as $j\to\infty$. Letting $M\to\infty$
		afterwards, we obtain
		\[
		\int_{\R^n\setminus B_2}
		\frac{|U_{R_j}(Y)|^{p-1}}{|Y|^{n+sp}}
		\dd Y\longrightarrow0.
		\]
		
		Each $U_{R_j}$ is continuous and is a weak
		solution. By \cite[Proposition~1.5]{BiswasTopp2025},
		it is also a viscosity solution.
		Since $sp>p-1$, the local Lipschitz estimate
		in \cite[Theorem~2.1]{BiswasTopp2025} gives
		\[
		\begin{aligned}
			[U_{R_j}]_{C^{0,1}(B_1)}
			&\le C\left(
			\|U_{R_j}\|_{L^\infty(B_2)}
			+\operatorname{Tail}(U_{R_j};0,2)
			\right)\longrightarrow0.
		\end{aligned}
		\]
		By scaling,
		$[U_{R_j}]_{C^{0,1}(B_1)}
		=[U]_{C^{0,1}(B_{R_j})}$.
		Since $R_j\to\infty$, this implies that
		$U$ is constant in $\R^n$.
		Together with $U(0)=0$, we obtain $U=0$.
		
		Suppose now that $b\ne0$.
		Set $V_j(X)=U_{R_j}(X)-b\cdot X$.
		Then $V_j\to0$ locally uniformly,
		$V_j(0)=0$, and
		\[
		[V_j]_{C^{0,1}(\R^n)}\le\Lambda+|b|.
		\]
		Moreover, $V_j$ solves the equation
		shifted by $b$ in $B_2$.
		
		Define
		\[
		d_j:=\|V_j\|_{L^\infty(B_4)},\qquad
		E_j:=\lambda_b^{p-2}\mathcal E_{s,b}(V_j;B_2),
		\qquad
		T_j:=\lambda_b^{p-2}\mathcal T^{\rm pt}_{s,b}(V_j).
		\]
		We claim that
		\[
		d_j+E_j+T_j\longrightarrow0.
		\]
		The convergence $d_j\to0$ follows from
		the local uniform convergence of $V_j$.
		
		For $X,Y\in B_2$, we have
		$|\delta V_j(X,Y)|
		\le C\min\{d_j,|X-Y|\}$.
		Using $D_\Phi(a;t)\le C|t|^p$,
		we obtain
		\[
		\begin{aligned}
			E_j
			\le C\int_0^4\min\{d_j,t\}^p
			t^{-sp-1}\dd t\le Cd_j^{p-sp}\longrightarrow0,
		\end{aligned}
		\]
		where $p-sp>0$.
		
		To estimate $T_j$, fix $M>4$.
		In both exterior integrals defining
		$\mathcal T^{\rm pt}_{s,b}$, the interior
		and exterior points are separated by
		a distance of at least $1/12$.
		For exterior points in $B_M$, the
		$(p-1)$-H\"older continuity of $\Jp$ gives
		\[
		|\mathcal D_bV_j(X,Y)|
		\le C|\delta V_j(X,Y)|^{p-1}
		\le C\|V_j\|_{L^\infty(B_M)}^{p-1}.
		\]
		For $|Y|\ge M$, the common global Lipschitz
		bound gives $|\delta V_j(X,Y)|\le C|Y|$.
		Consequently,
		\[
		\begin{aligned}
			T_j
			\le C\|V_j\|_{L^\infty(B_M)}^{p-1}
			+C\int_M^\infty t^{p-sp-2}\dd t\le C\|V_j\|_{L^\infty(B_M)}^{p-1}
			+CM^{-\gamma}.
		\end{aligned}
		\]
		Letting first $j\to\infty$ and then
		$M\to\infty$, we obtain $T_j\to0$.
		Thus the claim is proved.
		
		Choose a fixed constant $C_b\ge1$
		sufficiently large and set
		\[
		a_j:=C_b(d_j+E_j^{1/2}+T_j+R_j^{-1}),
		\qquad v_j:=\frac{V_j}{a_j},\qquad
		Q_j:=\frac{b}{a_j}.
		\]
		Since $d_j+E_j+T_j\to0$ and
		$R_j\to\infty$, we have $a_j\to0$.
		In particular, $0<a_j\le1$ for all
		sufficiently large $j$, and
		$|Q_j|=|b|/a_j\to\infty$.
		
		By the homogeneity of the equation,
		$v_j$ solves the equation shifted by
		$Q_j$ in $B_2$. Moreover,
		\[
		\fint_{B_2}|v_j|^p\dd X
		\le\left(\frac{d_j}{a_j}\right)^p
		\le C_b^{-p}.
		\]
		The homogeneity of $D_\Phi$ and $\Jp$
		also gives
		\[
		\mathcal E_{s,Q_j}(v_j;B_2)
		=\frac{(a_j+|b|)^{2-p}}{a_j^2}E_j,
		\qquad
		\mathcal T^{\rm pt}_{s,Q_j}(v_j)
		=\frac{(a_j+|b|)^{2-p}}{a_j}T_j.
		\]
		Since $a_j\ge C_bE_j^{1/2}$,
		$a_j\ge C_bT_j$, and $a_j\le1$,
		we obtain
		\[
		\mathcal E_{s,Q_j}(v_j;B_2)
		\le\frac{(1+|b|)^{2-p}}{C_b^2},
		\qquad
		\mathcal T^{\rm pt}_{s,Q_j}(v_j)
		\le\frac{(1+|b|)^{2-p}}{C_b}.
		\]
		Thus, by choosing $C_b$ sufficiently
		large, we have
		\begin{equation}
			\fint_{B_2}|v_j|^p\dd X\le1,\qquad
			\mathcal E_{s,Q_j}(v_j;B_2)\le1,\qquad
			\mathcal T^{\rm pt}_{s,Q_j}(v_j)\le1.
			\label{eq:affine-blowdown-initial-unit-state}
		\end{equation}
		
		Fix $0<\alpha<\gamma$.
		Since $|Q_j|\to\infty$, we may apply
		Lemma~\ref{lem:iterable-large-slope-state}
		with $L=1$ to $v_j$ for all sufficiently
		large $j$.
		The lemma gives an affine function
		$m_j(X)=A_j+B_j\cdot X$ such that
		\[
		|A_j|+|B_j|\le C,\qquad
		\|v_j-m_j\|_{L^\infty(B_t)}
		\le Ct^{1+\alpha}\quad(0<t\le1),
		\]
		where $C$ is independent of $j$.
		
		Since $v_j(0)=0$, the last estimate gives
		$|A_j|\le Ct^{1+\alpha}$ for every
		$0<t\le1$. Letting $t\to0$, we obtain
		$A_j=0$. Hence $m_j(X)=B_j\cdot X$
		with $|B_j|\le C$.
		
		Fix $M>0$. For all sufficiently large $j$,
		we have $M/R_j\le1$.
		Using the definition of $v_j$ and the
		preceding approximation estimate, we get
		\[
		\begin{aligned}
			\|U-b\cdot X-a_jB_j\cdot X\|_{L^\infty(B_M)}=a_jR_j
			\|v_j-m_j\|_{L^\infty(B_{M/R_j})}\le Ca_jR_j(M/R_j)^{1+\alpha}
			=Ca_jM^{1+\alpha}R_j^{-\alpha}.
		\end{aligned}
		\]
		Since $|B_j|\le C$, it follows that
		\[
		\begin{aligned}
			\|U-b\cdot X\|_{L^\infty(B_M)}
			&\le Ca_jM^{1+\alpha}R_j^{-\alpha}
			+Ca_jM\longrightarrow0.
		\end{aligned}
		\]
		Thus $U(X)=b\cdot X$ in $B_M$.
		Since $M>0$ is arbitrary, the equality
		holds throughout $\R^n$.
	\end{proof}

	\section{The Liouville theorem for globally Lipschitz solutions}
	\label{sec:dilation-liouville}
	In this section, we prove the Liouville theorem
	for globally Lipschitz weak solutions of
	\eqref{eq:equation}.
	Assuming that a nonaffine solution exists,
	we use the results of
	Section~\ref{sec:large-slope-and-blowdown}
	to construct a nonaffine solution whose gradient
	attains its global Lipschitz constant.
	We then derive a linear equation for a directional
	derivative and apply a strong minimum principle
	to obtain a contradiction.
	
	\subsection{Recurrence under centered dilations}
\label{subsec:centered-dilation-recurrence}

Let $\mathcal A$ denote the affine functions on $\R^n$.
For $V(0)=0$, define
\[
(D_rV)(x):=\frac{V(rx)}r,
\qquad r>0.
\]
Proposition~\ref{prop:affine-blowdown-rigidity} gives
\[
D_{R_j}U\longrightarrow\ell\in\mathcal A
\quad\text{in }C_{\rm loc}(\R^n),\qquad R_j\to\infty
\quad\Longrightarrow\quad U\in\mathcal A.
\]

For a Lipschitz function $V$, write
\[
\Lambda_V(r):=[V]_{C^{0,1}(B_r)},
\qquad
\Lambda_V(0+):=\lim_{r\to0}\Lambda_V(r).
\]

\begin{proposition}
	\label{prop:dilation-recurrent-cascade}
	Let $U$ be a nonaffine,
	globally Lipschitz  solution of \eqref{eq:equation}, normalized by
	$U(0)=0$.  Then there are a nonaffine centered blow-down $V$ of $U$
	and a sequence $r_j\to0$ such that
	\begin{equation}
		D_{r_j}V\longrightarrow V
		\qquad\text{locally uniformly in }\R^n.
		\label{eq:dilation-recurrence}
	\end{equation}
	Moreover,
	\begin{equation}
		[V]_{C^{0,1}(\R^n)}
		=\Lambda_V(R)=\Lambda_V(0+)=:L_V>0
		\qquad\text{for every }R>0.
		\label{eq:all-centered-scales-full-slope}
	\end{equation}
	If $a_k,b_k\to0$, $a_k\ne b_k$, and
	\begin{equation}
		\frac{|V(a_k)-V(b_k)|}{|a_k-b_k|}\longrightarrow L_V,
		\label{eq:centered-extremizing-chords}
	\end{equation}
	then necessarily
	\begin{equation}
		\frac{|a_k-b_k|}{|a_k|+|b_k|}\longrightarrow0.
		\label{eq:centered-microscopic-ratio}
	\end{equation}
	In particular, sequences satisfying
	\eqref{eq:centered-extremizing-chords} exist, and all of them satisfy
	\eqref{eq:centered-microscopic-ratio}.
\end{proposition}

\begin{proof}
	Set $U_R(x):=\frac{U(Rx)-U(0)}R$.  Since
	\[
	U_R(0)=0,\qquad [U_R]_{C^{0,1}(\R^n)}=[U]_{C^{0,1}(\R^n)},
	\]
	Arzel\`a--Ascoli gives a nonempty compact set
	\[
	\Omega_\infty(U)
	:=\bigcap_{R_0>1}
	\overline{\{D_RU:R\ge R_0\}}^{\,C_{\rm loc}(\R^n)}.
	\]
	For smooth compactly supported tests, the uniform Lipschitz bound gives
	\[
	\int_0^1\rho^{p-sp-1}\dd\rho<\infty,
	\qquad
	\int_1^\infty\rho^{p-sp-2}\dd\rho<\infty.
	\]
	Dominated convergence therefore preserves the weak equation.
	For $r>0$,
	\[
	U_{R_i}\to W
	\quad\Longrightarrow\quad
	U_{rR_i}=D_rU_{R_i}\longrightarrow D_rW,
	\qquad rR_i\longrightarrow\infty,
	\]
	so $D_r\Omega_\infty(U)=\Omega_\infty(U)$, using also $D_{1/r}$.
	
	Zorn's lemma gives a minimal nonempty compact dilation-invariant set
	$\mathscr M\subset\Omega_\infty(U)$: every decreasing chain has a
	nonempty compact invariant intersection.  Write $T_t=D_{e^t}$ and fix
	$V\in\mathscr M$.  Minimality and compactness give, for every relative
	neighborhood $\mathcal O$ of $V$,
	\[
	\begin{gathered}
		\overline{\{T_tW:t\in\R\}}=\mathscr M\quad(W\in\mathscr M),
		\\
		\mathscr M=\bigcup_{i=1}^N T_{-t_i}\mathcal O.
	\end{gathered}
	\]
	Hence, for every $T>0$,
	\[
	\begin{gathered}
		T_{-T}V\in T_{-t_i}\mathcal O\text{ for some }i,
		\qquad T_{t_i-T}V\in\mathcal O,\\
		t_i-T\le\max_i t_i-T\to-\infty.
	\end{gathered}
	\]
	A decreasing neighborhood basis and successive choices of $T$ yield
	\[
	s_j\to-\infty,\qquad T_{s_j}V\to V,
	\qquad r_j=e^{s_j}\to0,
	\]
	which proves \eqref{eq:dilation-recurrence}.
	
	Proposition~\ref{prop:affine-blowdown-rigidity} and $U\notin\mathcal A$
	give $\mathscr M\cap\mathcal A=\varnothing$.  Set
	$L_0=\Lambda_V(0+)$.  For fixed $R>0$ and distinct $x,y\in B_R$,
	\eqref{eq:dilation-recurrence} gives
	\[
	\frac{|V(x)-V(y)|}{|x-y|}
	=\lim_{j\to\infty}
	\frac{|D_{r_j}V(x)-D_{r_j}V(y)|}{|x-y|}
	\le\lim_{j\to\infty}\Lambda_V(r_jR)=L_0.
	\]
	Monotonicity of $\Lambda_V$ now gives
	\[
	L_0\le\Lambda_V(R)\le L_0,
	\qquad [V]_{C^{0,1}(\R^n)}=\lim_{R\to\infty}\Lambda_V(R)=L_0>0,
	\]
	proving \eqref{eq:all-centered-scales-full-slope}.
	
	By \eqref{eq:all-centered-scales-full-slope}, choose distinct
	$a_k,b_k\in B_{1/k}$ with
	\[
	L_V-\frac1k<\frac{|V(a_k)-V(b_k)|}{|a_k-b_k|}\le L_V,
	\]
	which gives \eqref{eq:centered-extremizing-chords}.  For any sequence
	satisfying that condition, failure of \eqref{eq:centered-microscopic-ratio}
	gives, along a subsequence,
	\[
	|a_k-b_k|\ge c\bigl(|a_k|+|b_k|\bigr)
	\]
	with $c>0$.  Set $\rho_k=|a_k|+|b_k|$.  Compactness and dilation
	invariance give, after subsequences,
	\[
	\begin{gathered}
		W_k=D_{\rho_k}V\longrightarrow W\in\mathscr M
		\quad\text{locally uniformly},\qquad \rho_k\to0,\\
		\alpha_k=\frac{a_k}{\rho_k}\to\alpha,\qquad
		\beta_k=\frac{b_k}{\rho_k}\to\beta,\qquad
		\alpha_k,\beta_k\in\overline B_1,\quad |\alpha-\beta|\ge c.
	\end{gathered}
	\]
	By \eqref{eq:centered-extremizing-chords},
	\[
	|W(\alpha)-W(\beta)|=L_V|\alpha-\beta|.
	\]
	Thus $[W]_{C^{0,1}(\R^n)}=L_V$, attained at $\alpha\ne\beta$.
	Lemma~\ref{thm:extremal-secant-rigidity} gives
	$W\in\mathcal A\cap\mathscr M=\varnothing$, proving
	\eqref{eq:centered-microscopic-ratio}.
\end{proof}

\subsection{Collapse of affine radii}
\label{subsec:collapse-affine-radii}

Hereafter, $\mathscr M$ and $V\in\mathscr M$ are as in
Proposition~\ref{prop:dilation-recurrent-cascade}: $\mathscr M$ is compact,
dilation invariant, and $V$ satisfies \eqref{eq:dilation-recurrence}.

For a differentiability point $x\ne0$ of a Lipschitz function $V$, set
$q_x=\nabla V(x)$ and, for $0<\eta<1$, define its relative affine radius by
\begin{equation}
	\begin{split}
		\mathfrak r_\eta(V;x):=\sup\biggl\{r\in(0,|x|/2]:\ &
		\sup_{0<\rho\le r}\frac1\rho
		\|V-V(x)-q_x\cdot(\,\cdot-x)\|_{L^\infty(B_\rho(x))}
		\le\eta\biggr\}.
	\end{split}
	\label{eq:relative-affine-radius}
\end{equation}
Fr\'echet differentiability gives $\mathfrak r_\eta(V;x)>0$.
For the profile in Proposition~\ref{prop:dilation-recurrent-cascade},
these radii satisfy the following estimate.

\begin{corollary}
	\label{cor:collapse-affine-radii}
	Let $V$ be the recurrent profile in Proposition~\ref{prop:dilation-recurrent-cascade}.
	There are nonzero differentiability points $x_k\to0$ such that,
	after passing to a subsequence,
	\begin{equation}
		q_k:=\nabla V(x_k)\longrightarrow q,
		\qquad |q|=L_V.
		\label{eq:near-extremal-gradients}
	\end{equation}
	For every such sequence one necessarily has
	\begin{equation}
		\frac{\mathfrak r_{1/k}(V;x_k)}{|x_k|}\longrightarrow0.
		\label{eq:affine-radius-collapse}
	\end{equation}
	More generally, \eqref{eq:affine-radius-collapse} holds with any sequence
	$\eta_k\to0$ in place of $1/k$.
\end{corollary}

\begin{proof}
	Rademacher's theorem, integration on line segments in $B_r$, and
	\eqref{eq:all-centered-scales-full-slope} give
	\[
	\operatorname*{ess\,sup}_{B_r}|\nabla V|
	=[V]_{C^{0,1}(B_r)}=L_V
	\qquad(r>0).
	\]
	Choose differentiability points satisfying
	\[
	x_k\in B_{1/k}\setminus\{0\},\qquad
	L_V-\frac1k<|q_k|=|\nabla V(x_k)|\le L_V.
	\]
	Compactness of $\overline B_{L_V}$ yields
	\eqref{eq:near-extremal-gradients} along a subsequence.
	
	For any such sequence, set $d_k=|x_k|$.  If
	\eqref{eq:affine-radius-collapse} fails, pass to a subsequence with
	\[
	\mathfrak r_{1/k}(V;x_k)\ge cd_k.
	\]
	Decrease $c$ so that $0<c<1/2$ and the inequality is strict.
	Set $W_k=D_{d_k}V$ and $\xi_k=x_k/d_k\in\mathbb S^{n-1}$.
	Definition
	\eqref{eq:relative-affine-radius} yields
	\begin{equation}
		\left\|W_k-\frac{V(x_k)}{d_k}
		-q_k\cdot(\,\cdot-\xi_k)\right\|_{L^\infty(B_c(\xi_k))}
		\le\frac ck.
		\label{eq:rescaled-affine-ball}
	\end{equation}
	By compactness and dilation invariance, after subsequences,
	\[
	\begin{gathered}
		W_k\longrightarrow W\in\mathscr M\quad\text{locally uniformly},
		\qquad \xi_k\to\xi\in\mathbb S^{n-1},\\
		\left|\frac{V(x_k)}{d_k}\right|\le L_V,
		\qquad \frac{V(x_k)}{d_k}\to a\in\R.
	\end{gathered}
	\]
	Letting $k\to\infty$ in \eqref{eq:rescaled-affine-ball} gives
	\[
	W(X)=a+q\cdot(X-\xi)
	\qquad(X\in B_c(\xi))
	\]
	and, for $e=q/L_V$ and $0<t<c$,
	\[
	\frac{|W(\xi+te)-W(\xi)|}{t}=|q|=L_V
	\ge[W]_{C^{0,1}(\R^n)}.
	\]
	Lemma~\ref{thm:extremal-secant-rigidity} gives
	$W\in\mathcal A\cap\mathscr M=\varnothing$.
	This proves \eqref{eq:affine-radius-collapse}; replacing $c/k$ by
	$c\eta_k$ proves the assertion for every $\eta_k\to0$.
\end{proof}

\subsection{A nonaffine limit with maximal gradient}
\label{subsec:record-scale-nondegeneracy}

We rescale at the points of Corollary~\ref{cor:collapse-affine-radii}
satisfying \eqref{eq:affine-radius-collapse}.
For a differentiability point $x$ with $q_x\ne0$,
Proposition~\ref{prop:noncritical-pointwise-expansion} gives, for every
$\alpha<sp-p+1$ and all sufficiently small $r$,
\[
\left(\fint_{B_r(x)}
|V-V(x)-q_x\cdot(\,\cdot-x)|^p\dd y\right)^{1/p}
\le C_xr^{1+\alpha}
\]
The global Lipschitz bound then yields
\begin{equation}
	\frac1r\|V-V(x)-q_x\cdot(\,\cdot-x)\|_{L^\infty(B_r(x))}
	\le C'_xr^{\alpha p/(n+p)}.
	\label{eq:noncritical-sup-expansion}
\end{equation}
Indeed, for $F=V-V(x)-q_x\cdot(\,\cdot-x)$ and
$A=\|F\|_{L^\infty(B_r(x))}$, choose
$z\in\overline B_r(x)$ with $|F(z)|=A$.  Since
$\operatorname{Lip}(F)\le2L_V$ and $A\le2L_Vr$,
\[
\begin{aligned}
	|F|&\ge A/2 \quad \text{on }B_{A/(4L_V)}(z)\cap B_r(x),\\
	A^{n+p}&\le C L_V^n\int_{B_r(x)}|F|^p\dd y
	\le C_x L_V^n r^{n+p(1+\alpha)}.
\end{aligned}
\]

Fix $0<\beta<\frac{p}{n+p}(sp-p+1)$.  For a differentiability point
$x$ with tangent $q_x\ne0$, define
\begin{equation*}
	\mathfrak e_{V,x}(r):=
	\frac1r\|V-V(x)-q_x\cdot(\,\cdot-x)\|_{L^\infty(B_r(x))},
	\qquad
	\mathfrak H_{V,x}(r):=r^{-\beta}\mathfrak e_{V,x}(r).
\end{equation*}
Choose $\alpha$ with $\beta<\alpha p/(n+p)$.
By \eqref{eq:noncritical-sup-expansion} and the global Lipschitz bound,
\[
\mathfrak H_{V,x}(r)\le
\begin{cases}
	C'_xr^{\alpha p/(n+p)-\beta}\to0,&r\to0,\\
	2L_Vr^{-\beta}\to0,&r\to\infty.
\end{cases}
\]
The supremum over $\overline B_r(x)$ is continuous in $r>0$.
Thus $\mathfrak H_{V,x}$ attains its maximum in $(0,\infty)$.

\begin{proposition}
	\label{prop:no-linear-record-neck}
	Let $V$ and $x_k$ be as in Corollary~\ref{cor:collapse-affine-radii}, and
	let $d_k=|x_k|$ and $q_k=\nabla V(x_k)$.  Discard finitely many indices
	so that $|q_k|\ge L_V/2$.  Choose $t_k>0$ so that
	\begin{equation}
		\mathfrak H_{V,x_k}(t_k)
		=\max_{r>0}\mathfrak H_{V,x_k}(r),
		\qquad
		a_k:=\mathfrak e_{V,x_k}(t_k).
		\label{eq:record-scale-choice}
	\end{equation}
	Then $t_k\to0$.  After passage to a subsequence,
	\begin{equation}
		a_k\longrightarrow a_*>0.
		\label{eq:positive-record-amplitude}
	\end{equation}
	Moreover, the rescaled solutions
	\[
	U_k(X):=\frac{V(x_k+t_kX)-V(x_k)}{t_k}
	=q_k\cdot X+a_kv_k(X)
	\]
	converge locally uniformly to a nonaffine, globally $L_V$-Lipschitz 
	solution
	\begin{equation}
		U_*(X)=q\cdot X+a_*v_*(X),
		\qquad |q|=L_V,
		\label{eq:gradient-contact-profile}
	\end{equation}
	such that
	\begin{equation}
		v_*(0)=0,\qquad
		\|v_*\|_{L^\infty(B_1)}=1,\qquad
		\|v_*\|_{L^\infty(B_R)}\le R^{1+\beta}\quad(R>0).
		\label{eq:record-profile-growth}
	\end{equation}
	In particular, $U_*$ is differentiable at the origin,
	\begin{equation}
		\nabla U_*(0)=q,
		\qquad |\nabla U_*(0)|=[U_*]_{C^{0,1}(\R^n)}=L_V.
		\label{eq:attained-maximal-gradient}
	\end{equation}
\end{proposition}

\begin{proof}
	Fix $c\in(0,1/4)$.  We claim that, for some $\varepsilon_0>0$,
	\begin{equation}
		\mathfrak e_{V,x_k}(cd_k)\ge\varepsilon_0
		\label{eq:outer-record-nonaffinity}
	\end{equation}
	for all large $k$.  Otherwise, compactness and dilation invariance give,
	along subsequences,
	\[
	\begin{gathered}
		\mathfrak e_{V,x_k}(cd_k)\to0,\qquad
		W_k=D_{d_k}V\to W\in\mathscr M\quad\text{locally uniformly},\\
		\xi_k=\frac{x_k}{d_k}\to\xi\in\mathbb S^{n-1},\qquad q_k\to q,
	\end{gathered}
	\]
	and hence
	\[
	W(X)=W(\xi)+q\cdot(X-\xi)
	\qquad (X\in B_c(\xi)).
	\]
	With $e=q/L_V$ and $0<t<c$,
	\[
	\frac{|W(\xi+te)-W(\xi)|}{t}=L_V\ge[W]_{C^{0,1}(\R^n)}.
	\]
	Lemma~\ref{thm:extremal-secant-rigidity} gives
	$W\in\mathcal A\cap\mathscr M=\varnothing$.
	
	Equations \eqref{eq:record-scale-choice} and
	\eqref{eq:outer-record-nonaffinity} imply
	\[
	\mathfrak H_{V,x_k}(t_k)
	\ge\varepsilon_0(cd_k)^{-\beta}\longrightarrow\infty.
	\]
	Thus, for all sufficiently large $k$,
	\[
	0<a_k\le2L_V,\qquad
	t_k^\beta=\frac{a_k}{\mathfrak H_{V,x_k}(t_k)}\longrightarrow0.
	\]
	Maximality of $t_k$ gives, for every $R>0$,
	\begin{equation}
		\|v_k\|_{L^\infty(B_R)}
		=\frac{R\mathfrak e_{V,x_k}(Rt_k)}{a_k}
		\le R^{1+\beta},
		\qquad
		\|v_k\|_{L^\infty(B_1)}=1.
		\label{eq:record-growth-before-limit}
	\end{equation}
	
	Pass to a subsequence with $a_k\to a_*\ge0$.  Suppose $a_*=0$.
	Set
	\[
	\begin{gathered}
		Q_k=\frac{q_k}{a_k},\qquad |Q_k|\to\infty,\qquad
		e_k=\frac{q_k}{|q_k|}\to e=\frac q{L_V},\\
		\gamma:=sp-p+1>\beta,\qquad R_k:=a_k^{-1/\beta}.
	\end{gathered}
	\]
	For \(X\in B_{5/4}\), \(R\ge2\), and
	\(Y=X+R\omega\), set
	\begin{equation*}
		\mathcal F_k(X,R,\omega):=(1+|Q_k|)^{2-p}
		\left|\Jp\bigl(Q_k\cdot(X-Y)+\delta v_k(X,Y)\bigr)
		-\Jp\bigl(Q_k\cdot(X-Y)\bigr)\right|.
	\end{equation*}
	Since $|q_k|\to L_V>0$, $(1+|Q_k|)^{2-p}\asymp a_k^{p-2}$.
	For $2\le R\le R_k$, split the sphere into
	\begin{equation*}
		E_1:=\bigl\{|e_k\cdot\omega|>C a_kR^\beta\bigr\},
		\qquad
		E_2:=\bigl\{|e_k\cdot\omega|\le C a_kR^\beta\bigr\}.
	\end{equation*}
	By the linearized flux estimate, \eqref{eq:record-growth-before-limit},
	and the global $(p-1)$-H\"older bound for $\Jp$,
	\[
	\mathcal F_k(X,R,\omega)\le
	\begin{cases}
		CR^{p-1+\beta}|e_k\cdot\omega|^{p-2},&\omega\in E_1,\\
		Ca_k^{p-2}R^{(1+\beta)(p-1)},&\omega\in E_2.
	\end{cases}
	\]
	Since $p-2>-1$ and $|E_2|\le Ca_kR^\beta$,
	\[
	\begin{aligned}
		&\int_{E_1}\mathcal F_k\dd\omega
		\le CR^{p-1+\beta},\\
		&\int_{E_2}\mathcal F_k\dd\omega
		\le Ca_k^{p-2}R^{(1+\beta)(p-1)}a_kR^\beta=CR^{p-1+\beta}(a_kR^\beta)^{p-1}
		\le CR^{p-1+\beta}.
	\end{aligned}
	\]
	Therefore
	\begin{equation*}
		\sup_{X\in B_{5/4}}
		\int_2^{R_k}\int_{\mathbb S^{n-1}}
		\mathcal F_k(X,R,\omega)\dd\omega\,
		\frac{\dd R}{R^{sp+1}}
		\le C\int_2^\infty R^{\beta-\gamma-1}\dd R<\infty.
	\end{equation*}
	For $R>R_k$, the global Lipschitz bound for
	$U_k=q_k\cdot X+a_kv_k$, homogeneity, and the global H\"older bound give
	\[
	\int_{R_k}^\infty
	a_k^{-1}R^{p-1}R^{-sp-1}\dd R
	\le C a_k^{\gamma/\beta-1}\longrightarrow0.
	\]
	The shifted tails are thus uniformly bounded; for every fixed $R\ge2$,
	the portion outside $B_R$ is bounded, uniformly for all large $k$, by
	\begin{equation}
		C R^{\beta-\gamma}+o_k(1)
		=C R^{1+\beta-\tau_s}+o_k(1).
		\label{eq:vanishing-record-linearized-tail}
	\end{equation}
	
	For \(A\ge2\),
	set
	\begin{align*}
		\mathfrak T_k(A):=\max\biggl\{&
		{\operatorname*{ess\,sup}}_{X\in B_{5A/4}}
		\int_{\R^n\setminus B_{4A/3}}
		\frac{\lambda_{Q_k}^{2-p}|\mathcal D_{Q_k}v_k(X,Y)|}
		{|X-Y|^{n+sp}}\dd Y,\\
		&{\operatorname*{ess\,sup}}_{X\in B_{3A/2}}
		\int_{\R^n\setminus B_{2A}}
		\frac{\lambda_{Q_k}^{2-p}|\mathcal D_{Q_k}v_k(X,Y)|}
		{|X-Y|^{n+sp}}\dd Y\biggr\}.
	\end{align*}
	For $|X|\le2A$ and $|X-Y|\ge cA$,
	\eqref{eq:record-growth-before-limit} gives
	\[
	|\delta v_k(X,Y)|
	\le C\bigl(A+|X-Y|\bigr)^{1+\beta}
	\le C|X-Y|^{1+\beta}.
	\]
	The same angular splitting therefore yields, for fixed $A$ and large $k$,
	\[
	\mathfrak T_k(A)
	\le C\int_{cA}^{R_k}R^{\beta-\gamma-1}\dd R
	+Ca_k^{\gamma/\beta-1}
	\le CA^{\beta-\gamma}+Ca_k^{\gamma/\beta-1}.
	\]
	Consequently,
	\begin{equation}
		\limsup_{k\to\infty}\mathfrak T_k(A)
		\le CA^{\beta-\gamma}
		\qquad(A\ge2).
		\label{eq:record-tail-all-scales}
	\end{equation}
	Both components of \eqref{eq:large-slope-pointwise-tail} are covered.
	
	For each fixed \(A\ge2\), define
	\[
	\widehat v_{k,A}(X):=A^{-1-\beta}v_k(AX),
	\qquad
	\widehat Q_{k,A}:=A^{-\beta}Q_k.
	\]
	Then \(\widehat v_{k,A}\) solves the equation shifted by
	\(\widehat Q_{k,A}\), and
	\[
	\|\widehat v_{k,A}\|_{L^\infty(B_R)}\le R^{1+\beta},
	\qquad
	|\widehat Q_{k,A}|\longrightarrow\infty,
	\qquad
	\frac{\widehat Q_{k,A}}{|\widehat Q_{k,A}|}\longrightarrow e.
	\]
	A change of variables in the two tail components gives
	\begin{equation}
		\mathcal T^{\rm pt}_{s,\widehat Q_{k,A}}
		(\widehat v_{k,A})
		=
		\left(\frac{\lambda_{\widehat Q_{k,A}}}{\lambda_{Q_k}}\right)^{2-p}
		A^{sp-(1+\beta)(p-1)}\mathfrak T_k(A)\le C+o_k(1).
		\label{eq:scaled-record-tail}
	\end{equation}
	Indeed,
	\[
	\begin{aligned}
		\left(\frac{\lambda_{\widehat Q_{k,A}}}{\lambda_{Q_k}}\right)^{2-p}
		A^{sp-(1+\beta)(p-1)}
		&\longrightarrow A^{-\beta(2-p)+sp-(1+\beta)(p-1)}=A^{\gamma-\beta}.
	\end{aligned}
	\]
	Combining this with \eqref{eq:record-tail-all-scales} at $2A$
	also gives
	\begin{equation}
		{\operatorname*{ess\,sup}}_{X\in B_3}
		\int_{\R^n\setminus B_4}
		\frac{\lambda_{\widehat Q_{k,A}}^{2-p}
			|\mathcal D_{\widehat Q_{k,A}}\widehat v_{k,A}(X,Y)|}
		{|X-Y|^{n+sp}}\dd Y\le C+o_k(1).
		\label{eq:scaled-record-expanded-tail}
	\end{equation}
	Here the rescaled integral is bounded by
	$(A^{\gamma-\beta}+o_k(1))\mathfrak T_k(2A)\le C+o_k(1)$.
	
	Choose $\zeta\in C_c^\infty(B_3)$ with $0\le\zeta\le1$ and
	$\zeta=1$ on $B_2$.  For fixed $k,A$, the total solution is globally
	Lipschitz.  Smooth approximation therefore makes
	$\zeta^2\widehat v_{k,A}$ admissible in the original, hence shifted,
	equation, since
	\[
	\int_0^1\rho^{p-sp-1}\dd\rho<\infty,
	\qquad\int_1^\infty\rho^{p-sp-2}\dd\rho<\infty.
	\]
	For each pair $(X,Y)$, set
	\[
	m=\frac{\zeta(X)^2+\zeta(Y)^2}{2},\qquad
	\theta=\frac{\zeta(X)+\zeta(Y)}2,\qquad
	\theta^2\le m\le2\theta^2.
	\]
	Then
	\[
	\delta(\zeta^2\widehat v_{k,A})
	=m\,\delta\widehat v_{k,A}
	+\theta\bigl(\widehat v_{k,A}(X)+\widehat v_{k,A}(Y)\bigr)
	\delta\zeta
	\]
	and \eqref{eq:uncapped-shifted-young} gives, for $\theta>0$,
	\[
	\begin{aligned}
		\lambda_{\widehat Q_{k,A}}^{2-p}
		\mathcal D_{\widehat Q_{k,A}}\widehat v_{k,A}
		\,\delta(\zeta^2\widehat v_{k,A})&\ge\frac m2\lambda_{\widehat Q_{k,A}}^{2-p}
		\mathcal D_{\widehat Q_{k,A}}\widehat v_{k,A}
		\,\delta\widehat v_{k,A}\\
		&\qquad-C\mathcal H_{\widehat Q_{k,A},X-Y}
		\bigl((|\widehat v_{k,A}(X)|+|\widehat v_{k,A}(Y)|)\delta\zeta\bigr).
	\end{aligned}
	\]
	For the remainder, use \eqref{eq:uncapped-shifted-young} with
	$b=(\widehat v_{k,A}(X)+\widehat v_{k,A}(Y))\delta\zeta$.
	Pairs with $\theta=0$ vanish.  Using \eqref{eq:scaled-record-expanded-tail} outside
	$B_4$, we obtain
	\begin{align*}
		\mathcal E_{s,\widehat Q_{k,A}}
		(\widehat v_{k,A};B_2)	&\le C(1-s)
		\iint_{B_4^2}
		\frac{\mathcal H_{\widehat Q_{k,A},X-Y}
			\bigl((|\widehat v_{k,A}(X)|
			+|\widehat v_{k,A}(Y)|)\delta\zeta(X,Y)\bigr)}
		{|X-Y|^{n+sp}}\dd X\!\dd Y\\
		&\qquad+C\|\widehat v_{k,A}\|_{L^\infty(B_3)}
		{\operatorname*{ess\,sup}}_{X\in B_3}
		\int_{\R^n\setminus B_4}
		\frac{\lambda_{\widehat Q_{k,A}}^{2-p}
			|\mathcal D_{\widehat Q_{k,A}}\widehat v_{k,A}(X,Y)|}
		{|X-Y|^{n+sp}}\dd Y.
	\end{align*}
	Write $X-Y=r\omega$.  The growth bound gives
	\[
	(|\widehat v_{k,A}(X)|+|\widehat v_{k,A}(Y)|)
	\delta\zeta(X,Y)=r\,t_{k,A,X,r}(\omega),
	\qquad |t_{k,A,X,r}(\omega)|\le C_A.
	\]
	The global $p$-growth bound and the angular estimate give
	\[
	r^{-p}\mathcal H_{\widehat Q_{k,A},r\omega}
	(r t_{k,A,X,r}(\omega))
	\le
	\begin{cases}
		C|t_{k,A,X,r}(\omega)|^p,&|\widehat Q_{k,A}|<1,\\
		C|e_{k,A}\cdot\omega|^{p-2}|t_{k,A,X,r}(\omega)|^2,
		&\widehat Q_{k,A}=q_{k,A}e_{k,A},\ q_{k,A}\ge1.
	\end{cases}
	\]
	Since $p-2>-1$,
	\[
	\begin{gathered}
		\int_{\mathbb S^{n-1}}
		\mathcal H_{\widehat Q_{k,A},r\omega}(r t_{k,A,X,r}(\omega))\dd\omega
		\le C_A r^p,\\
		(1-s)\int_0^8r^{p-sp-1}\dd r
		=\frac{8^{p(1-s)}}p\le C(p).
	\end{gathered}
	\]
	Together with \eqref{eq:scaled-record-expanded-tail}, this yields
	\begin{equation}
		\mathcal E_{s,\widehat Q_{k,A}}
		(\widehat v_{k,A};B_2)\le C_A.
		\label{eq:scaled-record-energy}
	\end{equation}
	Apply Proposition~\ref{prop:fixed-order-infinite-slope-compactness} to
	$\widehat v_{k,A}$ successively for $A=2^m$, $m\in\mathbb N$.
	A diagonal subsequence gives
	\begin{equation}
		v_k\longrightarrow v
		\quad\hbox{locally uniformly in }\R^n.
		\label{eq:record-global-local-uniform-convergence}
	\end{equation}
	Limits agree on overlaps along this common subsequence.
	
	For each dyadic $A=2^m$, Proposition~\ref{prop:fixed-order-infinite-slope-compactness} and
	\eqref{eq:record-tail-all-scales} give, in \(B_{5A/4}\),
	\[
	L^{B_{4A/3}}_{s,e}v=f_A,
	\qquad
	\|f_A\|_{L^\infty(B_{5A/4})}
	\le CA^{\beta-\gamma}.
	\]
	The regional source $\widehat f_A$ for $\widehat v_{k,A}$ rescales as
	\[
	f_A(X)=A^{\beta-\gamma}\widehat f_A(X/A),
	\qquad \|\widehat f_A\|_{L^\infty(B_{5/4})}\le C.
	\]
	Thus $f_A\to0$ uniformly on every fixed ball as $A\to\infty$.
	By \eqref{eq:record-growth-before-limit},
	\begin{equation}
		\int_{\R^n\setminus B_A}
		\frac{|v(Y)|}{|Y|^{n+\tau_s}}\dd Y
		\le CA^{1+\beta-\tau_s}
		=CA^{\beta-\gamma}.
		\label{eq:record-limit-stable-tail}
	\end{equation}
	Polar coordinates and $|e\cdot\omega|^{p-2}\in L^1(\mathbb S^{n-1})$ give
	\begin{equation*}
		\int_{\R^n\setminus B_A}
		\frac{|e\cdot(Y/|Y|)|^{p-2}|v(Y)|}
		{|Y|^{n+\tau_s}}\dd Y
		\le CA^{\beta-\gamma}.
	\end{equation*}
	Fix $M\ge1$ and $A\ge4M$.  For $X\in B_M$ and
	$Y=X+\rho\omega\notin B_A$,
	\[
	\rho\ge A-M,\qquad |\delta v(X,Y)|\le C_M\rho^{1+\beta}.
	\]
	Polar coordinates about $X$ therefore give
	\[
	\sup_{X\in B_M}\int_{\R^n\setminus B_A}
	\frac{|e\cdot(X-Y)|^{p-2}|\delta v(X,Y)|}
	{|X-Y|^{n+sp}}\dd Y
	\le C_M\int_{A-M}^{\infty}\rho^{\beta-\gamma-1}\dd\rho
	\le C_M A^{\beta-\gamma}.
	\]
	For a test supported in $B_M$, the crossed integrals tend to zero;
	the regional integral is finite by the limiting energy bounds.
	Letting the dyadic radius $A\to\infty$, with $M$ fixed, gives the
	homogeneous equation with kernel
	\begin{equation*}
		K_e(z)=\frac{|e\cdot z|^{p-2}}
		{|z|^{n+sp}}
		=\frac{|e\cdot(z/|z|)|^{p-2}}
		{|z|^{n+sp-p+2}}.
	\end{equation*}
	The angular density is integrable and spectrally elliptic.
	Equations \eqref{eq:record-global-local-uniform-convergence} and
	\eqref{eq:record-growth-before-limit} yield
	\begin{equation}
		v(0)=0,\qquad
		\|v\|_{L^\infty(B_1)}=1,\qquad
		\|v\|_{L^\infty(B_R)}\le R^{1+\beta}.
		\label{eq:nontrivial-linear-record-limit}
	\end{equation}
	
	By \eqref{eq:nontrivial-linear-record-limit},
	\[
	1+\beta<2\sigma=\tau_s<2,\qquad \sigma=\tau_s/2.
	\]
	Theorem~2.1 of \cite{RosOtonSerra2016} gives
	\[
	\deg v\le\lfloor1+\beta\rfloor=1,\qquad
	|v(X)|\le|X|^{1+\beta}\ (X\to0)
	\quad\Longrightarrow\quad v\equiv0,
	\]
	contradicting $\|v\|_{L^\infty(B_1)}=1$ and proving
	\eqref{eq:positive-record-amplitude}.
	
	Now $a_k\to a_*>0$.  Arzel\`a--Ascoli and $sp>p-1$ give, along a subsequence, a weak solution $U_*$ of \eqref{eq:equation} in $\R^n$ with	\[
	\begin{gathered}
		U_k\to U_*,\qquad
		v_k=\frac{U_k-q_k\cdot X}{a_k}\to
		v_*:=\frac{U_*-q\cdot X}{a_*}
		\quad\text{locally uniformly},\\
		[U_*]_{C^{0,1}(\R^n)}\le L_V.
	\end{gathered}
	\]
	Passing to the limit in \eqref{eq:record-growth-before-limit} proves
	\eqref{eq:gradient-contact-profile} and \eqref{eq:record-profile-growth}.
	Moreover,
	\[
	\begin{gathered}
		\frac{|v_*(X)-v_*(0)|}{|X|}\le |X|^\beta\to0
		\qquad(X\to0),\\
		\nabla v_*(0)=0,\qquad\nabla U_*(0)=q,
	\end{gathered}
	\]
	which proves \eqref{eq:attained-maximal-gradient}.
	Finally,
	\[
	U_*\in\mathcal A\quad\Longrightarrow\quad v_*\in\mathcal A,
	\quad v_*(0)=0,\quad\nabla v_*(0)=0
	\quad\Longrightarrow\quad v_*=0,
	\]
	contrary to $\|v_*\|_{L^\infty(B_1)}=1$.
\end{proof}

\subsection{Noncritical regularity at the point of maximal gradient}
\label{subsec:noncritical-maximal-gradient}

\begin{corollary}
	\label{cor:tangent-ball-dissipation}
	Let $U_*=q\cdot X+a_*v_*$ be the profile in
	Proposition~\ref{prop:no-linear-record-neck}, and set $e=q/L_V$.  Then
	$v_*\in W^{1,\infty}_{\rm loc}(\R^n)$ and
	\begin{equation}
		\partial_ev_*
		\le-\frac{a_*}{2L_V}|\nabla v_*|^2\le0
		\qquad\text{a.e.\ in }\R^n.
		\label{eq:tangent-ball-dissipation}
	\end{equation}
	Consequently, for almost every $X\in e^\perp$ and $t_1<t_2$,
	\begin{equation}
		\frac{a_*}{2L_V}
		\int_{t_1}^{t_2}|\nabla v_*(X+te)|^2\dd t
		\le v_*(X+t_1e)-v_*(X+t_2e).
		\label{eq:line-dissipation}
	\end{equation}
	Moreover,
	\begin{equation}
		\limsup_{R\to\infty}\frac1R
		\|v_*\|_{L^\infty(B_R)}>0,
		\label{eq:mandatory-linear-far-field}
	\end{equation}
	and no centered blow-down of $U_*$ is affine.
\end{corollary}

\begin{proof}
	Since $v_*=(U_*-q\cdot X)/a_*$, Rademacher's theorem gives
	\[
	\begin{aligned}
		|q+a_*\nabla v_*|^2&\le L_V^2=|q|^2,\\
		2a_*L_V\partial_ev_*+a_*^2|\nabla v_*|^2&\le0
		\qquad\text{a.e. in }\R^n.
	\end{aligned}
	\]
	Division by $2a_*L_V$ proves \eqref{eq:tangent-ball-dissipation};
	integration on almost every line proves \eqref{eq:line-dissipation}.
	
	Failure of \eqref{eq:mandatory-linear-far-field} would give, for every $M>0$,
	\[
	\|D_RU_*-q\cdot X\|_{L^\infty(B_M)}
	\le a_*M\frac{\|v_*\|_{L^\infty(B_{RM})}}{RM}
	\longrightarrow0\qquad(R\to\infty).
	\]
	The rigidity result used in Proposition~\ref{prop:dilation-recurrent-cascade}
	would make $U_*$ affine, contrary to
	Proposition~\ref{prop:no-linear-record-neck}.  The same rigidity result
	excludes every affine centered blow-down.
\end{proof}

\begin{lemma}
	\label{lem:regional-angular-compactness}
	Let $n\ge2$, $1<p<2$, $0<\alpha<1$, and let
	$U\in C^{1,\alpha}(B_{2r_0})$ satisfy
	\begin{equation}
		\|\nabla U-q\|_{L^\infty(B_{2r_0})}\le\frac14|q|,
		\qquad q\ne0.
		\label{eq:noncritical-gradient-tube}
	\end{equation}
	Fix a unit vector $\xi$, assume that $x,x+h\xi\in B_{2r_0}$, and set
	\[
	\Omega_{x,\rho,h}
	:=\{\omega\in\mathbb S^{n-1}:
	x-\rho\omega,\ x-\rho\omega+h\xi\in B_{2r_0}\}.
	\]
	For $\omega\in\Omega_{x,\rho,h}$, set
	\[
	A_0:=U(x)-U(x-\rho\omega),\qquad
	A_h:=U(x+h\xi)-U(x-\rho\omega+h\xi),
	\]
	and
	\[
	b_h(x,\rho,\omega)
	:=(p-1)\rho^{2-p}\int_0^1
	|(1-t)A_0+tA_h|^{p-2}\dd t.
	\]
	Then, uniformly for the admissible $x$, $\rho$ and $h$,
	\begin{equation}
		c\le b_h(x,\rho,\omega)
		\quad\text{for a.e.\ }\omega\in\Omega_{x,\rho,h},
		\qquad
		\int_{\Omega_{x,\rho,h}}b_h(x,\rho,\omega)\dd\omega\le C.
		\label{eq:regional-angular-ellipticity}
	\end{equation}
	Moreover, if $|h|\le\rho$, then
	\begin{equation}
		\int_{\Omega_{x,\rho,h}}
		\left|b_h(x,\rho,\omega)
		-(p-1)\left|\frac{A_0}{\rho}\right|^{p-2}
		\right|\dd\omega
		\le C|h|^{\alpha(p-1)}.
		\label{eq:regional-angular-convergence}
	\end{equation}
	The constants depend only on $n,p,|q|$ and
	$[\nabla U]_{C^\alpha(B_{2r_0})}$.  If all directions are admissible,
	$\Omega_{x,\rho,h}=\mathbb S^{n-1}$.
\end{lemma}

\begin{proof}
	Let $\nu=p-2\in(-1,0)$ and let $F,G$ be Lipschitz on
	$\mathbb S^{n-1}$, with
	\begin{align*}
		\|F-q\cdot\omega\|_\infty,\quad
		\|G-q\cdot\omega\|_\infty&\le\frac14|q|,\\
		\|\nabla_{\!S}(F-q\cdot\omega)\|_\infty,\quad
		\|\nabla_{\!S}(G-q\cdot\omega)\|_\infty
		&\le\frac14|q|.
	\end{align*}
	Write $e=q/|q|$ and $\omega=te+\sqrt{1-t^2}\,\theta$, with
	$\theta\in\mathbb S^{n-1}\cap e^\perp$.  For $H\in\{F,G\}$,
	\[
	\partial_tH\ge |q|-\frac{|q|}{4\sqrt{1-t^2}}
	\ge\left(1-\frac1{2\sqrt3}\right)|q|
	\quad\text{for a.e. }|t|\le\tfrac12.
	\]
	Moreover,
	\[
	|H|<a<|q|/4\quad\Longrightarrow\quad
	|t|\le\frac{|H|+|H-q\cdot\omega|}{|q|}<\frac12,
	\qquad\dd\omega=(1-t^2)^{(n-3)/2}\dd t\dd\theta.
	\]
	Changing variables in $t$ gives
	\begin{equation}
		\bigl|\{|F|<a\}\bigr|+\bigl|\{|G|<a\}\bigr|\le Ca
		\qquad(0<a<1).
		\label{eq:angular-sublevel-bound}
	\end{equation}
	For $a\ge|q|/4$, use $|\mathbb S^{n-1}|\le Ca$; here $C$ may depend
	on $|q|$.  If $\|F-G\|_\infty\le\delta\le1$, then
	\begin{equation}
		\int_{\mathbb S^{n-1}}
		\bigl||F|^\nu-|G|^\nu\bigr|\dd\omega
		\le C\delta^{\nu+1}=C\delta^{p-1}.
		\label{eq:singular-angular-modulus}
	\end{equation}
	Indeed, for $E=\{|F|\le2\delta\}$, the sublevel estimate and the
	mean value theorem give
	\[
	\begin{aligned}
		\int_E\bigl||F|^\nu-|G|^\nu\bigr|\dd\omega
		&\le\int_{\{|F|\le2\delta\}}|F|^\nu\dd\omega
		+\int_{\{|G|\le3\delta\}}|G|^\nu\dd\omega
		\le C\delta^{\nu+1},\\
		\int_{E^c}\bigl||F|^\nu-|G|^\nu\bigr|\dd\omega
		&\le C\delta\int_\delta^1 a^{\nu-1}\dd a+C\delta
		\le C\delta^{\nu+1}.
	\end{aligned}
	\]
	
	By convexity and \eqref{eq:noncritical-gradient-tube},
	\[
	f=U-q\cdot x,\qquad
	\operatorname{Lip}(f;B_{2r_0})\le|q|/4.
	\]
	Extend $f$ to $\R^n$ with the same Lipschitz constant.  The functions
	\[
	q\cdot\omega+
	\frac{f(x+j\xi)-f(x-\rho\omega+j\xi)}{\rho},
	\qquad j\in\{0,h\},
	\]
	and their convex combinations satisfy the four bounds on the sphere.
	On $\Omega_{x,\rho,h}$ they equal $F_j=A_j/\rho$.
	For $\widetilde F_t=(1-t)F_0+tF_h$, the sublevel bound gives
	\[
	|\widetilde F_t|\le\frac54|q|,
	\qquad \int_{\Omega_{x,\rho,h}}|\widetilde F_t|^\nu\dd\omega\le C,
	\]
	and hence
	\[
	b_h=(p-1)\int_0^1|\widetilde F_t|^\nu\dd t
	\ge(p-1)\left(\frac54|q|\right)^\nu,
	\qquad
	\int_{\Omega_{x,\rho,h}}b_h\dd\omega\le C.
	\]
	This proves \eqref{eq:regional-angular-ellipticity}.
	
	Convexity keeps the following segments in $B_{2r_0}$:
	\[
	F_j=\int_0^1
	\nabla U(x+j\xi-t\rho\omega)\cdot\omega\dd t,
	\qquad j\in\{0,h\},
	\]
	Thus
	\[
	\|F_h-F_0\|_{L^\infty(\Omega_{x,\rho,h})}
	\le [\nabla U]_{C^\alpha(B_{2r_0})}|h|^\alpha.
	\]
	The proof of \eqref{eq:singular-angular-modulus} uses only sublevel
	bounds and the difference on the integration set.  Applied to $\widetilde F_t,F_0$
	on $\Omega_{x,\rho,h}$ and integrated in $t$, it gives
	\[
	\begin{aligned}
		\int_{\Omega_{x,\rho,h}}
		\bigl|b_h-(p-1)|F_0|^{p-2}\bigr|\dd\omega\le(p-1)\int_0^1\int_{\Omega_{x,\rho,h}}
		\bigl||\widetilde F_t|^\nu-|F_0|^\nu\bigr|\dd\omega\dd t
		\le C|h|^{\alpha(p-1)},
	\end{aligned}
	\]
	which is \eqref{eq:regional-angular-convergence}.  If
	$[\nabla U]_{C^\alpha}|h|^\alpha>1$, the angular upper bounds give the
	same estimate directly.
\end{proof}

\begin{lemma}
	\label{lem:record-noncritical-neighborhood}
	Let \(U_*=q\cdot X+a_*v_*\) be the profile in
	Proposition~\ref{prop:no-linear-record-neck}.  If
	\(\beta<\alpha<sp-p+1\), then \(U_*\in C^{1,\alpha}(B_{2r_0})\)
	for some \(r_0>0\).  After
	decreasing \(r_0\),
	\begin{equation}
		\|\nabla U_*-q\|_{L^\infty(B_{2r_0})}\le\frac14|q|.
		\label{eq:record-noncritical-gradient-tube}
	\end{equation}
\end{lemma}

\begin{proof}
	For small $r>0$ and $x\in B_r$, set
	\[
	\ell_x(y):=U_*(x)+q\cdot(y-x),\qquad
	\widehat H=a_*r^\beta,\qquad
	\widehat Q=\frac q{\widehat H},
	\]
	\[
	\widehat z(X)
	:=\frac{U_*(x+rX)-\ell_x(x+rX)}{r\widehat H}
	=\frac{v_*(x+rX)-v_*(x)}{r^{1+\beta}}.
	\]
	Choose $r$ independently of $x$ so that $|\widehat Q|\ge1$.
	Then $\widehat z$ solves the equation shifted by $\widehat Q$ in every
	ball, and \eqref{eq:record-profile-growth} gives
	\begin{equation}
		\begin{aligned}
			|\widehat z(X)|
			&\le\frac{|x+rX|^{1+\beta}+|x|^{1+\beta}}{r^{1+\beta}}
			\le C(1+|X|)^{1+\beta},\\
			\|U_*-\ell_x\|_{L^\infty(B_{4r}(x))}
			&\le Ca_*r^{1+\beta}.
		\end{aligned}
		\label{eq:record-translated-affine-tube}
	\end{equation}
	The constants are independent of \(x\) and \(r\).
	
	Set $e=q/|q|$ and $\gamma=sp-p+1>\beta$.
	By \eqref{eq:large-slope-angular-flux},
	\[
	\lambda_{\widehat Q}^{2-p}
	|\mathcal D_{\widehat Q}\widehat z(X,Y)|
	\le C|X-Y|^{p-2}
	\left|e\cdot\frac{X-Y}{|X-Y|}\right|^{p-2}
	|\widehat z(X)-\widehat z(Y)|.
	\]
	For bounded $X$,
	\[
	\begin{gathered}
		|\delta\widehat z(X,Y)|\le C(1+|X-Y|)^{1+\beta},\qquad
		\int_{\mathbb S^{n-1}}|e\cdot\omega|^{p-2}\dd\omega<\infty,\\
		\int_1^\infty\rho^{p-2+\beta-sp}\dd\rho
		=\int_1^\infty\rho^{-1-(\gamma-\beta)}\dd\rho
		=\frac1{\gamma-\beta}.
	\end{gathered}
	\]
	The separated exterior regions therefore satisfy, with one constant $C_1$,
	\begin{equation}
		\mathcal T^{\rm pt}_{s,\widehat Q}(\widehat z)
		+\operatorname*{ess\,sup}_{X\in B_3}
		\int_{\R^n\setminus B_4}
		\frac{\lambda_{\widehat Q}^{2-p}
			|\mathcal D_{\widehat Q}\widehat z(X,Y)|}
		{|X-Y|^{n+sp}}\dd Y
		\le C_1.
		\label{eq:record-uniform-relative-tail}
	\end{equation}
	Choose $\eta\in C_c^\infty(B_3)$ with $0\le\eta\le1$, $\eta=1$ on
	$B_2$, and bounded gradient.  For fixed $r$, global Lipschitz continuity
	and \eqref{eq:record-uniform-relative-tail} justify the test
	$\eta^2\widehat z$ by smooth approximation in the original weak equation.
	For a regional pair, set
	\[
	m=\frac{\eta(X)^2+\eta(Y)^2}{2},\qquad
	\theta=\frac{\eta(X)+\eta(Y)}2,\qquad
	\theta^2\le m\le2\theta^2.
	\]
	Then
	\[
	\delta(\eta^2\widehat z)
	=m\,\delta\widehat z
	+\frac{\widehat z(X)+\widehat z(Y)}2\,
	\delta(\eta^2).
	\]
	The Bregman equivalence and Young inequality from
	Lemma~\ref{lem:shift-uniform-fixed-level-clearing}, together with
	$\|\widehat z\|_{L^\infty(B_4)}\le C$, give
	\[
	\begin{aligned}
		\theta^2\mathcal H_{\widehat Q,X-Y}(b/\theta)
		&\le C\mathcal H_{\widehat Q,X-Y}(b)\qquad(\theta>0),\\
		\lambda_{\widehat Q}^{2-p}\mathcal D_{\widehat Q}\widehat z
		\,\delta(\eta^2\widehat z)
		&\ge\frac m2\lambda_{\widehat Q}^{2-p}
		\mathcal D_{\widehat Q}\widehat z\,\delta\widehat z-C\mathcal H_{\widehat Q,X-Y}(C|\delta\eta|).
	\end{aligned}
	\]
	Pairs with $\theta=0$ vanish; crossed pairs are bounded by
	\eqref{eq:record-uniform-relative-tail}.  Hence
	\begin{align*}
		\mathcal E_{s,\widehat Q}(\widehat z;B_2)
		&\le C(1-s)\iint_{B_4^2}
		\frac{\mathcal H_{\widehat Q,X-Y}
			(C|\eta(X)-\eta(Y)|)}
		{|X-Y|^{n+sp}}\dd X\dd Y\\
		&\quad+C\|\widehat z\|_{L^\infty(B_3)}
		\operatorname*{ess\,sup}_{X\in B_3}
		\int_{\R^n\setminus B_4}
		\frac{\lambda_{\widehat Q}^{2-p}
			|\mathcal D_{\widehat Q}\widehat z(X,Y)|}
		{|X-Y|^{n+sp}}\dd Y .
	\end{align*}
	Since $|\widehat Q|\ge1$, \eqref{eq:large-slope-bregman-equivalence} gives
	\[
	\mathcal H_{\widehat Q,X-Y}(C|\delta\eta|)
	\le C|X-Y|^{p-2}
	\left|e\cdot\frac{X-Y}{|X-Y|}\right|^{p-2}
	|\eta(X)-\eta(Y)|^2.
	\]
	Angular integration and $|\delta\eta|\le C|X-Y|$ yield
	\[
	(1-s)\iint_{B_4^2}
	\frac{\mathcal H_{\widehat Q,X-Y}(C|\delta\eta|)}{|X-Y|^{n+sp}}
	\dd X\dd Y
	\le C(1-s)\int_0^8\rho^{p-sp-1}\dd\rho\le C.
	\]
	Together with \eqref{eq:record-uniform-relative-tail}, this gives
	\[
	\fint_{B_2}|\widehat z|^p\dd X
	+\mathcal E_{s,\widehat Q}(\widehat z;B_2)\le C_1,
	\]
	after increasing \(C_1\).
	
	Choose $C_2\ge1$ depending only on $C_1$, and set
	\[
	H_0=C_2\widehat H,\qquad
	z_{x,r}=\frac{\widehat z}{C_2},\qquad
	Q_{x,r}=\frac{\widehat Q}{C_2}.
	\]
	Exact homogeneity gives
	\begin{align*}
		\mathcal E_{s,Q_{x,r}}(z_{x,r};B_2)
		&=C_2^{-p}
		\left(\frac{\lambda_{Q_{x,r}}}{\lambda_{\widehat Q}}\right)^{2-p}
		\mathcal E_{s,\widehat Q}(\widehat z;B_2)
		\le C_1C_2^{-p},\\
		\mathcal T^{\rm pt}_{s,Q_{x,r}}(z_{x,r})
		&=C_2^{1-p}
		\left(\frac{\lambda_{Q_{x,r}}}{\lambda_{\widehat Q}}\right)^{2-p}
		\mathcal T^{\rm pt}_{s,\widehat Q}(\widehat z)
		\le C_1C_2^{1-p}.
	\end{align*}
	Also,
	\[
	\fint_{B_2}|z_{x,r}|^p\dd X
	=C_2^{-p}\fint_{B_2}|\widehat z|^p\dd X\le C_1C_2^{-p}.
	\]
	Fix $C_2$ with $C_1C_2^{-p},C_1C_2^{1-p}\le1$, then set
	$L_{\rm rec}=\max\{1,C_1\}$.  These choices are independent of $x,r$.
	Finally decrease $r$ until
	\[
	|Q_{x,r}|=\frac{|q|}{C_2a_*r^\beta}
	\ge Q_{\rm it}(n,p,s,\alpha,L_{\rm rec}).
	\]
	Lemma~\ref{lem:iterable-large-slope-state}, uniformly for $x\in B_r$, yields
	\[
	\inf_{m\in\mathcal A}
	\|U_*-m\|_{L^\infty(B_t(x))}
	\le C H_0r^{-\alpha}t^{1+\alpha}
	\qquad(0<t\le r/2).
	\]
	The Campanato characterization on overlapping balls gives
	$U_*\in C^{1,\alpha}(B_{r/2})$.  Set $r_0=r/4$ and decrease $r_0$ so that
	\[
	\|\nabla U_*-q\|_{L^\infty(B_{2r_0})}
	\le [\nabla U_*]_{C^\alpha(B_{r/2})}(2r_0)^\alpha
	\le\frac14|q|,
	\]
	using $\nabla U_*(0)=q$.  This is
	\eqref{eq:record-noncritical-gradient-tube}.
\end{proof}

\subsection{Weak linearization and a strong minimum principle}
\label{subsec:morrey-kato-linearization}

Lemma~\ref{lem:record-noncritical-neighborhood} gives
\eqref{eq:record-noncritical-gradient-tube}; hence
Lemma~\ref{lem:regional-angular-compactness} applies with this bound
in place of \eqref{eq:noncritical-gradient-tube}.

\begin{proposition}
	\label{prop:crossed-sublinear-absorption}
	Let $U_*$, $e=q/L_V$ and $r_0$ be as above.  For $0<h<r_0/4$, set
	\begin{equation*}
		V_h(X):=U_*(X+he)-L_Vh,
		\qquad Z_h:=U_*-V_h.
	\end{equation*}
	Then $V_h$ is a weak solution of \eqref{eq:equation} in $\R^n$, $0\le Z_h\le2L_Vh$ in $\R^n$, and
	the regional secant kernel in $B_{r_0}$ satisfies the pointwise lower
	bound and the upper bound for the angular integral of order $sp-p+2$ in
	Lemma~\ref{lem:regional-angular-compactness}.  Moreover, for every
	$\varphi\in C_c^\infty(B_{r_0/2})$, $\varphi\ge0$,
	\begin{equation}
		\begin{split}
			(1-s)\iint_{B_{r_0}\times B_{r_0}}
			\frac{c_h(X,Y)(Z_h(X)-Z_h(Y))
				(\varphi(X)-\varphi(Y))}
			{|X-Y|^{n+sp}}\dd X\!\dd Y\ge
			-C\int_{B_{r_0/2}}Z_h^{p-1}\varphi\dd X,
		\end{split}
		\label{eq:regional-sublinear-absorption}
	\end{equation}
	where
	\begin{equation*}
		c_h(X,Y):=(p-1)\int_0^1
		|(1-t)\delta V_h(X,Y)+t\delta U_*(X,Y)|^{p-2}\dd t.
	\end{equation*}
	On the set where both increments vanish, the integrand and \(c_h\)
	are defined to be zero.  This convention does not alter the secant
	identity, whose two sides vanish there.
	The exponent $p-1$ in \eqref{eq:regional-sublinear-absorption} is the
	optimal exponent in the scalar flux estimate when the constant is
	independent of the base value.
\end{proposition}

\begin{proof}
	The equation is unchanged under shifts of the variables and addition
	of constants, so \(V_h\) is a solution.  Moreover,
	\[
	|U_*(X+he)-U_*(X)|\le L_Vh
	\quad\Longrightarrow\quad 0\le Z_h\le2L_Vh.
	\]
	Subtracting the weak equations and using
	\[
	\Jp(\delta U_*)-\Jp(\delta V_h)=c_h\,\delta Z_h
	\]
	gives the regional identity.  The bounds for \(c_h|X-Y|^{2-p}\)
	follow from \eqref{eq:regional-angular-ellipticity}.
	For \(X\in B_{r_0/2}\), \(Y\notin B_{r_0}\), and
	\(a=\delta V_h(X,Y)\), monotonicity and the global
	\((p-1)\)-H\"older bound for \(\Jp\) give
	\[
	\begin{aligned}
		\Jp(\delta U_*)-\Jp(\delta V_h)
		&=\Jp(a+Z_h(X)-Z_h(Y))-\Jp(a)\\
		&\le\Jp(a+Z_h(X))-\Jp(a)
		\le CZ_h(X)^{p-1},\\
		\int_{\R^n\setminus B_{r_0}}|X-Y|^{-n-sp}\dd Y
		&\le Cr_0^{-sp}.
	\end{aligned}
	\]
	Moving the crossed term proves
	\eqref{eq:regional-sublinear-absorption}.  Optimality follows from
	\[
	|\Jp(0+t)-\Jp(0)|=|t|^{p-1}.
	\]
\end{proof}

We first prove a local strong minimum principle using only regional
kernel bounds and a Morrey bound for the potential.

For a symmetric kernel \(K\) on \(D\times D\), let
\(\mathcal F_K(D)\) denote the closure of \(C_c^\infty(D)\) under the
norm
\[
\|\phi\|_{\mathcal F_K(D)}^2
:=\|\phi\|_{L^2(D)}^2
+\iint_{D\times D}K(x,y)|\delta\phi(x,y)|^2\dd x\!\dd y.
\]
We write \(z\in\mathcal F_{K,\mathrm{loc}}(D)\) if
\(\eta z\in\mathcal F_K(D)\) for every
\(\eta\in C_c^\infty(D)\), and use
\(\mathcal F_{K,c}(D)\) for elements of \(\mathcal F_K(D)\) with
compact support in \(D\).
In integrals over \(\R^n\), elements of \(\mathcal F_{K,c}(D)\)
are extended by zero on \(\R^n\setminus D\).

\begin{lemma}
	\label{lem:kato-contact-propagation}
	Let \(n\ge2\), let \(D\subset\R^n\) be a ball, let
	\(1<\tau<2\), and let \(K\) be a symmetric measurable kernel in
	\(D\times D\).  The same statement holds when \(D\) is a connected
	convex open set.  Suppose that, in polar coordinates about every
	$x\in D$,
	\begin{equation}
		K(x,x+\rho\omega)
		=\frac{a(x,\rho,\omega)}{\rho^{n+\tau}},
		\qquad
		a\ge\lambda>0,\qquad
		\int_{\{\omega:\,x+\rho\omega\in D\}}
		a(x,\rho,\omega)\dd\omega\le\Lambda
		\label{eq:spectral-regional-bounds}
	\end{equation}
	for almost every \((\rho,\omega)\) such that
	\(x+\rho\omega\in D\).
	The angular upper bound covers all pairs in \(D\times D\);
	pairs outside \(D\) are unrestricted.
	Let $A\ge0$ satisfy, locally uniformly in $D$,
	\begin{equation}
		\int_{B_r(x)}A\dd y\le M r^{n-\kappa},
		\qquad 0<\kappa<\tau.
		\label{eq:morrey-absorption}
	\end{equation}
	Assume that \(z\ge0\), that
	\(z\in C(D)\cap L^\infty(D)\cap\mathcal F_{K,\mathrm{loc}}(D)\), and
	assume that
	\begin{equation}
		\iint_{D\times D}K(x,y)\delta z(x,y)\delta\varphi(x,y)
		\dd x\!\dd y
		\ge-\int_DAz\varphi\dd x
		\label{eq:regional-supersolution-with-potential}
	\end{equation}
	for every nonnegative
	\(\varphi\in\mathcal F_{K,c}(D)\).  If $z$ vanishes at one
	point of $D$, then $z\equiv0$ in $D$.
	The conclusion is local in the following sense: if
	\eqref{eq:regional-supersolution-with-potential} is known only for tests
	supported in an open set \(D_0\Subset D\), then every zero of \(z\) in
	\(D_0\) has a neighborhood on which \(z\) vanishes.
\end{lemma}

\begin{proof}
	For the local assertion, fix \(B_{4R_0}\Subset D\), or
	\(B_{4R_0}\Subset D_0\) if tests are restricted to \(D_0\).
	Choose \(R_0\) so that \eqref{eq:morrey-absorption} holds with the
	same \(M\) on the balls below.  Constants may depend on this
	localization, but not on the regularizing parameter.
	Set \(\dd\mu=A\mathbf1_{B_{3R_0}}\dd x\).  Dyadic decomposition gives
	\begin{equation}
		\sup_{x\in B_{3R_0}}
		\sum_{j=0}^{\infty}(2^{-j}R)^{\tau-n}
		\mu(B_{2^{-j}R}(x))
		\le CMR^{\tau-\kappa},
		\qquad 0<R\le R_0.
		\label{eq:morrey-to-kato}
	\end{equation}
	Let \(B_{2R}\subset B_{3R_0}\) and \(\nu=\mu\lfloor B_R\).
	For general trace inequalities, see \cite{Adams1973}.
	If \(B_r(x)\cap B_R\ne\varnothing\), choose
	\(x'\in B_r(x)\cap B_R\).  Then
	\[
	\nu(B_r(x))\le
	\begin{cases}
		\mu(B_{2r}(x'))\le CM r^{n-\kappa},&0<r<R/2,\\
		\nu(\R^n)\le MR^{n-\kappa}\le CM r^{n-\kappa},&r\ge R/2.
	\end{cases}
	\]
	In the first case \(B_{2r}(x')\subset B_{2R}\).
	Thus the bound holds for all \(x\in\R^n\), \(r>0\).
	Since \(0<\kappa<\tau<n\),
	\[
	S_R:=\sup_{x\in\R^n}\int_{\R^n}|x-y|^{\tau-n}\dd\nu(y)
	\le CM\left(\sum_{j=0}^{\infty}(2^{-j}R)^{\tau-\kappa}
	+R^{\tau-\kappa}\right)
	\le CMR^{\tau-\kappa}.
	\]
	For the Riesz potential \(I_a\) with multiplier \(|\xi|^{-a}\)
	and bounded \(g\in L^2(\nu)\), kernel composition gives
	\begin{align*}
		\|I_{\tau/2}(g\nu)\|_{L^2(\R^n)}^2
		&\le C\iint_{\R^n\times\R^n}
		|x-y|^{\tau-n}|g(x)g(y)|\dd\nu(x)\dd\nu(y)\\
		&\le CS_R\|g\|_{L^2(\nu)}^2.
	\end{align*}
	Here \(2|g(x)g(y)|\le|g(x)|^2+|g(y)|^2\).
	Integrating the product of the Riesz kernels and applying Tonelli
	first to \(|g|\) also justifies signed \(g\) by absolute integrability.
	For \(\psi\in C_c^\infty(B_R)\), let \(\widetilde\psi\) be its zero
	extension.  Fourier inversion gives
	\begin{align*}
		\left|\int\psi g\dd\nu\right|
		&=\left|\int_{\R^n}(-\Delta)^{\tau/4}\widetilde\psi\,
		I_{\tau/2}(g\nu)\dd x\right|\\
		&\le CS_R^{1/2}
		[\widetilde\psi]_{\dot H^{\tau/2}(\R^n)}
		\|g\|_{L^2(\nu)}.
	\end{align*}
	Fourier cutoffs away from zero and infinity can be removed by
	uniform convergence on the left and \(L^2\) convergence on the right.
	Taking \(\sup_{\|g\|_{L^2(\nu)}\le1,\ g\text{ bounded}}\) yields
	\begin{equation}
		\int_{B_R}A\psi^2\dd x
		\le CMR^{\tau-\kappa}
		[\widetilde\psi]_{\dot H^{\tau/2}(\R^n)}^2.
		\label{eq:kato-form-bound}
	\end{equation}
	This bound makes approximants Cauchy in \(L^2(\nu)\), with the same
	Lebesgue limit since \(\nu\ll\dd x\).  For a regional form element
	compactly supported in \(B_R\), multiply its smooth approximants by
	a fixed cutoff in \(C_c^\infty(B_R)\), equal to one near its support.
	Their common support lies away from \(\partial D\);
	\eqref{eq:kato-zero-extension-localization} below then gives convergence
	of the zero extensions in \(H^{\tau/2}(\R^n)\).  Hence
	\eqref{eq:kato-form-bound} extends to these tests.
	
	For \(\psi\in C_c^\infty(B_{2R_0})\), choose a covering with bounded
	overlap by balls of radius \(\rho\), whose doubles lie in
	\(B_{3R_0}\), and subordinate smooth cutoffs with
	\(\sum_i\zeta_i^2=1\).  Then
	\[
	\sum_i|\zeta_i(x)\psi(x)-\zeta_i(y)\psi(y)|^2
	\le2|\delta\psi(x,y)|^2
	+2|\psi(y)|^2\sum_i|\zeta_i(x)-\zeta_i(y)|^2
	\]
	and bounded overlap gives
	\[
	\sum_i|\zeta_i(x)-\zeta_i(y)|^2
	\le C\min\left\{1,\frac{|x-y|^2}{\rho^2}\right\}.
	\]
	Integration and \eqref{eq:kato-form-bound} yield
	\[
	\begin{aligned}
		\sum_i[\zeta_i\psi]_{\dot H^{\tau/2}(\R^n)}^2
		&\le C[\widetilde\psi]_{\dot H^{\tau/2}(\R^n)}^2
		+C\rho^{-\tau}\|\psi\|_2^2,\\
		\int A\psi^2\dd x
		&\le CM\rho^{\tau-\kappa}
		[\widetilde\psi]_{\dot H^{\tau/2}(\R^n)}^2
		+CM\rho^{-\kappa}\|\psi\|_2^2.
	\end{aligned}
	\]
	With \(d_0=\operatorname{dist}(\operatorname{supp}\psi,\partial D)>0\),
	\eqref{eq:spectral-regional-bounds} and
	\[
	\sup_{x\in\operatorname{supp}\psi}
	\int_{\R^n\setminus D}|x-y|^{-n-\tau}\dd y\le Cd_0^{-\tau}
	\]
	give
	\begin{equation}
		[\widetilde\psi]_{\dot H^{\tau/2}(\R^n)}^2
		\le C\lambda^{-1}\mathcal E_K(\psi,\psi)
		+Cd_0^{-\tau}\|\psi\|_2^2.
		\label{eq:kato-zero-extension-localization}
	\end{equation}
	Choose \(\rho=c(\epsilon/(1+M))^{1/(\tau-\kappa)}\), capped by the
	geometric radius above, with \(c\) small enough to make the energy
	coefficient at most \(\epsilon\).  Substitution gives
	\begin{equation}
		\int A\psi^2\dd x
		\le \epsilon\mathcal E_K(\psi,\psi)
		+C_*(1+M)^{\tau/(\tau-\kappa)}
		\epsilon^{-\kappa/(\tau-\kappa)}\|\psi\|_2^2,
		\qquad 0<\epsilon\le1,
		\label{eq:kato-infinitesimal-form-bound}
	\end{equation}
	where
	\[
	\mathcal E_K(\psi,\psi)
	:=\iint_{D\times D}K(x,y)|\psi(x)-\psi(y)|^2\dd x\!\dd y.
	\]
	
	If \(0\le\eta\le1\),
	\(\eta\in C_c^{0,1}(B_R)\), and
	\(\|\nabla\eta\|_\infty\le C/d\), then
	\begin{equation}
		\mathcal K_\eta(x):=
		\int_DK(x,y)|\eta(x)-\eta(y)|^2\dd y
		\le C\Lambda d^{-\tau}
		\qquad \bigl(x\in D\bigr).
		\label{eq:kato-cutoff-kernel-bound}
	\end{equation}
	Indeed, extend \(\eta\) by zero and use the angular upper bound:
	\[
	|\delta\eta(x,y)|\le C\min\{1,|x-y|/d\},
	\qquad
	\mathcal K_\eta(x)
	\le C\Lambda\int_0^\infty
	\min\left\{1,\frac{\rho^2}{d^2}\right\}
	\frac{\dd\rho}{\rho^{1+\tau}}
	\le C\Lambda d^{-\tau}.
	\]
	
	Set \(v=z+\varepsilon\), where \(\varepsilon>0\).  Let
	\(B_r\Subset B_R\Subset B_{2R_0}\), with \(R/2\le r<R\), choose
	\(\eta\in C_c^{0,1}(B_R)\), \(0\le\eta\le1\), so that
	\(\eta=1\) on \(B_r\) and
	\(\|\nabla\eta\|_\infty\le C/(R-r)\), and set
	\[
	S=\{\eta>0\},\qquad m=1+2q,\qquad \psi=\eta v^{-q}.
	\]
	For every fixed \(q_0>0\) and \(q\ge q_0\), we use the elementary
	inequality
	\begin{equation}
		(a-b)(\xi^2a^{-1-2q}-\zeta^2b^{-1-2q})
		\le-\gamma_q
		|\xi a^{-q}-\zeta b^{-q}|^2
		+C_q(a^{-2q}+b^{-2q})|\xi-\zeta|^2,
		\label{eq:kato-interior-negative-power-scalar}
	\end{equation}
	whenever \(a,b>0\) and \(0\le\xi,\zeta\le1\), where
	\begin{equation*}
		0<\gamma_q\le1,\qquad
		\gamma_q\ge\frac{c_{q_0}}{1+q},
		\qquad C_q\le C_{q_0}(1+q)^2.
	\end{equation*}
	By symmetry, take \(a\ge b\), and set
	\(A=a^{-q}\), \(B=b^{-q}\), \(c_q=(1+2q)/q^2\).
	Cauchy--Schwarz on \([b,a]\) gives
	\[
	\begin{aligned}
		(a-b)(a^{-1-2q}-b^{-1-2q})&\le-c_q(B-A)^2,\\
		(a-b)a^{-1-q}&\le\int_b^a t^{-1-q}\dd t=\frac{B-A}{q}.
	\end{aligned}
	\]
	If \(\xi\le\zeta\), then
	\[
	\begin{aligned}
		(a-b)(\xi^2a^{-1-2q}-\zeta^2b^{-1-2q})
		&\le-c_q\xi^2(B-A)^2\\
		&\le-\frac{c_q}{2}|\xi A-\zeta B|^2
		+c_qB^2|\xi-\zeta|^2.
	\end{aligned}
	\]
	If \(\xi\ge\zeta\), set \(h=\xi-\zeta\).  Young's inequality and
	\(A(B-A)\le(A^2+B^2)/2\) give
	\[
	\begin{aligned}
		(a-b)(\xi^2a^{-1-2q}-\zeta^2b^{-1-2q})
		&\le-c_q\zeta^2(B-A)^2
		+(2\zeta h+h^2)\frac{A(B-A)}q\\
		&\le-\frac{c_q}{2}\zeta^2(B-A)^2
		+\left(\frac{2}{c_qq^2}+\frac1{2q}\right)h^2(A^2+B^2),\\
		|\xi A-\zeta B|^2
		&\le2\zeta^2(B-A)^2+2h^2A^2.
	\end{aligned}
	\]
	Thus \(\gamma_q=\min\{1,c_q/4\}\asymp(1+q)^{-1}\), and
	\(C_q\le C_{q_0}\) for \(q\ge q_0\), proving the stated bounds.
	
	For \(x\in S\), \(y\in D\setminus S\), the test vanishes at \(y\), and
	\begin{equation*}
		(v(x)-v(y))\eta(x)^2v(x)^{-1-2q}
		\le
		\begin{cases}
			0, & v(x)\le v(y),\\
			\eta(x)^2v(x)^{-2q}, & v(x)>v(y).
		\end{cases}
	\end{equation*}
	Hence, after decreasing \(\gamma_q\) by a fixed factor,
	\begin{equation}
		(v(x)-v(y))\eta(x)^2v(x)^{-1-2q}
		\le-\gamma_q|\psi(x)-\psi(y)|^2
		+(1+\gamma_q)v(x)^{-2q}
		|\eta(x)-\eta(y)|^2.
		\label{eq:kato-crossed-negative-power-scalar}
	\end{equation}
	Only \(v(x)^{-2q}\) occurs on the right.
	
	Use bounded energy approximations of
	\(\eta^2v^{-1-2q}\) in
	\eqref{eq:regional-supersolution-with-potential}.  Since \(z\le v\),
	\[
	\int_DAz\eta^2v^{-1-2q}\dd x
	\le\int_DA\psi^2\dd x.
	\]
	Integrating \eqref{eq:kato-interior-negative-power-scalar} on
	\(S^2\), using
	\eqref{eq:kato-crossed-negative-power-scalar} on
	\(S\times(D\setminus S)\), and exploiting symmetry, we obtain
	\begin{equation}
		\gamma_q\mathcal E_K(\psi,\psi)
		\le C_q\int_Sv^{-2q}\mathcal K_\eta\dd x
		+\int_SA\psi^2\dd x.
		\label{eq:kato-negative-caccioppoli-preabsorption}
	\end{equation}
	By \eqref{eq:kato-infinitesimal-form-bound} with
	\(\epsilon=\gamma_q/2\), \eqref{eq:kato-cutoff-kernel-bound}, and
	\(\gamma_q\asymp(1+q)^{-1}\), absorption gives
	\begin{equation}
		\mathcal E_K(\psi,\psi)
		\le C(1+q)^\nu
		\bigl((R-r)^{-\tau}+C_0\bigr)
		\int_{B_R}v^{-2q}\dd x,
		\qquad
		\nu=\max\left\{3,\frac{\tau}{\tau-\kappa}\right\}.
		\label{eq:kato-negative-energy}
	\end{equation}
	Here \(C_0\) depends only on the localization and hypothesis constants.
	The powers of $1+q$ follow from
	\[
	\gamma_q^{-1}C_q\le C(1+q)^3,\qquad
	\gamma_q^{-1}(\gamma_q/2)^{-\kappa/(\tau-\kappa)}
	\le C(1+q)^{\tau/(\tau-\kappa)}.
	\]
	The lower spectral bound,
	\eqref{eq:kato-zero-extension-localization}, and the fractional
	Sobolev inequality give, with
	\(\chi=n/(n-\tau)>1\) since \(\tau<2\le n\),
	\begin{equation}
		\left(\fint_{B_r}v^{-2q\chi}\dd x\right)^{1/\chi}
		\le C(1+q)^\nu
		\left[1+\left(\frac{R}{R-r}\right)^\tau\right]
		\fint_{B_R}v^{-2q}\dd x.
		\label{eq:kato-negative-power-step}
	\end{equation}
	Set
	\[
	q_j=q_0\chi^j,\qquad R_j=\frac R2+\frac R4\,2^{-j},
	\qquad
	\sum_{j=0}^\infty\frac{1+j+\log(1+q_j)}{q_j}<\infty.
	\]
	Taking powers \(1/(2q_j)\) and iterating gives
	\begin{equation}
		\operatorname*{ess\,sup}_{B_{R/2}}v^{-1}
		\le C\left(\fint_{B_{3R/4}}v^{-2q_0}\dd x\right)^{1/(2q_0)}
		\qquad(q_0>0).
		\label{eq:kato-reverse-moser}
	\end{equation}
	
	For the initial negative moment, test with \(\eta^2/v\).  On
	\(S^2\), use
	\[
	(a-b)(\xi^2a^{-1}-\zeta^2b^{-1})
	\le-c\min\{\xi^2,\zeta^2\}
	|\log a-\log b|^2+C|\xi-\zeta|^2
	\]
	with \(c=1/2\), \(C=3\).  Indeed, for \(a\ge b\), set
	\[
	d=\frac{a-b}{\sqrt{ab}}\ge\log\frac ab,\qquad
	u=\frac{a-b}{a}\le\min\{1,d\}.
	\]
	Then
	\[
	(a-b)(\xi^2a^{-1}-\zeta^2b^{-1})
	\le
	\begin{cases}
		-\xi^2d^2,&\xi\le\zeta,\\
		-\zeta^2d^2+2\zeta hu+h^2u
		\le-\frac12\zeta^2d^2+3h^2,&h=\xi-\zeta\ge0.
	\end{cases}
	\]
	For \(x\in S\), \(y\notin S\),
	\[
	(v(x)-v(y))\frac{\eta(x)^2}{v(x)}
	\le
	\begin{cases}
		0,&v(x)\le v(y),\\
		\eta(x)^2,&v(x)>v(y).
	\end{cases}
	\]
	Since \(z/v\le1\), the cutoff and Morrey bounds give, on every
	sufficiently small ball,
	\begin{equation}
		\iint_{B_{3R/4}\times B_{3R/4}}
		\frac{|\log v(x)-\log v(y)|^2}{|x-y|^{n+\tau}}
		\dd x\!\dd y
		\le C R^{n-\tau}+CMR^{n-\kappa}
		\le C R^{n-\tau}.
		\label{eq:kato-logarithmic-energy}
	\end{equation}
	Here \(\kappa<\tau\) and \(R\le R_0\).  On subballs, fractional
	Poincar\'e gives
	\[
	\fint_{B_\varrho}
	\left|\log v-(\log v)_{B_\varrho}\right|\dd x
	\le C\varrho^{(\tau-n)/2}
	\left(
	\iint_{B_\varrho^2}
	\frac{|\delta\log v(x,y)|^2}{|x-y|^{n+\tau}}\dd x\!\dd y
	\right)^{1/2}
	\le C.
	\]
	Perform these estimates on a slightly larger concentric ball, still
	within the allowed test region.  Small balls meeting \(B_{3R/4}\)
	have their doubles there; balls of radius comparable to \(R\) are
	controlled by its logarithmic energy.  Thus
	\[
	[\log v]_{\mathrm{BMO}(B_{3R/4})}\le C,
	\]
	uniformly in \(\varepsilon\).  John--Nirenberg gives
	\(\theta>0\), independent of \(\varepsilon\), with
	\begin{equation}
		\left(\fint_{B_{3R/4}}v^\theta\dd x\right)
		\left(\fint_{B_{3R/4}}v^{-\theta}\dd x\right)\le C.
		\label{eq:kato-positive-negative-moment}
	\end{equation}
	Taking \(q_0=\theta/2\) in \eqref{eq:kato-reverse-moser} and using
	\eqref{eq:kato-positive-negative-moment} yields
	\begin{equation}
		\left(\fint_{B_{3R/4}}(z+\varepsilon)^\theta\dd x\right)^{1/\theta}
		\le C\operatorname*{ess\,inf}_{B_{R/2}}(z+\varepsilon).
		\label{eq:kato-weak-harnack}
	\end{equation}
	
	For test admissibility, fix \(\varepsilon,m>0\) and set
	\(F_{\varepsilon,m}(t)=(t+\varepsilon)^{-m}\).
	Choose \(\chi\in C_c^\infty(D)\), \(0\le\chi\le1\), which equals one
	near the cutoff support.  Then
	\[
	\chi z\in\mathcal F_K(D),\qquad
	G_{\varepsilon,m}=F_{\varepsilon,m}-F_{\varepsilon,m}(0),
	\qquad G_{\varepsilon,m}(0)=0.
	\]
	Extend \(G_{\varepsilon,m}\) from \([0,\|z\|_\infty]\) to a
	globally Lipschitz \(\widehat G_{\varepsilon,m}\) with the same value
	at zero.  For any Lipschitz \(F\) with \(F(0)=0\),
	\[
	|\delta F(u)|\le\operatorname{Lip}(F)|\delta u|,
	\qquad
	\mathcal E_K(F(u),F(u))
	\le\operatorname{Lip}(F)^2\mathcal E_K(u,u).
	\]
	For smooth compactly supported \(u\), approximate \(F\) by smooth
	functions vanishing at zero with the same Lipschitz bound; dominated
	convergence gives \(F(u)\in\mathcal F_K(D)\).
	For \(u_j\to u\) in the form norm,
	\[
	F(u_j)\to F(u)\text{ in }L^2,\qquad
	\sup_j\|F(u_j)\|_{\mathcal F_K(D)}<\infty.
	\]
	Weak compactness and lower semicontinuity give the same membership
	and energy bound for \(u\in\mathcal F_K(D)\).  In particular,
	\[
	\mathcal E_K(\widehat G_{\varepsilon,m}(\chi z),
	\widehat G_{\varepsilon,m}(\chi z))
	\le \operatorname{Lip}(\widehat G_{\varepsilon,m})^2
	\mathcal E_K(\chi z,\chi z).
	\]
	Compactly supported Lipschitz cutoffs belong to \(\mathcal F_K(D)\):
	their mollifications have common compact support and bounded
	Lipschitz constants, while the error increments converge to zero and
	are bounded by \(C\min\{1,|x-y|\}\).  The angular bound, symmetry,
	and dominated convergence give form convergence.
	For a cutoff \(\eta\),
	\[
	\mathcal E_K(\eta u,\eta u)
	\le2\|\eta\|_\infty^2\mathcal E_K(u,u)
	+2\int_D|u(y)|^2\mathcal K_\eta(y)\dd y.
	\]
	Thus \eqref{eq:kato-cutoff-kernel-bound} makes multiplication by
	\(\eta\) bounded in the form norm; approximation from smooth \(u\)
	preserves form membership.  Adding
	\(F_{\varepsilon,m}(0)\eta\) or \(F_{\varepsilon,m}(0)\eta^2\) gives
	\[
	\eta v^{-q},\quad\eta^2v^{-1-2q},\quad\eta^2/v
	\in\mathcal F_{K,c}(D).
	\]
	These tests are admissible in
	\eqref{eq:regional-supersolution-with-potential}; the potential term
	is continuous by \eqref{eq:kato-infinitesimal-form-bound}.
	Thus the scalar inequalities extend to the closed form domain, with
	all resulting estimates uniform as \(\varepsilon\to 0\).
	
	If \(z(x_0)=0\), center the balls at \(x_0\).  Continuity and
	nonnegativity give
	\[
	|\{z<\delta\}\cap B_{R/2}(x_0)|>0\quad(\delta>0),
	\qquad
	\operatorname*{ess\,inf}_{B_{R/2}(x_0)}(z+\varepsilon)=\varepsilon.
	\]
	Consequently, \eqref{eq:kato-weak-harnack} and
	\(\varepsilon\to 0\) yield
	\[
	\int_{B_{3R/4}(x_0)}z^\theta\dd x=0
	\quad\Longrightarrow\quad z=0\text{ in }B_{3R/4}(x_0).
	\]
	This proves the local assertion.  If tests are allowed throughout
	\(D\), the nonempty zero set is relatively open by this assertion
	and relatively closed by continuity.  Connectedness gives \(z\equiv0\).
\end{proof}

\begin{proposition}
	\label{prop:coarea-weak-linearization}
	Let $U$ be a globally $L$-Lipschitz weak solution of
	\eqref{eq:equation}. Suppose that, for some $r_0>0$, $0<\alpha<1$,
	$q\ne0$, and $e=q/|q|$,
	\begin{equation}
		U\in C^{1,\alpha}(B_{2r_0}),
		\qquad
		\|\nabla U-q\|_{L^\infty(B_{2r_0})}\le\frac14|q|.
		\label{eq:strict-directional-monotonicity}
	\end{equation}
	In particular, $\partial_eU\ge\lambda_0:=3|q|/4$ in $B_{2r_0}$.
	Set
	\[
	\tau:=sp-p+2,
	\qquad
	J_U(x,y):=(p-1)
	\frac{|U(x)-U(y)|^{p-2}}{|x-y|^{n+sp}}.
	\]
	Set \(J_U(x,y)=0\) when \(U(x)=U(y)\).  Strict monotonicity on
	lines parallel to \(e\) gives
	\[
	|\{x\in B_{2r_0}:U(x)=U(y)\}|=0\qquad(y\in\R^n),
	\]
	so this convention leaves all forms below unchanged.
	Then $J_U$ is a symmetric nonnegative measurable kernel for the
	local bilinear forms below. Set
	\[
	A_U(x):=\int_{\R^n\setminus B_{r_0}}J_U(x,y)\dd y.
	\]
	For every $B_r(x_0)\subset B_{r_0/2}$,
	\begin{equation}
		\int_{B_r(x_0)}A_U(x)\dd x\le Cr^{n+p-2}.
		\label{eq:crossed-morrey-bound}
	\end{equation}
	More generally, for every \(D_0\Subset B_{r_0}\) the same estimate
	holds for balls \(B_r(x_0)\subset D_0\), with a constant which may also
	depend on \(\operatorname{dist}(D_0,\partial B_{r_0})\).
	In particular, the codimension in this Morrey bound is strictly below
	the regional order $\tau$, since
	\begin{equation}
		\tau-(2-p)=sp>0.
		\label{eq:strict-kato-gap}
	\end{equation}
	
	Moreover, the a.e.\ directional derivative \(w=\partial_eU\) belongs to
	\(\mathcal F_{J_U,\mathrm{loc}}(B_{r_0})\) and satisfies
	\begin{equation}
		\iint_{\R^n\times\R^n}
		J_U(x,y)\delta w(x,y)\delta\varphi(x,y)\dd x\!\dd y=0
		\label{eq:full-weak-linearization}
	\end{equation}
	for every \(\varphi\in C_c^\infty(B_{r_0})\).  The identity extends to
	every \(\varphi\in\mathcal F_{J_U,c}(B_{r_0})\).
\end{proposition}

\begin{proof}
	For \(t\in[0,1]\), \(0\le h<r_0/4\), set
	\[
	W_{t,h}(x):=(1-t)U(x)+tU(x+he),
	\]
	Thus \(W_{t,0}=U\).  For each \(K\Subset B_{r_0}\), choose
	\(h_K>0\), depending only on \(K\), small enough that
	\begin{equation}
		\sup_{\substack{0\le h<h_K\\t\in[0,1]\\a\in\R}}
		\int_E|W_{t,h}(x)-a|^{p-2}\dd x
		\le C_K|E|^{p-1}
		\qquad(E\subset K).
		\label{eq:quantitative-directional-coarea}
	\end{equation}
	Indeed, write \(x=x'+\xi e\), \(x'\in e^\perp\).
	Since \(\partial_eW_{t,h}\ge\lambda_0\), the one-dimensional area
	formula, rearrangement about \(a\), and H\"older give
	\[
	\begin{aligned}
		\int_{E_{x'}}|W_{t,h}(x'+\xi e)-a|^{p-2}\dd\xi
		&\le C|E_{x'}|^{p-1},\\
		\int_E|W_{t,h}(x)-a|^{p-2}\dd x
		&\le C\int_{P_{e^\perp}K}|E_{x'}|^{p-1}\dd x'\le C|P_{e^\perp}K|^{2-p}|E|^{p-1}.
	\end{aligned}
	\]
	For \(B_r(x_0)\subset D_0\Subset B_{r_0}\), set
	\(d=\operatorname{dist}(D_0,\partial B_{r_0})>0\).
	With \(h=0\), \(a=U(y)\), the line estimate gives
	\[
	\begin{aligned}
		\int_{B_r(x_0)}|U(x)-U(y)|^{p-2}\dd x
		&\le Cr^{n-1}r^{p-1}=Cr^{n+p-2},\\
		\sup_{x\in D_0}|x-y|^{-n-sp}
		&\le C_d(1+|y|)^{-n-sp}\qquad(y\notin B_{r_0}),\\
		\int_{B_r(x_0)}A_U(x)\dd x
		&\le C_dr^{n+p-2}
		\int_{\R^n\setminus B_{r_0}}(1+|y|)^{-n-sp}\dd y
		\le C_dr^{n+p-2}.
	\end{aligned}
	\]
	This proves \eqref{eq:crossed-morrey-bound} and its local version.
	
	For \(0<h<r_0/4\), set \(U_h(x):=U(x+he)\),
	$w_h:=\frac{U_h-U}{h}$, and
	\[
	c_h(x,y):=(p-1)\int_0^1
	|(1-t)\delta U(x,y)+t\delta U_h(x,y)|^{p-2}\dd t.
	\]
	Set \(c_h=0\) when both increments vanish; otherwise the integral is
	finite since \(p-2>-1\).  Subtracting the  weak equations and using
	\[
	\Jp(\delta U_h)-\Jp(\delta U)=h c_h\,\delta w_h
	\]
	gives
	\begin{equation}
		\iint_{\R^n\times\R^n}
		\frac{c_h(x,y)\delta w_h(x,y)\delta\varphi(x,y)}
		{|x-y|^{n+sp}}\dd x\!\dd y=0.
		\label{eq:exact-secant-linear-equation}
	\end{equation}
	By \eqref{eq:quantitative-directional-coarea} and
	Lemma~\ref{lem:regional-angular-compactness}, reducing \(h_K\) if needed,
	\begin{equation}
		\sup_{0<h<h_K}
		\int_K\int_{\R^n}(1\wedge|x-y|^2)
		\frac{c_h(x,y)}{|x-y|^{n+sp}}\dd y\!\dd x<\infty.
		\label{eq:averaged-secant-levy-bound}
	\end{equation}
	Indeed, near the diagonal,
	\[
	\begin{aligned}
		c_h(x,x-\rho\omega)=\rho^{p-2}b_h(x,\rho,\omega),
		&\qquad\int_{\mathbb S^{n-1}}b_h(x,\rho,\omega)\dd\omega\le C,\\
		\int_0^1\rho^{p-2}\rho^2\rho^{-sp-1}\dd\rho
		&=\int_0^1\rho^{1-\tau}\dd\rho<\infty.
	\end{aligned}
	\]
	On separated sets, \eqref{eq:quantitative-directional-coarea}
	in \(x\), followed by H\"older in \(y\), gives uniform integrability:
	for \(K\Subset B_{r_0}\), \(0<\varepsilon<R<\infty\), and
	\[
	G\subset\{(x,y):x\in K,\ \varepsilon\le|x-y|\le R\},
	\]
	we have
	\begin{equation}
		\sup_{0<h<h_K}\iint_G
		\frac{c_h(x,y)}{|x-y|^{n+sp}}\dd x\!\dd y
		\le C_{K,\varepsilon,R}|G|^{p-1}.
		\label{eq:separated-secant-uniform-integrability}
	\end{equation}
	Indeed, let $G_y:=\{x\in K:(x,y)\in G\}$.
	The $y$-projection of $G$ lies in the bounded set $K+B_R$, and
	\eqref{eq:quantitative-directional-coarea} yields
	\[
	\begin{aligned}
		\iint_G\frac{c_h(x,y)}{|x-y|^{n+sp}}\dd x\dd y
		&\le C_{K,\varepsilon}\int_{K+B_R}|G_y|^{p-1}\dd y\\
		&\le C_{K,\varepsilon,R}
		\left(\int_{K+B_R}|G_y|\dd y\right)^{p-1}
		=C_{K,\varepsilon,R}|G|^{p-1}.
	\end{aligned}
	\]
	This proves uniform integrability on the separated sets.
	
	Fix \(\eta\in C_c^\infty(B_{r_0})\), and let \(h_\eta\) correspond
	to a compact neighborhood of \(\operatorname{supp}\eta\).
	For fixed \(h>0\),
	\[
	U,U_h\in W^{s,p}_{\rm loc},\qquad
	\eta^2w_h\in W^{s,p}_0(B_{r_0}),\qquad |w_h|\le L,
	\]
	and
	\[
	c_h|\delta w_h|^2
	=h^{-2}\bigl[\Jp(\delta U_h)-\Jp(\delta U)\bigr]
	(\delta U_h-\delta U)
	\le Ch^{-2}|\delta U_h-\delta U|^p.
	\]
	Use common smooth \(W^{s,p}_0(B_{r_0})\) approximants to
	\(\eta^2w_h\) in the original equations for \(U_h\) and \(U\), then
	subtract.  The limits follow from local \(W^{s,p}\) bounds and
	\[
	\int_1^\infty r^{p-2-sp}\dd r<\infty
	\qquad(sp>p-1),
	\]
	which controls the global Lipschitz tail.  The tested integrand in
	\eqref{eq:exact-secant-linear-equation} is absolutely integrable as
	the difference of the two original integrands divided by \(h\).
	The cutoff term is uniformly bounded in $L^1$ by
	\eqref{eq:averaged-secant-levy-bound}, since
	\[
	|\delta\eta(x,y)|^2\le C_\eta(1\wedge|x-y|^2),
	\quad
	\delta\eta\ne0\ \Longrightarrow\
	\{x,y\}\cap\operatorname{supp}\eta\ne\varnothing.
	\]
	Therefore
	\[
	(a-b)(\eta_x^2a-\eta_y^2b)
	=|\eta_xa-\eta_yb|^2-ab|\eta_x-\eta_y|^2
	\]
	gives
	\[
	\begin{aligned}
		\iint_{\R^n\times\R^n}
		\frac{c_h|\delta(\eta w_h)|^2}{|x-y|^{n+sp}}\dd x\dd y&=\iint_{\R^n\times\R^n}
		\frac{c_hw_h(x)w_h(y)|\eta(x)-\eta(y)|^2}
		{|x-y|^{n+sp}}\dd x\dd y\\
		&\le L^2\iint_{\R^n\times\R^n}
		\frac{c_h|\delta\eta|^2}{|x-y|^{n+sp}}\dd x\dd y\le C_\eta.
	\end{aligned}
	\]
	Hence
	\begin{equation}
		\sup_{0<h<h_\eta}
		\iint_{\R^n\times\R^n}
		\frac{c_h(x,y)|\delta(\eta w_h)(x,y)|^2}
		{|x-y|^{n+sp}}\dd x\!\dd y<\infty.
		\label{eq:uniform-secant-caccioppoli}
	\end{equation}
	
	Local \(C^1\) regularity and Rademacher's theorem give
	\[
	w_h\longrightarrow w
	\quad\text{locally uniformly in }B_{r_0}
	\text{ and a.e.\ in }\R^n.
	\]
	Strict monotonicity on \(e\)-lines gives
	\(|\{x\in B_{r_0}:U(x)=U(y)\}|=0\) for every \(y\).  Thus
	\[
	\frac{c_h(x,y)}{|x-y|^{n+sp}}\longrightarrow J_U(x,y)
	\quad\text{for a.e.\ pair meeting }B_{r_0}.
	\]
	Fatou in \eqref{eq:uniform-secant-caccioppoli} gives finite local
	\(J_U\)-energy for \(w\).
	
	For density in \(D_0:=B_{r_0}\), define
	\[
	\Theta(x):=P_{e^\perp}x+U(x)e,
	\qquad x\in D_0.
	\]
	Rotate so that \(e=e_n\).  For \(x,y\in D_0\), inserting
	\((x',y_n)\in B_{2r_0}\) and integrating along \(e\)-lines gives
	\[
	\begin{aligned}
		|x_n-y_n|
		&\le\lambda_0^{-1}\bigl(|U(x)-U(y)|+L|x'-y'|\bigr),\\
		c|x-y|&\le|\Theta(x)-\Theta(y)|\le C|x-y|,\\
		\lambda_0&\le\det D\Theta=\partial_eU\le L.
	\end{aligned}
	\]
	Thus \(\Theta:D_0\to D^\sharp:=\Theta(D_0)\) is bi-Lipschitz.
	
	Let \(f\in L^2(D_0)\) have compact support in \(D_0\) and finite
	regional \(J_U\)-energy, and set
	\(\widetilde f=f\circ\Theta^{-1}\).  With
	\(\xi=\Theta(x)\) and \(\zeta=\Theta(y)\), change of variables gives
	\begin{equation}
		\iint_{D_0\times D_0}J_U(x,y)|\delta f(x,y)|^2\dd x\!\dd y
		\asymp
		\iint_{D^\sharp\times D^\sharp}
		\frac{|e\cdot(\xi-\zeta)|^{p-2}
			|\delta\widetilde f(\xi,\zeta)|^2}
		{|\xi-\zeta|^{n+sp}}\dd\xi\!\dd\zeta.
		\label{eq:noncritical-coordinate-form-equivalence}
	\end{equation}
	Write
	\[
	L(H):=\frac{|e\cdot H|^{p-2}}{|H|^{n+sp}}
	=\frac{|e\cdot\omega|^{p-2}}{|H|^{n+\tau}},
	\qquad \omega=\frac H{|H|}.
	\]
	and
	\[
	\mathcal E_L(g):=\iint_{\R^n\times\R^n}
	L(\xi-\zeta)|g(\xi)-g(\zeta)|^2\dd\xi\!\dd\zeta.
	\]
	Since \(p-2>-1\), the angular density is integrable, and
	\[
	c_{n,\tau}\int_{\mathbb S^{n-1}}
	|\vartheta\cdot\omega|^\tau|e\cdot\omega|^{p-2}\dd\omega
	\asymp |\vartheta|^\tau.
	\]
	Hence the form domain on \(\R^n\) is \(H^{\tau/2}(\R^n)\).
	
	Extend \(\widetilde f\) by zero.  With
	\(d=\operatorname{dist}(\operatorname{supp}\widetilde f,
	\partial D^\sharp)>0\), radial integration gives
	\[
	\mathcal E_L(\widetilde f)
	\le
	\iint_{(D^\sharp)^2}L(\xi-\zeta)|\delta\widetilde f|^2
	\dd\xi\dd\zeta+Cd^{-\tau}\|\widetilde f\|_2^2<\infty.
	\]
	For a standard mollifier \(\rho_\varepsilon\), translation
	invariance and Jensen give
	\[
	\mathcal E_L(\rho_\varepsilon*\widetilde f-\widetilde f)
	\le\int_{\R^n}\rho_\varepsilon(h)
	\mathcal E_L(\widetilde f(\,\cdot+h)-\widetilde f)\dd h
	\longrightarrow0.
	\]
	Indeed, \(L^2\) translation continuity gives convergence for each
	increment \(H\), with integrable majorant
	\[
	4L(H)\|\widetilde f(\,\cdot+H)-\widetilde f\|_2^2.
	\]
	For small \(\varepsilon\), the convolutions have common compact
	support in \(D^\sharp\).  Therefore
	\[
	f_\varepsilon:=(\rho_\varepsilon*\widetilde f)\circ\Theta
	\in C_c^{1,\alpha}(D_0),\qquad
	f_\varepsilon\longrightarrow f
	\quad\text{in }L^2(D_0)\text{ and regional }J_U\text{-energy},
	\]
	by \eqref{eq:noncritical-coordinate-form-equivalence}.
	Mollify \(f_\varepsilon\) in the original variables, first fixing
	\(\varepsilon\).  The error tends to zero in \(C^1\), and
	Lemma~\ref{lem:regional-angular-compactness} gives, for
	\(h\in C_c^1(D_0)\) supported in a fixed \(K\Subset D_0\),
	\[
	\iint_{D_0^2}J_U(x,y)|\delta h(x,y)|^2\dd x\dd y
	\le C_K\bigl(\|\nabla h\|_\infty^2+\|h\|_\infty^2\bigr).
	\]
	using
	\[
	|\delta h|\le\min\{\|\nabla h\|_\infty|x-y|,2\|h\|_\infty\},
	\qquad \int_0^1r^{1-\tau}\dd r<\infty,
	\]
	and splitting at a fixed interior distance.  Hence the second
	mollification converges in regional energy.  Then let
	\(\varepsilon\to0\).  The smooth approximants have common
	compact support in \(D_0\), and preserve nonnegativity if \(f\ge0\).
	Thus
	\[
	\left.
	\begin{gathered}
		f\in L^2(D_0),\quad\operatorname{supp}f\Subset D_0,\\
		\iint_{D_0^2}J_U|\delta f|^2\dd x\dd y<\infty
	\end{gathered}
	\right\}
	\quad\Longrightarrow\quad f\in\mathcal F_{J_U}(D_0).
	\]
	Applying this to \(\eta w\), whose energy is finite by Fatou and
	\eqref{eq:uniform-secant-caccioppoli}, gives
	\[
	\eta w\in\mathcal F_{J_U}(B_{r_0})
	\quad(\eta\in C_c^\infty(B_{r_0})),
	\qquad w\in\mathcal F_{J_U,\mathrm{loc}}(B_{r_0}).
	\]
	
	For the limit in \eqref{eq:exact-secant-linear-equation}, set
	\(K=\operatorname{supp}\varphi\) and take \(\eta=1\) on a fixed
	neighborhood of \(K\).  By symmetry, consider pairs with \(x\in K\).
	On
	\[
	\{(x,y):x\in K,
	\varepsilon\le|x-y|\le R\},
	\]
	Vitali applies by \eqref{eq:quantitative-directional-coarea},
	a.e.\ kernel convergence, and \(|w_h|\le L\), giving the integral
	with \(J_U\).
	For \(|y|>R\), we have
	\(\delta\varphi(x,y)=\varphi(x)\), \(|\delta w_h(x,y)|\le2L\), and
	\begin{equation*}
		\sup_{\substack{0<h<h_K\\t\in[0,1]}}\int_K
		|(1-t)\delta U+t\delta U_h|^{p-2}\dd x\le C
	\end{equation*}
	for each \(y\), by directional coarea.  Hence, uniformly in \(h\),
	\[
	\int_K\int_{\{|y|>R\}}
	\frac{c_h|\delta w_h\,\delta\varphi|}{|x-y|^{n+sp}}
	\dd y\dd x
	\le C\int_R^\infty r^{-sp-1}\dd r=CR^{-sp}.
	\]
	For \(\varepsilon<\operatorname{dist}(K,\{\eta\ne1\})\), pairs
	meeting \(K\) with \(|x-y|<\varepsilon\) lie in \(\{\eta=1\}\).
	Cauchy--Schwarz, \eqref{eq:uniform-secant-caccioppoli}, and the
	angular bound give
	\[
	\begin{aligned}
		\int_K\int_{\{|x-y|<\varepsilon\}}
		\frac{c_h|\delta w_h\,\delta\varphi|}{|x-y|^{n+sp}}\dd y\dd x
		&\le C\left(\int_0^\varepsilon r^{1-\tau}\dd r\right)^{1/2}\\
		&\le C\varepsilon^{(2-\tau)/2}\longrightarrow0,
	\end{aligned}
	\]
	uniformly in \(h\), since
	\(|\delta\varphi|\le\|\nabla\varphi\|_\infty|x-y|\).
	Letting, in order, \(h\to0\), \(R\uparrow\infty\), and
	\(\varepsilon\to0\) proves
	\eqref{eq:full-weak-linearization}.
	
	For \(\varphi\in\mathcal F_{J_U,c}(B_{r_0})\), multiply its form
	approximants by a fixed cutoff equal to one near
	\(\operatorname{supp}\varphi\).  This gives
	\(\varphi_j\in C_c^\infty(B_{r_0})\) with common compact support and
	\[
	\|\varphi_j-\varphi\|_2^2
	+\iint_{B_{r_0}\times B_{r_0}}J_U
	|\delta(\varphi_j-\varphi)|^2\dd x\!\dd y\longrightarrow0.
	\]
	Choose \(K\Subset K_1\Subset B_{r_0}\) containing all supports.
	On \(K_1^2\), Cauchy--Schwarz and the local form membership of \(w\)
	give convergence.  For separated pairs from \(K\) to
	\(B_{r_0}\setminus K_1\), directional coarea gives the local Morrey
	bound, and \eqref{eq:kato-infinitesimal-form-bound} gives convergence.
	For crossed pairs, write
	\(A_U(x)=\int_{\R^n\setminus B_{r_0}}
	J_U(x,y)\dd y\).  Since \(|w|\le L\),
	\begin{align*}
		&\left|\int_{B_{r_0}}\int_{\R^n\setminus B_{r_0}}
		J_U(x,y)\delta w(x,y)(\varphi_j-\varphi)(x)\dd y\!\dd x\right|\\
		&\qquad\le2L
		\left(\int_KA_U\dd x\right)^{1/2}
		\left(\int_KA_U|\varphi_j-\varphi|^2\dd x\right)^{1/2}.
	\end{align*}
	By the local version of \eqref{eq:crossed-morrey-bound} and the
	density extension of \eqref{eq:kato-infinitesimal-form-bound},
	\[
	\int_KA_U|\varphi_j-\varphi|^2\dd x
	\le C\|\varphi_j-\varphi\|_{\mathcal F_{J_U}(B_{r_0})}^2
	\longrightarrow0.
	\]
	Here \(C\) may also depend on
	\(\operatorname{dist}(K,\partial B_{r_0})\).
	This proves \eqref{eq:full-weak-linearization} for all compactly
	supported form tests.
\end{proof}

At the contact point, local $C^{1,\alpha}$ regularity gives
\begin{equation}
	0\le Z_h(0)=L_Vh-U_*(he)+U_*(0)\le C h^{1+\alpha}.
	\label{eq:ordered-gap-contact-rate}
\end{equation}
In \eqref{eq:regional-sublinear-absorption},
\[
h^{-1}Z_h^{p-1}=h^{p-2}(Z_h/h)^{p-1},
\qquad h^{p-2}\longrightarrow\infty\quad(h\to0).
\]
Thus it gives no uniform linear absorption bound.
Proposition~\ref{prop:coarea-weak-linearization} controls the exterior
coefficient by directional integration before taking this limit.

\subsection{Completion of the Liouville theorem}
\label{subsec:liouville-completion}

\begin{corollary}
	\label{cor:gradient-contact-bernstein}
	There is no profile \(U_*\) satisfying the conclusions of
	Proposition~\ref{prop:no-linear-record-neck}.
\end{corollary}

\begin{proof}
	Set \(e=q/L_V\), \(w=\partial_eU_*\), and \(z=L_V-w\).
	The global Lipschitz bound and \eqref{eq:attained-maximal-gradient}
	give
	\[
	|w|\le L_V,\quad 0\le z\le2L_V\quad\text{a.e.\ in }\R^n,
	\qquad w(0)=L_V,\quad z(0)=0.
	\]
	Use the continuous representatives from local \(C^{1,\alpha}\)
	regularity.  In the ball from
	Lemma~\ref{lem:record-noncritical-neighborhood},
	\eqref{eq:strict-directional-monotonicity} holds with, for instance,
	\(\lambda_0=L_V/2\).  The cutoff bound and
	Proposition~\ref{prop:coarea-weak-linearization} give
	\[
	\eta z=L_V\eta-\eta w\in\mathcal F_{J_{U_*}}(B_{r_0})
	\quad(\eta\in C_c^\infty(B_{r_0})),
	\qquad
	z\in\mathcal F_{J_{U_*},\mathrm{loc}}(B_{r_0}).
	\]
	Since $\delta z=-\delta w$, Proposition~\ref{prop:coarea-weak-linearization}
	gives
	\[
	\iint_{\R^n\times\R^n}J_{U_*}(x,y)\delta z\,\delta\varphi\dd x\dd y=0
	\qquad\bigl(\varphi\in\mathcal F_{J_{U_*},c}(B_{r_0})\bigr).
	\]
	For nonnegative such $\varphi$ supported in $B_{r_0/2}$, split this integral.
	Since \(z(y)\ge0\),
	\[
	\begin{aligned}
		\iint_{B_{r_0}\times B_{r_0}}
		J_{U_*}(x,y)\delta z\,\delta\varphi\dd x\!\dd y
		&=-2\int_{B_{r_0/2}}\int_{\R^n\setminus B_{r_0}}
		J_{U_*}(x,y)(z(x)-z(y))\varphi(x)\dd y\dd x\\
		&\ge-2\int_{B_{r_0/2}}A_{U_*}z\varphi\dd x.
	\end{aligned}
	\]
	Apply Lemma~\ref{lem:kato-contact-propagation} with
	\[
	D=B_{r_0},\qquad D_0=B_{r_0/2},\qquad
	K=J_{U_*},\qquad A=2A_{U_*},\qquad
	\tau=sp-p+2,\qquad\kappa=2-p.
	\]
	The kernel is symmetric and measurable.
	Lemma~\ref{lem:regional-angular-compactness}
	with \(h=0\) gives \eqref{eq:spectral-regional-bounds} for all
	regional pairs, a.e.  Also,
	\[
	1<\tau<2\le n,\qquad 0<\kappa<\tau,\qquad
	\int_{B_r(x)}A\dd y\le Cr^{n-\kappa}
	\]
	locally in \(D\), by \eqref{eq:crossed-morrey-bound} and its local
	version.  The remaining hypotheses were verified above.
	The local conclusion gives a ball \(B_\rho(0)\) on which
	\[
	z=0,\qquad \partial_eU_*=L_V,\qquad
	|P_{e^\perp}\nabla U_*|^2
	=|\nabla U_*|^2-L_V^2\le0.
	\]
	Thus \(U_*\) is affine there with slope \(L_Ve\).
	A short segment parallel to \(e\) is an extremal secant;
	Lemma~\ref{thm:extremal-secant-rigidity} makes \(U_*\) affine in
	\(\R^n\), contrary to Proposition~\ref{prop:no-linear-record-neck}.
\end{proof}

\begin{theorem}
	\label{thm:fixed-order-centered-closure}
	If \(n\ge2\), $1<p<2$, $0<s<1$, and $sp>p-1$, every
	globally Lipschitz weak solution of \eqref{eq:equation} is affine.
\end{theorem}

\begin{proof}
	A nonaffine solution gives, by
	Proposition~\ref{prop:dilation-recurrent-cascade} and
	Corollary~\ref{cor:collapse-affine-radii}, the function and point
	sequence required by Proposition~\ref{prop:no-linear-record-neck}.
	The resulting profile contradicts
	Corollary~\ref{cor:gradient-contact-bernstein}.
\end{proof}

\section{Affine approximation and proof of the main theorem}
\label{sec:common-recursion}
We complete the proof of Theorem~\ref{thm:main}. After establishing the
normalization and exterior estimates, we use the Liouville theorem from
Section~\ref{sec:dilation-liouville} to improve the affine approximation
when the slope is bounded. Combining this improvement with the estimates
for large slopes from Section~\ref{sec:large-slope-and-blowdown}, we obtain
an iteration that yields H\"older continuity of the gradient.
\subsection{Normalization and exterior estimates}
\label{subsec:common-normalization-state}

Let \(\gamma:=sp-p+1>0\). For $r,H>0$ and
$\ell(x)=a+q\cdot(x-x_0)$, define
\begin{equation}
	v(X):=\frac{u(x_0+rX)-\ell(x_0+rX)}{rH},
	\qquad Q:=\frac qH.
	\label{eq:common-intrinsic-profile}
\end{equation}
Fix
\begin{equation*}
	q_{\rm ext}:=2\max\left\{p,\frac1{p-1}\right\},
	\qquad
	q_{\rm ext}':=\frac{q_{\rm ext}}{q_{\rm ext}-1}.
\end{equation*}
Then $
q_{\rm ext}>p$ and $(p-2)q_{\rm ext}'>-1$. Recall that \(\mathcal A\) denotes the finite dimensional space of affine
functions on \(\R^n\). For \(\beta\ge0\) and a fixed constant \(G\ge1\), define
\begin{equation}
	\mathcal T_{\beta,G}(v,Q)
	:=\max\left\{G^{-1}\mathscr C_{q_{\rm ext},\beta}(v),
	\mathcal T^{\rm pt}_{s,Q}(v)\right\},
	\label{eq:common-tail-state}
\end{equation}
where\[
	\mathbf c_{q_{\rm ext}}(v;j)
	:=2^{-j}\inf_{m\in\mathcal A}
	\left(\fint_{B_{2^{j+1}}}|v-m|^{q_{\rm ext}}\dd X
	\right)^{1/q_{\rm ext}},\quad 
	\mathscr C_{q_{\rm ext},\beta}(v)
	:=\sup_{j\ge0}2^{-j\beta}\mathbf c_{q_{\rm ext}}(v;j),\]
and
\begin{align*}
	\mathcal T^{\rm pt}_{s,Q}(v)
	&:=\max\left\{
	{\operatorname*{ess\,sup}}_{X\in B_{5/4}}
	\int_{\R^n\setminus B_{4/3}}
	\frac{(1+|Q|)^{2-p}|\mathcal D_Qv(X,Y)|}
	{|X-Y|^{n+sp}}\dd Y,\right.\\[-2mm]
	&\hspace{31mm}\left.
	{\operatorname*{ess\,sup}}_{X\in B_{3/2}}
	\int_{\R^n\setminus B_2}
	\frac{(1+|Q|)^{2-p}|\mathcal D_Qv(X,Y)|}
	{|X-Y|^{n+sp}}\dd Y\right\}.
\end{align*}
The constant \(G\) is fixed before the compactness argument for bounded
slopes. Affine invariance gives, for every \(m\in\mathcal A\),
\[
\mathscr C_{q_{\rm ext},\beta}(v+m)
=\mathscr C_{q_{\rm ext},\beta}(v).
\]

We shall choose a dyadic contraction \(\rho=2^{-N}\). If
\begin{equation*}
	w(X):=\frac{v(\rho X)-m(\rho X)}{\rho\eta},
	\qquad m\in\mathcal A,
\end{equation*}
then a change of variables gives, for \(j\ge N\),
\begin{equation*}
	\mathbf c_{q_{\rm ext}}(w;j)
	=\eta^{-1}\mathbf c_{q_{\rm ext}}(v;j-N).
\end{equation*}
In particular, for \(\eta=\rho^\beta\),
\begin{equation}
	\sup_{j\ge N}2^{-j\beta}\mathbf c_{q_{\rm ext}}(w;j)
	=\mathscr C_{q_{\rm ext},\beta}(v).
	\label{eq:exact-affine-weight-cancellation}
\end{equation}
If a single affine map satisfies
\begin{equation*}
	t^{-1}\|v-m\|_{L^\infty(B_{2t})}\le\varepsilon\eta
	\qquad(\rho\le t\le1/2),
\end{equation*}
then the zero affine competitor, with \(t=\rho2^j\), gives
\begin{equation*}
	\sup_{0\le j<N}2^{-j\beta}
	\mathbf c_{q_{\rm ext}}(w;j)\le\varepsilon.
\end{equation*}
Thus $
\mathscr C_{q_{\rm ext},\beta}(w)
\le\max\{\varepsilon,\mathscr C_{q_{\rm ext},\beta}(v)\}$.
For $\ell(x)=a+q\cdot(x-x_0)$,  we define
\begin{align*}
	\mathbf Z_r(u,\ell;x_0)
	&:=\fint_{B_{2r}(x_0)}
	\left|\frac{u-\ell}{r}\right|^p\dd x,\\
	\mathbf E_{s,r}(u,\ell;x_0)
	&:=(1-s)r^{sp-p-n}
	\iint_{B_{2r}(x_0)\times B_{2r}(x_0)}
	\frac{D_\Phi\bigl(q\cdot(x-y);\delta(u-\ell)(x,y)\bigr)}
	{|x-y|^{n+sp}}\dd x\!\dd y,
\end{align*}
and 
\[
\mathbf A_{s,r}(u,\ell;x_0)
:=\inf\left\{A>0:
\begin{aligned}
	\mathbf Z_r(u,\ell;x_0)&\le A^p,\\
	\mathbf E_{s,r}(u,\ell;x_0)&\le(|q|+A)^{p-2}A^2
\end{aligned}\right\}.
\]
For \(\zeta\ge0\), define \(F_\zeta(A)=(\zeta+A)^{p-2}A^2\), with
\(F_\zeta(0)=0\). Then \(F_\zeta\in C([0,\infty))\) and
\[
F_\zeta'(A)
=A(\zeta+A)^{p-3}(2\zeta+pA)>0
\qquad(A>0).
\]
Continuity and strict monotonicity therefore imply, also when \(q=0\)
or \(\mathbf A_{s,r}=0\),
\[
A\ge\mathbf A_{s,r}(u,\ell;x_0),\quad A>0
\quad\Longrightarrow\quad
\begin{cases}
	\mathbf Z_r(u,\ell;x_0)\le A^p,\\
	\mathbf E_{s,r}(u,\ell;x_0)\le F_{|q|}(A).
\end{cases}
\]
The analogous implication holds for \(\mathfrak a_{s,Q}\).

\subsection{A Liouville theorem under a Campanato bound}
\label{subsec:affine-campanato-liouville}

For \(R\ge1\), set
\begin{equation}
	\mathscr E_{q_{\rm ext}}(U;R):=R^{-1}\inf_{m\in\mathcal A}
	\left(\fint_{B_{2R}}|U-m|^{q_{\rm ext}}\dd X
	\right)^{1/q_{\rm ext}}.
	\label{eq:continuous-affine-excess-main}
\end{equation}

\begin{lemma}
	\label{lem:coherent-affine-fits-main}
	Suppose that
	\(\sup_{R\ge1}\mathscr E_{q_{\rm ext}}(U;R)\le G\).  For dyadic
	\(R\), let \(m_R=c_R+b_R\cdot X\) minimize
	\eqref{eq:continuous-affine-excess-main}.  Then
	\begin{align}
		|b_{2R}-b_R|+R^{-1}|c_{2R}-c_R|&\le CG,
		\label{eq:coherent-affine-coefficients-main}\\
		\left(\fint_{B_{2^{h+1}R}}|U-m_R|^{q_{\rm ext}}\dd X
		\right)^{1/q_{\rm ext}}
		&\le CG(h+1)2^hR\qquad(h\ge0).
		\label{eq:coherent-log-shells-main}
	\end{align}
\end{lemma}

\begin{proof}
	By minimality,  we have 
	\[
	\begin{aligned}
		&\left(\fint_{B_{2R}}|m_{2R}-m_R|^{q_{\rm ext}}\dd X\right)^{1/q_{\rm ext}}\\
		&\quad \le \left(\fint_{B_{2R}}|U-m_R|^{q_{\rm ext}}\dd X\right)^{1/q_{\rm ext}}+C\left(\fint_{B_{4R}}|U-m_{2R}|^{q_{\rm ext}}\dd X\right)^{1/q_{\rm ext}}\le CGR.
	\end{aligned}
	\]

Take \(a=c_{2R}-c_R\), \(b=b_{2R}-b_R\), and
\(q=q_{\rm ext}\). Since \(B_{2R}\) is centered
at the origin, H\"older's inequality gives
\[
|a|
\le \left(\fint_{B_{2R}}
|a+b\cdot X|^q\,\dd X\right)^{1/q}
\le CGR.
\]
By the symmetry of \(B_{2R}\),  rotational
invariance and
Minkowski's inequality,
\[
\begin{aligned}
c_{n,q}R|b|	\le \left(\fint_{B_{2R}}|b\cdot X|^q
	\,\dd X\right)^{1/q}
	\le
	\left(\fint_{B_{2R}}|a+b\cdot X|^q
	\,\dd X\right)^{1/q}\le CGR,
\end{aligned}
\]
where \(c_{n,q}>0\) depends only on \(n,q\).
Consequently,
\[
|b_{2R}-b_R|
+R^{-1}|c_{2R}-c_R|
\le CG.
\]
For \(h\in\mathbb N_0\),
	telescoping gives
	\[
	U-m_R=U-m_{2^hR}
	+\sum_{i=0}^{h-1}(m_{2^{i+1}R}-m_{2^iR}),
	\]
	while \eqref{eq:coherent-affine-coefficients-main} gives
	\[
	\begin{aligned}
		\|m_{2^{i+1}R}-m_{2^iR}\|_{L^\infty(B_{2^{h+1}R})}
		&\le CG(2^iR+2^hR)\le CG2^hR,\\
		\left(\fint_{B_{2^{h+1}R}}|U-m_R|^{q_{\rm ext}}\dd X\right)^{1/q_{\rm ext}}
		&\le CG2^hR+\sum_{i=0}^{h-1}CG2^hR
		\le CG(h+1)2^hR.
	\end{aligned}
	\]

For any real \(h\ge0\), choose an integer \(k\)
such that \(h\le k<h+1\). Then
\[
B_{2^{h+1}R}\subset B_{2^{k+1}R},
\qquad 1\le 2^{k-h}<2.
\]
Using the estimate for the integer \(k\), we obtain
\eqref{eq:coherent-log-shells-main}.

\end{proof}

\begin{lemma}
	\label{lem:sublinear-best-affine-normalization-main}
	Let $U$ be a weak solution of \eqref{eq:equation} in $\R^n$ satisfying
	\(\sup_{R\ge1}\mathscr E_{q_{\rm ext}}(U;R)\le G\), and set
	\(v_R(X)=\frac{U(RX)-m_R(RX)}R\) and \(Q_R=b_R\).
	With \(\theta=n/(n+p)<1\), there is
	\begin{equation}
		1\le A_R\le C(1+|b_R|^\theta)
		\label{eq:sublinear-best-affine-amplitude-main}
	\end{equation}
	such that \(\widehat v_R=v_R/A_R\), with affine slope
	\(\widehat Q_R=Q_R/A_R\), satisfies
	\begin{equation}
		\fint_{B_2}|\widehat v_R|^p\dd X
		+\mathcal E_{s,\widehat Q_R}(\widehat v_R;B_2)
		+\mathcal T^{\rm pt}_{s,\widehat Q_R}(\widehat v_R)\le1.
		\label{eq:sublinear-best-affine-unit-data-main}
	\end{equation}
	Here \(C=C(n,p,s,G)\).  In particular,
	\(A_R/|b_R|\to0\) whenever \(|b_R|\to\infty\).
\end{lemma}

\begin{proof}
	By minimality and \(q_{\rm ext}>p\),
	\[
	\left(\fint_{B_2}|v_R|^p\dd X\right)^{1/p}
	\le\left(\fint_{B_2}|v_R|^{q_{\rm ext}}\dd X\right)^{1/q_{\rm ext}}
	\le G.
	\]
	To estimate the exterior contribution, set
	\[
	W_R(X)=\frac{U(RX)-c_R}{R}=b_R\cdot X+v_R(X).
	\]
	For \(h\ge0\), let
	\[
	A_h=(B_{2^{h+3}}\setminus B_{2^{h+2}})
	\cap(\R^n\setminus B_6).
	\]
	By \eqref{eq:coherent-log-shells-main}
	and a change of variables,
	\[
	\left(\fint_{B_{2^{h+3}}}
	|v_R|^{q_{\rm ext}}\dd Y
	\right)^{1/q_{\rm ext}}
	\le CG(h+1)2^h.
	\]
	Since \(|Y|\ge 2^{h+2}\) on \(A_h\),
	H\"older's inequality gives
	\[
	\begin{aligned}
		\int_{A_h}
		\frac{|v_R(Y)|^{p-1}}{|Y|^{n+sp}}\dd Y&\le C2^{-h(n+sp)}
		|B_{2^{h+3}}|
		\left(\fint_{B_{2^{h+3}}}
		|v_R|^{q_{\rm ext}}\dd Y
		\right)^{(p-1)/q_{\rm ext}}\\
		&\le CG^{p-1}(h+1)^{p-1}
		2^{-h(sp-p+1)}\\
		&=CG^{p-1}(h+1)^{p-1}2^{-h\gamma}.
	\end{aligned}
	\]
Therefore, we obtain
	\[
	\begin{aligned}
		\int_{\R^n\setminus B_6}
		\frac{|W_R(Y)|^{p-1}}{|Y|^{n+sp}}\dd Y
		&\le C|b_R|^{p-1}\int_6^\infty t^{-\gamma-1}\dd t
		+CG^{p-1}\sum_{h\ge0}(h+1)^{p-1}2^{-h\gamma}\\
		&\le C(1+|b_R|)^{p-1}.
	\end{aligned}
	\]
	Combining this estimate with \eqref{eq:coherent-log-shells-main}, we obtain
	\[
	\left(\fint_{B_{12}}|W_R|^p\dd X\right)^{1/p}
	+\left(\int_{\R^n\setminus B_6}
	\frac{|W_R(Y)|^{p-1}}{|Y|^{n+sp}}\dd Y\right)^{1/(p-1)}
	\le C(1+|b_R|).
	\]
	Since \(W_R\) is a weak solution of the
	homogeneous equation, both \(W_R\) and
	\(-W_R\) are weak subsolutions.
	Applying \cite[Theorem~4]{KorvenpaaKuusiPalatucci2016}
	on \(B_{12}\), with the preceding local
	\(L^p\) and tail bounds, gives
	\[
	\|W_R\|_{L^\infty(B_6)}
	\le C(1+|b_R|).
	\]
	By \cite[Theorem~5]{KorvenpaaKuusiPalatucci2016},
	\(W_R\) has a continuous representative.
	Then \cite[Proposition~1.5]{BiswasTopp2025}
	implies that \(W_R\) is also a viscosity solution.
	Since \(sp>p-1\), the local Lipschitz estimate
	in \cite[Theorem~2.1]{BiswasTopp2025},
	applied on \(B_6\), yields
	\[
	\operatorname{Lip}(W_R;B_3)
	\le C(1+|b_R|).
	\]
	Finally, since \(v_R(X)=W_R(X)-b_R\cdot X\),
	\[
	\begin{aligned}
		\operatorname{Lip}(v_R;B_3)
		\le \operatorname{Lip}(W_R;B_3)+|b_R|\le C(1+|b_R|).
	\end{aligned}
	\]
	On the other hand, \eqref{eq:coherent-log-shells-main} with \(h=1\) gives
	\[
	\left(\fint_{B_4}|v_R|^p\dd X\right)^{1/p}\le C G.
	\]
	Interpolating this bound with the Lipschitz estimate, we obtain
	\begin{equation}
		\|v_R\|_{L^\infty(B_3)}
		\le C(1+|b_R|^{n/(n+p)}).
		\label{eq:sublinear-best-affine-sup-main}
	\end{equation}
	
	Combining the local \(L^p\) bound with the preceding Lipschitz estimate, we obtain
	\begin{equation}
		\|v_R\|_{L^\infty(B_3)}
		\le C(1+|b_R|^{n/(n+p)}).
		\label{eq:sublinear-best-affine-sup-main}
	\end{equation}
To see it, we define $
	M_R=\|v_R\|_{L^\infty(B_3)}$, and $
	L_R=\operatorname{Lip}(v_R;B_3)$.
	If \(M_R=0\), the conclusion is immediate. If \(L_R=0\), then \(v_R\) is constant on \(B_3\), and the \(L^p\) bound gives \(M_R\le CG\).
	
	We may therefore assume that \(M_R>0\) and \(L_R>0\). Choose \(x_R\in B_3\) such that
	\[
	|v_R(x_R)|\ge \frac{M_R}{2},
	\]
	and let $d_R=\min\left\{1,\frac{M_R}{4L_R}\right\}$.	For every \(X\in B_{d_R}(x_R)\cap B_3\), the Lipschitz estimate gives
	\[
	\begin{aligned}
		|v_R(X)|
		\ge |v_R(x_R)|-|v_R(X)-v_R(x_R)|
		\ge \frac{M_R}{2}-L_Rd_R
		\ge \frac{M_R}{4}.
	\end{aligned}
	\]
	Since \(x_R\in B_3\) and \(d_R\le1\), the intersection \(B_{d_R}(x_R)\cap B_3\) contains a ball of radius \(d_R/2\). Hence
	\[
	|B_{d_R}(x_R)\cap B_3|\ge c_nd_R^n.
	\]
	It follows that
	\[
	\begin{aligned}
		CG^p
		\ge \int_{B_4}|v_R|^p\dd X\ge \int_{B_{d_R}(x_R)\cap B_3}|v_R|^p\dd X\ge cM_R^pd_R^n.
	\end{aligned}
	\]
	If \(M_R\ge4L_R\), then \(d_R=1\), and consequently
	\[
	M_R^p\le CG^p,
	\qquad\text{so}\qquad M_R\le CG.
	\]
	If \(M_R<4L_R\), then \(d_R=M_R/(4L_R)\), which gives
	\[
	M_R^p\left(\frac{M_R}{4L_R}\right)^n
	\le CG^p.
	\]
	Thus $
	M_R^{n+p}\le CG^pL_R^n$
	and hence
	\[
	M_R\le CG^{p/(n+p)}L_R^{n/(n+p)}.
	\]
	Combining the two cases, we conclude that
	\[
	M_R\le C\left(
	G+G^{p/(n+p)}L_R^{n/(n+p)}
	\right).
	\]
	Finally, using $
	L_R\le C(1+|b_R|)$
	and recalling that \(G\) is fixed, we obtain
	\[
	M_R\le C(1+|b_R|^{n/(n+p)}).
	\]

	Enlarging the constant in
	\eqref{eq:sublinear-best-affine-sup-main} if necessary, set $
	S_R=C(1+|b_R|^\theta)$ and  $A_R=C_0S_R$,	so that
	\[
	\|v_R\|_{L^\infty(B_3)}\le S_R,
	\qquad S_R\ge\max\{1,G\}.
	\]
	Here \(C_0\ge1\) will be chosen at the end of the proof. Recall that \(\widehat v_R=v_R/A_R\) and \(\widehat Q_R=b_R/A_R\).
	
	We next estimate the two exterior integrals in \(\mathcal T^{\rm pt}_{s,\widehat Q_R}(\widehat v_R)\). For the first integral, take \(X\in B_{5/4}\) and \(\mathcal A_0=B_4\setminus B_{4/3}\). For the second, take \(X\in B_{3/2}\) and \(\mathcal A_0=B_4\setminus B_2\). In both cases, define
	\[
	\mathcal A_h=B_{2^{h+2}}\setminus B_{2^{h+1}},
	\qquad h\ge1.
	\]
	These shells cover the corresponding exterior domain. Moreover, for every \(h\ge0\),
	\[
	\mathcal A_h\subset B_{2^{h+2}},
	\qquad c2^h\le |X-Y|\le C2^h
	\quad(Y\in\mathcal A_h).
	\]
	Indeed, on the first shell the distance between \(X\) and \(Y\) is at least \(4/3-5/4=1/12\) in the first case and \(2-3/2=1/2\) in the second. For \(h\ge1\), the bounds follow directly from the definition of \(\mathcal A_h\).
	
	By \eqref{eq:coherent-log-shells-main}, with \(h+1\) in place of \(h\), and the definition of \(\widehat v_R\), we have
	\[
	\left(\fint_{B_{2^{h+2}}}
	|\widehat v_R|^{q_{\rm ext}}\dd Y
	\right)^{1/q_{\rm ext}}
	\le \frac{CG}{A_R}(h+1)2^h.
	\]
	Also, since \(X\in B_{3/2}\subset B_3\), $
	|\widehat v_R(X)|\le\frac{S_R}{A_R}$.
	
	We first consider \(|\widehat Q_R|\le2\). By \eqref{eq:large-slope-flux-bound} and \(1<p<2\),
	\[
	\lambda_{\widehat Q_R}^{2-p}
	|\mathcal D_{\widehat Q_R}\widehat v_R(X,Y)|
	\le C|\delta\widehat v_R(X,Y)|^{p-1}\le  |\widehat v_R(X)|^{p-1}
	+|\widehat v_R(Y)|^{p-1}..
	\]
	Using \(|X-Y|\ge c2^h\) on \(\mathcal A_h\) and applying H\"older's inequality with \(q_{\rm ext}>p-1\), we obtain
	\[
	\begin{aligned}
		\int_{\mathcal A_h}
		\frac{\lambda_{\widehat Q_R}^{2-p}
			|\mathcal D_{\widehat Q_R}\widehat v_R(X,Y)|}
		{|X-Y|^{n+sp}}\dd Y
		&\le C2^{-hsp}
		\left[
		|\widehat v_R(X)|^{p-1}
		+\left(\fint_{B_{2^{h+2}}}
		|\widehat v_R|^{q_{\rm ext}}\dd Y
		\right)^{(p-1)/q_{\rm ext}}
		\right]\\
		&\le C\left[
		\left(\frac{S_R}{A_R}\right)^{p-1}2^{-hsp}
		+\left(\frac{G}{A_R}\right)^{p-1}
		(h+1)^{p-1}2^{-h\gamma}
		\right],
	\end{aligned}
	\]
	where we used \(\gamma=sp-p+1\).
	
	We next consider \(|\widehat Q_R|>2\). Set \(e=\widehat Q_R/|\widehat Q_R|\) and \(\omega=(X-Y)/|X-Y|\). By \eqref{eq:large-slope-angular-flux},
	\[
	\frac{\lambda_{\widehat Q_R}^{2-p}
		|\mathcal D_{\widehat Q_R}\widehat v_R(X,Y)|}
	{|X-Y|^{n+sp}}
	\le C\frac{|e\cdot\omega|^{p-2}
		|\delta\widehat v_R(X,Y)|}
	{|X-Y|^{n+\gamma+1}}.
	\]
	To estimate the integral on \(\mathcal A_h\), recall that \(q_{\rm ext}'=q_{\rm ext}/(q_{\rm ext}-1)\) and \((p-2)q_{\rm ext}'>-1\). Using polar coordinates centered at \(X\) and the bounds \(c2^h\le|X-Y|\le C2^h\), we obtain
	\[
	\begin{aligned}
		\left(\int_{\mathcal A_h}
		\frac{|e\cdot\omega|^{(p-2)q_{\rm ext}'}}
		{|X-Y|^{(n+\gamma+1)q_{\rm ext}'}}\dd Y
		\right)^{1/q_{\rm ext}'}\le C\left(
		\int_{c2^h}^{C2^h}
		t^{n-1-(n+\gamma+1)q_{\rm ext}'}\dd t
		\right)^{1/q_{\rm ext}'}\le C2^{-h(n/q_{\rm ext}+\gamma+1)}.
	\end{aligned}
	\]
	On the other hand, since \(\mathcal A_h\subset B_{2^{h+2}}\) and \(|\delta\widehat v_R(X,Y)|\le|\widehat v_R(X)|+|\widehat v_R(Y)|\), we have
	\[
	\begin{aligned}
		\left(\int_{\mathcal A_h}
		|\delta\widehat v_R(X,Y)|^{q_{\rm ext}}
		\dd Y\right)^{1/q_{\rm ext}}\le C2^{hn/q_{\rm ext}}
		\left[
		|\widehat v_R(X)|
		+\left(\fint_{B_{2^{h+2}}}
		|\widehat v_R|^{q_{\rm ext}}\dd Y
		\right)^{1/q_{\rm ext}}
		\right].
	\end{aligned}
	\]
	Applying H\"older's inequality with exponents \(q_{\rm ext}'\) and \(q_{\rm ext}\), and combining the preceding two estimates, we find
	\[
	\begin{aligned}
		\int_{\mathcal A_h}
		\frac{\lambda_{\widehat Q_R}^{2-p}
			|\mathcal D_{\widehat Q_R}\widehat v_R(X,Y)|}
		{|X-Y|^{n+sp}}\dd Y&\le C2^{-h(\gamma+1)}
		\left[
		|\widehat v_R(X)|
		+\left(\fint_{B_{2^{h+2}}}
		|\widehat v_R|^{q_{\rm ext}}\dd Y
		\right)^{1/q_{\rm ext}}
		\right]\\
		&\le C\left[
		\frac{S_R}{A_R}2^{-h(\gamma+1)}
		+\frac{G}{A_R}(h+1)2^{-h\gamma}
		\right].
	\end{aligned}
	\]
	
	Thus, in both cases, the contribution of \(\mathcal A_h\) is bounded by
	\begin{align}
		C\biggl[
		&\left(\frac{S_R}{A_R}\right)^{p-1}2^{-hsp}
		+\left(\frac{G}{A_R}\right)^{p-1}
		(h+1)^{p-1}2^{-h\gamma}\notag\\
		&+\frac{S_R}{A_R}2^{-h(\gamma+1)}
		+\frac{G}{A_R}(h+1)2^{-h\gamma}
		\biggr].
		\label{eq:best-affine-shell-series-main}
	\end{align}
	Since \(\gamma>0\), all four series are summable. Summing over \(h\ge0\), and using \(G\le S_R\) and \(A_R=C_0S_R\), we conclude that
	\[
	\mathcal T^{\rm pt}_{s,\widehat Q_R}(\widehat v_R)
	\le C(C_0^{1-p}+C_0^{-1}).
	\]

It remains to estimate the shifted energy.  Choose
\(\zeta\in C_c^\infty(B_{5/2})\) such that
\(0\le\zeta\le1\), \(\zeta=1\) on \(B_2\), and
\(|\nabla\zeta|\le C\). For \(Z\in\R^n\), write
\[
H_{b_R,Z}(t)
:=D_\Phi(b_R\cdot Z;t)+D_\Phi(b_R\cdot Z;-t).
\]

We first use the shifted weak equation with the test function \(\zeta^2v_R\). For \(X,Y\in B_3\), set \(Z=X-Y\) and
\[
\mu=\frac{\zeta(X)^2+\zeta(Y)^2}{2},
\qquad
\eta=\frac{\zeta(X)+\zeta(Y)}2.
\]
Then \(\eta^2\le\mu\le2\eta^2\), and a direct calculation gives
\[
\delta(\zeta^2v_R)
=\mu\,\delta v_R
+\eta(v_R(X)+v_R(Y))\delta\zeta.
\]
By \eqref{eq:uncapped-shifted-young}, after removing the common normalization factor, and the scalar estimates in Lemma~\ref{lem:large-slope-scalar-angular}, we have
\[
\begin{aligned}
	&\mathcal D_{b_R}v_R(X,Y)
	\,\delta(\zeta^2v_R)(X,Y)\ge
	c\mu D_\Phi(b_R\cdot Z;\delta v_R(X,Y))
	-C H_{b_R,Z}(2S_R\delta\zeta(X,Y)).
\end{aligned}
\]
Indeed, for \(\eta>0\), the error term is estimated by applying \eqref{eq:uncapped-shifted-young} with the second argument
\[
\frac{(v_R(X)+v_R(Y))\delta\zeta}{\eta}.
\]
Multiplying the resulting inequality by \(\eta^2\), using \(\eta^2\le\mu\) and the second inequality in \eqref{eq:uncapped-shifted-young}, gives
\[
\begin{aligned}
	\left|\eta\,\mathcal D_{b_R}v_R
	(v_R(X)+v_R(Y))\delta\zeta\right|\le
	\varepsilon\mu H_{b_R,Z}(\delta v_R)
	+C_\varepsilon H_{b_R,Z}(2S_R\delta\zeta).
\end{aligned}
\]
Choosing \(\varepsilon\) sufficiently small and using
\[
\mathcal D_{b_R}v_R\,\delta v_R
\ge cH_{b_R,Z}(\delta v_R)
\ge cD_\Phi(b_R\cdot Z;\delta v_R)
\]
proves the claimed lower bound. If \(\eta=0\), then \(\mu=0\), and the same bound is immediate.

Since \(v_R\) is locally Lipschitz in \(B_3\), the integrals near \(X=Y\) are finite because \(p-sp>0\). The exterior integrals are finite by the preceding tail estimates and the shell bounds below. Thus smooth approximation allows us to use \(\zeta^2v_R\) in the shifted weak formulation.

Since \(\mu=1\) on \(B_2\times B_2\) and \(D_\Phi\ge0\), the weak formulation gives
\[
(1-s)\iint_{B_2^2}
\frac{D_\Phi(b_R\cdot(X-Y);\delta v_R(X,Y))}
{|X-Y|^{n+sp}}\dd X\dd Y
\le C(I_1+I_2),
\]
where
\[
\begin{aligned}
	I_1&:=(1-s)\iint_{B_3^2}
	\frac{H_{b_R,X-Y}(2S_R\delta\zeta)}
	{|X-Y|^{n+sp}}\dd X\dd Y,\\
	I_2&:=(1-s)S_R
	\int_{B_{5/2}}\int_{\R^n\setminus B_3}
	\frac{|\mathcal D_{b_R}v_R(X,Y)|}
	{|X-Y|^{n+sp}}\dd Y\dd X.
\end{aligned}
\]

We first estimate \(I_1\). Since
\(|\delta\zeta(X,Y)|\le C|X-Y|\), the bound
\(D_\Phi(a;t)\le C|t|^p\) gives
\[
H_{b_R,X-Y}(2S_R\delta\zeta)
\le CS_R^p|X-Y|^p.
\]
Consequently, using polar coordinates,
\[
\begin{aligned}
	I_1
	&\le CS_R^p(1-s)
	\int_0^6r^{p-sp-1}\dd r=\frac{C6^{p-sp}}p S_R^p
	\le CS_R^p.
\end{aligned}
\]
This estimate applies in particular when
\(|b_R|\le4S_R\).

Suppose now that \(|b_R|>4S_R\), and set
\(e=b_R/|b_R|\). By
\eqref{eq:symmetric-shifted-density}, for
\(Z=X-Y=r\omega\) and almost every
\(\omega\in\mathbb S^{n-1}\),
\[
\begin{aligned}
	H_{b_R,Z}(2S_R\delta\zeta)
	&\le C\bigl(|b_R\cdot Z|
	+S_R|\delta\zeta|\bigr)^{p-2}
	S_R^2|\delta\zeta|^2\\
	&\le C|b_R|^{p-2}S_R^2r^p
	|e\cdot\omega|^{p-2}.
\end{aligned}
\]
Here we used \(p-2<0\). Therefore, we obtain
\[
I_1\le C|b_R|^{p-2}S_R^2(1-s)
\int_0^6r^{p-sp-1}\dd r
\le C|b_R|^{p-2}S_R^2.
\]
We have proved
\[
I_1\le
\begin{cases}
	CS_R^p,&|b_R|\le4S_R,\\
	C|b_R|^{p-2}S_R^2,&|b_R|>4S_R.
\end{cases}
\]

We next estimate \(I_2\). For \(h\ge0\), define $
\mathcal B_h
:=B_{3\cdot2^{h+1}}\setminus B_{3\cdot2^h}$.
These shells cover \(\R^n\setminus B_3\). For every \(X\in B_{5/2}\) and \(Y\in\mathcal B_h\),
\[
c2^h\le|X-Y|\le C2^h.
\]
Moreover, \eqref{eq:coherent-log-shells-main} and a change of variables give
\[
\left(\fint_{B_{6\cdot2^h}}
|v_R|^{q_{\rm ext}}\dd Y\right)^{1/q_{\rm ext}}
\le CG(h+1)2^h.
\]
Suppose first that \(|b_R|\le4S_R\). The global \((p-1)\)-H\"older continuity of \(\Jp\) gives
\[
|\mathcal D_{b_R}v_R(X,Y)|
\le C|\delta v_R(X,Y)|^{p-1}.
\]
Since \(0<p-1<1\) and \(|v_R(X)|\le S_R\), H\"older's inequality with \(q_{\rm ext}>p-1\) yields
\[
\begin{aligned}
	\int_{\mathcal B_h}
	\frac{|\mathcal D_{b_R}v_R(X,Y)|}
	{|X-Y|^{n+sp}}\dd Y
	&\le C2^{-hsp}
	\left[
	S_R^{p-1}
	+\left(\fint_{B_{6\cdot2^h}}
	|v_R|^{q_{\rm ext}}\dd Y
	\right)^{(p-1)/q_{\rm ext}}
	\right]\\
	&\le C\left[
	S_R^{p-1}2^{-hsp}
	+G^{p-1}(h+1)^{p-1}2^{-h\gamma}
	\right].
\end{aligned}
\]

For \(|b_R|>4S_R\), we use
\eqref{eq:large-slope-flux-bound}. Writing
\(e=b_R/|b_R|\) and
\(\omega=(X-Y)/|X-Y|\), we have, for almost every \(\omega\),
\[
|\mathcal D_{b_R}v_R(X,Y)|
\le C|b_R|^{p-2}|X-Y|^{p-2}
|e\cdot\omega|^{p-2}
|\delta v_R(X,Y)|.
\]
Since
\[
\left(\int_{\mathcal B_h}
|v_R(Y)|^{q_{\rm ext}}\dd Y
\right)^{1/q_{\rm ext}}
\le CG(h+1)2^{h(1+n/q_{\rm ext})},
\]
H\"older's inequality and
\(|\delta v_R(X,Y)|\le S_R+|v_R(Y)|\)
give
\[
\begin{aligned}
	\int_{\mathcal B_h}
	\frac{|\mathcal D_{b_R}v_R(X,Y)|}
	{|X-Y|^{n+sp}}\dd Y\le C|b_R|^{p-2}
	\left[
	S_R2^{-h(\gamma+1)}
	+G(h+1)2^{-h\gamma}
	\right].
\end{aligned}
\]

Since \(\gamma>0\), the series
\[
\sum_{h\ge0}2^{-hsp},\qquad
\sum_{h\ge0}(h+1)^{p-1}2^{-h\gamma},
\qquad
\sum_{h\ge0}(h+1)2^{-h\gamma}
\]
are finite. Summing the preceding estimates over \(h\), and recalling that \(B_{5/2}\) has fixed volume, we obtain
\[
I_2\le
\begin{cases}
	CS_R(S_R^{p-1}+G^{p-1}),
	&|b_R|\le4S_R,\\
	C|b_R|^{p-2}S_R(S_R+G),
	&|b_R|>4S_R.
\end{cases}
\]
Since \(G\le S_R\), this becomes
\[
I_2\le
\begin{cases}
	CS_R^p,&|b_R|\le4S_R,\\
	C|b_R|^{p-2}S_R^2,&|b_R|>4S_R.
\end{cases}
\]
Combining the estimates for \(I_1\) and \(I_2\), and using \(1<p<2\), we conclude that
\[
I_1+I_2
\le C(|b_R|+S_R)^{p-2}S_R^2.
\]
Indeed, if \(|b_R|\le4S_R\), then
\((|b_R|+S_R)^{p-2}S_R^2\) is comparable to \(S_R^p\). If \(|b_R|>4S_R\), it is comparable to \(|b_R|^{p-2}S_R^2\).

We have therefore proved
\begin{equation}
	(1-s)\iint_{B_2^2}
	\frac{D_\Phi(b_R\cdot(X-Y);\delta v_R(X,Y))}
	{|X-Y|^{n+sp}}\dd X\dd Y
	\le C(|b_R|+S_R)^{p-2}S_R^2.
	\label{eq:best-affine-shifted-caccioppoli}
\end{equation}

We now return to the normalized function
\(\widehat v_R=v_R/A_R\), where
\(A_R=C_0S_R\) and
\(\widehat Q_R=b_R/A_R\).
By the \(p\)-homogeneity of \(D_\Phi\),
\[
D_\Phi\left(\frac a{A_R};\frac t{A_R}\right)
=A_R^{-p}D_\Phi(a;t).
\]
Also, $
\lambda_{\widehat Q_R}^{2-p}
=A_R^{p-2}(A_R+|b_R|)^{2-p}$. Consequently, \eqref{eq:best-affine-shifted-caccioppoli} gives
\[
\begin{aligned}
	\mathcal E_{s,\widehat Q_R}(\widehat v_R;B_2)
	&\le
	CA_R^{-2}(A_R+|b_R|)^{2-p}
	(|b_R|+S_R)^{p-2}S_R^2\\
	&=CC_0^{-2}
	\left(\frac{|b_R|+C_0S_R}
	{|b_R|+S_R}\right)^{2-p}\le CC_0^{-p}.
\end{aligned}
\]
Thus, we conclude that
\[
\begin{aligned}
	\fint_{B_2}|\widehat v_R|^p\dd X
	+\mathcal E_{s,\widehat Q_R}(\widehat v_R;B_2)
	+\mathcal T^{\rm pt}_{s,\widehat Q_R}(\widehat v_R)\le
	C(C_0^{-p}+C_0^{1-p}+C_0^{-1})
	\le1,
\end{aligned}
\]
provided \(C_0=C_0(n,p,s,G)\) is sufficiently large. This proves \eqref{eq:sublinear-best-affine-unit-data-main}.

Finally, since \(\theta=n/(n+p)<1\) and
\(A_R=C_0S_R\le C(1+|b_R|^\theta)\),
we have
\[
\frac{A_R}{|b_R|}
\le C\left(|b_R|^{-1}+|b_R|^{\theta-1}\right)
\longrightarrow0
\qquad\text{as }|b_R|\to\infty.
\]

\end{proof}

\begin{lemma}
	\label{lem:large-slope-escape-state-main}
	Fix \(0<\alpha<\bar\alpha<\gamma\) and set \(L_{\rm esc}=2\).  There are a dyadic
	\(\rho_{\rm esc}\in(0,1/16)\), a threshold \(Q_{\rm esc}\), and
	\(C_{\rm esc}<\infty\) with the following property.  Suppose that \(v\)
	solves the equation shifted by \(Q\) in \(B_2\),
	\(|Q|\ge Q_{\rm esc}\), and
	\[
	\fint_{B_2}|v|^p\dd X\le1,\qquad
	\mathcal E_{s,Q}(v;B_2)\le1,\qquad
	\mathcal T^{\rm pt}_{s,Q}(v)\le L_{\rm esc}.
	\]
	Then there is an affine map \(m=c+b\cdot X\), with
	\(|c|+|b|\le C_{\rm esc}\), such that, with
	\[
	\eta_{\rm esc}:=\rho_{\rm esc}^{\alpha},\qquad
	w(X):=\frac{v(\rho_{\rm esc}X)-m(\rho_{\rm esc}X)}
	{\rho_{\rm esc}\eta_{\rm esc}},\qquad
	P:=\frac{Q+b}{\eta_{\rm esc}},
	\]
	one has
	\begin{equation}
		\fint_{B_2}|w|^p\dd X\le1,\qquad
		\mathcal E_{s,P}(w;B_2)\le1,\qquad
		\mathcal T^{\rm pt}_{s,P}(w)\le L_{\rm esc},\qquad
		|P|\ge Q_{\rm esc}.
		\label{eq:large-slope-escape-state-reproduction}
	\end{equation}
\end{lemma}

\begin{proof}
	Let \(C_{\rm L}\) be the constant in Proposition~\ref{prop:large-slope-affine-improvement} with \(L=L_{\rm esc}\) and exponents \(\alpha,\bar\alpha\). Choose \(C_0\) larger than the affine coefficient bound and the constants in the scalar and angular flux estimates, and choose \(C_{\rm esc}\) larger than the coefficient bound. We then choose a dyadic \(\rho_{\rm esc}\in(0,1/16)\) such that
\begin{equation}	\label{eq:escape-large-slope-tail-budget}
	 C_{\rm L}\rho_{\rm esc}^{\bar\alpha-\alpha}\le\frac12\quad \text{and}\quad 
		 C_0\bigl[\rho_{\rm esc}^{\bar\alpha-\alpha}
		+\rho_{\rm esc}^{\gamma-\alpha}\bigr]+C_0L_{\rm esc}\rho_{\rm esc}^{\gamma-\alpha}
		\le L_{\rm esc}.
	\end{equation}
	This is possible since \(\bar\alpha-\alpha,\gamma-\alpha>0\). For this fixed radius, the proposition provides \(Q_{\rm L}\) and an affine map \(m\) such that
	\[
	\mathfrak a_{s,Q}(v,m;\rho_{\rm esc})
	\le C_{\rm L}\rho_{\rm esc}^{\bar\alpha}
	\le\frac12\eta_{\rm esc},
	\qquad
	t^{-1}\|v-m\|_{L^\infty(B_{2t})}
	\le C_{\rm L}t^{\bar\alpha}
	\]
	for \(\rho_{\rm esc}\le t\le1/2\). Choose an admissible \(0<a<\eta_{\rm esc}\) and set \(q=|Q+b|\). Then
	\[
	\fint_{B_2}|w|^p\dd X\le(a/\eta_{\rm esc})^p\le1,
	\qquad
	\frac{q+\eta_{\rm esc}}{q+a}\le\frac{\eta_{\rm esc}}a.
	\]
	By homogeneity,
	\begin{equation*}
		\mathcal E_{s,P}(w;B_2)
		\le
		\left(\frac{q+\eta_{\rm esc}}{q+a}\right)^{2-p}
		\left(\frac a{\eta_{\rm esc}}\right)^2
		\le\left(\frac a{\eta_{\rm esc}}\right)^p\le1.
	\end{equation*}
	
	To estimate the tail, write \(h=v-m\), \(Z=\rho_{\rm esc}X\), and \(W=\rho_{\rm esc}Y\). A change of variables gives
	\begin{equation}
		\lambda_P^{2-p}
		\frac{|\mathcal D_Pw(X,Y)|}{|X-Y|^{n+sp}}\dd Y=
		\rho_{\rm esc}^{\gamma}\eta_{\rm esc}^{1-p}
		\lambda_P^{2-p}
		\frac{|\mathcal D_{Q+b}h(Z,W)|}
		{|Z-W|^{n+sp}}\dd W.
		\label{eq:escape-pointwise-flux-scaling}
	\end{equation}
	We shall choose \(Q_{\rm esc}\) below so that both \(Q\) and \(Q+b\) lie in the large slope regime. Then
	\begin{equation*}
		\frac{\lambda_P}{\lambda_{Q+b}}\le C\eta_{\rm esc}^{-1},
		\qquad
		\frac{\lambda_P}{\lambda_Q}\le C\eta_{\rm esc}^{-1}
	\end{equation*}
	Combining \eqref{eq:escape-pointwise-flux-scaling} with the angular estimate on \(\rho_{\rm esc}\lesssim|W|\lesssim1\), we obtain
	\begin{equation*}
		C\rho_{\rm esc}^{\gamma}\eta_{\rm esc}^{-1}
		\sum_{\substack{\rho_{\rm esc}\le t\le1\\t\text{ dyadic}}}
		t^{\bar\alpha-\gamma}
		\le C\bigl(\rho_{\rm esc}^{\bar\alpha-\alpha}
		+\rho_{\rm esc}^{\gamma-\alpha}\bigr).
	\end{equation*}
	The shells above \(1/2\) are controlled by the local bounds and the affine coefficient estimate. For $|W|\ge4/3$, the first integral in $\mathcal T^{\rm pt}_{s,Q}(v)$ gives the required bound. These contributions are bounded by the first bracket in \eqref{eq:escape-large-slope-tail-budget}.
	
	On the remaining exterior region, we use the decomposition
	\[
	\mathcal D_{Q+b}h
	=\mathcal D_Qv+
	\Jp(Q\cdot(Z-W))-\Jp((Q+b)\cdot(Z-W)).
	\]
	The coefficient multiplying $\mathcal D_Qv$ satisfies
	\begin{equation*}
		\rho_{\rm esc}^{\gamma}\eta_{\rm esc}^{1-p}
		\left(\frac{\lambda_P}{\lambda_Q}\right)^{2-p}
		\le C\rho_{\rm esc}^{\gamma}\eta_{\rm esc}^{-1}.
	\end{equation*}
	The angular estimate controls the affine correction by
	\[
	C|b|\rho_{\rm esc}^{\gamma}\eta_{\rm esc}^{-1}
	\int_1^\infty t^{-\gamma-1}\dd t
	\le C\rho_{\rm esc}^{\gamma-\alpha}.
	\]
	Combining these estimates yields
	\begin{align*}
		\mathcal T^{\rm pt}_{s,P}(w)
		\le C_0\bigl[
		\rho_{\rm esc}^{\bar\alpha-\alpha}
		+\rho_{\rm esc}^{\gamma-\alpha}\bigr]
		+C_0\rho_{\rm esc}^{\gamma-\alpha}
		\mathcal T^{\rm pt}_{s,Q}(v).
	\end{align*}
	By \eqref{eq:escape-large-slope-tail-budget}, \(\mathcal T^{\rm pt}_{s,P}(w)\le L_{\rm esc}\).
	
	Finally, choose
	\[
	Q_{\rm esc}\ge
	\max\left\{Q_{\rm L}+1,
	2C_{\rm esc}+2,
	\frac{C_{\rm esc}}{1-\eta_{\rm esc}}\right\}.
	\]
	Thus
	\[
	|P|\ge\eta_{\rm esc}^{-1}(|Q|-|b|)
	\ge\eta_{\rm esc}^{-1}(Q_{\rm esc}-C_{\rm esc})
	\ge Q_{\rm esc}.
	\]
	This proves \eqref{eq:large-slope-escape-state-reproduction}.
\end{proof}

\begin{proposition}
	\label{prop:affine-campanato-liouville-main}
	If $U$ is a weak solution of \eqref{eq:equation} in $\R^n$,
	$U(0)=0$, and $
	\sup_{R\ge1}\mathscr E_{q_{\rm ext}}(U;R)<\infty$,
	then \(U\) is affine.
\end{proposition}

\begin{proof}
	Let \(m_R=c_R+b_R\cdot X\) be the dyadic minimizers. We first prove that their slopes are bounded. Suppose, to the contrary, that
	\[
	\sup_R|b_R|=\infty,\qquad
	R_j\to\infty,\qquad |b_{R_j}|\to\infty.
	\]
	Using the normalization in Lemma~\ref{lem:sublinear-best-affine-normalization-main}, define
	\[
	\widetilde v_j(X)=
	\frac{U(R_jX)-m_{R_j}(R_jX)}{A_jR_j},
	\qquad \widetilde Q_j=\frac{b_{R_j}}{A_j},
	\qquad A_j=A_{R_j}.
	\]
	Then \(|\widetilde Q_j|\to\infty\). Fix \(0<\alpha<\bar\alpha<\gamma\) and choose \(\rho_{\rm esc},Q_{\rm esc}\) from Lemma~\ref{lem:large-slope-escape-state-main}. For large \(j\), we have \(|\widetilde Q_j|\ge Q_{\rm esc}\). Iterating the lemma with \(r_k=\rho_{\rm esc}^k\) and \(H_k=r_k^\alpha\), as long as \(R_jr_k\ge4\), gives affine maps \(a_{j,k}\) satisfying
	\begin{align}
		\left(\fint_{B_{2r_k}}|\widetilde v_j-a_{j,k}|^p\dd X\right)^{1/p}
		&\le Cr_kH_k,
		\label{eq:escaping-affine-error-main}\\
		|\nabla a_{j,k+1}-\nabla a_{j,k}|&\le CH_k.
		\label{eq:escaping-affine-increments-main}
	\end{align}
	The constants are independent of \(j,k\) by \eqref{eq:large-slope-escape-state-reproduction}.
	
	Choose \(k=k(j)\) so that \(4\le R_jr_k<4\rho_{\rm esc}^{-1}\), and define
	\[
	M_j(x):=m_{R_j}(x)+A_jR_j a_{j,k(j)}(x/R_j).
	\]
	Since \(B_8\subset B_{2R_jr_k}\subset B_{8\rho_{\rm esc}^{-1}}\), estimates \eqref{eq:escaping-affine-error-main}--\eqref{eq:escaping-affine-increments-main} give
	\[
	\left(\fint_{B_8}|U-M_j|^p\dd x\right)^{1/p}
	\le C A_j R_j^{-\alpha},
	\qquad
	|\nabla M_j-b_{R_j}|\le CA_j.
	\]
	Moreover, \eqref{eq:coherent-affine-coefficients-main} and \eqref{eq:sublinear-best-affine-amplitude-main} imply
	\[
	|b_R|\le C(1+\log R),\qquad
	A_jR_j^{-\alpha}
	\le C\bigl[1+(\log R_j)^\theta\bigr]R_j^{-\alpha}\longrightarrow0.
	\]
	On the other hand, norm equivalence on \(\mathcal A\) gives
	\[
		|\nabla M_j|
		\le C\bigl(\|U\|_{L^p(B_8)}+\|U-M_j\|_{L^p(B_8)}\bigr)\le C,\]
		and
\[|\nabla M_j|
		\ge |b_{R_j}|-CA_j
		=|b_{R_j}|\left(1-C\frac{A_j}{|b_{R_j}|}\right)\longrightarrow\infty,
	\]
	which is impossible.
	
	We conclude that \(\sup_R|b_R|<\infty\). Lemma~\ref{lem:coherent-affine-fits-main} then gives
	\[
	\left(\fint_{B_{2R}}|U|^{q_{\rm ext}}\dd X\right)^{1/q_{\rm ext}}
	\le CR
	\]
	For \(Z_R(X)=U(RX)/R\), H\"older's inequality and \(\sum_{h\ge0}2^{-h\gamma}<\infty\) yield
	\[
	\left(\fint_{B_8}|Z_R|^p\dd X\right)^{1/p}
	+\left(\int_{\R^n\setminus B_4}
	\frac{|Z_R(Y)|^{p-1}}{|Y|^{n+sp}}\dd Y\right)^{1/(p-1)}\le C.
	\]
	Applying \cite[Theorem~4]{KorvenpaaKuusiPalatucci2016} to \(\pm Z_R\) on \(B_8\), we obtain \(\|Z_R\|_{L^\infty(B_4)}\le C\). By \cite[Theorem~5]{KorvenpaaKuusiPalatucci2016} and \cite[Proposition~1.5]{BiswasTopp2025}, \(Z_R\) is a continuous viscosity solution. Therefore, \cite[Theorem~2.1]{BiswasTopp2025}, applied on \(B_4\), yields
	\[
	\operatorname{Lip}(U;B_R)=\operatorname{Lip}(Z_R;B_1)
	\le\operatorname{Lip}(Z_R;B_2)\le C.
	\]
	Letting dyadic \(R\to\infty\), we obtain \(\operatorname{Lip}(U;\R^n)\le C\). Theorem~\ref{thm:fixed-order-centered-closure} now yields \(U\in\mathcal A\).
\end{proof}

\subsection{Affine approximation for bounded slopes}
\label{subsec:bounded-slope-improvement}

\begin{proposition}
	\label{prop:bounded-slope-affine-invariant-block-main}
	Fix \(0<Q_0,G_0<\infty\), \(0<\rho<1/16\), and
	\(0<\alpha_1<\gamma\).  There are
	\(\beta_*=\beta_*(n,p,s,\alpha_1,Q_0,G_0,\rho)>0\) and
	\(R_*=R_*(n,p,s,\alpha_1,Q_0,G_0,\rho)\ge8\) such
	that, for every \(0\le\beta\le\beta_*\) and every \(|Q|\le Q_0\), the
	following holds.
	If \(v\) solves the equation shifted by \(Q\) in \(B_{2R_*}\) and
	\begin{equation}
		\mathfrak a_{s,Q}(v,0;1)\le1,
		\qquad\mathscr C_{q_{\rm ext},\beta}(v)\le G_0,
		\qquad\mathcal T^{\rm pt}_{s,Q}(v)\le1,
		\label{eq:bounded-slope-affine-block-assumptions-main}
	\end{equation}
	then there is an affine map \(m\) satisfying
	\begin{equation*}
		|m(0)|+|\nabla m|\le C_{\rm B},
	\end{equation*}
	where \(C_{\rm B}=C_{\rm B}(n,p,s,\alpha_1)\) is independent of
	\(Q_0,G_0,\rho\), such that
	\begin{equation}
		\mathfrak a_{s,Q}(v,m;t)
		+t^{-1}\|v-m\|_{L^\infty(B_{2t})}\le t^{\alpha_1}
		\qquad(\rho\le t\le1/2).
		\label{eq:bounded-slope-affine-full-block-main}
	\end{equation}
\end{proposition}

\begin{proof} Let $C_{\rm aff}\ge1$ be a constant such that
	\[
	|m(0)|+|\nabla m|
	\le C_{\rm aff}
	\left(\fint_{B_1}|m|^p\,dx\right)^{1/p}
	\]
	for every affine function $m$. Choose
	$C_{\rm B}>4C_{\rm aff}$. Suppose that the conclusion is false. Then, for each
	$j\ge1$, there exist
	\[
	0\le\beta_j\le j^{-1},\qquad
	R_j\ge\max\{j,8\},\qquad |Q_j|\le Q_0,
	\]
	and a solution $v_j$ of the equation shifted by $Q_j$
	in $B_{2R_j}$ satisfying
	\eqref{eq:bounded-slope-affine-block-assumptions-main},
	such that for every affine function $m$ with
	$|m(0)|+|\nabla m|\le C_{\rm B}$, there exists
	$t\in[\rho,1/2]$ for which
	\[
	\mathfrak a_{s,Q_j}(v_j,m;t)
	+t^{-1}\|v_j-m\|_{L^\infty(B_{2t})}
	>t^{\alpha_1}.
	\]
	Since $|Q_j|\le Q_0$, we may assume, after passing
	to a subsequence, that $Q_j\to Q$. Let \(m_{j,h}=c_{j,h}+b_{j,h}\cdot X\) be a best affine fit
	to \(v_j\) on \(B_{2^{h+1}}\), for \(h\ge0\). By the definition
	of \(\mathscr C_{q_{\rm ext},\beta_j}\) and the choice of
	\(m_{j,h}\), we have
	\[
	\left(
	\fint_{B_{2^{h+1}}}|v_j-m_{j,h}|^{q_{\rm ext}}\dd X
	\right)^{1/q_{\rm ext}}
	\le G_0\,2^{h(1+\beta_j)}
	\qquad(h\ge0).
	\]
	Since \(B_{2^{h+1}}\subset B_{2^{h+2}}\), the triangle
	inequality gives
	\[
	\begin{aligned}
		&\left(
		\fint_{B_{2^{h+1}}}
		|m_{j,h+1}-m_{j,h}|^{q_{\rm ext}}\dd X
		\right)^{1/q_{\rm ext}}\\
		&\quad \le
		\left(
		\fint_{B_{2^{h+1}}}|v_j-m_{j,h}|^{q_{\rm ext}}\dd X
		\right)^{1/q_{\rm ext}}+
		2^{n/q_{\rm ext}}
		\left(
		\fint_{B_{2^{h+2}}}|v_j-m_{j,h+1}|^{q_{\rm ext}}\dd X
		\right)^{1/q_{\rm ext}}\\
		&\quad \le
		G_0\,2^{h(1+\beta_j)}
		+2^{n/q_{\rm ext}}G_0\,2^{(h+1)(1+\beta_j)}\le CG_0\,2^{h(1+\beta_j)},
	\end{aligned}
	\]
	where \(C\) is independent of \(j,h\), since
	\(0\le\beta_j\le1\).
	
	The difference
	\[
	m_{j,h+1}(X)-m_{j,h}(X)
	=(c_{j,h+1}-c_{j,h})
	+(b_{j,h+1}-b_{j,h})\cdot X
	\]
	is affine. Applying the norm equivalence on the space of
	affine maps, scaled from \(B_1\) to \(B_{2^{h+1}}\), we obtain
	\[
	\begin{aligned}
		|b_{j,h+1}-b_{j,h}|
		+2^{-h}|c_{j,h+1}-c_{j,h}|\le C2^{-h}
		\left(
		\fint_{B_{2^{h+1}}}
		|m_{j,h+1}-m_{j,h}|^{q_{\rm ext}}\dd X
		\right)^{1/q_{\rm ext}}\le CG_0\,2^{h\beta_j}.
	\end{aligned}
	\]
	At the unit scale, we have
	\[
	\begin{aligned}
		|c_{j,0}|+|b_{j,0}|
		&\le C\left(\fint_{B_2}|m_{j,0}|^p\dd X\right)^{1/p}\\
		&\le C\left(\fint_{B_2}|v_j|^p\dd X\right)^{1/p}
		+C\left(\fint_{B_2}|v_j-m_{j,0}|^{q_{\rm ext}}\dd X\right)^{1/q_{\rm ext}}\\
		&\le C(1+G_0).
	\end{aligned}
	\]
	We now sum the coefficient differences. For \(h\ge1\),
	the slope estimate and \(\beta_j\ge0\) give
	\[
	\begin{aligned}
		|b_{j,h}|
		&\le |b_{j,0}|
		+|b_{j,1}-b_{j,0}|+\cdots
		+|b_{j,h}-b_{j,h-1}|\\
		&\le C(1+G_0)+CG_0h\,2^{h\beta_j}\le C(1+G_0)(h+1)2^{h\beta_j}.
	\end{aligned}
	\]
	For the constant terms, the coefficient difference estimate
	implies
	\[
	|c_{j,h+1}-c_{j,h}|
	\le CG_0\,2^{h(1+\beta_j)}.
	\]
	Hence
	\[
	\begin{aligned}
		|c_{j,h}|
		&\le |c_{j,0}|
		+|c_{j,1}-c_{j,0}|+\cdots
		+|c_{j,h}-c_{j,h-1}|\\
		&\le C(1+G_0)+CG_0h\,2^{h(1+\beta_j)}\\
		&\le C(1+G_0)(h+1)2^{h(1+\beta_j)}.
	\end{aligned}
	\]
	These two bounds also hold for \(h=0\) by the initial
	coefficient estimate.
	
	Since \(m_{j,h}(X)=c_{j,h}+b_{j,h}\cdot X\), we obtain
	\[
	\begin{aligned}
		\|m_{j,h}\|_{L^\infty(B_{2^{h+1}})}
		\le |c_{j,h}|+2^{h+1}|b_{j,h}|\le C(1+G_0)(h+1)2^{h(1+\beta_j)}.
	\end{aligned}
	\]
	Combining this with
	\[
	\left(
	\fint_{B_{2^{h+1}}}|v_j-m_{j,h}|^{q_{\rm ext}}\dd X
	\right)^{1/q_{\rm ext}}
	\le G_0\,2^{h(1+\beta_j)},
	\]
	the triangle inequality yields, for $h\ge0$,
	\[
	\begin{aligned}
		\left(
		\fint_{B_{2^{h+1}}}|v_j|^{q_{\rm ext}}\dd X
		\right)^{1/q_{\rm ext}}
&\le
		\left(
		\fint_{B_{2^{h+1}}}|v_j-m_{j,h}|^{q_{\rm ext}}\dd X
		\right)^{1/q_{\rm ext}}
		+\|m_{j,h}\|_{L^\infty(B_{2^{h+1}})}\\
		&\quad\le
		G_0\,2^{h(1+\beta_j)}
		+C(1+G_0)(h+1)2^{h(1+\beta_j)}\\
		&\quad\le C(1+G_0)(h+1)2^{h(1+\beta_j)}.
	\end{aligned}
	\]

	Set \(U_j(X)=Q_j\cdot X+v_j(X)\).  Then
	\[
	(-\Delta_p)^sU_j=0\quad\text{in }B_{2R_j},\qquad
	(p-1)(1+\beta_j)\le sp-\gamma/2\quad(j\text{ large}).
	\]
	For each fixed \(T>0\), the growth bound and a dyadic summation give
	\[
	\|U_j\|_{L^p(B_{2T})}
	+\int_{\R^n\setminus B_{2T}}
	\frac{|U_j(Y)|^{p-1}}{|Y|^{n+sp}}\dd Y
	\le C_T.
	\]
	The local boundedness and H\"older estimates, together with \cite[Proposition~1.5 and Theorem~2.1]{BiswasTopp2025}, yield
	\[
	\|U_j\|_{L^\infty(B_T)}+\Lip(U_j;B_T)\le C_T,
	\qquad C_T=C_T(n,p,s,Q_0,G_0).
	\]
Since \(v_j(X)=U_j(X)-Q_j\cdot X\) and \(|Q_j|\le Q_0\),
	the preceding estimates give, for every fixed \(T>0\),
	\[
	\begin{aligned}
		\|v_j\|_{L^\infty(B_T)}
		&\le \|U_j\|_{L^\infty(B_T)}+T|Q_j|
		\le C_T+TQ_0,\\
		\Lip(v_j;B_T)
		&\le \Lip(U_j;B_T)+|Q_j|
		\le C_T+Q_0.
	\end{aligned}
	\]
	These bounds are independent of \(j\) and hold for all
	sufficiently large \(j\), since \(R_j\to\infty\).
	In particular,
	\[
	|v_j(X)-v_j(Y)|
	\le(C_T+Q_0)|X-Y|
	\qquad(X,Y\in B_T).
	\]
	Applying these estimates on larger balls, we see that
	\((v_j)\) is uniformly bounded and equicontinuous on
	every fixed closed ball.
	
	The Arzel\`a--Ascoli theorem therefore gives a subsequence
	converging uniformly on \(\overline{B_1}\).
	From this subsequence, we extract a further subsequence
	converging uniformly on \(\overline{B_2}\), and continue
	in this way on \(\overline{B_3},\overline{B_4},\ldots\).
	The diagonal subsequence converges uniformly on every
	fixed closed ball. Retaining the notation \(v_j\), we obtain
	\[
	v_j\longrightarrow v_\infty
	\quad\text{locally uniformly in }\R^n.
	\]

	Set \(U_\infty(X)=Q\cdot X+v_\infty(X)\). Since \(Q_j\to Q\)
	and \(v_j\to v_\infty\) locally uniformly, we also have
	\(U_j\to U_\infty\) locally uniformly in \(\R^n\).
	
	We first verify that \(U_\infty\) satisfies the weak equation.
	The growth estimate for \(v_j\), together with \(|Q_j|\le Q_0\),
	gives
	\[
	\begin{aligned}
		\left(
		\fint_{B_{2^{h+1}}}|U_j|^{q_{\rm ext}}\dd X
		\right)^{1/q_{\rm ext}}
		&\le 2^{h+1}Q_0+
		\left(
		\fint_{B_{2^{h+1}}}|v_j|^{q_{\rm ext}}\dd X
		\right)^{1/q_{\rm ext}}\\
		&\le C(h+1)2^{h(1+\beta_j)}
		\qquad(h\ge0).
	\end{aligned}
	\]
	
	Fix \(\varphi\in C_c^\infty(\R^n)\), and choose \(T\ge1\)
	such that \(\operatorname{supp}\varphi\subset B_T\).
	Let \(H\) be an integer such that \(2^H\ge2T\).
	For \(h\ge H\), \(X\in B_T\), and
	\(Y\in B_{2^{h+1}}\setminus B_{2^h}\), we have
	\[
	|X-Y|\ge |Y|-|X|\ge2^{h-1}.
	\]
	Moreover, since \(0<p-1<1\),
	\[
	|\Jp(U_j(X)-U_j(Y))|
	\le |U_j(X)|^{p-1}+|U_j(Y)|^{p-1}.
	\]
	Using the local bound for \(U_j\), followed by
	H\"older's inequality and the growth estimate above, we obtain
	\begin{equation}
		\begin{aligned}
			&\int_{B_T}
			\int_{B_{2^{h+1}}\setminus B_{2^h}}
			\frac{|\Jp(U_j(X)-U_j(Y))|\,|\varphi(X)|}
			{|X-Y|^{n+sp}}\dd Y\dd X\\
			&\quad\le C2^{-h(n+sp)}
			\left[
			2^{hn}+\int_{B_{2^{h+1}}}|U_j(Y)|^{p-1}\dd Y
			\right]\\
			&\quad\le C2^{-hsp}
			\left[
			1+
			\left(
			\fint_{B_{2^{h+1}}}|U_j|^{q_{\rm ext}}\dd Y
			\right)^{(p-1)/q_{\rm ext}}
			\right]\\
			&\quad\le
			C(h+1)^{p-1}
			2^{-h[sp-(p-1)(1+\beta_j)]}.
		\end{aligned}
		\label{eq:bounded-slope-tail-tightness-main}
	\end{equation}
	The constants in this estimate may depend on the fixed
	\(\varphi,T\), but are independent of \(j,h,H\).
	
	Since \(\beta_j\to0\) and \(\gamma=sp-p+1>0\), there is
	\(j_0\) such that
	\[
	sp-(p-1)(1+\beta_j)
	=\gamma-(p-1)\beta_j
	\ge\gamma/2
	\qquad(j\ge j_0).
	\]
	Summing \eqref{eq:bounded-slope-tail-tightness-main}
	over \(h\ge H\), we therefore obtain
	\[
	\begin{aligned}
		&\sup_{j\ge j_0}
		\int_{B_T}\int_{\R^n\setminus B_{2^H}}
		\frac{|\Jp(U_j(X)-U_j(Y))|\,|\varphi(X)|}
		{|X-Y|^{n+sp}}\dd Y\dd X\\
		&\quad\le
		C\sum_{h\ge H}(h+1)^{p-1}2^{-h\gamma/2}
		\longrightarrow0
		\qquad(H\to\infty).
	\end{aligned}
	\]
	By local uniform convergence and Fatou's lemma, the same
	tail bound holds with \(U_j\) replaced by \(U_\infty\).
	
	We next pass to the limit on a fixed bounded region.
	For fixed \(H\), the local Lipschitz bounds give
	\[
	|\delta U_j(X,Y)|\le C|X-Y|,
	\qquad
	|\delta\varphi(X,Y)|\le C|X-Y|
	\qquad(X,Y\in B_{2^H}),
	\]
	where \(C\) is independent of sufficiently large \(j\).
	Hence
	\[
	\frac{|\Jp(\delta U_j)\delta\varphi|}
	{|X-Y|^{n+sp}}
	\le C|X-Y|^{p-n-sp}.
	\]
	This bound is integrable on
	\(B_{2^H}\times B_{2^H}\), because \(p-sp>0\).
	The continuity of \(\Jp\) and dominated convergence imply
	\[
	\begin{aligned}
		&\lim_{j\to\infty}
		\iint_{B_{2^H}\times B_{2^H}}
		\frac{\Jp(\delta U_j)\delta\varphi}
		{|X-Y|^{n+sp}}\dd X\dd Y=
		\iint_{B_{2^H}\times B_{2^H}}
		\frac{\Jp(\delta U_\infty)\delta\varphi}
		{|X-Y|^{n+sp}}\dd X\dd Y.
	\end{aligned}
	\]
	
	For all sufficiently large \(j\), the support of \(\varphi\)
	lies in the region where \(U_j\) solves the equation.
	Since \(\varphi=0\) outside \(B_T\), its weak formulation
	can be written as
	\[
	\begin{aligned}
		0
		=\iint_{B_{2^H}\times B_{2^H}}
		\frac{\Jp(\delta U_j)\delta\varphi}
		{|X-Y|^{n+sp}}\dd X\dd Y+
		2\int_{B_T}\int_{\R^n\setminus B_{2^H}}
		\frac{\Jp(U_j(X)-U_j(Y))\varphi(X)}
		{|X-Y|^{n+sp}}\dd Y\dd X.
	\end{aligned}
	\]
	We first let \(j\to\infty\) on the bounded region and use
	the uniform tail bound for the second term. We then let
	\(H\to\infty\). The corresponding tail of \(U_\infty\)
	also tends to zero, so
	\[
	\iint_{\R^n\times\R^n}
	\frac{\Jp(\delta U_\infty)\delta\varphi}
	{|X-Y|^{n+sp}}\dd X\dd Y=0.
	\]
	The local Lipschitz bounds pass to \(U_\infty\).
	The growth estimates also pass to the limit at each fixed
	radius and ensure that
	\[
	U_\infty\in
	W^{s,p}_{\rm loc}(\R^n)\cap L^{p-1}_{sp}(\R^n).
	\]
	Thus \(U_\infty\) is a weak solution in \(\R^n\).
	
	We now verify the affine growth condition required by
	Proposition~\ref{prop:affine-campanato-liouville-main}.
	Adding \(Q_j\cdot X\) does not change the distance to
	the space of affine functions. Therefore, for every fixed
	\(h\ge0\),
	\[
	2^{-h}\inf_{m\in\mathcal A}
	\left(
	\fint_{B_{2^{h+1}}}|U_j-m|^{q_{\rm ext}}\dd X
	\right)^{1/q_{\rm ext}}
	\le G_0\,2^{h\beta_j}.
	\]
	The triangle inequality gives
	\[
	\begin{aligned}
		2^{-h}\inf_{m\in\mathcal A}
		\left(
		\fint_{B_{2^{h+1}}}|U_\infty-m|^{q_{\rm ext}}\dd X
		\right)^{1/q_{\rm ext}}\le
		2^{-h}\|U_\infty-U_j\|_{L^\infty(B_{2^{h+1}})}
		+G_0\,2^{h\beta_j}.
	\end{aligned}
	\]
	Keeping \(h\) fixed and letting \(j\to\infty\), we obtain
	\[
	2^{-h}\inf_{m\in\mathcal A}
	\left(
	\fint_{B_{2^{h+1}}}|U_\infty-m|^{q_{\rm ext}}\dd X
	\right)^{1/q_{\rm ext}}
	\le G_0
	\qquad(h\ge0).
	\]
	
	For an arbitrary \(R\ge1\), choose \(h\ge0\) such that
	\(R\le2^h<2R\). Comparing the averages on
	\(B_{2R}\subset B_{2^{h+1}}\), we find
	\[
	\begin{aligned}
		R^{-1}\inf_{m\in\mathcal A}
		\left(
		\fint_{B_{2R}}|U_\infty-m|^{q_{\rm ext}}\dd X
		\right)^{1/q_{\rm ext}}&\le
		2^{n/q_{\rm ext}}R^{-1}\inf_{m\in\mathcal A}
		\left(
		\fint_{B_{2^{h+1}}}|U_\infty-m|^{q_{\rm ext}}\dd X
		\right)^{1/q_{\rm ext}}\\
		&\le
		2^{n/q_{\rm ext}}\frac{2^h}{R}G_0
		\le 2^{1+n/q_{\rm ext}}G_0.
	\end{aligned}
	\]
	Consequently,
	\[
	\sup_{R\ge1}R^{-1}\inf_{m\in\mathcal A}
	\left(
	\fint_{B_{2R}}|U_\infty-m|^{q_{\rm ext}}\dd X
	\right)^{1/q_{\rm ext}}
	\le C(1+G_0).
	\]
	
	Subtracting \(U_\infty(0)\) preserves both the weak equation
	and this affine growth condition.
	Proposition~\ref{prop:affine-campanato-liouville-main}
	therefore gives
	\[
	U_\infty\in\mathcal A.
	\]
	Since \(v_\infty=U_\infty-Q\cdot X\), we also have
	\[
	v_\infty=m_\infty\in\mathcal A.
	\]
	
	Finally, the initial normalization and local uniform
	convergence imply
	\[
	\left(
	\fint_{B_2}|v_\infty|^p\dd X
	\right)^{1/p}
	=
	\lim_{j\to\infty}
	\left(
	\fint_{B_2}|v_j|^p\dd X
	\right)^{1/p}
	\le1.
	\]
	Applying the defining inequality for \(C_{\rm aff}\)
	on \(B_1\) to the affine function \(m_\infty(2X)\), we obtain
	\[
	\begin{aligned}
		|m_\infty(0)|+|\nabla m_\infty|
		&\le |m_\infty(0)|+2|\nabla m_\infty|\\
		&\le C_{\rm aff}
		\left(
		\fint_{B_1}|m_\infty(2X)|^p\dd X
		\right)^{1/p}=C_{\rm aff}
		\left(
		\fint_{B_2}|v_\infty|^p\dd X
		\right)^{1/p}\le C_{\rm aff}<C_{\rm B}/2.
	\end{aligned}
	\]
	
	It remains to prove convergence in the intrinsic amplitude,
	uniformly for \(\rho\le t\le1/2\).
	Let \(h_j=v_j-m_\infty\) and \(b_\infty=\nabla m_\infty\).
	The local uniform convergence and the preceding local
	Lipschitz bounds imply
	\[
	\|h_j\|_{L^\infty(B_1)}\longrightarrow0,
	\qquad
	\Lip(h_j;B_1)
	\le \Lip(v_j;B_1)+|b_\infty|
	\le C.
	\]
	
	By the convexity of \(\Phi\) and the bound
	\(D_\Phi(a;t)\le C|t|^p\), we have
	\[
	\begin{aligned}
		0
		\le
		D_\Phi\bigl((Q_j+b_\infty)\cdot(X-Y);
		\delta h_j(X,Y)\bigr)\le C|\delta h_j(X,Y)|^p
		\le C|X-Y|^p,
		\qquad(X,Y\in B_1).
	\end{aligned}
	\]
	For each fixed pair \(X\ne Y\), the middle bound also
	shows that this Bregman remainder tends to zero, since
	\(h_j\to0\) uniformly on \(B_1\).
	Moreover,
	\[
	\frac{
		D_\Phi\bigl((Q_j+b_\infty)\cdot(X-Y);
		\delta h_j(X,Y)\bigr)}
	{|X-Y|^{n+sp}}
	\le C|X-Y|^{p-n-sp},
	\]
	and
	\[
	\iint_{B_1\times B_1}|X-Y|^{p-n-sp}\dd X\dd Y
	\le C\int_0^2 r^{p-sp-1}\dd r<\infty,
	\]
	because \(p-sp=p(1-s)>0\).
	Dominated convergence therefore yields
	\[
	J_j:=
	\iint_{B_1\times B_1}
	\frac{
		D_\Phi\bigl((Q_j+b_\infty)\cdot(X-Y);
		\delta h_j(X,Y)\bigr)}
	{|X-Y|^{n+sp}}\dd X\dd Y
	\longrightarrow0.
	\]
	
	We next control the two quantities defining the intrinsic
	amplitude. For every \(\rho\le t\le1/2\), we have
	\(B_{2t}\subset B_1\), and hence
	\[
	\begin{aligned}
		\mathfrak Z(v_j,m_\infty;t)
		=\fint_{B_{2t}}\left|\frac{h_j}{t}\right|^p\dd X\le t^{-p}\|h_j\|_{L^\infty(B_1)}^p\le \rho^{-p}\|h_j\|_{L^\infty(B_1)}^p.
	\end{aligned}
	\]
	For the energy term, its definition gives
	\[
	\begin{aligned}
		\mathfrak B_{s,Q_j}(v_j,m_\infty;t)
		&=(1-s)t^{sp-p-n}
		\iint_{B_{2t}\times B_{2t}}
		\frac{
			D_\Phi\bigl((Q_j+b_\infty)\cdot(X-Y);
			\delta h_j(X,Y)\bigr)}
		{|X-Y|^{n+sp}}\dd X\dd Y\\
		&\le(1-s)\rho^{sp-p-n}J_j.
	\end{aligned}
	\]
	Here we used the nonnegativity of the integrand,
	\(B_{2t}\times B_{2t}\subset B_1\times B_1\), and
	\(t^{sp-p-n}\le\rho^{sp-p-n}\), since \(sp-p-n<0\).
	Thus
	\[
	\sup_{\rho\le t\le1/2}
	\mathfrak Z(v_j,m_\infty;t)\longrightarrow0,
	\qquad
	\sup_{\rho\le t\le1/2}
	\mathfrak B_{s,Q_j}(v_j,m_\infty;t)\longrightarrow0.
	\]
	
	To deduce convergence of the intrinsic amplitude,
	choose \(M=Q_0+|b_\infty|\), so that
	\(|Q_j+b_\infty|\le M\) for every \(j\).
	Fix any \(a>0\). Since \(p-2<0\),
	\[
	(|Q_j+b_\infty|+a)^{p-2}a^2
	\ge(M+a)^{p-2}a^2>0.
	\]
	The two uniform convergences above imply that there is
	\(j_a\) such that, for every \(j\ge j_a\) and every
	\(t\in[\rho,1/2]\),
	\[
	\begin{cases}
		\mathfrak Z(v_j,m_\infty;t)\le a^p,\\
		\mathfrak B_{s,Q_j}(v_j,m_\infty;t)
		\le(M+a)^{p-2}a^2
		\le(|Q_j+b_\infty|+a)^{p-2}a^2.
	\end{cases}
	\]
	These are precisely the two inequalities required in
	\eqref{eq:block-intrinsic-amplitude}. Consequently,
	\[
	\mathfrak a_{s,Q_j}(v_j,m_\infty;t)\le a
	\qquad
	(j\ge j_a,\ \rho\le t\le1/2).
	\]
	Since \(a>0\) was arbitrary, this proves
	\[
	\sup_{\rho\le t\le1/2}
	\mathfrak a_{s,Q_j}(v_j,m_\infty;t)
	\longrightarrow0.
	\]
	The function error also tends to zero uniformly over
	the same interval, because
	\[
	\sup_{\rho\le t\le1/2}
	t^{-1}\|v_j-m_\infty\|_{L^\infty(B_{2t})}
	\le\rho^{-1}\|h_j\|_{L^\infty(B_1)}
	\longrightarrow0.
	\]
	
	We may therefore choose \(j\) sufficiently large that
	\[
	\sup_{\rho\le t\le1/2}
	\mathfrak a_{s,Q_j}(v_j,m_\infty;t)
	<\frac12\rho^{\alpha_1},
	\qquad
	\rho^{-1}\|h_j\|_{L^\infty(B_1)}
	<\frac12\rho^{\alpha_1}.
	\]
	For this \(j\), the same affine function \(m_\infty\)
	satisfies
	\[\mathfrak a_{s,Q_j}(v_j,m_\infty;t)
		+t^{-1}\|v_j-m_\infty\|_{L^\infty(B_{2t})}<\rho^{\alpha_1}\le t^{\alpha_1},
		\qquad(\rho\le t\le1/2).
	\]
	Its coefficients satisfy
	\(|m_\infty(0)|+|\nabla m_\infty|\le C_{\rm B}/2\).
	The defining property of the counterexample \(v_j\),
	applied to \(m_\infty\), requires the opposite strict
	inequality at some \(t\in[\rho,1/2]\).
	This contradiction proves the proposition.
\end{proof}

\subsection{Choice of parameters and preservation of the estimates}
\label{subsec:parameter-choice-state}

	We fix the parameters in the following order:
\begin{equation}
	0<\alpha_1<\bar\alpha<\gamma.
	\label{eq:affine-route-intermediate-exponents-main}
\end{equation}
	Let \(C_{\rm L}\) dominate the constants in the estimates and affine coefficient bounds of Proposition~\ref{prop:large-slope-affine-improvement} for \(\alpha=\alpha_1\), \(L=1\), with \(C_{\rm L}\ge\max\{2^{\alpha_1}C_{\rm B},1\}\).  Set
\begin{equation}
	G_*:=2+C_{\rm L}.
	\label{eq:affine-route-campanato-budget-main}
\end{equation}
	Next choose
\begin{equation}
	0<\bar\beta<\frac12\min\{\alpha_1,\gamma\}.
	\label{eq:affine-route-beta-bar-main}
\end{equation}
	Fix $C_0$ large enough for the scalar and angular estimates below, depending on the constants $C_{\rm L}$ and $G_*$ already chosen. Choose a dyadic number \(\rho=2^{-N}<1/16\) such that
\begin{align}
	C_{\rm L}\rho^{\bar\alpha-\alpha_1}&\le1,
	\label{eq:affine-route-large-endpoint-main}\\
	C_0\bigl[
	\rho^{(\alpha_1-\bar\beta)(p-1)}
	+\rho^{\gamma-(p-1)\bar\beta}
	+\rho^{\alpha_1-\bar\beta}
	+\rho^{\gamma-\bar\beta}\bigr]&\le\frac14.
	\label{eq:affine-route-tail-small-main}
\end{align}
	Proposition~\ref{prop:large-slope-affine-improvement} determines \(Q_{\rm th}=Q_{\rm th}(\rho)\). Applying Proposition~\ref{prop:bounded-slope-affine-invariant-block-main} with \(Q_0=Q_{\rm th}\), \(G_0=G_*\), \(\rho\), and \(\alpha_1\) from \eqref{eq:affine-route-intermediate-exponents-main}, we obtain $\beta_*$ and a radius $R_*=R_*(n,p,s,\alpha_1,Q_{\rm th},G_*,\rho)\ge8$. Finally, set
\begin{equation}
	\beta:=\frac12\min\{\beta_*,\bar\beta\},
	\qquad\eta:=\rho^\beta.
	\label{eq:affine-route-final-beta-main}
\end{equation}
	These parameters remain fixed throughout the iteration.
\begin{lemma}
	\label{lem:affine-invariant-state-reproduction-main}
	Suppose that \(v\) solves the equation shifted by \(Q\)
	in \(B_{2R_*}\) and that
	\begin{equation}
		\mathfrak a_{s,Q}(v,0;1)\le1,
		\qquad
		\mathcal T_{\beta,G_*}(v,Q)\le1.
		\label{eq:affine-route-unit-assumptions-main}
	\end{equation}
	If \(|Q|\le Q_{\rm th}\), let \(m=c+b\cdot X\) be the
	affine map in
	Proposition~\ref{prop:bounded-slope-affine-invariant-block-main}.
	If \(|Q|>Q_{\rm th}\), use the affine map in
	Proposition~\ref{prop:large-slope-affine-improvement}.
	Define
	\[
	w(X)=\frac{v(\rho X)-m(\rho X)}{\rho\eta},
	\qquad
	P=\frac{Q+b}{\eta}.
	\]
	Then \(w\) solves the equation shifted by \(P\) in
	\(B_{2R_*/\rho}\), and
	\begin{equation}
		\mathfrak a_{s,P}(w,0;1)\le1,
		\qquad
		\mathcal T_{\beta,G_*}(w,P)\le1.
		\label{eq:affine-route-unit-reproduction-main}
	\end{equation}
\end{lemma}

\begin{proof}
	All parameters are fixed as above. The constants \(C\)
	below depend only on the structural parameters and the
	already fixed constants \(C_{\rm L},G_*\).
	They are independent of \(Q,Q_{\rm th},\rho\).
	The constant \(C_0\) in
	\eqref{eq:affine-route-tail-small-main} is chosen at the
	outset to dominate the constants in the tail estimates
	below.
	
We first write the equation satisfied by \(w\) and recall the
affine approximation estimates in the two cases. Since
	\[
	P\cdot X+w(X)
	=
	\frac{Q\cdot(\rho X)+v(\rho X)-c}{\rho\eta},
	\]
	the homogeneity of the equation gives
	\[
	(-\Delta_p)^s(P\cdot X+w)=0
	\qquad\text{in }B_{2R_*/\rho}.
	\]
	Thus \(w\) solves the equation shifted by \(P\).
	
	By \eqref{eq:affine-route-unit-assumptions-main} and the
	definitions of the intrinsic amplitude and the exterior
	quantity, we have
	\[
		\fint_{B_2}|v|^p\dd X\le1,
		\quad
		\mathcal E_{s,Q}(v;B_2)\le1 \quad 
		\mathscr C_{q_{\rm ext},\beta}(v)\le G_*,
		\quad
		\mathcal T^{\rm pt}_{s,Q}(v)\le1.
	\]
	If \(|Q|\le Q_{\rm th}\), we apply
	Proposition~\ref{prop:bounded-slope-affine-invariant-block-main}.
	If \(|Q|>Q_{\rm th}\), we apply
	Proposition~\ref{prop:large-slope-affine-improvement}.
	The choice of \(C_{\rm L}\) gives, in either case, $
	|c|+|b|\le C_{\rm L}$.
	In the large slope case, we also use $
	C_{\rm L}\rho^{\bar\alpha}
	\le\rho^{\alpha_1}$,
	which follows from
	\eqref{eq:affine-route-large-endpoint-main}.
	Consequently, both cases give
	\begin{equation}
		\mathfrak a_{s,Q}(v,m;\rho)
		\le\rho^{\alpha_1}<\eta,
		\label{eq:affine-route-common-endpoint-main}
	\end{equation}
	and
	\begin{equation}
		t^{-1}\|v-m\|_{L^\infty(B_{2t})}
		\le C_{\rm L}t^{\alpha_1}
		\qquad(\rho\le t\le1/2).
		\label{eq:affine-route-common-intermediate-main}
	\end{equation}
	For the second estimate in the large slope case, we used
	\(t^{\bar\alpha}\le t^{\alpha_1}\).
	The strict inequality in
	\eqref{eq:affine-route-common-endpoint-main} follows from
	\(\eta=\rho^\beta\), \(0<\rho<1\), and \(\beta<\alpha_1\).
	
	\emph{Step 1: preservation of the intrinsic amplitude.}
	A change of variables and the \(p\)-homogeneity of
	\(D_\Phi\) give
	\[
	\mathfrak Z(w,0;1)
	=\eta^{-p}\mathfrak Z(v,m;\rho),
	\qquad
	\mathfrak B_{s,P}(w,0;1)
	=\eta^{-p}\mathfrak B_{s,Q}(v,m;\rho).
	\]
	Indeed, in the energy integral the change of variables
	contributes \(\rho^{sp-n}\), while the Bregman remainder
	contributes \((\rho\eta)^{-p}\).
	The remaining factor
	\(\rho^{sp-p-n}\) is precisely the factor in
	\(\mathfrak B_{s,Q}(v,m;\rho)\).
	
	For every \(a>0\), the identity \(P=(Q+b)/\eta\) also gives
	\[
	\left(|P|+\frac a\eta\right)^{p-2}
	\left(\frac a\eta\right)^2
	=
	\eta^{-p}(|Q+b|+a)^{p-2}a^2.
	\]
	Hence the two inequalities defining an admissible
	amplitude \(a\) for \(\mathfrak a_{s,Q}(v,m;\rho)\)
	are equivalent to the two inequalities defining the
	amplitude \(a/\eta\) for \(\mathfrak a_{s,P}(w,0;1)\).
	Taking infima, we obtain
	\[
	\mathfrak a_{s,P}(w,0;1)
	=
	\eta^{-1}\mathfrak a_{s,Q}(v,m;\rho)
	\le\rho^{\alpha_1-\beta}<1.
	\]
	
	\emph{Step 2: preservation of the affine growth bound.}
	Recall that \(\rho=2^{-N}\).
	For \(j\ge N\), affine invariance and a change of variables
	give
	\[
	\mathbf c_{q_{\rm ext}}(w;j)
	=
	\eta^{-1}\mathbf c_{q_{\rm ext}}(v;j-N).
	\]
	Since \(\eta=\rho^\beta=2^{-N\beta}\), it follows that
	\[
	\begin{aligned}
		2^{-j\beta}\mathbf c_{q_{\rm ext}}(w;j)
		=
		2^{-(j-N)\beta}
		\mathbf c_{q_{\rm ext}}(v;j-N)\le\mathscr C_{q_{\rm ext},\beta}(v)
		\le G_*.
	\end{aligned}
	\]
	
	For \(0\le j<N\), let \(t=\rho2^j\).
	Then \(\rho\le t\le1/2\).
	Using the zero affine function as a competitor in the
	definition of \(\mathbf c_{q_{\rm ext}}(w;j)\), we find
	\[
	\begin{aligned}
		2^{-j\beta}\mathbf c_{q_{\rm ext}}(w;j)
		&\le
		2^{-j(1+\beta)}
		\|w\|_{L^\infty(B_{2^{j+1}})}=
		\frac{2^{-j(1+\beta)}}{\rho\eta}
		\|v-m\|_{L^\infty(B_{2t})}=
		t^{-1-\beta}\|v-m\|_{L^\infty(B_{2t})}\\
		&\le C_{\rm L}t^{\alpha_1-\beta}
		\le C_{\rm L}.
	\end{aligned}
	\]
	Together with the estimate for \(j\ge N\), this proves
	\[
	\mathscr C_{q_{\rm ext},\beta}(w)
	\le
	\max\bigl\{
	C_{\rm L},\mathscr C_{q_{\rm ext},\beta}(v)
	\bigr\}
	\le G_*.
	\]
	
	\emph{Step 3: preservation of the pointwise relative tail.}
	Set \(h=v-m\), \(Z=\rho X\), and \(W=\rho Y\).
	The homogeneity of \(\Jp\) gives
	\[
	\mathcal D_Pw(X,Y)
	=
	(\rho\eta)^{1-p}\mathcal D_{Q+b}h(Z,W).
	\]
	Therefore,
	\[
	\begin{aligned}
		\lambda_P^{2-p}
		\frac{|\mathcal D_Pw(X,Y)|}{|X-Y|^{n+sp}}\dd Y=
		\rho^\gamma\eta^{1-p}\lambda_P^{2-p}
		\frac{|\mathcal D_{Q+b}h(Z,W)|}
		{|Z-W|^{n+sp}}\dd W.
	\end{aligned}
	\]
	In the first integral defining \(\mathcal T^{\rm pt}_{s,P}(w)\),
	we have \(|Z|\le5\rho/4\) and \(|W|\ge4\rho/3\).
	In the second, \(|Z|\le3\rho/2\) and \(|W|\ge2\rho\).
	In both cases, \(|Z|\le3\rho/2\) and
	\(|Z|\le15|W|/16\), so
	\[
	\frac{|W|}{16}\le |Z-W|\le2|W|.
	\]
	The estimates below apply to both integrals and are uniform
	in \(Z\).
	
We distinguish the cases \(|P|\le1\) and \(|P|>1\).
	If \(|P|\le1\), then \(\lambda_P\le2\), and the
	\((p-1)\)-H\"older continuity of \(\Jp\) gives
	\[
	\eta^{1-p}\lambda_P^{2-p}
	|\mathcal D_{Q+b}h(Z,W)|
	\le
	C\eta^{1-p}|\delta h(Z,W)|^{p-1}.
	\]
	If \(|P|>1\), set \(e=P/|P|\).
	Since \(Q+b=\eta P\),
	\[
	\mathcal D_{Q+b}h
	=\eta^{p-1}\mathcal D_P(h/\eta).
	\]
	Applying the angular estimate in
	Lemma~\ref{lem:large-slope-scalar-angular} with slope \(P\),
	we obtain
	\[
	\begin{aligned}
		&\eta^{1-p}\lambda_P^{2-p}
		|\mathcal D_{Q+b}h(Z,W)|\le
		C\eta^{-1}|Z-W|^{p-2}
		|e\cdot\omega|^{p-2}|\delta h(Z,W)|,
		\qquad
		\omega=\frac{Z-W}{|Z-W|}.
	\end{aligned}
	\]
	This estimate includes the case \(|Q+b|<1<|P|\).
	The division according to \(P\) is independent of the
	division according to \(Q\) used to choose \(m\).

	We estimate the integrals separately over
	\[
	|W|<\frac12,\qquad
	\frac12\le|W|<2,\qquad
	|W|\ge2.
	\]
	
	Consider first \(|W|<1/2\).
	The relevant part of each exterior domain is covered by
	\[
	2^j\rho\le|W|<2^{j+1}\rho,
	\qquad 0\le j\le N-2.
	\]
	For \(t=2^j\rho\), we estimate the integral over the
	intersection of \(t\le|W|<2t\) with the corresponding
	exterior domain.
	On this intersection,
	\[
	Z,W\in B_{2t},
	\qquad
	\frac{t}{16}\le|Z-W|\le4t.
	\]
	It follows from
	\eqref{eq:affine-route-common-intermediate-main} that
	\[
	|\delta h(Z,W)|
	\le2\|h\|_{L^\infty(B_{2t})}
	\le Ct^{1+\alpha_1}.
	\]
	If \(|P|\le1\), the preceding estimate bounds the
	contribution of this shell by
	\[
	\begin{aligned}
		&C\rho^\gamma\eta^{1-p}t^{(1+\alpha_1)(p-1)}
		\int_{\{t/16\le|Z-W|\le4t\}}
		\frac{\dd W}{|Z-W|^{n+sp}}\\
		&\qquad\le
		C\rho^\gamma\eta^{1-p}
		t^{(1+\alpha_1)(p-1)-sp}=
		C\rho^\gamma\eta^{1-p}
		t^{\alpha_1(p-1)-\gamma}.
	\end{aligned}
	\]
	If \(|P|>1\), we use polar coordinates centered at \(Z\)
	and the angular bound
	\[
	\int_{\mathbb S^{n-1}}
	|e\cdot\omega|^{p-2}\dd\omega\le C.
	\]
	The shell contribution is then at most
	\[
	\begin{aligned}
		&C\rho^\gamma\eta^{-1}t^{1+\alpha_1}
		\int_{t/16}^{4t}r^{p-3-sp}\dd r
		\int_{\mathbb S^{n-1}}
		|e\cdot\omega|^{p-2}\dd\omega\le
		C\rho^\gamma\eta^{-1}t^{\alpha_1-\gamma}.
	\end{aligned}
	\]
	Since \(\alpha_1<\gamma\) and \(p-1<1\), both exponents $
	\alpha_1(p-1)-\gamma$ and $
	\alpha_1-\gamma$
	are negative. Hence, we have
	\[
		\sum_{j=0}^{N-2}
		(2^j\rho)^{\alpha_1(p-1)-\gamma}
		\le C\rho^{\alpha_1(p-1)-\gamma}\quad \text{and} \quad \sum_{j=0}^{N-2}
		(2^j\rho)^{\alpha_1-\gamma}
		\le C\rho^{\alpha_1-\gamma}.
	\]
	After summing the shell estimates and using
	\(\eta=\rho^\beta\), we conclude that the contribution
	from \(|W|<1/2\) is bounded by
	\[
	\begin{aligned}
		&C\rho^\gamma\eta^{1-p}
		\rho^{\alpha_1(p-1)-\gamma}
		+
		C\rho^\gamma\eta^{-1}\rho^{\alpha_1-\gamma}=
		C\left[
		\rho^{(\alpha_1-\beta)(p-1)}
		+\rho^{\alpha_1-\beta}
		\right].
	\end{aligned}
	\]
	
	We next consider \(1/2\le|W|<2\).
	We first obtain an \(L^{q_{\rm ext}}(B_4)\) bound for \(v\).
	For every affine function \(\ell\), the triangle
	inequality and equivalence of norms on affine functions
	give
	\[
	\begin{aligned}
		&\left(
		\fint_{B_4}|v|^{q_{\rm ext}}\dd X
		\right)^{1/q_{\rm ext}}\\
		&\quad\le
		\left(
		\fint_{B_4}|v-\ell|^{q_{\rm ext}}\dd X
		\right)^{1/q_{\rm ext}}
		+
		\left(
		\fint_{B_4}|\ell|^{q_{\rm ext}}\dd X
		\right)^{1/q_{\rm ext}}\\
		&\quad\le
		\left(
		\fint_{B_4}|v-\ell|^{q_{\rm ext}}\dd X
		\right)^{1/q_{\rm ext}}
		+
		C\left(\fint_{B_2}|\ell|^p\dd X\right)^{1/p}\\
		&\quad\le
		\left(
		\fint_{B_4}|v-\ell|^{q_{\rm ext}}\dd X
		\right)^{1/q_{\rm ext}}
		+
		C\left(\fint_{B_2}|v|^p\dd X\right)^{1/p}
		+
		C\left(\fint_{B_2}|v-\ell|^p\dd X\right)^{1/p}\\
		&\quad\le
		C\left(
		\fint_{B_4}|v-\ell|^{q_{\rm ext}}\dd X
		\right)^{1/q_{\rm ext}}
		+
		C\left(\fint_{B_2}|v|^p\dd X\right)^{1/p}.
	\end{aligned}
	\]
	The last inequality follows from \(q_{\rm ext}>p\),
	H\"older's inequality, and \(B_2\subset B_4\). Taking the infimum over \(\ell\in\mathcal A\) and using
	the \(j=1\) term of
	\(\mathscr C_{q_{\rm ext},\beta}(v)\le G_*\), we obtain
	\[
	\left(\fint_{B_4}|v|^{q_{\rm ext}}\dd X\right)^{1/q_{\rm ext}}
	\le C\bigl(1+2^{1+\beta}G_*\bigr)\le C.
	\]
	Together with \(|b|+|c|\le C_{\rm L}\), this also gives
	\[
	\left(\fint_{B_4}|h|^{q_{\rm ext}}\dd X\right)^{1/q_{\rm ext}}
	\le C.
	\]
	For the center \(Z\), the estimate at \(t=\rho\) yields
	\[
	|h(Z)|
	\le\|h\|_{L^\infty(B_{2\rho})}
	\le C_{\rm L}\rho^{1+\alpha_1}\le C.
	\]
	
	On the present annulus, \(|Z-W|\) is bounded above and
	below by positive constants.
	The scalar estimate and H\"older's inequality therefore
	bound its contribution by
	\[
	C\rho^\gamma\eta^{1-p}
	=
	C\rho^{\gamma-(p-1)\beta}.
	\]
	For the angular estimate, H\"older's inequality is
	applicable because $
	(p-2)q_{\rm ext}'>-1$. Indeed, polar coordinates centered at \(Z\) give
	\[
	\int_{\{1/2\le|W|<2\}}
	|e\cdot\omega|^{(p-2)q_{\rm ext}'}\dd W
	\le C.
	\]
	Combining this bound with the \(L^{q_{\rm ext}}\) bound
	for \(h\), the angular contribution is at most $
	C\rho^\gamma\eta^{-1}
	=
	C\rho^{\gamma-\beta}$.
	Thus the contribution from \(|W|<2\) is bounded by
	\begin{equation}
		\begin{aligned}
			C\bigl[
			\rho^{(\alpha_1-\beta)(p-1)}
			+\rho^{\gamma-(p-1)\beta}+\rho^{\alpha_1-\beta}
			+\rho^{\gamma-\beta}
			\bigr].
		\end{aligned}
		\label{eq:affine-route-intervening-flux-main}
	\end{equation}

	We finally consider \(|W|\ge2\).
	Since \(h=v-m\),
	\[
	\begin{aligned}
		\mathcal D_{Q+b}h(Z,W)
		={}&\mathcal D_Qv(Z,W)
		+\Jp(Q\cdot(Z-W))
		-\Jp((Q+b)\cdot(Z-W)).
	\end{aligned}
	\]
	We estimate the terms on the right separately.
	
	Since \(Z\in B_{3\rho/2}\subset B_{3/2}\), the definition
	of \(\mathcal T^{\rm pt}_{s,Q}(v)\) gives
	\[
	\begin{aligned}
		&\rho^\gamma\eta^{1-p}\lambda_P^{2-p}
		\int_{\R^n\setminus B_2}
		\frac{|\mathcal D_Qv(Z,W)|}{|Z-W|^{n+sp}}\dd W\le
		\rho^\gamma\eta^{1-p}
		\left(\frac{\lambda_P}{\lambda_Q}\right)^{2-p}
		\mathcal T^{\rm pt}_{s,Q}(v).
	\end{aligned}
	\]
	If \(|P|\le1\), then \(\lambda_P/\lambda_Q\le2\).
	If \(|P|>1\), we use \(|b|\le C_{\rm L}\) and
	\(\eta\le1\) to obtain
	\[
	\begin{aligned}
		\frac{\lambda_P}{\lambda_Q}
		&=\frac{1+|Q+b|/\eta}{1+|Q|}\le\eta^{-1}
		\frac{1+|Q|+|b|}{1+|Q|}
		\le C\eta^{-1}.
	\end{aligned}
	\]
	Thus, recalling \(\eta=\rho^\beta\), we have
	\[
	\begin{aligned}
		&\rho^\gamma\eta^{1-p}\lambda_P^{2-p}
		\int_{\R^n\setminus B_2}
		\frac{|\mathcal D_Qv(Z,W)|}{|Z-W|^{n+sp}}\dd W\le
		C\left[
		\rho^{\gamma-(p-1)\beta}
		+\rho^{\gamma-\beta}
		\right]\mathcal T^{\rm pt}_{s,Q}(v).
	\end{aligned}
	\]
	
	We next estimate the difference involving \(\Jp\).
	Write \(Z-W=r\omega\).
	If \(|P|\le1\), the \((p-1)\)-H\"older continuity of
	\(\Jp\), together with \(|b|\le C_{\rm L}\), gives
	\[
	\begin{aligned}
		&\eta^{1-p}\lambda_P^{2-p}
		\left|
		\Jp(Q\cdot r\omega)-\Jp((Q+b)\cdot r\omega)
		\right|\\
		&\qquad\le
		C\eta^{1-p}|b\cdot r\omega|^{p-1}
		\le C\eta^{1-p}r^{p-1}.
	\end{aligned}
	\]
	If \(|P|>1\), then
	\[
	Q\cdot r\omega
	=(Q+b)\cdot r\omega-b\cdot r\omega.
	\]
	Lemma~\ref{lem:large-slope-scalar-angular} therefore yields,
	for almost every \(\omega\),
	\[
	\begin{aligned}
		\left|
		\Jp(Q\cdot r\omega)-\Jp((Q+b)\cdot r\omega)
		\right|
		&\le
		C\bigl(
		|(Q+b)\cdot r\omega|+|b\cdot r\omega|
		\bigr)^{p-2}|b\cdot r\omega|\\
		&\le
		C|Q+b|^{p-2}|b|r^{p-1}|e\cdot\omega|^{p-2}.
	\end{aligned}
	\]
	Here \(e=P/|P|\).
	Since \(Q+b=\eta P\),
	\[
	\eta^{1-p}\lambda_P^{2-p}|Q+b|^{p-2}
	=
	\eta^{-1}(1+|P|)^{2-p}|P|^{p-2}
	\le C\eta^{-1}.
	\]
	Using the integrability of
	\(|e\cdot\omega|^{p-2}\) on \(\mathbb S^{n-1}\),
	the estimates in the two cases give
	\[
	\begin{aligned}
		\eta^{1-p}\lambda_P^{2-p}
		\int_{\mathbb S^{n-1}}
		\left|
		\Jp(Q\cdot r\omega)-\Jp((Q+b)\cdot r\omega)
		\right|\dd\omega\le
		Cr^{p-1}\bigl(\eta^{1-p}+\eta^{-1}\bigr).
	\end{aligned}
	\]
	
	Since \(|W|\ge2\) and \(|Z|\le3\rho/2\), we have
	\(|Z-W|\ge1\).
	Polar coordinates centered at \(Z\) consequently give
	\[
	\begin{aligned}
		&\rho^\gamma\eta^{1-p}\lambda_P^{2-p}
		\int_{\R^n\setminus B_2}
		\frac{
			|\Jp(Q\cdot(Z-W))-\Jp((Q+b)\cdot(Z-W))|
		}{|Z-W|^{n+sp}}\dd W\\
		&\qquad\le
		C\rho^\gamma\bigl(\eta^{1-p}+\eta^{-1}\bigr)
		\int_1^\infty r^{p-2-sp}\dd r\\
		&\qquad\le
		C\left[
		\rho^{\gamma-(p-1)\beta}
		+\rho^{\gamma-\beta}
		\right].
	\end{aligned}
	\]

	The estimates in all three regions are uniform in the
	centers of both tail integrals.
	Taking the essential suprema and using the initial
	choice of \(C_0\), we conclude that
	\[
	\begin{aligned}
		\mathcal T^{\rm pt}_{s,P}(w)
		\le{}& C_0\bigl[
		\rho^{(\alpha_1-\beta)(p-1)}
		+\rho^{\gamma-(p-1)\beta}
		+\rho^{\alpha_1-\beta}
		+\rho^{\gamma-\beta}
		\bigr]
		\bigl[1+\mathcal T^{\rm pt}_{s,Q}(v)\bigr].
	\end{aligned}
	\]
	All four exponents are positive.
	Since \(\beta\le\bar\beta\) and \(0<\rho<1\), each power
	of \(\rho\) is bounded above by its counterpart with
	\(\beta\) replaced by \(\bar\beta\).
	The initial tail bound and
	\eqref{eq:affine-route-tail-small-main} therefore give
	\[
	\mathcal T^{\rm pt}_{s,P}(w)\le\frac12\le1.
	\]
	Combining this with
	\(\mathscr C_{q_{\rm ext},\beta}(w)\le G_*\), we obtain
	\[
	\mathcal T_{\beta,G_*}(w,P)\le1.
	\]
	Together with Step~1, this proves
	\eqref{eq:affine-route-unit-reproduction-main}.
\end{proof}

\subsection{Iteration and proof of the main theorem}
\label{subsec:unified-iteration-main}

\begin{proposition}
	\label{prop:unified-variable-block-recursion}
	Let $\rho,\beta,\eta,G_*$ and $R_*$ be the constants fixed in
	Subsection~\ref{subsec:parameter-choice-state}, with
	$\eta=\rho^\beta$. Let $x_0\in\R^n$ and $r_0,H_0>0$.
	Suppose that $u$ is a weak solution of \eqref{eq:equation}
	in $B_{2R_*r_0}(x_0)$.
	
	Let $\ell_0(x)=a_0+q_0\cdot(x-x_0)$, and define
	\[
	v_0(X)=
	\frac{u(x_0+r_0X)-\ell_0(x_0+r_0X)}
	{r_0H_0},
	\qquad
	Q_0=\frac{q_0}{H_0}.
	\]
	Assume that
	\[
	\mathbf A_{s,r_0}(u,\ell_0;x_0)\le H_0,
	\qquad
	\mathcal T_{\beta,G_*}(v_0,Q_0)\le1.
	\]
	
	For $k\ge0$, set $r_k=\rho^k r_0$ and $
	H_k=\eta^k H_0$. Then there exists a sequence of affine functions
	$\ell_k(x)=a_k+q_k\cdot(x-x_0)$, starting from
	$\ell_0$, such that the normalized functions
	\[
	v_k(X)=
	\frac{u(x_0+r_kX)-\ell_k(x_0+r_kX)}
	{r_kH_k},
	\qquad
	Q_k=\frac{q_k}{H_k},
	\]
	satisfy, for every $k\ge0$,
	\begin{align}
		\mathbf A_{s,r_k}(u,\ell_k;x_0)
		&\le H_k,
		\label{eq:unified-physical-amplitude}\\
		|a_{k+1}-a_k|+r_k|q_{k+1}-q_k|
		&\le C_\sharp r_kH_k,
		\label{eq:unified-affine-increments}\\
		\mathcal T_{\beta,G_*}(v_k,Q_k)
		&\le1.
		\label{eq:unified-tail-invariant}
	\end{align}
	Moreover,
	\begin{equation}
		H_k=H_0r_0^{-\beta}r_k^\beta.
		\label{eq:unified-amplitude-decay}
	\end{equation}
	The constant $C_\sharp$ depends only on the fixed
	parameters and is independent of $k$, $r_0$,
	$H_0$, and the normalized slopes $Q_k$.
\end{proposition}

\begin{proof}
	By \eqref{eq:common-intrinsic-profile} and the
	definition of the intrinsic amplitude,
	\[
	\mathbf A_{s,r_0}(u,\ell_0;x_0)\le H_0
	\quad\Longleftrightarrow\quad
	\mathfrak a_{s,Q_0}(v_0,0;1)\le1.
	\]
	Thus the initial assumptions give
	\[
	\mathfrak a_{s,Q_0}(v_0,0;1)\le1,
	\qquad
	\mathcal T_{\beta,G_*}(v_0,Q_0)\le1.
	\]
	
	We proceed by induction. Suppose that $\ell_k$
	has been constructed and that
	\[
	\mathfrak a_{s,Q_k}(v_k,0;1)\le1,
	\qquad
	\mathcal T_{\beta,G_*}(v_k,Q_k)\le1.
	\]
	Since
	\[
	Q_k\cdot X+v_k(X)
	=\frac{u(x_0+r_kX)-a_k}{r_kH_k},
	\]
	the function $v_k$ solves the equation shifted
	by $Q_k$ in $B_{2R_*}$. We may therefore apply
	Lemma~\ref{lem:affine-invariant-state-reproduction-main}
	to obtain an affine function
	$m_k(X)=c_k+b_k\cdot X$ satisfying
	\[
	|c_k|+|b_k|\le C_\sharp.
	\]
	
	Since $m_k$ approximates the normalized remainder
	$v_k$, we define the next affine approximation
	of $u$ by
	\[
	\ell_{k+1}(x)
	:=\ell_k(x)+r_kH_k
	m_k\left(\frac{x-x_0}{r_k}\right).
	\]
	Consequently,
	\[
	a_{k+1}=a_k+r_kH_kc_k,
	\qquad
	q_{k+1}=q_k+H_kb_k.
	\]
	It follows that
	\[
	\begin{aligned}
		&|a_{k+1}-a_k|
		+r_k|q_{k+1}-q_k|
		=r_kH_k(|c_k|+|b_k|)
		\le C_\sharp r_kH_k,
	\end{aligned}
	\]
	which proves \eqref{eq:unified-affine-increments}.
	
	We next verify that the estimates are preserved
	at level $k+1$. Since $r_{k+1}=\rho r_k$
	and $H_{k+1}=\eta H_k$, the definitions give
	\[
	v_{k+1}(X)
	=\frac{v_k(\rho X)-m_k(\rho X)}
	{\rho\eta},
	\qquad
	Q_{k+1}=\frac{Q_k+b_k}{\eta}.
	\]
	These are precisely the rescaled function
	and slope in
	Lemma~\ref{lem:affine-invariant-state-reproduction-main}.
	Hence
	\[
	\mathfrak a_{s,Q_{k+1}}(v_{k+1},0;1)\le1,
	\qquad
	\mathcal T_{\beta,G_*}(v_{k+1},Q_{k+1})\le1.
	\]
	By the scaling relation for the intrinsic amplitude,
	\[
	\begin{aligned}
		&\mathbf A_{s,r_{k+1}}
		(u,\ell_{k+1};x_0)
		=H_{k+1}
		\mathfrak a_{s,Q_{k+1}}(v_{k+1},0;1)
		\le H_{k+1}.
	\end{aligned}
	\]
	This proves \eqref{eq:unified-physical-amplitude}
	and \eqref{eq:unified-tail-invariant}
	at level $k+1$, and completes the induction.
	
	Finally, since $\eta=\rho^\beta$, we have $
	H_k=\eta^kH_0
	=H_0r_0^{-\beta}r_k^\beta$,	which proves \eqref{eq:unified-amplitude-decay}.
\end{proof}

\begin{corollary}
	\label{cor:common-recursion-campanato}
	Under the hypotheses of
	Proposition~\ref{prop:unified-variable-block-recursion},
	the coefficients of the affine functions $\ell_k$
	satisfy
	\[
	a_k\longrightarrow u(x_0),
	\qquad q_k\longrightarrow q(x_0)
	\]
	for some $q(x_0)\in\R^n$. Moreover,
	\begin{equation}
		\left(
		\fint_{B_r(x_0)}
		|u(x)-u(x_0)-q(x_0)\cdot(x-x_0)|^p
		\,\dd x
		\right)^{1/p}
		\le C H_0r_0^{-\beta}r^{1+\beta}
		\label{eq:common-recursion-campanato-decay}
	\end{equation}
	for every $0<r\le r_0$, where $u$ is taken
	in its continuous representative.
	
	If the hypotheses of
	Proposition~\ref{prop:unified-variable-block-recursion}
	hold at every center of an open set $\Omega$,
	with a common initial radius and uniformly
	bounded initial amplitudes on each compact
	subset, then
	\[
	\nabla u\in C^\beta_{\mathrm{loc}}(\Omega).
	\]
\end{corollary}

\begin{proof}
	By \eqref{eq:unified-affine-increments},
	\[
	|q_{k+1}-q_k|\le C_\sharp H_k,
	\qquad
	|a_{k+1}-a_k|\le C_\sharp r_kH_k.
	\]
	Since $r_k=\rho^kr_0$ and $H_k=\eta^kH_0$,
	both series
	\[
	\sum_{k=0}^\infty H_k
	=\frac{H_0}{1-\eta},
	\qquad
	\sum_{k=0}^\infty r_kH_k
	=\frac{r_0H_0}{1-\rho\eta}
	\]
	are finite. Thus $a_k$ and $q_k$ converge
	to limits $a_\infty$ and $q(x_0)$,
	respectively. Moreover, we obtain 
	\[|q_k-q(x_0)|
		\le C_\sharp\sum_{i=k}^\infty H_i
		\le C H_k \quad \text{and} \quad 
		|a_k-a_\infty|
		\le C_\sharp\sum_{i=k}^\infty r_iH_i
		\le C r_kH_k.
	\]
	
	We next identify $a_\infty$ and estimate the
	affine approximation error. By
	\eqref{eq:unified-physical-amplitude}
	and the definition of the intrinsic amplitude,
	\[
	\left(
	\fint_{B_{2r_k}(x_0)}
	|u-\ell_k|^p\,\dd x
	\right)^{1/p}
	\le r_kH_k.
	\]
	Consequently,
	\[
	\left(
	\fint_{B_{r_k}(x_0)}
	|u-\ell_k|^p\,\dd x
	\right)^{1/p}
	\le C r_kH_k.
	\]
	Since $\ell_k(x)=a_k+q_k\cdot(x-x_0)$,
	its average over $B_{r_k}(x_0)$ equals $a_k$.
	Hence
	\[
	\begin{aligned}
		\left|
		a_k-\fint_{B_{r_k}(x_0)}u\,\dd x
		\right|
		\le
		\left(
		\fint_{B_{r_k}(x_0)}
		|u-\ell_k|^p\,\dd x
		\right)^{1/p}\le C r_kH_k\longrightarrow0.
	\end{aligned}
	\]
	By continuity of $u$ at $x_0$, we know 
	\[
	\fint_{B_{r_k}(x_0)}u\,\dd x
	\longrightarrow u(x_0).
	\]
	Therefore $a_\infty=u(x_0)$, and
	\[
	|a_k-u(x_0)|\le C r_kH_k.
	\]
	
	Using the triangle inequality, we obtain
	\[
	\begin{aligned}
		&\left(
		\fint_{B_{r_k}(x_0)}
		|u(x)-u(x_0)-q(x_0)\cdot(x-x_0)|^p
		\,\dd x
		\right)^{1/p}\\
		&\quad\le
		\left(
		\fint_{B_{r_k}(x_0)}
		|u-\ell_k|^p\,\dd x
		\right)^{1/p}
		+|a_k-u(x_0)|
		+r_k|q_k-q(x_0)|\le C r_kH_k
		=C H_0r_0^{-\beta}r_k^{1+\beta}.
	\end{aligned}
	\]
	For any $0<r\le r_0$, choose $k\ge0$
	such that $r_{k+1}<r\le r_k$.
	Comparing the averages over $B_r(x_0)$
	and $B_{r_k}(x_0)$ gives
	\[
	\begin{aligned}
		&\left(
		\fint_{B_r(x_0)}
		|u(x)-u(x_0)-q(x_0)\cdot(x-x_0)|^p
		\,\dd x
		\right)^{1/p}\\
		&\quad\le
		C\left(\frac{r_k}{r}\right)^{n/p}
		H_0r_0^{-\beta}r_k^{1+\beta}\le C H_0r_0^{-\beta}r^{1+\beta},
	\end{aligned}
	\]
	where we used $r_k/r<\rho^{-1}$.
	This proves \eqref{eq:common-recursion-campanato-decay}.
	
	Finally, suppose that the hypotheses hold
	uniformly at every center of $\Omega$.
	For $x\in\Omega$, define
	\[
	L_x(y)=u(x)+q(x)\cdot(y-x).
	\]
	Fix two sufficiently close points
	$x_0,x_1\in\Omega$, and set  $
	d=|x_0-x_1|>0$, $z=\frac{x_0+x_1}{2}$.
	We may assume that $4d\le r_0$ and
	$B_{4d}(x_i)\subset\Omega$ for $i=0,1$.
	Since
	\[
	B_d(z)\subset
	B_{4d}(x_0)\cap B_{4d}(x_1),
	\]
	the triangle inequality and
	\eqref{eq:common-recursion-campanato-decay}
	yield
	\[
	\begin{aligned}
		\left(
		\fint_{B_d(z)}
		|L_{x_0}-L_{x_1}|^p\,\dd y
		\right)^{1/p}\le C\sum_{i=0}^1
		\left(
		\fint_{B_{4d}(x_i)}
		|u-L_{x_i}|^p\,\dd y
		\right)^{1/p}\le Cd^{1+\beta}.
	\end{aligned}
	\]
	
	For every affine function $L$, the equivalence
	of norms on affine functions gives
	\[
	|L(z)|+d|\nabla L|
	\le C\left(
	\fint_{B_d(z)}|L|^p\,\dd y
	\right)^{1/p}.
	\]
	Since $|x_1-z|=d/2$ and
	$L(x_1)=L(z)+\nabla L\cdot(x_1-z)$,
	we obtain
	\[
	|L(x_1)|+d|\nabla L|
	\le C\left(
	\fint_{B_d(z)}|L|^p\,\dd y
	\right)^{1/p}.
	\]
	Applying this estimate to
	$L=L_{x_0}-L_{x_1}$, we obtain
	\[
	\begin{aligned}
		|L_{x_0}(x_1)-L_{x_1}(x_1)|+d|q(x_0)-q(x_1)|
		\le Cd^{1+\beta}.
	\end{aligned}
	\]
	As $L_{x_1}(x_1)=u(x_1)$, it follows that
	\[
	\begin{aligned}
		|u(x_1)-u(x_0)
		-q(x_0)\cdot(x_1-x_0)|
		\le C|x_1-x_0|^{1+\beta}\quad \text{and} \quad 
		|q(x_0)-q(x_1)|
		\le C|x_0-x_1|^\beta.
	\end{aligned}
	\]
	Dividing the first inequality by
	$|x_1-x_0|$ and letting $x_1\to x_0$,
	we conclude that $u$ is differentiable
	at $x_0$ with $\nabla u(x_0)=q(x_0)$.
	The second inequality then gives
	$\nabla u\in C^\beta_{\mathrm{loc}}(\Omega)$.
\end{proof}

	We now prove Theorem~\ref{thm:main}.

\begin{proof}[Proof of Theorem~\ref{thm:main}]
	Set
	\[
	M:=\frac12\osc_{\R^n}u,
	\qquad
	\kappa_0:=\frac12
	\left(
	\operatorname*{ess\,sup}_{\R^n}u
	+\operatorname*{ess\,inf}_{\R^n}u
	\right).
	\]
	If $M=0$, then $u$ is constant and the conclusion
	is immediate. Suppose that $M>0$, and define
	\[
	U(X):=\frac{u(x_0+2RX)-\kappa_0}{M}.
	\]
	By the homogeneity of the equation,
	\[
	(-\Delta_p)^s U=0
	\quad\text{in }B_2,
	\qquad
	\|U\|_{L^\infty(\R^n)}=1.
	\]
	
	By \cite[Theorem~5]{KorvenpaaKuusiPalatucci2016}
	and \cite[Proposition~1.5]{BiswasTopp2025},
	$U$ is a continuous viscosity solution in $B_2$.
	Since $U$ is globally bounded and $sp>p-1$,
	the local Lipschitz estimate in
	\cite[Theorem~2.1]{BiswasTopp2025} gives
	\begin{equation}
		\operatorname{Lip}(U;B_{3/2})
		\le L_0=L_0(n,p,s).
		\label{eq:main-initial-lipschitz}
	\end{equation}
	
	We now verify the initial assumptions of
	Proposition~\ref{prop:unified-variable-block-recursion}.
	Let $\rho,\beta,\eta,G_*$ and $R_*$ be the
	fixed parameters from
	Subsection~\ref{subsec:parameter-choice-state}.
	We shall choose
	\[
	H_0=C_1(1+L_0),
	\]
	with $C_1\ge1$ sufficiently large, and then fix
	a sufficiently small $r_0>0$ such that
	\[
	2R_*r_0<1,
	\qquad r_0<\frac1{16}.
	\]
	In particular,
	\[
	B_{2R_*r_0}(z)\subset B_{3/2}
	\qquad(z\in B_{1/2}).
	\]
	
	Fix $z\in B_{1/2}$ and take
	\[
	\ell_0\equiv U(z),
	\qquad
	v_0(X)=
	\frac{U(z+r_0X)-U(z)}{r_0H_0},
	\qquad Q_0=0.
	\]
	By \eqref{eq:main-initial-lipschitz},
	\[
	\mathbf Z_{r_0}(U,\ell_0;z)
	=
	\fint_{B_{2r_0}(z)}
	\left|\frac{U-U(z)}{r_0}\right|^p
	\,\dd x
	\le CL_0^p.
	\]
	Since $D_\Phi(0;t)=|t|^p/p$, the same
	Lipschitz estimate gives
	\[
	\begin{aligned}
		\mathbf E_{s,r_0}(U,\ell_0;z)
		&=\frac{1-s}{p}r_0^{sp-p-n}
		\iint_{B_{2r_0}(z)^2}
		\frac{|U(x)-U(y)|^p}
		{|x-y|^{n+sp}}
		\,\dd x\dd y\le CL_0^p.
	\end{aligned}
	\]
	As $\nabla\ell_0=0$, choosing $C_1$
	sufficiently large yields
	\[
	\mathbf A_{s,r_0}(U,\ell_0;z)
	\le H_0.
	\]
	
	We next estimate
	$\mathcal T_{\beta,G_*}(v_0,0)$.
	First consider the Campanato term.
	Comparison with the zero affine function gives
	\[
	\mathbf c_{q_{\rm ext}}(v_0;j)
	\le
	\begin{cases}
		CL_0/H_0,
		& B_{2^{j+1}r_0}(z)\subset B_{3/2},\\
		C/(2^jr_0H_0),
		& \text{otherwise}.
	\end{cases}
	\]
	In the second case,
	$2^{j+1}r_0\ge1$, since $z\in B_{1/2}$.
	Consequently,
	\[
	\mathscr C_{q_{\rm ext},\beta}(v_0)
	\le\frac{C(1+L_0)}{H_0}.
	\]
	
	To estimate the pointwise tail, recall that
	\[
	|\mathcal D_0v_0(X,Y)|
	=|v_0(X)-v_0(Y)|^{p-1}.
	\]
	For the two exterior integrals defining
	$\mathcal T^{\rm pt}_{s,0}(v_0)$,
	the interior point satisfies $|X|\le3/2$,
	and the distance $|X-Y|$ is bounded below
	by a positive constant.
	
	When $|Y|\le(2r_0)^{-1}$, both
	$z+r_0X$ and $z+r_0Y$ belong to $B_{3/2}$.
	Hence
	\[
	|v_0(X)-v_0(Y)|
	\le\frac{L_0}{H_0}|X-Y|.
	\]
	Writing $\gamma=sp-p+1>0$ and using
	polar coordinates, the corresponding exterior
	contribution is bounded by
	\[
	C\left(\frac{L_0}{H_0}\right)^{p-1}
	\int_{1/12}^{\infty}t^{-1-\gamma}\,\dd t
	\le C\left(\frac{L_0}{H_0}\right)^{p-1}.
	\]
	
	For $|Y|>(2r_0)^{-1}$, the global bound
	$\|U\|_{L^\infty(\R^n)}=1$ gives
	\[
	|v_0(X)-v_0(Y)|
	\le\frac{2}{r_0H_0}.
	\]
	Since $|X-Y|\ge c|Y|$ in this region,
	the remaining contribution is bounded by
	\[
	\begin{aligned}
		C(r_0H_0)^{1-p}
		\int_{1/(2r_0)}^\infty
		t^{-1-sp}\,\dd t
		&\le C H_0^{1-p}r_0^\gamma.
	\end{aligned}
	\]
	Combining these estimates yields
	\[
	\mathcal T^{\rm pt}_{s,0}(v_0)
	\le
	C\left(\frac{L_0}{H_0}\right)^{p-1}
	+CH_0^{1-p}r_0^\gamma.
	\]
	
	We first choose $C_1$ sufficiently large
	and then decrease $r_0$, if necessary.
	The preceding estimates give
	\[
	\mathscr C_{q_{\rm ext},\beta}(v_0)\le G_*,
	\qquad
	\mathcal T^{\rm pt}_{s,0}(v_0)\le1.
	\]
	Thus
	\[
	\mathcal T_{\beta,G_*}(v_0,0)\le1.
	\]
	The constants $H_0$ and $r_0$ are
	independent of the center $z\in B_{1/2}$.
	
	We may now apply
	Proposition~\ref{prop:unified-variable-block-recursion}
	and Corollary~\ref{cor:common-recursion-campanato}
	at every $z\in B_{1/2}$.
	Since the initial radius and amplitude
	are uniform in $z$, these results give
	\[
	\nabla U\in C^\beta_{\rm loc}(B_{1/2})
	\]
	with uniform local estimates. 
	By Corollary~\ref{cor:common-recursion-campanato},
	the uniform affine approximation estimates give
	\[
	|\nabla U(z_0)-\nabla U(z_1)|
	\le C|z_0-z_1|^\beta
	\]
	whenever $z_0,z_1\in B_{1/2}$ and
	$|z_0-z_1|\le r_0/4$.
	
	If $|z_0-z_1|>r_0/4$, then
	\eqref{eq:main-initial-lipschitz} gives
	\[
	\begin{aligned}
		|\nabla U(z_0)-\nabla U(z_1)|
		&\le 2L_0\le 2L_0(4/r_0)^\beta
		|z_0-z_1|^\beta.
	\end{aligned}
	\]
	Since $L_0$ and $r_0$ depend only on $n,p,s$,
	we conclude that
	\[
	\|\nabla U\|_{L^\infty(B_{1/2})}
	+[\nabla U]_{C^\beta(B_{1/2})}
	\le C(n,p,s).
	\]

	Finally, returning to the original variables, we obtain
	\[
		\|\nabla_X U\|_{L^\infty(B_{1/2})}
		=\frac{2R}{M}
		\|\nabla u\|_{L^\infty(B_R(x_0))}\quad \text{and}\quad 
		[\nabla_X U]_{C^\beta(B_{1/2})}
		=\frac{(2R)^{1+\beta}}{M}
		[\nabla u]_{C^\beta(B_R(x_0))}.
	\]
	This proves
	\eqref{eq:fixed-order-gradient-holder-estimate}
	with $\alpha_0=\beta$.
\end{proof}

\addtocontents{toc}{\protect\setcounter{tocdepth}{0}}
\section*{Statements and Declarations}

\subsection*{Funding}
This work was supported by the National Natural Science Foundation of
China (No.~12471128).

\subsection*{Competing interests}
The author declares that there are no competing interests.

\subsection*{Data availability}
No datasets were generated or analyzed in this work.

\subsection*{AI disclosure} During the preparation of this manuscript, the author used ChatGPT as an auxiliary tool for computations, editorial assistance, and preliminary exploration of proof ideas. All computations, proofs, references, and conclusions were independently verified by the author, who takes full responsibility for the content of this paper.


\begin{thebibliography}{99}
	\raggedbottom
	\setlength{\itemsep}{2pt}
	
	\bibitem{Adams1973}
	D.~R. Adams, \emph{A trace inequality for generalized potentials}, Studia Math.
	\textbf{48} (1973), 99--105.
	
	\bibitem{BiswasSen2025}
	A.~Biswas and A.~Sen, \emph{Improved H\"older regularity of fractional
		$(p,q)$-Poisson equation with regular data}, Preprint, arXiv:2507.09920,
	2025.
	
	\bibitem{BiswasTopp2025}
	A.~Biswas and E.~Topp, \emph{Lipschitz regularity of fractional $p$-Laplacian},
	Ann. PDE \textbf{11} (2025), Paper No.~27, 43~pp. 
	
	\bibitem{BogeleinDuzaarLiaoMolicaBisciServadei2025a}
	V.~B\"ogelein, F.~Duzaar, N.~Liao, G.~Molica Bisci and R.~Servadei,
	\emph{Gradient regularity for $(s,p)$-harmonic functions}, Calc. Var. Partial
	Differential Equations \textbf{64} (2025), Paper No.~253, 63~pp.
	
	\bibitem{BogeleinDuzaarLiaoMolicaBisciServadei2025b}
	V.~B\"ogelein, F.~Duzaar, N.~Liao, G.~Molica Bisci and R.~Servadei,
	\emph{Regularity for the fractional $p$-Laplace equation}, J. Funct. Anal.
	\textbf{289} (2025), Paper No.~111078, 69~pp.
	
	\bibitem{BogeleinDuzaarLiaoMolicaBisciServadei2024}
	V.~B\"ogelein, F.~Duzaar, N.~Liao, G.~Molica Bisci and R.~Servadei,
	\emph{Higher regularity theory for $(s,p)$-harmonic functions}, Rend. Lincei
	Mat. Appl. \textbf{35} (2024), 311--321.
	
	\bibitem{BogeleinDuzaarLiaoMoring2025}
	V.~B\"ogelein, F.~Duzaar, N.~Liao and K.~Moring, \emph{Gradient estimates for
		the fractional $p$-Poisson equation}, J. Math. Pures Appl. (9) \textbf{204}
	(2025), Paper No.~103764, 25~pp.
	
	\bibitem{BogeleinDuzaarLiaoMoring2026}
	V.~B\"ogelein, F.~Duzaar, N.~Liao and K.~Moring, \emph{Sharp gradient
		integrability for $(s,p)$-Poisson type equations}, Preprint,
	arXiv:2602.08944, 2026.
	
	\bibitem{BrascoLindgren2017}
	L.~Brasco and E.~Lindgren, \emph{Higher Sobolev regularity for the fractional
		$p$-Laplace equation in the superquadratic case}, Adv. Math. \textbf{304}
	(2017), 300--354.
	
	\bibitem{BrascoLindgrenSchikorra2018}
	L.~Brasco, E.~Lindgren and A.~Schikorra, \emph{Higher H\"older regularity for
		the fractional $p$-Laplacian in the superquadratic case}, Adv. Math.
	\textbf{338} (2018), 782--846.
	
	\bibitem{DeFilippisMingione2024}
	C.~De Filippis and G.~Mingione, \emph{Gradient regularity in mixed local and
		nonlocal problems}, Math. Ann. \textbf{388} (2024), 261--328.
	
	\bibitem{DiCastroKuusiPalatucci2016}
	A.~Di Castro, T.~Kuusi and G.~Palatucci, \emph{Local behavior of fractional
		$p$-minimizers}, Ann. Inst. H. Poincar\'e C Anal. Non Lin\'eaire \textbf{33}
	(2016), 1279--1299.
	
	\bibitem{DiCastroKuusiPalatucci2014}
	A.~Di Castro, T.~Kuusi and G.~Palatucci, \emph{Nonlocal Harnack inequalities},
	J. Funct. Anal. \textbf{267} (2014), 1807--1836.
	
	\bibitem{DieningKimLeeNowak2025a}
	L.~Diening, K.~Kim, H.-S.~Lee and S.~Nowak, \emph{Higher differentiability for
		the fractional $p$-Laplacian}, Math. Ann. \textbf{391} (2025), 5631--5693.
	
	\bibitem{DieningKimLeeNowak2025b}
	L.~Diening, K.~Kim, H.-S.~Lee and S.~Nowak, \emph{Nonlinear nonlocal potential
		theory at the gradient level}, J. Eur. Math. Soc. (2025), published online
	first, \href{https://doi.org/10.4171/JEMS/1706}{doi:10.4171/JEMS/1706}.
	
	\bibitem{DieningNowak2025}
	L.~Diening and S.~Nowak, \emph{Calder\'on--Zygmund estimates for the fractional
		$p$-Laplacian}, Ann. PDE \textbf{11} (2025), Paper No.~6, 33~pp.
	
	\bibitem{GarainLindgren2024}
	P.~Garain and E.~Lindgren, \emph{Higher H\"older regularity for the fractional
		$p$-Laplace equation in the subquadratic case}, Math. Ann. \textbf{390}
	(2024), 5753--5792.
	
	\bibitem{GiovagnoliJesusSilvestre2025}
	D.~Giovagnoli, D.~Jesus and L.~Silvestre, \emph{$C^{1+\alpha}$ regularity for
		fractional $p$-harmonic functions}, Preprint, arXiv:2509.26565, 2025.
	
	\bibitem{IannizzottoMosconi2024}
	A.~Iannizzotto and S.~Mosconi, \emph{Fine boundary regularity for the singular
		fractional $p$-Laplacian}, J. Differential Equations \textbf{412} (2024),
	322--379.
	
	\bibitem{IannizzottoMosconiSquassina2016}
	A.~Iannizzotto, S.~Mosconi and M.~Squassina, \emph{Global H\"older regularity
		for the fractional $p$-Laplacian}, Rev. Mat. Iberoam. \textbf{32} (2016),
	1353--1392.
	
	
	
	\bibitem{KorvenpaaKuusiLindgren2019}
	J.~Korvenp\"a\"a, T.~Kuusi and E.~Lindgren, \emph{Equivalence of solutions to
		fractional $p$-Laplace type equations}, J. Math. Pures Appl. (9) \textbf{132}
	(2019), 1--26.
	
	\bibitem{KorvenpaaKuusiPalatucci2016}
	J.~Korvenp\"a\"a, T.~Kuusi and G.~Palatucci, \emph{Fractional superharmonic
		functions and the Perron method for nonlinear integro-differential
		equations}, Math. Ann. \textbf{369} (2017), 1443--1489.
	
	\bibitem{KuusiMingioneSire2015a}
	T.~Kuusi, G.~Mingione and Y.~Sire, \emph{Nonlocal equations with measure data},
	Comm. Math. Phys. \textbf{337} (2015), 1317--1368.
	
	\bibitem{KuusiMingioneSire2015b}
	T.~Kuusi, G.~Mingione and Y.~Sire, \emph{Nonlocal self-improving properties},
	Anal. PDE \textbf{8} (2015), 57--114.
	
	\bibitem{KuusiMingioneSire2018}
	T.~Kuusi, G.~Mingione and Y.~Sire, \emph{Regularity issues involving the
		fractional $p$-Laplacian}, in: \emph{Recent Developments in Nonlocal Theory},
	De Gruyter, Berlin, 2018, 303--334.
	
	\bibitem{KuusiNowakSire2022}
	T.~Kuusi, S.~Nowak and Y.~Sire, \emph{Gradient regularity and first-order
		potential estimates for a class of nonlocal equations}, to appear in Amer. J.
	Math., arXiv:2212.01950.
	
	\bibitem{Lindgren2016}
	E.~Lindgren, \emph{H\"older estimates for viscosity solutions of equations of
		fractional $p$-Laplace type}, NoDEA Nonlinear Differential Equations Appl.
	\textbf{23} (2016), Paper No.~55, 18~pp.
	
	\bibitem{RosOtonSerra2016}
	X.~Ros-Oton and J.~Serra, \emph{Regularity theory for general stable
		operators}, J. Differential Equations \textbf{260} (2016), 8675--8715.
	
\end{thebibliography}
\end{document}